\documentclass[11pt]{amsart}
\usepackage{amsmath,amssymb,amsthm,mathtools}
\newif\ifmachinereadable
\machinereadabletrue
\ifmachinereadable
  \usepackage[accsupp]{axessibility}
  \input{glyphtounicode}
\fi
\usepackage{enumitem}
\usepackage[margin=1.2in]{geometry}
\usepackage{array}
\allowdisplaybreaks[2]
\usepackage{booktabs}
\usepackage[protrusion=true,expansion=false]{microtype}
\usepackage{hyperref}
\hypersetup{hidelinks,pdfauthor={Hamed Ebadi},pdftitle={Moments and Large Geometric Monodromy for Linearly Constrained Kloosterman Families}}
\numberwithin{equation}{section}
\theoremstyle{plain}
\newtheorem{theorem}[equation]{Theorem}
\newtheorem{proposition}[equation]{Proposition}
\newtheorem{lemma}[equation]{Lemma}
\newtheorem{corollary}[equation]{Corollary}
\newtheorem*{theoremA}{Theorem A}
\newtheorem*{theoremB}{Theorem B}
\newtheorem*{theoremC}{Theorem C}
\newtheorem*{corollaryD}{Corollary D}
\newtheorem*{corollaryE}{Corollary E}
\newtheorem*{theoremF}{Theorem F}
\newtheorem*{theoremG}{Theorem G}
\theoremstyle{definition}

\theoremstyle{remark}
\newtheorem{remark}[equation]{Remark}
\newcommand{\Qell}{\mathbb{Q}_\ell}
\newcommand{\Qlb}{\bar{\mathbb{Q}}_\ell}
\newcommand{\Fq}{\mathbb{F}_q}
\newcommand{\Fp}{\mathbb{F}_p}
\newcommand{\Fpbar}{\bar{\mathbb{F}}_p}
\newcommand{\Fqbar}{\bar{\mathbb{F}}_q}
\newcommand{\Gm}[1][]{\mathbb{G}_{m#1}}
\newcommand{\Frob}{\mathrm{Frob}}
\newcommand{\Tr}{\mathrm{Tr}}
\newcommand{\Swan}{\mathrm{Swan}}
\newcommand{\ggeom}{G_{\mathrm{geom}}}
\newcommand{\pgeom}{\pi_1^{\mathrm{geom}}}
\newcommand{\Lpsi}[1]{\mathcal{L}_\psi(#1)}
\newcommand{\Kl}{\mathcal{K}\ell}
\DeclareMathOperator{\SL}{SL}
\DeclareMathOperator{\GL}{GL}
\DeclareMathOperator{\Sp}{Sp}
\DeclareMathOperator{\Sym}{Sym}
\DeclareMathOperator{\Hom}{Hom}
\DeclareMathOperator{\End}{End}
\DeclareMathOperator{\rank}{rank}
\title[Moments and large monodromy]{Moments and Large Geometric Monodromy for Linearly Constrained Kloosterman Families}
\author{Hamed Ebadi}
\address{K. N. Toosi University of Technology, Tehran, Iran}
\email{h.ebadi1@email.kntu.ac.ir}
\date{October 2026}
\begin{document}
\begin{abstract}
Let \(p\) be a prime, \(\psi\) a nontrivial additive character of \(\mathbb F_p\), \(\lambda\in\mathbb F_p^\times\), and \(\psi_q=\psi\circ\mathrm{Tr}_{\mathbb F_q/\mathbb F_p}\). For \(h\in\mathbb F_q^n\) we study the linearly constrained Kloosterman sums
\[
\Sigma_n(h)=\sum_{\substack{u\in(\mathbb F_q^\times)^n\\ h_1u_1+\cdots+h_nu_n=-\lambda}}\psi_q\bigl(u_1^{-1}+\cdots+u_n^{-1}\bigr).
\]
We construct a lisse sheaf \(\mathcal W_n\) on \(\mathbb G_m^n\), pure of weight \(n+1\) and of rank \(r(n)=2^n-\binom{n}{\lfloor n/2\rfloor}\), whose Frobenius traces are \((-1)^{n-1}q\,\Sigma_n(h)\) minus an explicit polynomial in \(q\), and we study its monodromy for every \(n\ge2\) in every characteristic. For odd \(p\) the geometric monodromy group has identity component \(\mathrm{SL}_{r(n)}\) and a cyclic group of components of order divisible by \(p\); it is exactly \(\{g:\det(g)^p=1\}\) for \(n=2\), and for \(n\ge3\) given the unipotence of the local monodromy at the boundary proved in the companion preprint. For \(p=2\) the sheaf is self-dual and the group is exactly \(\mathrm{SO}_{r(n)}\) or \(\mathrm{Sp}_{r(n)}\), according as \(n\) is odd or even, and the arithmetic and geometric groups coincide after a Tate twist. This yields a Sato--Tate law with rate \(q^{-1/2}\) for polynomial test functions, all the power moments of the sums, which are Gaussian up to order \(r(n)-1\) and, for odd \(n\), exceed the orthogonal value by one at order \(r(n)\), and, when \(n+1\) is a power of \(2\), a Frobenius eigenvalue \(q^{(n+1)/2}\) common to all \(h\). The inputs are the fourth moment \(\sum_h|\Sigma_n(h)|^4=c_pq^{3n-2}+O(q^{3n-2-\delta(n)})\) for \(n\ge3\), with \(c_p=2\) for odd \(p\) and \(c_2=3\), computed through plane sections of the Cayley cubic surface; two exact identities which determine the Frobenius--Schur indicator; Laumon's local Fourier transform; and an Artin--Schreier argument on \(\mathbb G_m^n\). For \(n=2\) the sums are twisted Kloosterman sums, and their fourth moment is computed exactly in every characteristic.
\end{abstract}
\maketitle
\section{Introduction}

\subsection{The family and the results}
Throughout, \(p\) is a prime, \(\ell\ne p\) is a prime, \(\psi\) is a nontrivial additive character of \(\Fp\) with values in \(\Qlb^\times\), and \(\lambda\in\Fp^\times\); if \(p=2\), then necessarily \(\lambda=1\) and \(\psi(x)=(-1)^x\). For a finite extension \(\Fq/\Fp\) we write \(\psi_q=\psi\circ\Tr_{\Fq/\Fp}\). For \(n\ge1\) and \(h\in\Fq^n\) we put
\begin{equation}\label{eq:family}
\Sigma_n(h)=\Sigma_{n,q}(h):=\sum_{\substack{u\in(\Fq^\times)^n\\ \langle h,u\rangle=-\lambda}}\psi_q\bigl(u_1^{-1}+\cdots+u_n^{-1}\bigr),
\qquad \langle h,u\rangle=\sum_{i=1}^nh_iu_i ,
\end{equation}
and
\[
r(n):=2^n-\binom{n}{\lfloor n/2\rfloor},\qquad \nu_{n,k}:=\binom nk-\binom n{k-1}\quad\Bigl(0\le k\le\Bigl\lfloor\frac n2\Bigr\rfloor,\ \binom n{-1}:=0\Bigr),
\]
so that \(r(2)=2\), \(r(3)=5\), \(r(4)=10\), \(r(5)=22\), and \(\sum_k\nu_{n,k}=\binom n{\lfloor n/2\rfloor}\). For \(h\in(\Fq^\times)^n\) the sum \eqref{eq:family} is an exponential sum over the \((n-1)\)-dimensional variety
\[
Y_h:=\{u\in\Gm^n:\langle h,u\rangle=-\lambda\},
\]
and its natural size is \(q^{(n-1)/2}\). Let \(S_n:=\Gm^n\) over \(\Fp\), with coordinates \(h=(h_1,\dots,h_n)\).

\begin{theoremA}[Theorem~\ref{thm:package}, Proposition~\ref{prop:lower}, Corollary~\ref{cor:exact}]
For every \(n\ge1\) there is a lisse \(\Qlb\)-sheaf \(\mathcal W_n\) on \(S_n\), pure of weight \(n+1\) and of rank \(r(n)\), such that for every finite extension \(\Fq/\Fp\) and every \(h\in S_n(\Fq)\)
\[
(-1)^{n-1}q\,\Sigma_n(h)=\Tr\bigl(\Frob_h\mid\mathcal W_n\bigr)+\sum_{k=1}^{\lfloor n/2\rfloor}\nu_{n,k}\,q^k .
\]
The fibre of \(\mathcal W_n\) at \(h\) is the top-weight quotient of \(H^{n-1}_c\bigl(Y_h\otimes\Fqbar,\mathcal L_{\psi}(\sum_iu_i^{-1})\bigr)(-1)\). The space \(H^{n-1}_c\bigl(Y_h\otimes\Fqbar,\mathcal L_{\psi}(\sum_iu_i^{-1})\bigr)\) has dimension \(2^n-1\), and its Frobenius eigenvalues other than those on \(\mathcal W_{n,h}(1)\) are the numbers \(q^{k-1}\), \(1\le k\le\lfloor n/2\rfloor\), with multiplicity \(\nu_{n,k}\). In particular
\[
|\Sigma_n(h)|\le r(n)\,q^{(n-1)/2}+\sum_{k=1}^{\lfloor n/2\rfloor}\nu_{n,k}\,q^{k-1}\le(2^n-1)\,q^{(n-1)/2}.
\]
\end{theoremA}

\begin{theoremB}[Theorem~\ref{thm:fourth}]
Let \(n\ge3\), and put \(c_p=2\), \(b_p(n)=2^n-4n-7\) if \(p\) is odd, and \(c_2=3\), \(b_2(n)=2^n-6n-9\). For every finite field \(\Fq\) of characteristic \(p\),
\[
\sum_{h\in\Fq^n}|\Sigma_n(h)|^4=c_p\,q^{3n-2}+O_n\bigl(q^{3n-2-\delta(n)}\bigr),
\qquad\delta(n)=\min\Bigl(1,\frac{n-2}2\Bigr),
\]
and, if \(n\ge5\),
\[
\sum_{h\in\Fq^n}|\Sigma_n(h)|^4=c_p\,q^{3n-2}+b_p(n)\,q^{3n-3}+O_n\bigl(q^{3n-3-\delta'(n)}\bigr),
\qquad\delta'(n)=\min\Bigl(1,\frac{n-4}2\Bigr).
\]
Both formulas remain valid, with the same main terms, when \(h\) is restricted to \((\Fq^\times)^n\). The implied constants depend only on \(n\).
\end{theoremB}

For \(n=2\) the bound of the generic stratum is of the order of the main term, and Theorem~B is replaced by an exact formula. The sums \(\Sigma_2(h)\) are twisted Kloosterman sums: for \(h\in(\Fq^\times)^2\),
\[
\Sigma_2(h)=\psi_q\Bigl(-\frac{h_1+h_2}\lambda\Bigr)\,K_q\Bigl(\frac{h_1h_2}{\lambda^2}\Bigr)-1,\qquad K_q(a)=\sum_{x\in\Fq^\times}\psi_q\bigl(x^{-1}+ax\bigr)
\]
(Proposition~\ref{prop:n2-kl}), and consequently
\[
\sum_{h\in\Fq^2}|\Sigma_2(h)|^4=
\begin{cases}
2q^4-9q^3+q^2+22q+14,&p\ \text{odd},\\[2pt]
2q^4-6q^3-8q^2+18q+14,&p=2
\end{cases}
\]
(Corollary~\ref{cor:n2-exact}). For \(p=2\) the constant is thus \(2\) rather than \(3\). Proposition~\ref{prop:n2-fourth} explains this through the generic stratum: in characteristic two the plane sections of the Cayley cubic with \(t_1+t_2+t_3+t_4=0\) are cuspidal (Lemma~\ref{lem:cusp}), and the remaining sections have mean square trace \(q+O(1)\), by an incidence count.

\begin{theoremC}[Theorems~\ref{thm:main}, \ref{thm:det} and~\ref{thm:full2}]
Let \(n\ge2\), let \(W_n\) be the geometric generic fibre of \(\mathcal W_n\), and let \(G=\ggeom(\mathcal W_n)\).
\begin{enumerate}[label=\textup{(\roman*)}]
\item If \(p\) is odd, then \(G^0=\SL_{r(n)}\). More precisely, the geometric determinant of \(\mathcal W_n\) has finite order \(d\), divisible by \(p\), and \(G=\{g\in\GL(W_n):\det(g)^d=1\}\).
\item If \(p=2\), then \(W_n\) carries a nondegenerate \(G\)-invariant bilinear form \(B\), symmetric if \(n\) is odd and alternating if \(n\) is even, and
\[
G=\mathrm{SO}(W_n,B)\cong\mathrm{SO}_{r(n)}\quad(n\ \text{odd}),\qquad G=\Sp(W_n,B)\cong\Sp_{r(n)}\quad(n\ \text{even}).
\]
Moreover \(\det(\Frob_h\mid\mathcal W_n)=q^{(n+1)r(n)/2}\) for every \(h\in S_n(\Fq)\), and the Tate twist \(\mathcal W_n\bigl(\frac{n+1}2\bigr)\), which is pure of weight \(0\) \textup{(}for even \(n\) it depends on the choice of a square root of \(2\)\textup{)}, has arithmetic monodromy group equal to \(G\).
\end{enumerate}
\end{theoremC}

In particular, for \(n=4\) the sheaf \(\mathcal W_4\) has rank \(10\), and \(\ggeom(\mathcal W_4)^0\) is \(\SL_{10}\) if \(p\) is odd and \(\ggeom(\mathcal W_4)=\Sp_{10}\) if \(p=2\); for \(n=3\) it has rank \(5\), and \(\ggeom(\mathcal W_3)^0\) is \(\SL_5\) if \(p\) is odd and \(\ggeom(\mathcal W_3)=\mathrm{SO}_5\) if \(p=2\). For odd \(p\), the companion preprint shows that the local monodromy of \(\mathcal W_n\) along the coordinate lines is unipotent at the origin \cite[Theorem~6.8(ii)]{EbadiCoh}; for \(n=2\) this is proved here (Theorem~\ref{thm:summary}(iv); see also Proposition~\ref{prop:n2-kl}). With this single input, Theorem~\ref{thm:summary} gives \(d=p\), \(\det(\mathcal W_n)\cong\Lpsi{-2^{n-1}(h_1+\cdots+h_n)/\lambda}\otimes\Qlb\bigl(-\frac{(n+1)r(n)}2\bigr)\), and \(G_{\mathrm{arith}}=\ggeom=\{g:\det(g)^p=1\}\) after a Tate twist, conclusions also obtained in \cite{EbadiCoh}; so \(\ggeom(\mathcal W_n)\) is known for all \(p\) and all \(n\ge2\). Section~\ref{subsec:companion} explains which results are shared with \cite{EbadiCoh}.

\begin{corollaryD}[Corollary~\ref{cor:structure}]
For \(n\ge2\) the sheaf \(\mathcal W_n\) is geometrically irreducible and primitive. If \(p\) is odd and \(n\ge3\), there is no lisse rank-one sheaf \(\mathcal L\) on \(S_n\) with \(\mathcal W_n\cong\mathcal W_n^\vee\otimes\mathcal L\) geometrically. If \(p=2\), then \(\mathcal W_n\) is geometrically self-dual, and \(B\) is, up to a scalar, its only invariant bilinear form.
\end{corollaryD}

\begin{corollaryE}[Corollary~\ref{cor:moments} and Proposition~\ref{prop:constants}]
Let \(n\ge2\) and \(k\ge1\), let \(q\) run over the powers of \(p\), and put \(N=r(n)\).
\begin{enumerate}[label=\textup{(\roman*)}]
\item If \(p\) is odd or \(n=2\), then
\[
\sum_{h\in\Fq^n}|\Sigma_n(h)|^{2k}=u_N(k)\,q^{n+k(n-1)}\bigl(1+O(q^{-1/2})\bigr),
\]
where \(u_N(k)\) is the number of permutations of \(\{1,\dots,k\}\) without an increasing subsequence of length \(N+1\). Thus \(u_N(k)=k!\) for \(k\le N\), \(u_N(N+1)=(N+1)!-1\), and \(u_N(k)<k!\) for \(k>N\).
\item If \(p=2\) and \(n\ge3\), then \(\Sigma_n(h)\in\mathbb Z\), and
\[
\sum_{h\in\Fq^n}\Sigma_n(h)^{2k}=(2k-1)!!\;q^{n+k(n-1)}\bigl(1+O(q^{-1/2})\bigr)
\]
for \(2k<N\), and also for \(2k\le N\) if \(n\) is even and for \(2k\le2N\) if \(N\) is odd; if \(n\) is even, the constant for \(2k=N+2\) is \((N+1)!!-1\).
\end{enumerate}
\end{corollaryE}

Thus, for \(n\ge3\), the even absolute moments of the normalised sums \(q^{-(n-1)/2}\Sigma_n(h)\), \(h\in\Fq^n\), agree with those of a standard complex Gaussian variable up to order \(2r(n)\) when \(p\) is odd, and fall strictly below them from order \(2r(n)+2\) on; when \(p=2\) they agree with those of a standard real Gaussian variable up to order \(r(n)-1\), and up to order \(r(n)+1\) if \(n\) is even. The self-duality in characteristic two, and its absence in odd characteristic, can also be read off from two exact identities.

\begin{theoremF}[Proposition~\ref{prop:exact-moments} and Theorem~\ref{thm:FS}]
For every \(n\ge1\) and every finite field \(\Fq\) of characteristic \(p\),
\[
\sum_{h\in\Fq^n}\Sigma_n(h)^2=
\begin{cases}
q^{n-2}\bigl((q-1)^{n+1}+1-(q-2)(-1)^n\bigr),&p=2,\\[2pt]
q^{n-2}\bigl(1-(q-1)^n+2(-1)^n\bigr),&p\ \text{odd},
\end{cases}
\]
and, writing \(\Sigma_{n,q^2}(h)\) for the sum \eqref{eq:family} formed over \(\mathbb F_{q^2}\),
\[
\sum_{h\in(\Fq^\times)^n}\Sigma_{n,q^2}(h)=\bigl(q\cdot\mathbf 1_{p=2}-1\bigr)\,\frac{(q^2-q+1)^n-(1-q)^n}{q^2}.
\]
Consequently, if \(p=2\), the \(\ggeom(\mathcal W_n)\)-invariant bilinear forms on \(W_n\) form a line, spanned by a nondegenerate form which is symmetric if \(n\) is odd and alternating if \(n\) is even, and Frobenius at \(h\in S_n(\Fq)\) multiplies it by \(q^{n+1}\); if \(p\) is odd, there is no nonzero invariant bilinear form.
\end{theoremF}

Theorem~F gives a proof of the self-duality in Theorem~C(ii) that does not use cup products, and it shows the dichotomy already at the level of second moments: \(\sum_h\Sigma_n(h)^2\) is of order \(q^{2n-1}\) when \(p=2\) and of order \(q^{2n-2}\) when \(p\) is odd. The complete determination of the groups in characteristic two, including the arithmetic monodromy group, has consequences for the distribution of the sums.

\begin{theoremG}[Theorem~\ref{thm:ST}, Corollaries~\ref{cor:ST-rate}, \ref{cor:signature} and~\ref{cor:forced}, Proposition~\ref{prop:constants}]
Let \(p=2\) and \(n\ge2\), and let \(K=\mathrm{SO}(r(n))\) if \(n\) is odd and \(K=\mathrm{USp}(r(n))\) if \(n\) is even, with standard representation \(V\).
\begin{enumerate}[label=\textup{(\roman*)}]
\item As \(q\to\infty\) over the powers of \(2\), the real numbers
\[
x_h=q^{-(n-1)/2}\Bigl((-1)^{n-1}\Sigma_n(h)-\sum_{k=1}^{\lfloor n/2\rfloor}\nu_{n,k}\,q^{k-1}\Bigr),\qquad h\in(\Fq^\times)^n,
\]
become equidistributed with respect to the image under the trace of the Haar measure of \(K\). For \(n=2\), \(x_h=-q^{-1/2}(\Sigma_2(h)+1)\), and the limiting distribution is the Sato--Tate law \(\frac1{2\pi}\sqrt{4-x^2}\,dx\) on \([-2,2]\). For every polynomial \(f\), the average of \(f(x_h)\) over \(h\in(\Fq^\times)^n\) differs from its limit by \(O_{n,f}(q^{-1/2})\).
\item For every \(m\ge1\), \(\sum_{h\in\Fq^n}\Sigma_n(h)^m=P_m\,q^{n+m(n-1)/2}+O\bigl(q^{n+m(n-1)/2-1/2}\bigr)\), where \(P_m=\dim(V^{\otimes m})^K\). If \(n\) is odd, then \(P_{r(n)}\) exceeds by one the corresponding dimension for the full orthogonal group \(\mathrm O(r(n))\), the excess coming from the volume form.
\item If \(n+1\) is a power of \(2\), then for every \(h\in S_n(\Fq)\) the number \(q^{(n+1)/2}\) is an eigenvalue of \(\Frob_h\) on the fibre of \(\mathcal W_n\) at \(h\).
\end{enumerate}
If \(p\) is odd, then in contrast \(\sum_{h\in\Fq^n}\Sigma_n(h)^m=O\bigl(q^{n+m(n-1)/2-1/2}\bigr)\) for \(1\le m<p\,r(n)\); and, with the input from \cite[Theorem~6.8(ii)]{EbadiCoh} described after Theorem~C, which is not needed for \(n=2\), \(\sum_{h\in\Fq^n}\Sigma_n(h)^{p\,r(n)}=f_{p,r(n)}\,q^{n+p\,r(n)(n-1)/2}\bigl(1+O(q^{-1/2})\bigr)\), where \(f_{p,N}=(pN)!\prod_{i=0}^{p-1}i!/(N+i)!\) is the number of standard Young tableaux of a rectangle with \(N\) rows and \(p\) columns.
\end{theoremG}

So the passage from \(\SL\) in odd characteristic to \(\mathrm{SO}\) or \(\Sp\) in characteristic two is visible in the statistics of the sums themselves: in odd characteristic the normalised sums are genuinely complex and all their power moments of order less than \(p\,r(n)\) are negligible, whereas in characteristic two they are integers, Sato--Tate distributed with respect to a compact orthogonal or symplectic group, their moment of order \(r(n)\) registers the determinant, and for \(n=2^k-1\), \(k\ge2\), every Frobenius has a fixed eigenvalue.

\subsection{Strategy}
The proof of Theorem~C has five ingredients.

\emph{The fourth moment.} By Katz's moment principle (Theorem~\ref{thm:moment-principle}), Theorems~A and~B for \(h\in(\Fq^\times)^n\) give \(M_4(\ggeom(\mathcal W_n),W_n)=c_p\) for \(n\ge3\). Additive orthogonality turns the fourth moment into a sum over the four-dimensional space of dilation parameters \(t\in\Fq^4\) of \(\psi_q(L(t))V(t)^n\), where \(V(t)=\sum_{h\in\Fq}\prod_{j=1}^4K(t_jh)\) is built from the classical Kloosterman sum \(K\) and does not depend on \(n\) (Lemma~\ref{lem:decouple}); the dimension \(n\) survives only as an exponent. On the diagonal and two-pair strata \(V\) is evaluated exactly, and on the degenerate stratum it is bounded by counting points on lines and conics. On the generic stratum \(V(t)\) is \(-q^2(a_t+2+3/q)\), where \(a_t\) is the trace of Frobenius of a plane section of the Cayley cubic surface \(\{e_3=0\}\). The nine lines on that surface, which are the same in every characteristic, show that these sections are absolutely irreducible plane cubics, hence either elliptic curves, to which Hasse's bound \(|a_t|\le2\sqrt q\) applies, or singular cubics, for which \(|a_t|\le1\). The cancellation of \(\psi_q(L(t))\) along projective rays supplies the remaining saving. The phase \(L\) vanishes on two of the three pairings of the four factors when \(p\) is odd, and on all three when \(p=2\); this is the origin of the dichotomy \(c_p\in\{2,3\}\).

\emph{Self-duality in characteristic two.} The rank-two Kloosterman sheaf is symplectically self-dual with values in \(\det\Kl\cong\Qlb(-1)\), and \(\Lpsi{\lambda t}^{\otimes2}\cong\Lpsi{2\lambda t}\) is trivial when \(p=2\). The cup product on \(R^1\rho_!\) then gives a pairing on \(\mathcal H_n\) with values in \(\Qlb(-n-1)\) whose radical is exactly the lower-weight part \(\mathcal E_n\), and \(\mathcal W_n\) becomes self-dual, orthogonally or symplectically according to the parity of \(n\) (Proposition~\ref{prop:selfdual}). Independently, the exact identities of Theorem~F, combined with an Adams-operation form of the moment principle, determine the dimensions of the spaces of invariant symmetric and alternating forms (Theorem~\ref{thm:FS}).

\emph{A unipotent element.} Along a ray \(x\mapsto(xa_1,\dots,xa_n)\) the family is a Fourier transform of a tensor product of dilated Kloosterman sheaves. Laumon's stationary phase identifies its local monodromy at \(x=0\) with the \(n\)-th tensor power of a unipotent Jordan block of size two. If \(g_0\) denotes the image of a generator of tame inertia, then \(g_0-1\) has rank \(r(n)\), which exceeds the rank \(\binom{n}{\lfloor n/2\rfloor}\) of the lower-weight part when \(n\ge3\); so a nonidentity unipotent element survives on \(\mathcal W_n\) (Corollary~\ref{cor:unipotent}). For \(n=2\) the rank count fails, and it is replaced by a weight argument (Proposition~\ref{prop:n2}): the lower-weight part is determined exactly, Frobenius acting on \(\mathcal E_{n,h}\) with the eigenvalues \(q^k\) of multiplicity \(\nu_{n,k}\) (Proposition~\ref{prop:lower}), and the relation \(\Phi\mathfrak N\Phi^{-1}=q^{-1}\mathfrak N\) between Frobenius and the logarithm of tame monodromy shows that a trivial action on \(\mathcal W_2\) would force a Frobenius eigenvalue of weight \(4\) at the end of the ray; alternatively, for \(n=2\), Proposition~\ref{prop:n2-kl} identifies \(\mathcal W_2\) with a twisted pull-back of \(\Kl\). The same determination of \(\mathcal E_n\) turns the inequality of Theorem~A into an identity and shows that \(\mathcal E_n\) is geometrically unipotent (Corollary~\ref{cor:E-structure}). The arithmetic form of the stationary phase, calibrated by the case \(n=1\), also determines the Frobenius at the end of a ray: its eigenvalues on \(\mathcal W_n\) are the numbers \(q^k\), \(1\le k\le n\), with explicit multiplicities, and their product is \(q^{(n+1)r(n)/2}\) (Proposition~\ref{prop:ray-frob}).

\emph{The determinant.} Along a coordinate line \(h_1\mapsto(h_1,c_2,\dots,c_n)\) the family is again a Fourier transform, now of a sheaf built from the \((n-1)\)-dimensional family along a ray, which has a second singular point, at \(-1/\lambda\). Stationary phase at this point shows that the local monodromy at \(h_1=\infty\) is the sum of a tame unipotent part of dimension \(2^{n-1}\) and of \(\Lpsi{-h_1/\lambda}\) tensored with a representation of slopes \(<1\) and dimension \(2^{n-1}\), and an Euler characteristic count shows that the local monodromy at \(h_1=0\) is tame (Theorem~\ref{thm:coord}). Hence the determinant along the line is \(\Lpsi{-2^{n-1}h_1/\lambda}\) up to a tame character. When \(p=2\) the factor \(2^{n-1}\) kills the Artin--Schreier part, a character of order at most two which is tame at \(0\) and \(\infty\) is trivial, and an Artin--Schreier argument on \(\Gm^n\) (Lemma~\ref{lem:AS}) shows that the determinant is geometrically trivial; the Frobenius at the end of a ray then fixes the arithmetic normalisation (Theorem~\ref{thm:det}). When \(p\) is odd the same computation shows that \(p\) divides \(d\).

\emph{Group theory.} \(\ggeom(\mathcal W_n)\) is reductive by Deligne's semisimplicity theorem, and Deligne's theorem on rank-one sheaves makes the determinant finite. For odd \(p\) and \(n\ge3\), Larsen's alternative turns \(M_4=2\) and a nonidentity unipotent element into \(\ggeom\supseteq\SL(W_n)\); for \(p=2\) and \(n\ge3\), its orthogonal and symplectic forms turn \(M_4=3\), the invariant form, and a nonidentity unipotent element into \(\ggeom^0=\mathrm{SO}(W_n,B)\) or \(\ggeom=\Sp(W_n,B)\); for \(n=2\), a connected reductive subgroup of \(\SL_2\) containing a nonidentity unipotent element is \(\SL_2\). The triviality of the determinant then gives \(\ggeom=\mathrm{SO}(W_n,B)\) for odd \(n\) when \(p=2\), and the Frobenius compatibility of \(B\) together with the arithmetic determinant gives \(G_{\mathrm{arith}}=\ggeom\) (Theorem~\ref{thm:full2}). Theorem~G follows from Deligne's equidistribution theorem and from invariant theory. For odd \(p\) the same determinant computation, combined with the unipotence at the origin of the local monodromy along coordinate lines proved in \cite[Theorem~6.8(ii)]{EbadiCoh}, gives the full arithmetic monodromy group (Theorem~\ref{thm:summary}); this is the only point at which \cite{EbadiCoh} is used, and only for \(n\ge3\), the case \(n=2\) being settled here (Theorem~\ref{thm:summary}(iv), by Proposition~\ref{prop:n2-kl} or by a weight argument).

\subsection{Foundational inputs}
The argument uses the following general theorems as black boxes. Apart from the coordinate-line unipotence input identified in Section~\ref{subsec:companion}, item~(a), the deductions specific to the family \eqref{eq:family} are proved in this paper. That additional input is used only for the conclusions listed there, concerning the full monodromy groups in odd characteristic and their consequences; the case \(n=2\) is proved directly here.
\begin{enumerate}[label=(F\arabic*)]
\item The formalism of \(\ell\)-adic cohomology \cite{SGA45,SGA5}: proper base change, the projection and K\"unneth formulas, the Grothendieck--Lefschetz trace formula, the Grothendieck--Ogg--Shafarevich formula, Poincar\'e duality and Artin vanishing; the K\"unneth formula for \(Rf_!\) \cite[Exp.~XVII, 5.4.3]{SGA4}, from which the cup product on \(Rf_!\) and its compatibility with base change follow; the perversity of the extension by zero of a perverse sheaf along an affine open immersion \cite[Corollaire 4.1.3]{BBD}; and, for Artin--Schreier and Kummer theory on \(\Gm^n\), the comparison of the \'etale and Zariski cohomology of quasi-coherent sheaves \cite[Ch.~III, Proposition 3.7]{Milne}.
\item Deligne's Weil~II \cite{Deligne}: the weight bound for \(Rf_!\) \cite[Th\'eor\`eme 3.3.1]{Deligne}, the local monodromy-weight theorem for pure lisse sheaves on curves \cite[1.8.4]{Deligne}, which bounds the weights of local invariants by the weight \(w\) of the sheaf and shows that they are pure of weight \(w\) when the logarithm of the unipotent part of the local monodromy vanishes, the purity of \(H^1(C,j_*\mathcal F)\) on a proper curve \cite[Th\'eor\`eme 3.2.3]{Deligne}, the existence of the weight filtration on a lisse mixed sheaf \cite[Th\'eor\`eme 3.4.1]{Deligne}, the finiteness of the geometric monodromy of lisse rank-one sheaves on normal geometrically connected schemes of finite type over a finite field \cite[Proposition 1.3.4(i)]{Deligne}, and, for Theorem~G(i) only, Deligne's equidistribution theorem \cite[Th\'eor\`eme 3.5.3]{Deligne} in the form \cite[Theorem 9.2.6]{KS}.
\item Deligne's semicontinuity theorem for Swan conductors \cite[Th\'eor\`eme 2.1.1(i) and Corollaire 2.1.2]{Laumon-semicont}, which is stated there for sheaves of \(\mathbb F_\ell\)-vector spaces; the passage to \(\Qlb\)-coefficients is carried out in the proof of Theorem~\ref{thm:package}.
\item The basic properties of the rank-two Kloosterman sheaf, due to Deligne \cite{SGA45} and recorded in Katz \cite{Katz-Gauss}; see Theorem~\ref{thm:kl}. The break decomposition of representations of inertia groups \cite[Ch.~1]{Katz-Gauss}.
\item Laumon's Fourier transform and its perversity, his local Fourier transforms and his stationary-phase principle, over an algebraically closed field and over a finite field \cite[(1.2.1.1), (1.3.2.3), (1.4.1), (2.3.3.1), (2.4.1), (2.4.2.1)--(2.4.2.3), (2.4.3)]{Laumon}, in the form stated in Theorem~\ref{thm:laumon} and Lemma~\ref{lem:kl-origin}; see also \cite[Theorem 3]{KatzB}.
\item Hasse's bound for elliptic curves over finite fields \cite[Theorem V.1.1]{Silverman}.
\item Schur--Weyl duality \cite[\S4.2.4]{GW}, \cite[Theorem 6.3]{FH}, the first fundamental theorem of invariant theory for the orthogonal and symplectic groups in tensor form \cite[\S5.3.2]{GW}, Rains' combinatorial description of the moments of the trace on the unitary, orthogonal and symplectic groups \cite[Theorems 1.1 and 3.4]{Rains}, and the hook length formula \cite[\S4.1]{FH}; these are used only for Corollary~E, Theorem~G(ii) and Proposition~\ref{prop:constants}.
\end{enumerate}
Katz's moment principle, Deligne's geometric semisimplicity theorem for pure lisse sheaves and Larsen's alternative, in its linear, orthogonal and symplectic forms, are stated and proved in the generality needed (Section~\ref{sec:monodromy}).

\subsection{Relation to the companion preprint and to the longer manuscript}\label{subsec:companion}
This paper and the companion preprint \cite{EbadiCoh} grew out of the author's longer manuscript \cite{EbadiFull}. The manuscript \cite{EbadiFull} treats odd \(p\) with \(p>\max(2,n)\), and excludes \(p=2\); for this range it contains the identity component \(\SL_{r(n)}\), obtained from the fourth moment and Larsen's alternative, together with the Guralnick--Tiep classification or a multiplicative convolution, the full geometric group \(\mu_{r(n)p}\cdot\SL_{r(n)}\), and analytic applications such as short-sum estimates. The companion preprint \cite{EbadiCoh} treats every odd \(p\) through the cohomology of the varieties \(Y_h\). It determines the weights and the lower-weight eigenvalues, the local monodromy along rays and along coordinate lines, including its unipotence at the origin of a coordinate line, the geometric determinant, the geometric and arithmetic monodromy groups, and the corresponding equidistribution theorem; its identity component \(\SL_{r(n)}\) is again obtained from a fourth moment, computed through plane sections of the Cayley cubic, and Larsen's alternative. The present paper treats every characteristic. Its relation to these two texts is as follows.
\begin{enumerate}[label=(\alph*)]
\item \emph{What is used.} A single result, for odd \(p\) and \(n\ge3\): the local monodromy of \(\mathcal W_n\) along every coordinate line is unipotent at the origin \cite[Theorem~6.8(ii)]{EbadiCoh} (for \(p>\max(2,n)\) also \cite{EbadiFull}); for \(n=2\) it is proved here (Theorem~\ref{thm:summary}(iv), Proposition~\ref{prop:n2-kl}). It enters only Theorem~\ref{thm:summary}(iii), Corollary~\ref{cor:ST-rate}(ii), Proposition~\ref{prop:constants}(iv), the last assertion of Theorem~G and the remarks referring to them, and there it is stated as a hypothesis. Every other statement of this paper is proved here completely.
\item \emph{What is proved again.} Theorem~A with the exact lower eigenvalues of Proposition~\ref{prop:lower}, the main term of Theorem~B for odd \(p\), the identity component and the divisibility \(p\mid d\) in Theorem~C(i), Corollary~E(i) for \(k\le r(n)\), the nonexistence of invariant forms for odd \(p\) in Theorem~F, the vanishing of the power moments of order less than \(p\,r(n)\) for odd \(p\) (Corollary~\ref{cor:signature}(ii)), the structure at infinity along coordinate lines (Theorem~\ref{thm:coord}) and the Frobenius at the end of a ray (Proposition~\ref{prop:ray-frob}) are established here by arguments that are uniform in \(p\) and independent of \cite{EbadiCoh} and \cite{EbadiFull}; the results in characteristic two rest on them. For odd \(p\) these statements overlap with results of \cite{EbadiCoh} or \cite{EbadiFull}, in some cases in a different form, and they are not claimed as new; the arithmetic normalisation of Proposition~\ref{prop:ray-frob} is obtained by a method different from that of \cite{EbadiCoh}, which computes the corresponding Tate twist from vanishing cycles.
\item \emph{What is specific to this paper.} In characteristic two: the self-duality of \(\mathcal W_n\) and its two proofs (Proposition~\ref{prop:selfdual} and Theorem~\ref{thm:FS}), the fourth moment with constant \(3\) (Theorem~B), the groups \(\mathrm{SO}_{r(n)}\) and \(\Sp_{r(n)}\) (Theorem~C(ii)), the arithmetic determinant and the arithmetic monodromy group (Theorems~\ref{thm:det} and~\ref{thm:full2}), and the Sato--Tate law, the power moments and the forced eigenvalue (Theorem~G). In every characteristic: the explicit second-order term of Theorem~B, the identity of Proposition~\ref{prop:exact-moments}(ii), the description of the moment constants beyond the Gaussian range (Proposition~\ref{prop:constants}), the exact fourth moment for \(n=2\) (Corollary~\ref{cor:n2-exact}), and the computation of the fourth moment for \(n=2\) in characteristic two through the generic stratum, with its comparison with the monodromy (Lemma~\ref{lem:cusp}, Proposition~\ref{prop:n2-fourth} and Remark~\ref{rem:cusp}).
\end{enumerate}
The texts \cite{EbadiCoh} and \cite{EbadiFull} are preprints which have not been refereed. The present paper has its own numbering, and its proofs refer to them only through item~(a). The reference to Theorem~6.8(ii) of \cite{EbadiCoh} uses the numbering of its third version, arXiv:2610.03081v3, where this is the coordinate-boundary theorem. The proofs in \cite{EbadiCoh} do not use the present paper as an input, so the dependence described in item~(a) is not circular.

\subsection{Relation to other work}
The method of determining a geometric monodromy group by combining a moment computation with Larsen's alternative is that of Katz \cite{Katz-Larsen}; it is developed systematically, for large families of exponential sums, in \cite{KatzMMP}, and the equidistribution statements to which such results lead through Deligne's equidistribution theorem are discussed in \cite{KS}. We do not rely on the results of \cite{KatzMMP}: the fourth moment of the family \eqref{eq:family} is computed directly in Section~\ref{sec:fourth}, and the unipotent elements are produced by local Fourier analysis and weights in Section~\ref{subsec:ray}. The function \(V\) of Section~\ref{sec:fourth} is a sum of products of dilated Kloosterman sums, a type of sum studied in general, over prime fields, by Fouvry, Kowalski and Michel \cite{FKM}; here the bound needed on the generic stratum is obtained by an elementary argument, through plane sections of the Cayley cubic surface, with explicit constants. For Kloosterman sheaves and their local monodromy we refer to \cite{Katz-Gauss} and \cite{ESDE}. The dichotomy between odd \(p\) and \(p=2\) in Theorem~C parallels Katz's theorem on the Kloosterman sheaves \(\mathrm{Kl}_n\) of odd rank \(n\), whose geometric monodromy group is \(\SL_n\) for odd \(p\) but \(\mathrm{SO}_n\) (or \(G_2\), when \(n=7\)) for \(p=2\) \cite[Theorem 11.1]{Katz-Gauss}; here the self-duality in characteristic two comes from the factor \(\Lpsi{\lambda t}\) carrying the linear constraint, and the triviality of the determinant from the vanishing of \(2^{n-1}\) in \(\mathbb F_2\). For \(n=2\) the family is classical: by Proposition~\ref{prop:n2-kl}, \(\mathcal W_2\) is the pull-back of \(\Kl\) by \((h_1,h_2)\mapsto h_1h_2/\lambda^2\), twisted by \(\Lpsi{-(h_1+h_2)/\lambda}\otimes\Qlb(-1)\), so that for \(n=2\) the monodromy statements of this paper can also be derived from the results of Deligne and Katz on \(\Kl\) \cite{SGA45}, \cite[Theorem~11.1]{Katz-Gauss}, and Theorem~G(i) from Katz's equidistribution theorem for Kloosterman sums (Remark~\ref{rem:n2-kl}). The case \(n=2\) is nevertheless treated within the general framework, and the comparison checks the normalisations.

The arguments below need no hypothesis on \(p\) beyond \(p\ne\ell\), and Theorems~A--G hold as stated in every characteristic, with the single dependence on \cite[Theorem~6.8(ii)]{EbadiCoh} described in Section~\ref{subsec:companion}. In the proofs, the distinction between \(p=2\) and odd \(p\) enters in Section~\ref{sec:fourth}, through the evaluation of the Kloosterman sums of Lemma~\ref{lem:kloost}(iii),(iv) and of the strata \(D\) and \(Q\); and in Section~\ref{sec:monodromy}, through the triviality of \(\Lpsi{\lambda t}^{\otimes2}\) in Proposition~\ref{prop:selfdual}, through the constant \(q\cdot\mathbf 1_{p=2}-1\) in Proposition~\ref{prop:exact-moments}, through the resulting choice between the three forms of Larsen's alternative, and through the vanishing of \(2^{n-1}\) in \(\mathbb F_2\) in Theorem~\ref{thm:det}. The determination of \(\ggeom(\mathcal W_n)^0\) for odd \(p\) could alternatively be organised around the Guralnick--Tiep classification of groups with \(M_4=2\) \cite{GT}, which requires excluding finite exceptional groups dimension by dimension; the unipotent elements of Corollary~\ref{cor:unipotent} and Proposition~\ref{prop:n2} make that unnecessary.

\subsection{Conventions}
For a lisse sheaf \(\mathcal F\) on a scheme \(X\) of finite type over \(\Fq\) and \(x\in X(\Fq)\) we write \(\Tr(\Frob_x\mid\mathcal F)\) for the trace of the geometric Frobenius. We fix a field isomorphism \(\iota:\Qlb\to\mathbb C\); all weights are \(\iota\)-weights. Every sheaf considered below is \(\iota\)-mixed for every \(\iota\), and all statements about weights are independent of \(\iota\). For a point \(x\) of a smooth curve over \(\Fpbar\), \(I_x\) denotes the inertia group, \(P_x\subset I_x\) its wild inertia subgroup, and \(I_x^{(r)}\) (\(r>0\)) the upper-numbering ramification groups; breaks and Swan conductors are as in \cite[Ch.~1]{Katz-Gauss}. For \(c\in\Fpbar^\times\) we write \([\times c]\) for the automorphism \(x\mapsto cx\) of \(\Gm\), and \(\mathcal L_\psi(f)\) for the Artin--Schreier sheaf attached to a function \(f\), whose trace function is \(\psi_q(f)\). A unipotent Jordan block of size \(m\) is denoted \(J_m\). For a nondegenerate bilinear form \(B\) on a vector space \(W\) we write \(\mathrm O(W,B)\), \(\mathrm{SO}(W,B)\) and \(\Sp(W,B)\) for its isometry groups, according as \(B\) is symmetric or alternating.
\section{The top-weight sheaf}\label{sec:sheaf}

\subsection{The Kloosterman sheaf}
For a finite extension \(\Fq/\Fp\) and \(a\in\Fq\) put
\[
K(a)=K_q(a):=\sum_{x\in\Fq^\times}\psi_q\bigl(x^{-1}+ax\bigr).
\]
Thus \(K(0)=\sum_{x\ne0}\psi_q(x^{-1})=-1\), and for \(a\ne0\), \(K(a)\) is the classical Kloosterman sum \(\sum_{yz=a}\psi_q(y+z)\) (put \(y=x^{-1}\), \(z=ax\)). The substitution \(x\mapsto-x\) shows that \(K(a)\) equals its complex conjugate, so \(K\) is real-valued.

\begin{theorem}[Deligne]\label{thm:kl}
Let \(\mathrm{pr}:\Gm[,x]\times\Gm[,a]\to\Gm[,a]\) be the projection. Then \(R^i\mathrm{pr}_!\Lpsi{x^{-1}+ax}=0\) for \(i\ne1\), and \(\Kl:=R^1\mathrm{pr}_!\Lpsi{x^{-1}+ax}\) is a lisse \(\Qlb\)-sheaf on \(\Gm\) over \(\Fp\) with the following properties.
\begin{enumerate}[label=\textup{(\roman*)}]
\item \(\Kl\) has rank \(2\), is pure of weight \(1\), and \(\Tr(\Frob_a\mid\Kl)=-K_q(a)\) for \(a\in\Fq^\times\).
\item \(\Kl\) is tamely ramified at \(0\), and a generator of the tame inertia group at \(0\) acts unipotently, with a single Jordan block \(J_2\).
\item \(\Kl\) is totally wild at \(\infty\), with both breaks equal to \(\tfrac12\); thus \(\Swan_\infty(\Kl)=1\).
\end{enumerate}
\end{theorem}

\begin{proof}
Via the isomorphism \(\Gm\to\{yz=a\}\), \(x\mapsto(x^{-1},ax)\), the sheaf \(\Kl\) is Deligne's Kloosterman sheaf \(\mathrm{Kl}_2\), and (i)--(iii) are Deligne's theorem \cite[Sommes trig., \S7]{SGA45}; see also \cite[Ch.~4 and Ch.~7]{Katz-Gauss}. That both breaks at \(\infty\) equal \(\frac12\) follows, in every characteristic \(p\), from total wildness and \(\Swan_\infty(\Kl)=1\): the Swan conductor \(x\dim M(x)\) of each break component \(M(x)\) of \(\Kl|_{I_\infty}\) is an integer \cite[Ch.~1]{Katz-Gauss}, so two distinct positive breaks would force \(\Swan_\infty\ge2\). The vanishing of \(R^i\mathrm{pr}_!\) for \(i\ne1\) is seen fibrewise: for \(a\ne0\) the rank-one sheaf \(\Lpsi{x^{-1}+ax}\) on \(\Gm\) has Swan conductor \(1\) at \(0\) and at \(\infty\), hence is geometrically nonconstant, so \(H^0_c=H^2_c=0\), and \(\chi_c=-2\) by Grothendieck--Ogg--Shafarevich. The trace formula in (i) is the Grothendieck--Lefschetz formula for this fibre.
\end{proof}

For \(c\in\Fpbar^\times\) the sheaf \([\times c]^*\Kl\), which is defined over \(\Fp(c)\), has the local properties (ii) and (iii), and it has rank \(2\), is pure of weight \(1\) and has trace \(-K_q(ca)\) at \(a\in\Fq^\times\) when \(c\in\Fq^\times\). Indeed \([\times c]\) fixes \(0\) and \(\infty\), preserves breaks, and induces the identity on the tame inertia group at \(0\), which by Kummer theory is canonically \(\varprojlim_{(m,p)=1}\mu_m\).

\emph{The characteristic.} No hypothesis on \(p\) is made in this section, and every statement in it holds for \(p=2\), where \(\lambda=1\) and \(\psi(x)=(-1)^x\). The characteristic could intervene only through the local invariants of the Artin--Schreier sheaves and of \(\Kl\) that are used below. The functions \(x^{-1}+ax\) \textup{(}\(a\ne0\)\textup{)}, \(u^{-1}\) and \(\lambda t\) have poles of order one, which is prime to every \(p\), so the corresponding Swan conductors at these poles are equal to \(1\) in every characteristic; and the breaks of \(\Kl\) at \(\infty\) are equal to \(\frac12\) for every \(p\), by the integrality argument in the proof of Theorem~\ref{thm:kl}.

\subsection{Auxiliary lemmas}

\begin{lemma}\label{lem:breaks}
Let \(V,W\) be finite-dimensional continuous \(\Qlb\)-representations of the inertia group \(I\) at a point of a smooth curve over \(\Fpbar\).
\begin{enumerate}[label=\textup{(\roman*)}]
\item If all breaks of \(V\) and of \(W\) are \(\le c\), then all breaks of \(V\otimes W\) are \(\le c\).
\item If all breaks of \(V\) are \(<b\) and \(W\) is isoclinic of break \(b>0\), then \(V\otimes W\) is isoclinic of break \(b\); in particular \(\Swan(V\otimes W)=b\dim V\dim W\).
\end{enumerate}
\end{lemma}

\begin{proof}
Recall \cite[Ch.~1]{Katz-Gauss} that all breaks of a representation \(U\) are \(\le c\) if and only if \(I^{(r)}\) acts trivially on \(U\) for every \(r>c\), and that \(U\) is isoclinic of break \(b>0\) if and only if moreover \(U^{I^{(b)}}=0\), where \(I^{(r)}\) are the upper-numbering ramification groups. Part (i) is immediate. For (ii), let \(a<b\) be the largest break of \(V\). For \(r>b\) the group \(I^{(r)}\) acts trivially on \(V\) and on \(W\), hence on \(V\otimes W\). Since \(b>a\), the group \(I^{(b)}\) acts trivially on \(V\), so \((V\otimes W)^{I^{(b)}}=V\otimes W^{I^{(b)}}=0\). Hence \(V\otimes W\) is isoclinic of break \(b\), and its Swan conductor is \(b\) times its dimension.
\end{proof}

\begin{lemma}\label{lem:cg}
Let \(V_2\) be a two-dimensional vector space over a field of characteristic zero on which an element \(T\) acts as a single unipotent Jordan block \(J_2\). For \(n\ge1\) the element \(T^{\otimes n}\) acts on \(V_2^{\otimes n}\) as a unipotent endomorphism with Jordan decomposition
\[
V_2^{\otimes n}\cong\bigoplus_{k=0}^{\lfloor n/2\rfloor}\nu_{n,k}\,J_{n-2k+1},
\qquad
\nu_{n,k}=\binom nk-\binom n{k-1}\quad\Bigl(\binom n{-1}:=0\Bigr).
\]
Consequently the number of Jordan blocks is \(\sum_k\nu_{n,k}=\binom n{\lfloor n/2\rfloor}\), and
\[
\rank\bigl(T^{\otimes n}-1\bigr)=2^n-\binom n{\lfloor n/2\rfloor}=r(n).
\]
\end{lemma}

\begin{proof}
Write \(T=\exp(\mathfrak n)\) with \(\mathfrak n\) nilpotent of rank one. Then \(T^{\otimes n}=\exp(\mathfrak n_n)\) with \(\mathfrak n_n=\sum_i\mathfrak n^{(i)}\), where \(\mathfrak n^{(i)}\) acts by \(\mathfrak n\) in the \(i\)-th factor; \(T^{\otimes n}-1\) and \(\mathfrak n_n\) have the same Jordan type. By the Jacobson--Morozov theorem, \(\mathfrak n\) is the nilpositive element of an \(\mathfrak{sl}_2\)-triple acting on \(V_2\) by the standard representation, and \(\mathfrak n_n\) is the nilpositive element of the induced action on \(V_2^{\otimes n}\); the Jordan blocks of \(\mathfrak n_n\) correspond to the irreducible \(\mathfrak{sl}_2\)-constituents, a constituent of dimension \(j\) giving a block \(J_j\). The Clebsch--Gordan rule \(V_2\otimes V_j\cong V_{j+1}\oplus V_{j-1}\) (\(j\ge2\)), \(V_2\otimes V_1=V_2\), gives the recursion \(\nu_{n,k}=\nu_{n-1,k}+\nu_{n-1,k-1}\) (terms outside the range \(0\le k\le\lfloor(n-1)/2\rfloor\) being zero), which is also satisfied by \(\binom nk-\binom n{k-1}\) by Pascal's rule; both agree for \(n=1\). The number of blocks telescopes to \(\binom n{\lfloor n/2\rfloor}\), and the rank of a unipotent operator minus the identity is the dimension minus the number of blocks.
\end{proof}

\begin{lemma}\label{lem:binom-gap}
For \(n\ge3\) one has \(\binom n{\lfloor n/2\rfloor}<2^{n-1}\); consequently \(r(n)>2^{n-1}>\binom n{\lfloor n/2\rfloor}\).
\end{lemma}

\begin{proof}
Put \(c_n=\binom n{\lfloor n/2\rfloor}2^{-n}\). For \(m\ge1\), \(c_{2m+1}=c_{2m}\frac{2m+1}{2m+2}\) and \(c_{2m+2}=c_{2m+1}\) (because \(\binom{2m+2}{m+1}=2\binom{2m+1}m\)). Hence \((c_n)_{n\ge2}\) is non-increasing, and \(c_3=\frac38<\frac12\).
\end{proof}

\subsection{The Lagrange-multiplier sheaf}
For \(n\ge1\), \(h\in S_n\) and \(t\in\mathbb A^1\) put
\[
F_n(u,t;h):=\sum_{i=1}^nu_i^{-1}+t\Bigl(\lambda+\sum_{i=1}^nh_iu_i\Bigr),
\qquad u\in\Gm^n .
\]
On \(\Gm[,t]\times S_n\) consider
\[
\mathcal N_n:=\Bigl(\bigotimes_{i=1}^nm_i^*\Kl\Bigr)\otimes\Lpsi{\lambda t},\qquad m_i(t,h)=h_it,
\]
a lisse sheaf of rank \(2^n\), pure of weight \(n\). For \(h\in S_n(\Fpbar)\) we write \(\mathcal N_h\) for the restriction of \(\mathcal N_n\) to \(\Gm[,t]\times\{h\}\); it is defined over the finite field \(\Fp(h)\) generated by the coordinates of \(h\).

\begin{lemma}\label{lem:factor}
Let \(\mathrm{pr}:\Gm[,u]^n\times\Gm[,t]\times S_n\to\Gm[,t]\times S_n\) be the projection. Then \(R\,\mathrm{pr}_!\Lpsi{F_n}\cong\mathcal N_n[-n]\).
\end{lemma}

\begin{proof}
We have \(\Lpsi{F_n}=\Lpsi{\lambda t}\otimes\bigotimes_i\Lpsi{u_i^{-1}+h_itu_i}\). By proper base change along the morphism \(m_i\), \(R\,\mathrm{pr}_{i!}\Lpsi{u_i^{-1}+h_itu_i}=m_i^*\Kl[-1]\), where \(\mathrm{pr}_i\) forgets \(u_i\); the relative K\"unneth formula and the projection formula give the assertion.
\end{proof}

\begin{proposition}\label{prop:tline}
Let \(n\ge1\) and \(h\in S_n(\Fpbar)\), and regard \(\mathcal N_h\) as a lisse sheaf on \(\Gm\) over \(\Fpbar\).
\begin{enumerate}[label=\textup{(\roman*)}]
\item \(\mathcal N_h\) is tamely ramified at \(t=0\); a generator of tame inertia acts unipotently, with the Jordan type of \(J_2^{\otimes n}\) described in Lemma~\textup{\ref{lem:cg}}. In particular \(\dim\mathcal N_h^{I_0}=\binom n{\lfloor n/2\rfloor}\).
\item \(\mathcal N_h\) is isoclinic of break \(1\) at \(t=\infty\); thus \(\Swan_\infty(\mathcal N_h)=2^n\) and \(\mathcal N_h^{P_\infty}=0\).
\item \(H^i_c(\Gm,\mathcal N_h)=0\) for \(i\ne1\), and \(\dim H^1_c(\Gm,\mathcal N_h)=2^n\).
\end{enumerate}
\end{proposition}

\begin{proof}
(i) By Theorem~\ref{thm:kl}(ii) and the remark following it, each \([\times h_i]^*\Kl\) is tame at \(0\) with unipotent monodromy \(J_2\), and \(\Lpsi{\lambda t}\) is unramified at \(0\); Lemma~\ref{lem:cg} applies, and the invariants of a unipotent operator have dimension equal to the number of its Jordan blocks. (ii) The tensor product of the \([\times h_i]^*\Kl\) has all breaks \(\le\frac12\) at \(\infty\) by Theorem~\ref{thm:kl}(iii) and Lemma~\ref{lem:breaks}(i), while \(\Lpsi{\lambda t}\) is isoclinic of break \(1\); Lemma~\ref{lem:breaks}(ii) applies. (iii) \(H^0_c\) vanishes on the affine curve \(\Gm\), and \(H^2_c\) is dual to the geometric invariants of \(\mathcal N_h^\vee\), which vanish because \(\mathcal N_h^\vee\), like \(\mathcal N_h\), is isoclinic of break \(1\) at \(\infty\) and so has no nonzero vectors fixed by \(I_\infty\). By Grothendieck--Ogg--Shafarevich, \(\chi_c(\Gm,\mathcal N_h)=-\Swan_0-\Swan_\infty=-2^n\).
\end{proof}

\subsection{The constrained variety}
For \(h\in S_n(\Fq)\) put \(X_h:=\Gm[,u]^n\times\mathbb A^1_t\), \(Z_h:=\{t=0\}\subset X_h\), and \(\mathcal L_h:=\Lpsi{\sum_iu_i^{-1}}|_{Y_h}\).

\begin{lemma}\label{lem:tzero}
\(H^i_c(Z_h\otimes\Fqbar,\Lpsi{F_n}|_{Z_h})\) vanishes for \(i\ne n\) and is one-dimensional for \(i=n\), with Frobenius eigenvalue \(1\).
\end{lemma}

\begin{proof}
On \(Z_h\cong\Gm^n\) the sheaf is \(\boxtimes_i\Lpsi{u_i^{-1}}\). The sheaf \(\Lpsi{u^{-1}}\) on \(\Gm\) has Swan conductor \(1\) at \(0\) and \(0\) at \(\infty\), so \(H^0_c=H^2_c=0\) and \(\dim H^1_c=1\); its Frobenius eigenvalue is \(1\) because \(-\Tr(\Frob\mid H^1_c)=\sum_{u\ne0}\psi_q(u^{-1})=-1\). Apply the K\"unneth formula.
\end{proof}

\begin{lemma}\label{lem:integrate-t}
For every \(i\), \(H^i_c(X_h\otimes\Fqbar,\Lpsi{F_n})\cong H^{i-2}_c(Y_h\otimes\Fqbar,\mathcal L_h)(-1)\), compatibly with Frobenius. In particular \(H^i_c(X_h,\Lpsi{F_n})=0\) for \(i<n+1\).
\end{lemma}

\begin{proof}
Let \(\pi:X_h\to\Gm^n\) forget \(t\), let \(\kappa:Y_h\hookrightarrow\Gm^n\) be the closed immersion, and put \(\mathcal C:=R\pi_!\Lpsi{F_n}\). For \(c\ne0\) in an algebraically closed field containing \(\Fq\) one has \(R\Gamma_c(\mathbb A^1,\Lpsi{ct})=0\): \(H^0_c\) vanishes because \(\mathbb A^1\) is affine, \(H^2_c\) because \(\Lpsi{ct}\) is geometrically nontrivial, and \(H^1_c\) because \(\chi_c=1-\Swan_\infty=0\) by Grothendieck--Ogg--Shafarevich. At a geometric point \(u\) of \(\Gm^n\), proper base change gives \(\mathcal C_u=\Lpsi{\textstyle\sum_iu_i^{-1}}_u\otimes R\Gamma_c\bigl(\mathbb A^1,\Lpsi{c_ut}\bigr)\) with \(c_u=\lambda+\langle h,u\rangle\), which vanishes if \(u\notin Y_h\). Hence the adjunction morphism \(\mathcal C\to\kappa_*\kappa^*\mathcal C\) is an isomorphism, as one checks on stalks. On \(\pi^{-1}(Y_h)=Y_h\times\mathbb A^1\) one has \(F_n=\sum_iu_i^{-1}\), so the restriction of \(\Lpsi{F_n}\) to \(\pi^{-1}(Y_h)\) is \(\mathrm{pr}_1^*\mathcal L_h\); proper base change, the projection formula and \(R\Gamma_c(\mathbb A^1,\Qlb)=\Qlb(-1)[-2]\) give \(\kappa^*\mathcal C\cong\mathcal L_h(-1)[-2]\). Thus \(R\pi_!\Lpsi{F_n}\cong\kappa_*\mathcal L_h(-1)[-2]\), all identifications being defined over \(\Fq\), and applying \(R\Gamma_c(\Gm^n\otimes\Fqbar,-)\) gives the isomorphism. Since all \(h_i\ne0\), \(Y_h\) is a smooth affine variety of pure dimension \(n-1\) (an open subset of an affine hyperplane), so \(H^j_c(Y_h,\mathcal L_h)=0\) for \(j<n-1\) by Poincar\'e duality and Artin vanishing.
\end{proof}

\begin{proposition}\label{prop:fibre}
Let \(n\ge1\) and \(h\in S_n(\Fq)\). Write \(j:\Gm[,t]\hookrightarrow\mathbb P^1\).
\begin{enumerate}[label=\textup{(\roman*)}]
\item There are exact sequences of \(\Frob_h\)-modules
\[
0\to H^n_c(Z_h,\Lpsi{F_n})\to H^1_c(\Gm,\mathcal N_h)\to H^{n+1}_c(X_h,\Lpsi{F_n})\to0,
\]
\[
0\to\mathcal N_h^{I_0}\to H^1_c(\Gm,\mathcal N_h)\to H^1(\mathbb P^1,j_*\mathcal N_h)\to0,
\]
all cohomology being taken over \(\Fqbar\), and \(H^i_c(X_h,\Lpsi{F_n})=0\) for \(i\ne n+1\).
\item \(\mathcal N_h^{I_0}\) has weights \(\le n\), \(H^1(\mathbb P^1,j_*\mathcal N_h)\) is pure of weight \(n+1\) and has dimension \(r(n)\), and the line \(H^n_c(Z_h)\) maps into \(\mathcal N_h^{I_0}\).
\item \(H^i_c(Y_h,\mathcal L_h)=0\) for \(i\ne n-1\), the weight-\((n-1)\) quotient of \(H^{n-1}_c(Y_h,\mathcal L_h)\) is \(H^1(\mathbb P^1,j_*\mathcal N_h)(1)\), and its other weights are \(\le n-2\). Moreover
\[
\Sigma_n(h)=(-1)^{n-1}\Tr\bigl(\Frob_h\mid H^{n-1}_c(Y_h,\mathcal L_h)\bigr).
\]
\end{enumerate}
\end{proposition}

\begin{proof}
(i) Let \(U=X_h\setminus Z_h=\{t\ne0\}\). By Lemma~\ref{lem:factor} and the Leray spectral sequence, \(H^i_c(U,\Lpsi{F_n})=H^{i-n}_c(\Gm,\mathcal N_h)\), which by Proposition~\ref{prop:tline}(iii) vanishes for \(i\ne n+1\). By Lemma~\ref{lem:tzero}, \(H^i_c(Z_h)\) vanishes for \(i\ne n\). The long exact sequence of the pair \((X_h,Z_h)\) shows \(H^i_c(X_h)=0\) for \(i\notin\{n,n+1\}\) and gives an exact sequence \(0\to H^n_c(X_h)\to H^n_c(Z_h)\to H^{n+1}_c(U)\to H^{n+1}_c(X_h)\to0\); since \(H^n_c(X_h)=0\) by Lemma~\ref{lem:integrate-t}, the first sequence follows. For the second, the exact sequence \(0\to j_!\mathcal N_h\to j_*\mathcal N_h\to\mathcal N_h^{I_0}\oplus\mathcal N_h^{I_\infty}\to0\) of sheaves on \(\mathbb P^1\) (the last term being skyscrapers at \(0\) and \(\infty\)) has \(\mathcal N_h^{I_\infty}=0\) by Proposition~\ref{prop:tline}(ii), and \(H^0(\mathbb P^1,j_*\mathcal N_h)=H^0(\Gm,\mathcal N_h)=0\) because a global section would be fixed by \(I_\infty\); the long exact cohomology sequence gives the claim.

(ii) \(\mathcal N_h\) is pure of weight \(n\). By \cite[1.8.4]{Deligne} the invariants \(\mathcal N_h^{I_0}\) have weights \(\le n\), and by \cite[Th\'eor\`eme 3.2.3]{Deligne} \(H^1(\mathbb P^1,j_*\mathcal N_h)\) is pure of weight \(n+1\). Its dimension is \(2^n-\binom n{\lfloor n/2\rfloor}=r(n)\) by Proposition~\ref{prop:tline}. The line \(H^n_c(Z_h)\) has weight \(0\); its image in the pure weight-\((n+1)\) quotient \(H^1(\mathbb P^1,j_*\mathcal N_h)\) is therefore zero, since morphisms of mixed Frobenius modules are strictly compatible with weights.

(iii) By Lemma~\ref{lem:integrate-t}, \(H^i_c(Y_h,\mathcal L_h)(-1)\cong H^{i+2}_c(X_h)\), which vanishes for \(i\ne n-1\) by (i). By (i) and (ii), \(H^{n+1}_c(X_h)\) is the quotient of \(H^1_c(\Gm,\mathcal N_h)\) by a line contained in \(\mathcal N_h^{I_0}\); hence it has the pure weight-\((n+1)\) quotient \(H^1(\mathbb P^1,j_*\mathcal N_h)\) with kernel \(\mathcal N_h^{I_0}/H^n_c(Z_h)\) of weights \(\le n\). Twisting by \((1)\) gives the weight statements. The trace formula is the Grothendieck--Lefschetz formula on \(Y_h\), which has cohomology only in degree \(n-1\).
\end{proof}

\subsection{The family}
Let \(\rho:\Gm[,t]\times S_n\to S_n\) be the projection.

\begin{theorem}\label{thm:package}
Let \(n\ge1\).
\begin{enumerate}[label=\textup{(\roman*)}]
\item \(R^i\rho_!\mathcal N_n=0\) for \(i\ne1\), and \(\mathcal H_n:=R^1\rho_!\mathcal N_n\) is a lisse sheaf of rank \(2^n\) on \(S_n\) whose formation commutes with arbitrary base change; it is mixed of weights \(\le n+1\).
\item There is an exact sequence of lisse sheaves
\[
0\to\mathcal E_n\to\mathcal H_n\to\mathcal W_n\to0,
\]
in which \(\mathcal E_n\) has rank \(\binom n{\lfloor n/2\rfloor}\) and weights \(\le n\), and \(\mathcal W_n\) has rank \(r(n)\) and is pure of weight \(n+1\). For every \(h\in S_n(\Fq)\) the fibres are \(\mathcal E_{n,h}=\mathcal N_h^{I_0}\) and \(\mathcal W_{n,h}=H^1(\mathbb P^1,j_*\mathcal N_h)\), the latter being the weight-\((n-1)\) quotient of \(H^{n-1}_c(Y_h,\mathcal L_h)\), twisted by \((-1)\).
\item For every \(h\in S_n(\Fq)\),
\[
(-1)^{n-1}q\,\Sigma_n(h)=\Tr(\Frob_h\mid\mathcal W_n)+\Tr(\Frob_h\mid\mathcal E_n)-1 .
\]
Consequently
\[
\Bigl|(-1)^{n-1}q\,\Sigma_n(h)-\Tr(\Frob_h\mid\mathcal W_n)\Bigr|\le\Bigl(\binom n{\lfloor n/2\rfloor}-1\Bigr)q^{n/2}
\]
and
\[
\bigl|\Tr(\Frob_h\mid\mathcal W_n)\bigr|\le r(n)\,q^{(n+1)/2}.
\]
\end{enumerate}
\end{theorem}

\begin{proof}
(i) Let \(\bar\rho:\mathbb P^1\times S_n\to S_n\) be the projection, a smooth proper morphism of relative dimension one; let \(\partial_n=\{0,\infty\}\times S_n\), which is finite \'etale of degree \(2\) over \(S_n\); and let \(j:\Gm\times S_n\hookrightarrow\mathbb P^1\times S_n\) be the complementary open immersion. Then \(R^i\rho_!\mathcal N_n=R^i\bar\rho_*(j_!\mathcal N_n)\), and by the proper base change theorem the formation of these sheaves commutes with arbitrary base change; in particular the fibre of \(R^i\rho_!\mathcal N_n\) at \(h\in S_n(\Fpbar)\) is \(H^i_c(\Gm,\mathcal N_h)\). The scheme \(S_n\) is of finite type over \(\Fp\); hence it is noetherian and excellent, \(\ell\) is invertible on it, and every nonempty locally closed subset of \(S_n\) contains a closed point of \(S_n\).

We show that the sheaves \(R^i\rho_!\mathcal N_n\) are lisse. The sheaf \(\mathcal N_n\) is defined over a finite extension \(E\) of \(\Qell\) (for instance \(E=\Qell(\mu_p)\), over which \(\mathcal L_\psi\), and hence \(\Kl\), are defined); let \(\mathcal O_E\), \(\varpi\) and \(k_E\) be its ring of integers, a uniformiser and its residue field. Since \(\pi_1(\Gm\times S_n)\) is compact, the representation defined by \(\mathcal N_n\) stabilises an \(\mathcal O_E\)-lattice; so \(\mathcal N_n=\mathcal N^\circ\otimes_{\mathcal O_E}E\) for a lisse sheaf \(\mathcal N^\circ\) of free \(\mathcal O_E\)-modules. Put \(\bar{\mathcal N}=\mathcal N^\circ/\varpi\mathcal N^\circ\), a lisse sheaf of \(\mathbb F_\ell\)-vector spaces of constant rank \(2^n[k_E:\mathbb F_\ell]\).

\emph{Swan conductors of \(\bar{\mathcal N}\).} Let \(h\in S_n(\Fpbar)\), let \(c\in\{0,\infty\}\), and consider the action of the inertia group \(I_c\) on the geometric generic stalk of the restriction of \(\mathcal N^\circ\) to \(\Gm\times\{h\}\). The wild inertia group \(P_c\) acts through a finite \(p\)-group, because a pro-\(p\) subgroup of \(\GL_{2^n}(\mathcal O_E)\) meets the open pro-\(\ell\) subgroup \(1+\varpi M_{2^n}(\mathcal O_E)\) trivially. For a closed subgroup \(H\subseteq P_c\) with finite image \(\bar H\), the idempotent \(|\bar H|^{-1}\sum_{g\in\bar H}g\) is defined over \(\mathcal O_E\) and projects onto the \(H\)-invariants; hence the \(H\)-invariants of \(\mathcal N_h\) and of \(\bar{\mathcal N}_h\) have the same dimension, over \(E\) and over \(k_E\) respectively. The Swan conductor of a representation of \(I_c\) on which \(P_c\) acts through a finite quotient depends only on the dimensions of the invariants of the upper-numbering subgroups \(I_c^{(r)}\), \(r>0\) \cite[Ch.~1]{Katz-Gauss}. Therefore the Swan conductor at \(c\) of \(\bar{\mathcal N}_h\), as a sheaf of \(\mathbb F_\ell\)-vector spaces, is \([k_E:\mathbb F_\ell]\) times that of \(\mathcal N_h\), which is \(0\) for \(c=0\) and \(2^n\) for \(c=\infty\) by Proposition~\ref{prop:tline}.

\emph{Local constancy.} We use Deligne's semicontinuity theorem in the form of \cite[2.1]{Laumon-semicont}. There \(S\) is a noetherian excellent scheme on which \(\ell\) is invertible, \(f:X\to S\) is separated and smooth of relative dimension one, \(Y\subset X\) is a closed subscheme which is finite and flat over \(S\), with complement \(j:U\hookrightarrow X\), and \(\mathcal F\) is a lisse sheaf of \(\mathbb F_\ell\)-vector spaces of constant rank on \(U\); to these data is attached a function \(\varphi\) on \(S\), built from the ranks and the Swan conductors of the restrictions of \(\mathcal F\) to the fibres \(U_{\bar s}\) at the points of \(Y_{\bar s}\). Then \(\varphi\) is constructible \cite[Th\'eor\`eme 2.1.1(i)]{Laumon-semicont}, and if moreover \(f\) is proper and \(\varphi\) is locally constant, the sheaves \(R^if_*(j_!\mathcal F)\) are lisse for all \(i\) \cite[Corollaire 2.1.2]{Laumon-semicont}. We apply this with \(S=S_n\), \(f=\bar\rho\), which is proper and smooth of relative dimension one, \(Y=\partial_n\), which is finite \'etale over \(S_n\), and \(\mathcal F=\bar{\mathcal N}\). Since \(\partial_n\) is finite \'etale of degree \(2\) over \(S_n\) and \(\bar{\mathcal N}\) has constant rank, \(\varphi(s)\) is the sum of a constant, which depends only on \(\rank\bar{\mathcal N}\), and of the Swan conductors at \(0\) and at \(\infty\) of the restriction of \(\bar{\mathcal N}\) to \(\Gm\times\{\bar s\}\), where \(\bar s\) is a geometric point over \(s\). By the preceding paragraph, \(\varphi\) takes one and the same value \(\varphi_0\) at all closed points of \(S_n\). Since \(\varphi\) is constructible, the set of points where \(\varphi\ne\varphi_0\) is a finite union of locally closed subsets of \(S_n\) that contain no closed point, hence is empty. Thus \(\varphi\) is constant, and the sheaves \(R^i\bar\rho_*(j_!\bar{\mathcal N})\) are lisse for all \(i\).

\emph{Passage to \(\Qlb\)-coefficients.} For \(m\ge2\), multiplication by \(\varpi^{m-1}\) gives exact sequences \(0\to\bar{\mathcal N}\to\mathcal N^\circ/\varpi^m\to\mathcal N^\circ/\varpi^{m-1}\to0\), and hence, \(j_!\) being exact, long exact sequences of the sheaves \(R^i\bar\rho_*j_!(-)\). Locally constant constructible sheaves are stable under kernels, cokernels and extensions. (For extensions: let \(0\to\mathcal A'\to\mathcal A\to\mathcal A''\to0\) be exact, with \(\mathcal A'\) and \(\mathcal A''\) constant, of finite values \(A'\) and \(A''\), on a connected \'etale neighbourhood of a geometric point. On a smaller connected \'etale neighbourhood \(V\) choose sections \(g_x\in\mathcal A(V)\) mapping to the constant sections \(x\in A''\); then \((a,x)\mapsto a+g_x\) is an isomorphism of sheaves of sets \(A'_V\times A''_V\to\mathcal A|_V\), as one checks on stalks.) By induction on \(m\), all the sheaves \(R^i\bar\rho_*\bigl(j_!(\mathcal N^\circ/\varpi^m)\bigr)\) are locally constant constructible. By definition, the \(\varpi\)-adic sheaf \(R^i\bar\rho_*(j_!\mathcal N^\circ)\) is the projective system \(\bigl(R^i\bar\rho_*(j_!(\mathcal N^\circ/\varpi^m))\bigr)_{m\ge1}\), which is Artin--Rees \(\varpi\)-adic; it is therefore Artin--Rees isomorphic to the \(\varpi\)-adic system formed by the images of the transition maps from a fixed number of steps higher up \cite[Exp.~V and VI]{SGA5}. These images are locally constant constructible, being images of morphisms between such sheaves. Hence \(R^i\bar\rho_*(j_!\mathcal N^\circ)\) is a lisse \(\varpi\)-adic sheaf, and \(R^i\rho_!\mathcal N_n=R^i\bar\rho_*(j_!\mathcal N^\circ)\otimes_{\mathcal O_E}E\) is lisse.

\emph{Conclusion.} By Proposition~\ref{prop:tline}(iii) the fibre \(H^i_c(\Gm,\mathcal N_h)\) at any \(h\in S_n(\Fpbar)\) vanishes for \(i\ne1\) and has dimension \(2^n\) for \(i=1\). Since \(S_n\) is connected, \(R^i\rho_!\mathcal N_n=0\) for \(i\ne1\), and \(\mathcal H_n\) is lisse of rank \(2^n\). Since \(\mathcal N_n\) is pure of weight \(n\), \(R^1\rho_!\mathcal N_n\) is mixed of weights \(\le n+1\) by \cite[Th\'eor\`eme 3.3.1]{Deligne}.

(ii) Because \(S_n\) is normal and connected, the lisse mixed sheaf \(\mathcal H_n\) has a unique increasing filtration by lisse subsheaves whose graded pieces are pure \cite[Th\'eor\`eme 3.4.1]{Deligne}. Let \(\mathcal E_n\subset\mathcal H_n\) be its part of weights \(\le n\) and \(\mathcal W_n=\mathcal H_n/\mathcal E_n\). At a closed point \(h\) the filtration induces the weight filtration of the Frobenius module \(\mathcal H_{n,h}=H^1_c(\Gm,\mathcal N_h)\), which by Proposition~\ref{prop:fibre}(i)--(ii) has weight-\(\le n\) part \(\mathcal N_h^{I_0}\), of dimension \(\binom n{\lfloor n/2\rfloor}\), and weight-\((n+1)\) quotient \(H^1(\mathbb P^1,j_*\mathcal N_h)\), of dimension \(r(n)\). This gives the ranks, the weights and the description of the fibres; the comparison with \(H^{n-1}_c(Y_h,\mathcal L_h)\) is Proposition~\ref{prop:fibre}(iii). We write \(W_n\) for the geometric generic fibre of \(\mathcal W_n\).

(iii) Inserting \(q^{-1}\sum_{t\in\Fq}\psi_q\bigl(t(\lambda+\langle h,u\rangle)\bigr)=\mathbf 1_{\langle h,u\rangle=-\lambda}\) into \eqref{eq:family} and summing over \(u\) first gives
\begin{equation}\label{eq:t-sum}
q\,\Sigma_n(h)=\sum_{t\in\Fq}\psi_q(\lambda t)\prod_{i=1}^nK(th_i).
\end{equation}
The term \(t=0\) equals \(K(0)^n=(-1)^n\). For \(t\ne0\), Theorem~\ref{thm:kl}(i) gives \(\psi_q(\lambda t)\prod_iK(th_i)=(-1)^n\Tr(\Frob_t\mid\mathcal N_h)\), and the Grothendieck--Lefschetz formula together with Proposition~\ref{prop:tline}(iii) gives \(\sum_{t\ne0}\Tr(\Frob_t\mid\mathcal N_h)=-\Tr(\Frob\mid H^1_c(\Gm,\mathcal N_h))\). Hence
\[
(-1)^{n-1}q\,\Sigma_n(h)=\Tr\bigl(\Frob\mid H^1_c(\Gm,\mathcal N_h)\bigr)-1=\Tr(\Frob_h\mid\mathcal W_n)+\Tr(\Frob_h\mid\mathcal E_n)-1,
\]
using (i) and (ii). By Proposition~\ref{prop:fibre}(i),(ii) and Lemma~\ref{lem:tzero}, the image of \(H^n_c(Z_h)\) in \(H^1_c(\Gm,\mathcal N_h)=\mathcal H_{n,h}\) is a Frobenius-stable line contained in \(\mathcal E_{n,h}=\mathcal N_h^{I_0}\), on which Frobenius acts as the identity. Hence \(\Tr(\Frob_h\mid\mathcal E_n)-1\) is the trace of Frobenius on \(\mathcal E_{n,h}/H^n_c(Z_h)\), which has dimension \(\binom n{\lfloor n/2\rfloor}-1\) and weights \(\le n\); this gives the first bound. The second follows from the purity of \(\mathcal W_n\).
\end{proof}

\begin{remark}\label{rem:minusone}
The constant \(-1\) in Theorem~\ref{thm:package}(iii) is the contribution of the weight-zero line \(H^n_c(Z_h)\) of Proposition~\ref{prop:fibre}: by that proposition, \((-1)^{n-1}q\,\Sigma_n(h)\) is the trace of Frobenius on \(H^{n+1}_c(X_h,\Lpsi{F_n})=H^1_c(\Gm,\mathcal N_h)/H^n_c(Z_h)\), in agreement with the computation in the proof.
\end{remark}

\subsection{The lower-weight part}\label{subsec:lower}
We now determine the Frobenius eigenvalues on \(\mathcal E_n\) exactly. Let \(C\) be a smooth curve over \(\Fq\), let \(x\in C(\Fq)\), and let \(\mathcal F\) be a lisse \(\Qlb\)-sheaf on \(C\setminus\{x\}\). Let \(D_x\) be a decomposition group at \(x\), with inertia subgroup \(I_x\) and wild inertia subgroup \(P_x\), and let \(\varrho\) denote the action of \(D_x\) on the geometric generic stalk \(\mathcal F_{\bar\eta}\). Suppose that \(I_x\) acts unipotently. Then \(P_x\) acts trivially, because its image is a finite group (as in the proof of Theorem~\ref{thm:package}(i)) consisting of unipotent automorphisms, and a unipotent automorphism of finite order is the identity; and if \(\sigma\in I_x\) maps to a topological generator of \(I_x/P_x\), the image of \(I_x\) lies in the closure of the group generated by \(\varrho(\sigma)\). We put \(\mathfrak N:=\log\varrho(\sigma)\), a nilpotent endomorphism of \(\mathcal F_{\bar\eta}\); thus \(\mathcal F_{\bar\eta}^{I_x}=\ker\mathfrak N\).

\begin{lemma}\label{lem:frob-monodromy}
In this situation, let \(\Phi\in D_x\) map to the geometric Frobenius of \(\Fq\). Then \(\varrho(\Phi)\,\mathfrak N\,\varrho(\Phi)^{-1}=q^{-1}\mathfrak N\). Consequently \(\ker\mathfrak N\) is stable under \(\varrho(\Phi)\), and \(\varrho(\Phi)\) acts on it as \(\Frob_x\) acts on the stalk \((j_*\mathcal F)_x\), where \(j:C\setminus\{x\}\hookrightarrow C\).
\end{lemma}

\begin{proof}
Let \(\varpi\) be a uniformiser at \(x\). By Kummer theory, \(I_x/P_x\) is canonically isomorphic to \(\varprojlim_{(m,p)=1}\mu_m(\Fqbar)\), the class of \(\tau\in I_x\) corresponding to the system \(\bigl(\tau(\varpi^{1/m})/\varpi^{1/m}\bigr)_m\); and for \(\delta\in D_x\) the class of \(\delta\tau\delta^{-1}\) corresponds to the system obtained by applying \(\delta\) to these roots of unity. As \(\Phi\) acts on roots of unity by \(\zeta\mapsto\zeta^{1/q}\), we get \(\Phi\sigma\Phi^{-1}\equiv\sigma^{q^{-1}}\) modulo \(P_x\), where \(q^{-1}\in\prod_{\ell'\ne p}\mathbb Z_{\ell'}^\times\). Since \(P_x\) acts trivially and, by continuity, \(\varrho(\sigma^c)=\exp(c\,\mathfrak N)\) for \(c\in\prod_{\ell'\ne p}\mathbb Z_{\ell'}\) (with \(c\) acting through its \(\ell\)-component), we obtain \(\varrho(\Phi)\exp(\mathfrak N)\varrho(\Phi)^{-1}=\exp(q^{-1}\mathfrak N)\); taking logarithms of these unipotent automorphisms gives the first assertion. The second follows, and the third is the description of \((j_*\mathcal F)_x=\mathcal F_{\bar\eta}^{I_x}\) with its action of \(D_x/I_x\).
\end{proof}

\begin{lemma}\label{lem:kl-origin}
Let \(j:\Gm\hookrightarrow\mathbb A^1\) over \(\Fq\), let \(D_0\) be a decomposition group at \(0\), and let \(\varrho_1\) be its action on \(\Kl_{\bar\eta}\), with \(\mathfrak N\) and \(\Phi\) as above (Theorem~\ref{thm:kl}(ii) shows that \(I_0\) acts unipotently).
\begin{enumerate}[label=\textup{(\roman*)}]
\item \(\Frob_0\) acts as the identity on the line \((j_*\Kl)_0\).
\item There is a basis \(e_0,e_1\) of \(\Kl_{\bar\eta}\) with \(\mathfrak Ne_1=e_0\), \(\mathfrak Ne_0=0\), \(\varrho_1(\Phi)e_0=e_0\) and \(\varrho_1(\Phi)e_1=qe_1\).
\item For every \(c\in\Fq^\times\), statements \textup{(i)} and \textup{(ii)} hold with \(\Kl\) replaced by \([\times c]^*\Kl\).
\end{enumerate}
\end{lemma}

\begin{proof}
(i) Let \(\mathrm{pr}:\Gm[,x]\times\mathbb A^1_a\to\mathbb A^1_a\) be the projection and \(\mathcal F:=R^1\mathrm{pr}_!\Lpsi{x^{-1}+ax}\), a constructible sheaf on \(\mathbb A^1_a\) whose restriction to \(\Gm\) is \(\Kl\). For every geometric point \(a\) of \(\mathbb A^1\) the sheaf \(\Lpsi{x^{-1}+ax}\) on \(\Gm\) has Swan conductor \(1\) at \(x=0\), hence is geometrically nontrivial, and its \(H^0_c\) and \(H^2_c\) vanish; so \(R^i\mathrm{pr}_!\Lpsi{x^{-1}+ax}=0\) for \(i\ne1\) by proper base change. Let \(j_x:\Gm[,x]\hookrightarrow\mathbb A^1_x\) and \(\mathcal P_0:=j_{x!}\Lpsi{x^{-1}}[1]\), a perverse sheaf on \(\mathbb A^1_x\) because \(j_x\) is an affine open immersion \cite[Corollaire 4.1.3]{BBD}. Laumon's Fourier transform \(\mathcal F_\psi(\mathcal P_0)=R\mathrm{pr}'_!(\mathrm{pr}_x^*\mathcal P_0\otimes\Lpsi{xa})[1]\) \cite[(1.2.1.1)]{Laumon}, where \(\mathrm{pr}_x,\mathrm{pr}'\) are the projections of \(\mathbb A^1_x\times\mathbb A^1_a\), is therefore \(\mathcal F[1]\), and it is perverse \cite[(1.3.2.3)]{Laumon}. Let \(\delta_0\) be the skyscraper sheaf \(\Qlb\) at \(a=0\). It is perverse, so \(\delta_0[1]\) lies in \({}^pD^{\le-1}\), while \(\mathcal F[1]\) lies in \({}^pD^{\ge0}\); hence \(\Hom(\delta_0,\mathcal F)=\Hom(\delta_0[1],\mathcal F[1])=0\), that is, \(\mathcal F\) has no nonzero section supported at \(0\), and the adjunction map \(\mathcal F_0\to(j_*j^*\mathcal F)_0=(j_*\Kl)_0\) is injective. Both spaces are one-dimensional: \(\mathcal F_0=H^1_c(\Gm\otimes\Fqbar,\Lpsi{x^{-1}})\) by proper base change, which is one-dimensional with Frobenius eigenvalue \(1\) (proof of Lemma~\ref{lem:tzero} with \(n=1\)), and \((j_*\Kl)_0=\ker\mathfrak N\) is a line because \(\mathfrak N\ne0\) on the two-dimensional space \(\Kl_{\bar\eta}\) (Theorem~\ref{thm:kl}(ii)). The map is Frobenius-equivariant, being the stalk at the rational point \(0\) of a morphism of sheaves on \(\mathbb A^1\) over \(\Fq\); this proves (i).

(ii) By Lemma~\ref{lem:frob-monodromy} and (i), \(\varrho_1(\Phi)\) acts by \(1\) on \(\ker\mathfrak N\). Let \(v\notin\ker\mathfrak N\) and write \(\varrho_1(\Phi)v=\beta v+w\) with \(w\in\ker\mathfrak N\). Lemma~\ref{lem:frob-monodromy} gives \(\mathfrak N\varrho_1(\Phi)=q\,\varrho_1(\Phi)\mathfrak N\), so \(\beta\,\mathfrak Nv=\mathfrak N\varrho_1(\Phi)v=q\,\varrho_1(\Phi)\mathfrak Nv=q\,\mathfrak Nv\), whence \(\beta=q\). As \(q\ne1\), \(\varrho_1(\Phi)\) is diagonalisable; let \(e_1\) be an eigenvector for \(q\) and \(e_0:=\mathfrak Ne_1\), which is nonzero and satisfies \(\varrho_1(\Phi)e_0=q^{-1}\mathfrak N\varrho_1(\Phi)e_1=e_0\).

(iii) The dilation \([\times c]\) is an automorphism of \(\mathbb A^1\) over \(\Fq\) fixing \(0\), so \(j_*([\times c]^*\Kl)\cong[\times c]^*(j_*\Kl)\) and \((j_*[\times c]^*\Kl)_0\cong(j_*\Kl)_0\), compatibly with Frobenius; hence (i) holds for \([\times c]^*\Kl\). Since \([\times c]^*\Kl\) is tame at \(0\) with unipotent monodromy consisting of a single Jordan block of size two (remark following Theorem~\ref{thm:kl}), the proof of (ii) applies verbatim.
\end{proof}

\begin{proposition}\label{prop:lower}
Let \(n\ge1\) and \(h\in S_n(\Fq)\). The eigenvalues of \(\Frob_h\) on \(\mathcal E_{n,h}\) are the numbers \(q^k\), \(0\le k\le\lfloor n/2\rfloor\), with multiplicity \(\nu_{n,k}=\binom nk-\binom n{k-1}\). In particular
\[
\Tr(\Frob_h\mid\mathcal E_n)=\sum_{k=0}^{\lfloor n/2\rfloor}\nu_{n,k}\,q^k ,
\]
and \(\mathcal E_{n,h}\) has the weights \(2k\), \(0\le k\le\lfloor n/2\rfloor\), with multiplicities \(\nu_{n,k}\).
\end{proposition}

\begin{proof}
By Theorem~\ref{thm:package}(ii) and Proposition~\ref{prop:fibre}(i), \(\mathcal E_{n,h}=\mathcal N_h^{I_0}=(j_*\mathcal N_h)_0\), Frobenius acting as on the stalk at the rational point \(t=0\); by Lemma~\ref{lem:frob-monodromy} this is the action of \(\varrho(\Phi)\) on \(\ker\mathfrak N_h\), where \(\varrho\) is the representation of \(D_0\) on \(\mathcal N_{h,\bar\eta}\) and \(\mathfrak N_h\) is the logarithm of its tame monodromy (unipotent by Proposition~\ref{prop:tline}(i)). The sheaf \(\Lpsi{\lambda t}\) is lisse at \(t=0\) with Frobenius eigenvalue \(\psi_q(0)=1\) there, so \(D_0\) acts trivially on it. By Lemma~\ref{lem:kl-origin}(ii),(iii), the \(i\)-th factor \([\times h_i]^*\Kl\) of \(\mathcal N_h\) has a basis \(e^{(i)}_0,e^{(i)}_1\) of its geometric generic stalk with \(\mathfrak N^{(i)}e^{(i)}_1=e^{(i)}_0\), \(\mathfrak N^{(i)}e^{(i)}_0=0\), \(\varrho(\Phi)e^{(i)}_0=e^{(i)}_0\) and \(\varrho(\Phi)e^{(i)}_1=q\,e^{(i)}_1\), where \(\mathfrak N^{(i)}\) is the logarithm of the tame monodromy of that factor. The tame monodromy of \(\mathcal N_h\) is the tensor product of those of the factors, so \(\mathfrak N_h=\sum_i\mathfrak N^{(i)}\), the \(i\)-th summand acting on the \(i\)-th factor. In the basis \(e_\varepsilon=e^{(1)}_{\varepsilon_1}\otimes\cdots\otimes e^{(n)}_{\varepsilon_n}\), \(\varepsilon\in\{0,1\}^n\), one has \(\varrho(\Phi)e_\varepsilon=q^{|\varepsilon|}e_\varepsilon\) with \(|\varepsilon|=\sum_i\varepsilon_i\). On each factor the operators \(Y=\mathfrak N^{(i)}\), \(X\) given by \(Xe^{(i)}_0=e^{(i)}_1\), \(Xe^{(i)}_1=0\), and \(H=[X,Y]\), which acts by \(He^{(i)}_1=e^{(i)}_1\), \(He^{(i)}_0=-e^{(i)}_0\), form an \(\mathfrak{sl}_2\)-triple; the induced triple on the tensor product \(\mathcal N_{h,\bar\eta}\) has \(Y_n=\mathfrak N_h\) and \(H_ne_\varepsilon=(2|\varepsilon|-n)e_\varepsilon\). By the proof of Lemma~\ref{lem:cg}, \(\mathcal N_{h,\bar\eta}\) is the sum of \(\nu_{n,k}\) irreducible \(\mathfrak{sl}_2\)-modules of dimension \(n-2k+1\), \(0\le k\le\lfloor n/2\rfloor\), and in an irreducible module of dimension \(m\) the kernel of \(Y_n\) is the line of \(H_n\)-weight \(-(m-1)\). So \(\ker\mathfrak N_h\) has a basis of \(H_n\)-eigenvectors, \(\nu_{n,k}\) of which have \(H_n\)-weight \(-(n-2k)\), that is, lie in the span of the \(e_\varepsilon\) with \(|\varepsilon|=k\), on which \(\varrho(\Phi)\) acts by \(q^k\).
\end{proof}

\begin{corollary}\label{cor:exact}
Let \(n\ge1\) and \(h\in S_n(\Fq)\).
\begin{enumerate}[label=\textup{(\roman*)}]
\item One has
\[
(-1)^{n-1}q\,\Sigma_n(h)=\Tr(\Frob_h\mid\mathcal W_n)+\sum_{k=1}^{\lfloor n/2\rfloor}\nu_{n,k}\,q^k ;
\]
in particular \(0\le(-1)^{n-1}q\,\Sigma_n(h)-\Tr(\Frob_h\mid\mathcal W_n)\le\bigl(\binom n{\lfloor n/2\rfloor}-1\bigr)q^{\lfloor n/2\rfloor}\).
\item The space \(H^{n-1}_c(Y_h\otimes\Fqbar,\mathcal L_h)\) has dimension \(2^n-1\). The eigenvalues of \(\Frob_h\) on it are those on \(\mathcal W_{n,h}(1)\), which have weight \(n-1\), together with the numbers \(q^{k-1}\), \(1\le k\le\lfloor n/2\rfloor\), with multiplicity \(\nu_{n,k}\).
\end{enumerate}
\end{corollary}

\begin{proof}
(i) Combine Theorem~\ref{thm:package}(iii) with Proposition~\ref{prop:lower}; the term \(k=0\) is \(\nu_{n,0}=1\), and \(\sum_{k\ge1}\nu_{n,k}=\binom n{\lfloor n/2\rfloor}-1\). (ii) By Lemma~\ref{lem:integrate-t} and Proposition~\ref{prop:fibre}(i), \(H^{n-1}_c(Y_h,\mathcal L_h)(-1)\cong H^{n+1}_c(X_h,\Lpsi{F_n})\) is the quotient of \(H^1_c(\Gm,\mathcal N_h)=\mathcal H_{n,h}\) by the line \(H^n_c(Z_h)\), which lies in \(\mathcal E_{n,h}\) and on which Frobenius acts as the identity (proof of Theorem~\ref{thm:package}(iii)). So its eigenvalues are those on \(\mathcal W_{n,h}\) together with those on \(\mathcal E_{n,h}\), the eigenvalue \(1=q^0\) being removed once; by Proposition~\ref{prop:lower}, \(\nu_{n,0}=1\), so the eigenvalue \(1\) disappears altogether. Twisting by \((1)\) divides all eigenvalues by \(q\), and the dimension is \(2^n-1\).
\end{proof}

The Frobenius eigenvalues determine the geometric structure of \(\mathcal E_n\), by the following elementary consequence of the trace formula.

\begin{lemma}\label{lem:constant}
Let \(S\) be smooth and geometrically connected of dimension \(e\) over a finite field \(k\), and let \(\mathcal F\) be a lisse \(\Qlb\)-sheaf on \(S\) such that, for every \(a\ge1\) and every \(x\in S(k_a)\), all eigenvalues of \(\Frob_{k_a,x}\) on \(\mathcal F\) are equal to \(1\). Then \(\pgeom(S)\) acts trivially on \(\mathcal F\).
\end{lemma}

\begin{proof}
Let \(r\) be the rank of \(\mathcal F\), let \(V\) be its geometric generic fibre, and let \(m\) be the dimension of the coinvariants \(V_{\pgeom(S)}\), that is, of the quotient of \(V\) by the span of the vectors \(gv-v\), \(g\in\pgeom(S)\). By Poincar\'e duality, \(H^{2e}_c(S_{\bar k},\mathcal F)\) is dual to \(H^0(S_{\bar k},\mathcal F^\vee)(e)\), and \(H^0(S_{\bar k},\mathcal F^\vee)\) is the space of \(\pgeom(S)\)-invariants in \(V^\vee\), which is dual to \(V_{\pgeom(S)}\); so it has dimension \(m\). Restriction to a geometric point over a closed point \(x\) of degree \(\delta\) is an injective map \(H^0(S_{\bar k},\mathcal F^\vee)\to\mathcal F^\vee_{\bar x}\) which carries the action of \(F_k^\delta\) to that of \(\Frob_x\), where \(F_k\) is the geometric Frobenius of \(k\); as all eigenvalues of \(\Frob_x\) on \(\mathcal F^\vee_{\bar x}\) equal \(1\), the eigenvalues \(\gamma_1,\dots,\gamma_m\) of \(F_k\) on \(H^0(S_{\bar k},\mathcal F^\vee)\) satisfy \(\gamma_j^\delta=1\). Hence \(F_k\) acts on \(H^{2e}_c(S_{\bar k},\mathcal F)\) with the eigenvalues \((\#k)^e\gamma_j^{-1}\), of absolute value \((\#k)^e\). By the Grothendieck--Lefschetz trace formula and \cite[Th\'eor\`eme 3.3.1]{Deligne}, applied to \(\mathcal F\) and to \(\Qlb\) (recall that \(H^{2e}_c(S_{\bar k},\Qlb)=\Qlb(-e)\), \(S\) being geometrically connected),
\[
\begin{aligned}
r\cdot\#S(k_a)&=\sum_{x\in S(k_a)}\Tr(\Frob_{k_a,x}\mid\mathcal F)=(\#k_a)^e\sum_{j=1}^m\gamma_j^{-a}+O\bigl((\#k_a)^{e-1/2}\bigr),\\
\#S(k_a)&=(\#k_a)^e+O\bigl((\#k_a)^{e-1/2}\bigr).
\end{aligned}
\]
Dividing by \((\#k_a)^e\) and letting \(a\to\infty\) gives \(r\le m\). Hence \(m=r\), the vectors \(gv-v\) all vanish, and \(\pgeom(S)\) acts trivially on \(V\).
\end{proof}

\begin{corollary}\label{cor:E-structure}
Let \(n\ge1\). For \(0\le k\le\lfloor n/2\rfloor\) the graded piece \(\mathrm{Gr}^W_{2k}\mathcal E_n\) of the weight filtration of \(\mathcal E_n\) has rank \(\nu_{n,k}\), all its Frobenius eigenvalues at points \(h\in S_n(\Fq)\) are equal to \(q^k\), and \(\pgeom(S_n)\) acts trivially on it; the graded pieces of odd weight vanish. Consequently \(\pgeom(S_n)\) acts on the geometric generic fibre of \(\mathcal E_n\) through unipotent automorphisms, and for every smooth curve \(C\) over \(\Fpbar\) and every morphism \(f:C\to S_n\), the sheaf \(f^*\mathcal E_n\) is tamely ramified at every point of the smooth compactification of \(C\).
\end{corollary}

\begin{proof}
By \cite[Th\'eor\`eme 3.4.1]{Deligne}, as in the proof of Theorem~\ref{thm:package}(ii), \(\mathcal E_n\) has an increasing filtration by lisse subsheaves whose graded pieces \(\mathrm{Gr}^W_i\mathcal E_n\) are pure of weight \(i\), and at each \(h\in S_n(\Fq)\) it induces the weight filtration of the Frobenius module \(\mathcal E_{n,h}\). By Proposition~\ref{prop:lower}, the eigenvalues of \(\Frob_h\) on \(\mathcal E_{n,h}\) of absolute value \(q^{i/2}\) are all equal to \(q^k\) if \(i=2k\), with multiplicity \(\nu_{n,k}\), and there are none if \(i\) is odd. This gives the ranks and the eigenvalues, and Lemma~\ref{lem:constant}, applied to \(\mathrm{Gr}^W_{2k}\mathcal E_n\otimes\Qlb(k)\), shows that \(\pgeom(S_n)\) acts trivially on \(\mathrm{Gr}^W_{2k}\mathcal E_n\). In a basis of the geometric generic fibre adapted to the filtration, every element of \(\pgeom(S_n)\) is therefore block upper triangular with identity diagonal blocks, hence unipotent. Finally, if \(x\) is a point of the compactification of \(C\), the image of the wild inertia group \(P_x\) in \(\GL(\mathcal E_{n,\bar\eta})\) is finite, as in the proof of Theorem~\ref{thm:package}(i), and consists of unipotent automorphisms; so it is trivial.
\end{proof}

\subsection{Pointwise bounds and the complement of the torus}

\begin{corollary}\label{cor:pointwise}
For \(n\ge1\) and \(h\in(\Fq^\times)^n\),
\[
|\Sigma_n(h)|\le r(n)\,q^{(n-1)/2}+\sum_{k=1}^{\lfloor n/2\rfloor}\nu_{n,k}\,q^{k-1}\le(2^n-1)\,q^{(n-1)/2}.
\]
\end{corollary}

\begin{proof}
The first inequality follows from Corollary~\ref{cor:exact} and \(|\Tr(\Frob_h\mid\mathcal W_n)|\le r(n)q^{(n+1)/2}\), and the second from \(k-1<(n-1)/2\) for \(k\le\lfloor n/2\rfloor\) and \(r(n)+\sum_{k\ge1}\nu_{n,k}=2^n-1\).
\end{proof}

\begin{lemma}\label{lem:offtorus}
Let \(n\ge2\) and \(h\in\Fq^n\), and let \(r\) be the number of nonzero coordinates of \(h\) and \(h''\in(\Fq^\times)^r\) the vector of these coordinates. Then \(\Sigma_n(h)=(-1)^{n-r}\Sigma_r(h'')\) if \(r\ge1\) and \(\Sigma_n(h)=0\) if \(r=0\). Consequently
\[
\sum_{h\in\Fq^n\setminus(\Fq^\times)^n}|\Sigma_n(h)|^4\le\sum_{r=1}^{n-1}\binom nr(q-1)^r\bigl((2^r-1)\,q^{(r-1)/2}\bigr)^4=O_n\bigl(q^{3n-5}\bigr).
\]
\end{lemma}

\begin{proof}
If \(h_j=0\), the variable \(u_j\) does not occur in the constraint \(\langle h,u\rangle=-\lambda\); it ranges freely over \(\Fq^\times\) and contributes the factor \(\sum_{u_j\ne0}\psi_q(u_j^{-1})=-1\). If \(r=0\) the constraint reads \(0=-\lambda\), which has no solution. The bound follows from Corollary~\ref{cor:pointwise}; the term \(r=n-1\) is of order \(q^{(n-1)+2(n-2)}=q^{3n-5}\), and the others are smaller.
\end{proof}

\subsection{The case \(n=2\)}\label{subsec:n2}
For \(n=2\) the sums \(\Sigma_2(h)\) are twisted Kloosterman sums. Let \(\mathrm m:S_2\to\Gm\) and \(\mathrm s:S_2\to\mathbb A^1\) be the morphisms \(\mathrm m(h)=h_1h_2/\lambda^2\) and \(\mathrm s(h)=-(h_1+h_2)/\lambda\), and write \(\Lpsi{\mathrm s}=\Lpsi{-(h_1+h_2)/\lambda}\).

\begin{proposition}\label{prop:n2-kl}
Let \(\Fq\) be a finite extension of \(\Fp\).
\begin{enumerate}[label=\textup{(\roman*)}]
\item For every \(h\in(\Fq^\times)^2\),
\[
\Sigma_2(h)=\psi_q\bigl(\mathrm s(h)\bigr)\,K_q\bigl(\mathrm m(h)\bigr)-1=\psi_q\Bigl(-\frac{h_1+h_2}\lambda\Bigr)\,K_q\Bigl(\frac{h_1h_2}{\lambda^2}\Bigr)-1 .
\]
Moreover \(\Sigma_2(c,0)=\Sigma_2(0,c)=-\psi_q(-c/\lambda)\) for \(c\in\Fq^\times\), and \(\Sigma_2(0,0)=0\).
\item There is an isomorphism of lisse sheaves on \(S_2\) over \(\Fp\)
\[
\mathcal W_2\cong\Lpsi{\mathrm s}\otimes\mathrm m^*\Kl\otimes\Qlb(-1),
\]
and \(\mathcal W_2\) is geometrically irreducible.
\end{enumerate}
\end{proposition}

\begin{proof}
(i) Let \(h\in(\Fq^\times)^2\). The substitution \(v_i=h_iu_i\) gives \(\Sigma_2(h)=\sum\psi_q(h_1/v_1+h_2/v_2)\), the sum being over the pairs \((v_1,v_2)\in(\Fq^\times)^2\) with \(v_1+v_2=-\lambda\). These are the pairs \(\bigl(-\lambda/(1+y),-\lambda y/(1+y)\bigr)\) with \(y\in\Fq\setminus\{0,-1\}\), and for them
\[
\frac{h_1}{v_1}+\frac{h_2}{v_2}=-\frac{(1+y)\,h_1}{\lambda}-\frac{(1+y)\,h_2}{\lambda y}=\mathrm s(h)+ay+\frac by,\qquad a=-\frac{h_1}\lambda,\quad b=-\frac{h_2}\lambda .
\]
Here \(a+b=\mathrm s(h)\) and \(ab=\mathrm m(h)\). The substitution \(y=bx\) shows that \(\sum_{y\in\Fq^\times}\psi_q(ay+b/y)=K_q(ab)\), and the term \(y=-1\) of this sum is \(\psi_q(-a-b)\); hence
\[
\Sigma_2(h)=\psi_q(a+b)\bigl(K_q(ab)-\psi_q(-a-b)\bigr)=\psi_q\bigl(\mathrm s(h)\bigr)K_q\bigl(\mathrm m(h)\bigr)-1 .
\]
The values off the torus are given by Lemma~\ref{lem:offtorus}, since \(\Sigma_1(c)=\psi_q(-c/\lambda)\).

(ii) For \(h\in S_2(\Fq)\), Corollary~\ref{cor:exact}(i) with \(n=2\), where \(\nu_{2,1}=1\), gives \(\Tr(\Frob_h\mid\mathcal W_2)=-q\Sigma_2(h)-q\), which by (i) and Theorem~\ref{thm:kl}(i) equals \(\psi_q(\mathrm s(h))\cdot\bigl(-K_q(\mathrm m(h))\bigr)\cdot q\), the trace of \(\Frob_h\) on \(\mathcal G:=\Lpsi{\mathrm s}\otimes\mathrm m^*\Kl\otimes\Qlb(-1)\). So the lisse sheaves \(\mathcal W_2\) and \(\mathcal G\) have the same Frobenius traces at all closed points of \(S_2\), and by the Chebotarev density theorem \cite{Serre-Zeta} and the Brauer--Nesbitt theorem their semisimplifications are isomorphic as representations of \(\pi_1(S_2)\). The sheaf \(\Kl\) is irreducible even as a representation of the inertia group \(I_\infty\): by Theorem~\ref{thm:kl}(iii), a one-dimensional \(I_\infty\)-stable subspace would have break \(\frac12\), hence the Swan conductor \(\frac12\), which is not an integer \cite[Ch.~1]{Katz-Gauss}. The morphism \(\mathrm m\) has the section \(x\mapsto(\lambda^2x,1)\), so the map \(\pgeom(S_2)\to\pgeom(\Gm)\) induced by \(\mathrm m\) is surjective, and \(\mathrm m^*\Kl\) is geometrically irreducible; so is its twist \(\mathcal G\) by a sheaf of rank one. Since the semisimplification of \(\mathcal W_2\) is thus irreducible, \(\mathcal W_2\) itself is irreducible, and \(\mathcal W_2\cong\mathcal G\).
\end{proof}

\begin{remark}\label{rem:n2-kl}
Proposition~\ref{prop:n2-kl} reduces several results below, in the case \(n=2\), to properties of \(\Kl\).
\begin{enumerate}[label=\textup{(\alph*)}]
\item Let \(c\in\Fpbar^\times\). On the coordinate line \(y\mapsto(y,c)\) the sheaf \(\mathcal W_2\) restricts to \(\Lpsi{-(y+c)/\lambda}\otimes[\times c/\lambda^2]^*\Kl\otimes\Qlb(-1)\). The first factor is lisse on \(\mathbb A^1\), so by Theorem~\ref{thm:kl}(ii) and the remark following it the local monodromy at \(y=0\) is unipotent, with a single Jordan block of size two; this is Theorem~\ref{thm:summary}(iv). On a ray \(x\mapsto(xa_1,xa_2)\), \(a\in(\Fpbar^\times)^2\), the sheaf \(\mathcal W_2\) restricts to \(\Lpsi{-(a_1+a_2)x/\lambda}\otimes f^*\Kl\otimes\Qlb(-1)\), where \(f(x)=a_1a_2x^2/\lambda^2\). The inertia group at \(x=0\) maps onto the subgroup of index two of the tame inertia group of \(\Kl\) at \(0\) if \(p\) is odd, and onto that whole group if \(p=2\), when \(f\) is purely inseparable; in both cases the image contains the square of a generator, which acts on \(\Kl\) by a unipotent automorphism different from the identity. So the local monodromy of \(\mathcal W_2\) at the origin of a ray is unipotent and nontrivial; this is Proposition~\ref{prop:n2}.
\item By Lemma~\ref{lem:detkl}, \(\det(\mathcal W_2)\cong\Lpsi{-2(h_1+h_2)/\lambda}\otimes\Qlb(-3)\), in agreement with Theorem~\ref{thm:summary}(i),(iv).
\item As \(\mathrm m\) has a section, \(\ggeom(\mathrm m^*\Kl)=\ggeom(\Kl)=\SL_2\) \cite[Theorem~11.1]{Katz-Gauss}. The geometric monodromy of \(\Lpsi{\mathrm s}\) is \(\mu_p\), and the kernel of its character is an open subgroup of finite index of \(\pgeom(S_2)\), whose image in \(\SL_2\) is Zariski dense, \(\SL_2\) being connected; hence \(\SL(W_2)\subseteq\ggeom(\mathcal W_2)\subseteq\mu_p\cdot\SL(W_2)\), and by (b), \(\ggeom(\mathcal W_2)\) is \(\SL(W_2)\) if \(p=2\) and \(\mu_p\cdot\SL(W_2)=\{g:\det(g)^p=1\}\) if \(p\) is odd. This is the case \(n=2\) of Theorem~C and of Theorem~\ref{thm:summary}.
\item If \(p=2\), the numbers \(x_h=-q^{-1/2}(\Sigma_2(h)+1)\) of Theorem~\ref{thm:ST} are \(-q^{-1/2}\psi_q(h_1+h_2)K_q(h_1h_2)\). For \(x\in\Fq^\times\) the signs \(\psi_q(h_1+h_2)\) at the \(q-1\) points with \(h_1h_2=x\) have sum \(K_q(x)\) (proof of Corollary~\ref{cor:n2-exact}), and the Sato--Tate measure is symmetric; so the case \(n=2\) of Theorem~\ref{thm:ST} also follows from the equidistribution of the numbers \(q^{-1/2}K_q(x)\), \(x\in\Fq^\times\), with respect to the Sato--Tate measure, due to Katz \cite{Katz-Gauss}.
\item The fourth moment for \(n=2\) is computed exactly in Corollary~\ref{cor:n2-exact}.
\end{enumerate}
Apart from Corollary~\ref{cor:n2-exact} and the first proof of Theorem~\ref{thm:summary}(iv), the arguments below do not use Proposition~\ref{prop:n2-kl}; for \(n=2\), (a)--(d) confirm them independently.
\end{remark}
\section{The fourth moment}\label{sec:fourth}

\subsection{Decoupling}
For \(t=(t_1,t_2,t_3,t_4)\in\Fq^4\) put
\[
V(t):=\sum_{h\in\Fq}\prod_{j=1}^4K(t_jh),\qquad L(t):=\lambda(t_1+t_2-t_3-t_4).
\]

\begin{lemma}[Decoupling]\label{lem:decouple}
For every \(n\ge1\) and every finite field \(\Fq\) of characteristic \(p\),
\[
q^4\sum_{h\in\Fq^n}|\Sigma_n(h)|^4=\sum_{t\in\Fq^4}\psi_q\bigl(L(t)\bigr)\,V(t)^n .
\]
\end{lemma}

\begin{proof}
Formula \eqref{eq:t-sum} holds for every \(h\in\Fq^n\) (its proof did not use \(h_i\ne0\)). The complex conjugate of \(q\Sigma_n(h)\) is obtained from \eqref{eq:t-sum} by replacing \(\psi_q(\lambda t)\) with \(\psi_q(-\lambda t)\), because \(K\) is real-valued. Writing \(q^4|\Sigma_n(h)|^4\) as the product of two copies of \(q\Sigma_n(h)\), whose summation variables we call \(t_1,t_2\), and two copies of its complex conjugate, whose summation variables we call \(t_3,t_4\), we find that it equals
\[
\sum_{t\in\Fq^4}\psi_q\bigl(\lambda(t_1+t_2-t_3-t_4)\bigr)\prod_{i=1}^n\Bigl[K(t_1h_i)K(t_2h_i)K(t_3h_i)K(t_4h_i)\Bigr].
\]
Summing over \(h\in\Fq^n\), the coordinates \(h_i\) are independent and each occurs in exactly one bracket; the sum over \(\Fq^n\) therefore factors into \(n\) copies of \(V(t)\).
\end{proof}

The right-hand side involves the dimension \(n\) only through the exponent: the function \(V\) and the parameter space \(\Fq^4\) are attached to the degree of the moment, not to the family. The rest of this section analyses \(V\) on \(\Fq^4\).

\emph{Convention for the phase.} The conjugated factors of \(|\Sigma_n(h)|^4\) occupy positions \(3\) and \(4\), so that \(L(t)=\lambda(t_1+t_2-t_3-t_4)\); this convention is used throughout the paper, in particular in the definitions of \(D\) and \(Q\) below. If the conjugated factors are placed in positions \(2\) and \(4\) instead, as in parts of \cite{EbadiFull}, the phase becomes \(\lambda(t_1-t_2+t_3-t_4)\). The two conventions are exchanged by the transposition \(t_2\leftrightarrow t_3\), under which \(V\) is invariant. It carries \(D\) to the set of tuples of the forms \((a,a,b,b)\) and \((a,b,b,a)\), and \(Q\) to the set of tuples \((a,b,a,b)\) with \(a\ne b\) (all coordinates nonzero). All statements of this section are therefore unchanged, provided each set is read in the coordinates of its own convention. If \(p=2\) the two phases coincide, both being \(\lambda(t_1+t_2+t_3+t_4)\).

\begin{lemma}\label{lem:homog}
Let \(m\ge2\) and \(\tau\in(\Fq^\times)^m\). Put
\[
\Sigma^{(0)}(\tau):=\sum_{\substack{u\in(\Fq^\times)^m\\ \sum_i\tau_iu_i=0}}\psi_q\Bigl(\sum_iu_i^{-1}\Bigr).
\]
\begin{enumerate}[label=\textup{(\roman*)}]
\item \(\sum_{h\in\Fq}\prod_{i=1}^mK(\tau_ih)=q\,\Sigma^{(0)}(\tau)\).
\item \(\Sigma^{(0)}(\tau)=qN_0(\tau)-N(\tau)\), where \(N(\tau)\) is the number of points \([u]\in\mathbb P^{m-1}(\Fq)\) with all \(u_i\ne0\) and \(\sum_i\tau_iu_i=0\), and \(N_0(\tau)\) is the number of those points that also satisfy \(e_{m-1}(u)=0\), \(e_{m-1}\) being the elementary symmetric polynomial of degree \(m-1\).
\item \(V(ct)=V(t)\) for \(c\in\Fq^\times\) and \(t\in\Fq^4\), and \(V(t)\in\mathbb Z\).
\end{enumerate}
\end{lemma}

\begin{proof}
(i) Expand each \(K\) and sum over \(h\) first. (ii) The summation set is stable under \(u\mapsto\mu u\) (\(\mu\in\Fq^\times\)), which acts freely, and \(f(u):=\sum_iu_i^{-1}\) satisfies \(f(\mu u)=\mu^{-1}f(u)\). On an orbit, \(\sum_{\mu\ne0}\psi_q(\mu^{-1}f(u))\) equals \(q-1\) if \(f(u)=0\) and \(-1\) otherwise; and for \(u\) with all coordinates nonzero, \(f(u)=0\) if and only if \(e_{m-1}(u)=0\). (iii) The substitution \(h\mapsto ch\) gives \(V(ct)=V(t)\). If all \(t_j\ne0\), then \(V(t)=q\Sigma^{(0)}(t)\in\mathbb Z\) by (i) and (ii). If exactly \(r<4\) coordinates are nonzero, then \(V(t)=K(0)^{4-r}\sum_h\prod_{t_j\ne0}K(t_jh)\), which equals \(q\) for \(r=0\), equals \(-\sum_{c\in\Fq}K(c)=0\) for \(r=1\) (Lemma~\ref{lem:kloost}(i) below), and is an integer by (i) and (ii) for \(r\ge2\).
\end{proof}

We decompose \(\Fq^4=Z\sqcup D\sqcup Q\sqcup\Omega\), where
\begin{align*}
Z&=\{t:\ t_1t_2t_3t_4=0\},\\
D&=\{t\notin Z:\ \{t_3,t_4\}=\{t_1,t_2\}\ \text{as multisets}\},\\
Q&=\{t\notin Z\cup D:\ t_1=t_2\ \text{and}\ t_3=t_4\}=\{(a,a,b,b):\ a,b\in\Fq^\times,\ a\ne b\},
\end{align*}
and \(\Omega\) is the complement of \(Z\cup D\cup Q\). All four sets are stable under scalar multiplication. On \(D\) one has \(L(t)=0\).

\begin{lemma}[Projective rays]\label{lem:rays}
Let \(\mathcal U\subseteq\Fq^4\setminus\{0\}\) be stable under \(\Fq^\times\) and let \(f:\mathcal U\to\mathbb C\) satisfy \(f(ct)=f(t)\) for \(c\in\Fq^\times\). Write \(\bar{\mathcal U}\subseteq\mathbb P^3(\Fq)\) for the image of \(\mathcal U\). Then
\[
\sum_{t\in\mathcal U}\psi_q\bigl(L(t)\bigr)f(t)=q\sum_{\substack{[t]\in\bar{\mathcal U}\\ L(t)=0}}f(t)-\sum_{[t]\in\bar{\mathcal U}}f(t).
\]
\end{lemma}

\begin{proof}
Since \(L(ct)=cL(t)\), the sum over the ray through \(t\) is \(f(t)\sum_{c\ne0}\psi_q(cL(t))\), which is \((q-1)f(t)\) if \(L(t)=0\) and \(-f(t)\) otherwise.
\end{proof}

\subsection{Kloosterman evaluations}
If \(p\) is odd, let \(\chi\) denote the quadratic character of \(\Fq^\times\), extended by \(\chi(0)=0\); we use the classical evaluation \(\sum_{x\in\Fq}\chi(\alpha x^2+\beta x+\gamma)=-\chi(\alpha)\) for \(\alpha\ne0\) and \(\beta^2-4\alpha\gamma\ne0\). If \(p=2\), then \(\lambda=1\) and \(\psi_q(x)=(-1)^{\Tr_{\Fq/\mathbb F_2}(x)}\), \(\psi\) being the only nontrivial character of \(\mathbb F_2\).

\begin{lemma}\label{lem:kloost}
Let \(a,b\in\Fq^\times\) and \(\theta=a/b\).
\begin{enumerate}[label=\textup{(\roman*)}]
\item \(\sum_{c\in\Fq}K(c)=0\); hence \(\sum_{c\ne0}K(c)=1\).
\item \(\sum_{h\in\Fq}K(ah)K(bh)\) equals \(q(q-1)\) if \(\theta=1\) and \(-q\) if \(\theta\ne1\).
\item If \(p\) is odd, then
\[
\sum_{h\in\Fq}K(ah)^2K(bh)^2=
\begin{cases}
2q^3-3q^2-3q,&\theta=1,\\
q^3-(2+\chi(\theta))q^2-3q,&\theta\ne1;
\end{cases}
\]
if \(p=2\), then
\[
\sum_{h\in\Fq}K(ah)^2K(bh)^2=
\begin{cases}
2q^3-2q^2-3q,&\theta=1,\\
q^3-2q^2-3q,&\theta\ne1.
\end{cases}
\]
\item \(\sum_{c\in\Fq^\times}K(c)^4\) equals \(2q^3-3q^2-3q-1\) if \(p\) is odd, and \(2q^3-2q^2-3q-1\) if \(p=2\).
\end{enumerate}
\end{lemma}

\begin{proof}
(i) \(\sum_cK(c)=\sum_{x\ne0}\psi_q(x^{-1})\sum_c\psi_q(cx)=0\), and \(K(0)=-1\).

(ii) By orthogonality the sum is \(q\sum\psi_q(x^{-1}+y^{-1})\) over \(x,y\ne0\) with \(ax+by=0\), that is, \(q\sum_{x\ne0}\psi_q\bigl((1-\theta^{-1})x^{-1}\bigr)\), which is \(q(q-1)\) if \(\theta=1\) and \(-q\) otherwise.

(iii) In every characteristic, expanding \(K(ah)^2K(bh)^2\) and summing over \(h\) gives
\[
\sum_hK(ah)^2K(bh)^2=q\sum_{s\in\Fq}\Upsilon(s)\Upsilon(-\theta s),\qquad \Upsilon(c):=\sum_{\substack{x,y\ne0\\x+y=c}}\psi_q(x^{-1}+y^{-1}),
\]
because \(a(x+y)+b(z+w)=0\) means \(z+w=-\theta(x+y)\). We have \(\Upsilon(0)=q-1\), since \(y=-x\) gives \(x^{-1}+y^{-1}=0\). Suppose first that \(p\) is odd. For \(c\ne0\), putting \(x=cu\) gives \(x^{-1}+y^{-1}=\bigl(cu(1-u)\bigr)^{-1}\) with \(u\ne0,1\); for \(m\ne0\) the equation \(u(1-u)=m\) has \(1+\chi(1-4m)\) solutions, all different from \(0,1\). Hence
\[
\Upsilon(c)=\sum_{m\ne0}\bigl(1+\chi(1-4m)\bigr)\psi_q\bigl((cm)^{-1}\bigr)=-1+\Xi(c),\qquad \Xi(c):=\sum_{m\ne0}\chi\bigl(m(m-4)\bigr)\psi_q(m/c),
\]
where we substituted \(m\mapsto(cm)^{-1}\) in the second sum and used \(\chi(1-4/(cm))=\chi\bigl(cm(cm-4)\bigr)\), then replaced \(cm\) by \(m\). Since \(m^2-4m\) has nonzero discriminant, \(\sum_{m\ne0}\chi(m^2-4m)=\sum_{m\in\Fq}\chi(m^2-4m)=-1\), so \(\sum_{s\ne0}\Xi(s)=\sum_{m\ne0}\chi(m^2-4m)\sum_{s\ne0}\psi_q(m/s)=1\), and likewise \(\sum_{s\ne0}\Xi(-\theta s)=1\). Next, using \(\sum_{s\ne0}\psi_q(A/s)=-1+q\mathbf 1_{A=0}\),
\[
\begin{aligned}
\sum_{s\ne0}\Xi(s)\Xi(-\theta s)&=-\Bigl(\sum_{m\ne0}\chi(m^2-4m)\Bigr)^2+q\sum_{m\ne0}\chi\bigl(m(m-4)\cdot\theta m(\theta m-4)\bigr)\\
&=-1+q\chi(\theta)\sum_{m\ne0}\chi\bigl((m-4)(\theta m-4)\bigr).
\end{aligned}
\]
The quadratic \((m-4)(\theta m-4)=\theta m^2-4(1+\theta)m+16\) has discriminant \(16(1-\theta)^2\). If \(\theta\ne1\) the sum over all \(m\in\Fq\) is \(-\chi(\theta)\), and removing \(m=0\) (where \(\chi(16)=1\)) gives \(-\chi(\theta)-1\); hence \(\sum_{s\ne0}\Xi(s)\Xi(-\theta s)=-1-q-q\chi(\theta)\). If \(\theta=1\) the quadratic is \((m-4)^2\), the sum over \(m\ne0\) is \(q-2\), and \(\sum_{s\ne0}\Xi(s)\Xi(-s)=q^2-2q-1\). Finally
\[
\sum_{s\ne0}\Upsilon(s)\Upsilon(-\theta s)=(q-1)-2+\sum_{s\ne0}\Xi(s)\Xi(-\theta s),
\]
which is \(-4-q\chi(\theta)\) if \(\theta\ne1\) and \(q^2-q-4\) if \(\theta=1\). Adding \(\Upsilon(0)^2=(q-1)^2\) and multiplying by \(q\) gives (iii) for \(p\) odd.

Now let \(p=2\), so that \(-\theta s=\theta s\). For \(c\ne0\) we again put \(x=cu\), so that \(x^{-1}+y^{-1}=(cm)^{-1}\) with \(m=u(1-u)=u^2+u\ne0\). The map \(u\mapsto u^2+u\) is \(\mathbb F_2\)-linear on \(\Fq\) with kernel \(\{0,1\}\), and its image is the kernel of \(\Tr_{\Fq/\mathbb F_2}\), which contains it (as \(\Tr(u^2)=\Tr(u)\)) and has the same cardinality \(q/2\). Hence the equation \(u^2+u=m\) has \(1+\psi_q(m)\) solutions, all different from \(0\) and \(1\), and
\[
\Upsilon(c)=\sum_{m\ne0}\bigl(1+\psi_q(m)\bigr)\psi_q\bigl((cm)^{-1}\bigr)=-1+\sum_{m\ne0}\psi_q\bigl(m+c^{-1}m^{-1}\bigr)=-1+K(c^{-1}).
\]
Using (i) twice and then (ii) with \((a,b)=(1,\theta^{-1})\),
\[
\sum_{s\ne0}\Upsilon(s)\Upsilon(\theta s)=(q-1)-2+\sum_{c\ne0}K(c)K(\theta^{-1}c)=q-3+\Bigl(\sum_{c\in\Fq}K(c)K(\theta^{-1}c)-1\Bigr),
\]
which is \(q^2-4\) if \(\theta=1\) and \(-4\) if \(\theta\ne1\). Adding \(\Upsilon(0)^2=(q-1)^2\) and multiplying by \(q\) gives (iii) for \(p=2\).

(iv) Take \(a=b\) in (iii) and remove the term \(h=0\), which equals \(K(0)^4=1\).
\end{proof}

Put
\[
X_0:=2q^3-3q^2-3q,\qquad X_+:=q^3-3q^2-3q,\qquad X_-:=q^3-q^2-3q,
\]
so that, if \(p\) is odd, by Lemma~\ref{lem:kloost}(iii) and the symmetry of \(V\), \(V(a,b,a,b)=V(a,b,b,a)=V(a,a,b,b)\) equals \(X_0\) if \(a=b\), \(X_+\) if \(a\ne b\) and \(\chi(a/b)=1\), and \(X_-\) if \(\chi(a/b)=-1\). If \(p=2\) we put
\[
X_0':=2q^3-2q^2-3q,\qquad X':=q^3-2q^2-3q,
\]
so that \(V(a,b,a,b)=V(a,b,b,a)=V(a,a,b,b)\) equals \(X_0'\) if \(a=b\) and \(X'\) if \(a\ne b\).

\subsection{The diagonal and two-pair strata}

\begin{proposition}\label{prop:D}
For every \(n\ge1\),
\[
\sum_{t\in D}\psi_q\bigl(L(t)\bigr)V(t)^n=
\begin{cases}
(q-1)\Bigl[X_0^n+(q-3)X_+^n+(q-1)X_-^n\Bigr],&p\ \text{odd},\\[2pt]
(q-1)\Bigl[X_0'^{\,n}+2(q-2)X'^{\,n}\Bigr],&p=2.
\end{cases}
\]
\end{proposition}

\begin{proof}
\(D\) is the disjoint union of \(D_0=\{(a,a,a,a)\}\), \(D_1=\{(a,b,a,b):a\ne b\}\) and \(D_2=\{(a,b,b,a):a\ne b\}\), with \(a,b\in\Fq^\times\), and \(L=0\) on \(D\). Let \(p\) be odd. On \(D_0\) we have \(V=X_0\), which gives \((q-1)X_0^n\). On \(D_1\) and on \(D_2\), writing \(a=\theta b\), each \(\theta\in\Fq^\times\setminus\{1\}\) occurs for \(q-1\) pairs \((a,b)\); exactly \((q-3)/2\) of these \(\theta\) satisfy \(\chi(\theta)=1\) and \((q-1)/2\) satisfy \(\chi(\theta)=-1\). Hence each of \(D_1,D_2\) contributes \((q-1)\bigl[\frac{q-3}2X_+^n+\frac{q-1}2X_-^n\bigr]\). If \(p=2\), then \(V=X_0'\) on \(D_0\) and \(V=X'\) on \(D_1\cup D_2\), and each of \(D_1,D_2\) has \((q-1)(q-2)\) elements.
\end{proof}

\begin{proposition}\label{prop:P}
For every \(n\ge1\),
\[
\sum_{t\in Q}\psi_q\bigl(L(t)\bigr)V(t)^n=
\begin{cases}
(2-q)\,\dfrac{X_+^n+X_-^n}2+\dfrac{X_+^n-X_-^n}2,&p\ \text{odd},\\[6pt]
(q-1)(q-2)\,X'^{\,n},&p=2.
\end{cases}
\]
In particular, if \(p\) is odd and \(n=4\), the sum equals
\[
-q^{13}+10q^{12}-38q^{11}+44q^{10}+101q^9-226q^8-210q^7+324q^6+459q^5+162q^4 .
\]
\end{proposition}

\begin{proof}
On \(t=(a,a,b,b)\in Q\) we have \(L(t)=2\lambda(a-b)\). If \(p=2\), then \(L=0\) on \(Q\), \(V=X'\) on \(Q\), and \(Q\) has \((q-1)(q-2)\) elements. Let \(p\) be odd. Then \(V(t)^n=A_n+B_n\chi(a/b)\) with \(A_n=\frac12(X_+^n+X_-^n)\) and \(B_n=\frac12(X_+^n-X_-^n)\). The two character sums are
\[
\kappa_1:=\sum_{\substack{a,b\ne0\\a\ne b}}\psi_q\bigl(2\lambda(a-b)\bigr)=\Bigl|\sum_{a\ne0}\psi_q(2\lambda a)\Bigr|^2-(q-1)=2-q,
\]
using \(2\lambda\ne0\), and, writing \(a=\theta b\),
\[
\kappa_2:=\sum_{\substack{a,b\ne0\\a\ne b}}\chi(a/b)\psi_q\bigl(2\lambda(a-b)\bigr)=\sum_{\theta\ne0,1}\chi(\theta)\sum_{b\ne0}\psi_q\bigl(2\lambda b(\theta-1)\bigr)=-\sum_{\theta\ne0,1}\chi(\theta)=1 .
\]
The sum is \(A_n\kappa_1+B_n\kappa_2\). The expression for \(n=4\) is obtained by expanding this polynomial in \(q\).
\end{proof}

\begin{remark}[The constants \(2\) and \(3\)]\label{rem:two}
For \(p\) odd, the leading coefficient in Theorem~\ref{thm:fourth} below is the number of matchings of the two unconjugated positions \(\{1,2\}\) with the two conjugated positions \(\{3,4\}\), that is, the number of two-dimensional components \(D_1,D_2\) of the diagonal stratum, on all of which the phase \(L\) vanishes; the one-dimensional component \(D_0\) contributes only to lower order, and the third pairing \(\{1,2\}\{3,4\}\) produces \(Q\), on which \(\psi_q(L)\) oscillates. This is the character-sum counterpart of \(\dim\End_{\GL(W)}(W^{\otimes2})=2\), that is, of the decomposition \(W^{\otimes2}=\Sym^2W\oplus\Lambda^2W\) into two non-isomorphic irreducible constituents. For \(p=2\) the phase \(L(a,a,b,b)=2\lambda(a-b)\) vanishes identically, so the third pairing contributes at the same order as the other two, and the leading coefficient becomes \(3\). On the geometric side this reflects the self-duality of \(\mathcal W_n\) in characteristic \(2\) (Proposition~\ref{prop:selfdual}): if \(W\) carries a nondegenerate invariant symmetric or alternating form, then \(W\cong W^\vee\), one of \(\Sym^2W\), \(\Lambda^2W\) contains the trivial representation, and, if \(\dim W\ge3\), \(W^{\otimes2}\) has at least three irreducible constituents.
\end{remark}

\subsection{The degenerate stratum}

\begin{proposition}\label{prop:Z}
For every \(n\ge1\),
\[
\begin{aligned}
\Bigl|\sum_{t\in Z}\psi_q\bigl(L(t)\bigr)V(t)^n\Bigr|&\le q^n+6(q-1)\bigl(q(q-1)\bigr)^n+6(q-1)(q-2)q^n+4(q-1)^3(q^2+2q)^n\\
&=O_n\bigl(q^{2n+3}\bigr).
\end{aligned}
\]
\end{proposition}

\begin{proof}
Let \(r\) be the number of nonzero coordinates of \(t\in Z\), so \(r\le3\), and recall \(K(0)=-1\). If \(r=0\), \(V(0)=q\). If \(r=1\), \(V(t)=-\sum_{c\in\Fq}K(c)=0\) by Lemma~\ref{lem:kloost}(i). If \(r=2\), with nonzero entries \(a,b\), Lemma~\ref{lem:kloost}(ii) gives \(V(t)=q(q-1)\) if \(a=b\) and \(V(t)=-q\) if \(a\ne b\); there are \(6(q-1)\) and \(6(q-1)(q-2)\) such tuples respectively. If \(r=3\), with nonzero entries \(\tau=(\tau_1,\tau_2,\tau_3)\), Lemma~\ref{lem:homog} gives \(V(t)=-q\Sigma^{(0)}(\tau)=-q\bigl(qN_0(\tau)-N(\tau)\bigr)\). Here \(N(\tau)=q-2\): the line \(\sum_i\tau_iu_i=0\) in \(\mathbb P^2\) has \(q+1\) points and meets each coordinate line \(\{u_i=0\}\) in exactly one point, these three points being distinct because no point of the line has two vanishing coordinates. Further \(0\le N_0(\tau)\le2\), because \(N_0(\tau)\) counts points of the line on the conic \(e_2(u)=u_1u_2+u_1u_3+u_2u_3=0\), and the restriction of \(e_2\) to the line is a nonzero binary quadratic form: the line contains the point \((\tau_2:-\tau_1:0)\), at which \(e_2=-\tau_1\tau_2\ne0\). This argument is valid in every characteristic. Therefore \(-q+2\le\Sigma^{(0)}(\tau)\le q+2\), so \(|V(t)|\le q^2+2q\); there are \(4(q-1)^3\) such tuples.
\end{proof}

\subsection{The generic stratum}

\begin{lemma}\label{lem:planecubic}
Let \(t\in\Fq^4\) with all \(t_i\ne0\). Let \(H_t=\{\sum_it_iu_i=0\}\subset\mathbb P^3\), let \(\mathcal S=\{e_3=0\}\subset\mathbb P^3\) with \(e_3=\sum_{i<j<k}u_iu_ju_k\), put \(C_t=H_t\cap\mathcal S\) and \(a_t=q+1-\#C_t(\Fq)\). Then
\[
V(t)=-q^2a_t-2q^2-3q .
\]
\end{lemma}

\begin{proof}
By Lemma~\ref{lem:homog}, \(V(t)=q^2N_0(t)-qN(t)\) with \(N,N_0\) as there (\(m=4\)). The plane \(H_t\) has \(q^2+q+1\) points; it meets each coordinate plane \(\{u_i=0\}\) in a line (since the remaining \(t_j\) are not all zero), any two of these lines in exactly one point \(\{u_i=u_j=0\}\cap H_t\) (the remaining two coordinates satisfy \(t_ku_k+t_lu_l=0\) with \(t_k,t_l\ne0\)), and no three of them meet in \(H_t\), because a point of \(H_t\) with three vanishing coordinates would have the fourth vanishing as well. Inclusion--exclusion gives \(N(t)=(q^2+q+1)-4(q+1)+6=q^2-3q+3\). Next, \(N_0(t)\) counts the points of \(C_t\) with all coordinates nonzero. If \(u\in C_t\) has \(u_1=0\), then \(e_3(u)=u_2u_3u_4=0\), so a second coordinate vanishes; three cannot vanish. Each of the six points of \(H_t\) with exactly two vanishing coordinates lies on \(\mathcal S\). Hence \(N_0(t)=\#C_t(\Fq)-6=q-5-a_t\), and \(V(t)=q^2(q-5-a_t)-q(q^2-3q+3)=-q^2a_t-2q^2-3q\).
\end{proof}

\begin{proposition}[The Cayley cubic]\label{prop:cayley}
Let \(\Bbbk\) be an algebraically closed field of arbitrary characteristic, and let \(\mathcal S_{\Bbbk}=\{e_3=0\}\subset\mathbb P^3_{\Bbbk}\).
\begin{enumerate}[label=\textup{(\roman*)}]
\item \(\mathcal S_{\Bbbk}\) is singular exactly at the four coordinate points.
\item Every line contained in \(\mathcal S_{\Bbbk}\) is one of the following nine lines: the six coordinate edges \(\{u_i=u_j=0\}\), and the three lines
\[
\{u_i+u_j=0,\ u_k+u_l=0\},\qquad\{i,j\}\sqcup\{k,l\}=\{1,2,3,4\}.
\]
In particular, every line on \(\mathcal S_{\Bbbk}\) is defined over the prime field.
\item Let \(t\in\Bbbk^4\) with \(t_1t_2t_3t_4\ne0\). The plane \(\{\sum_it_iu_i=0\}\subset\mathbb P^3_{\Bbbk}\) contains a line of \(\mathcal S_{\Bbbk}\) if and only if \(t\) has one of the forms \((a,a,b,b)\), \((a,b,a,b)\), \((a,b,b,a)\). In particular, if \(t\in\Fq^4\) has all \(t_i\ne0\) and \(\Fqbar\) is an algebraic closure of \(\Fq\), then \(C_t\otimes\Fqbar\) contains a line if and only if \(t\in D\cup Q\).
\end{enumerate}
\end{proposition}

\begin{proof}
All points and lines below are taken over \(\Bbbk\); no hypothesis on the characteristic is used.

(i) One has \(\partial e_3/\partial u_1=u_2u_3+u_2u_4+u_3u_4\), and similarly for the other coordinates. At a point with all \(u_i\ne0\) write \(e_3=u_1u_2u_3u_4\,f(u)\) with \(f(u)=\sum_iu_i^{-1}\); if \(e_3(u)=0\) then \(f(u)=0\), and then \(\partial e_3/\partial u_i\) equals \(-u_1u_2u_3u_4/u_i^2\), which is nonzero; so \(u\) is a smooth point. Let \(u\) be a singular point with some vanishing coordinate, say \(u_1=0\) after renumbering. Then \(\partial e_3/\partial u_2=u_3u_4=0\), say \(u_3=0\), and then \(\partial e_3/\partial u_1=u_2u_4=0\); so \(u\) has three vanishing coordinates, that is, \(u\) is a coordinate point. Conversely, at the coordinate point \((1:0:0:0)\) the equation in the affine chart \(u_1=1\) is \(e_2(u_2,u_3,u_4)+u_2u_3u_4=0\), which has no linear term; so each coordinate point is singular.

(ii) \emph{Lines through a coordinate point.} By symmetry we may consider the coordinate point \((1:0:0:0)\). Every line through it meets the plane \(\{u_1=0\}\) in a single point \((0:v_2:v_3:v_4)\), and consists of the points \((1:sv_2:sv_3:sv_4)\), \(s\in\Bbbk\), together with \((0:v_2:v_3:v_4)\). On it \(e_3=s^2e_2(v_2,v_3,v_4)+s^3v_2v_3v_4\), and at the point \((0:v_2:v_3:v_4)\) one has \(e_3=v_2v_3v_4\). The line lies on \(\mathcal S_{\Bbbk}\) if and only if \(e_2(v_2,v_3,v_4)=v_2v_3v_4=0\), that is, if and only if two of \(v_2,v_3,v_4\) vanish. Thus the lines on \(\mathcal S_{\Bbbk}\) through a coordinate point are the three coordinate edges through it.

\emph{Lines avoiding the coordinate points.} Such a line \(R\) is contained in no coordinate plane, since \(\mathcal S_{\Bbbk}\cap\{u_i=0\}\) is a union of edges, which pass through coordinate points. Hence \(R\) meets each plane \(\{u_i=0\}\) in exactly one point; that point lies on \(\mathcal S_{\Bbbk}\), so it has a second vanishing coordinate, and exactly two because it is not a coordinate point. Each such point lies in exactly two coordinate planes, and each coordinate plane contains exactly one of them; so \(R\) is spanned by two points with complementary pairs of vanishing coordinates, say \((A:B:0:0)\) and \((0:0:C:D)\) up to renumbering, with \(ABCD\ne0\). On the point \((sA:sB:rC:rD)\),
\[
e_3=rs\bigl[AB(C+D)\,s+CD(A+B)\,r\bigr],
\]
which vanishes identically if and only if \(A+B=C+D=0\), that is, if and only if \(R\) is the line \(\{u_1+u_2=0,\ u_3+u_4=0\}\). The nine lines are defined by equations with coefficients in \(\{0,1\}\).

(iii) By (ii), the plane \(\{\sum_it_iu_i=0\}\) contains a line of \(\mathcal S_{\Bbbk}\) if and only if it contains one of the nine lines. It contains the edge \(\{u_i=u_j=0\}\) only if \(t_k=t_l=0\), which is excluded. It contains \(\{u_i+u_j=0,u_k+u_l=0\}\), parametrised by \((u_i,u_j,u_k,u_l)=(s,-s,r,-r)\), if and only if \(t_i=t_j\) and \(t_k=t_l\). For the three partitions \(\{i,j\}\sqcup\{k,l\}\) this means that \(t\) is of the form \((a,a,b,b)\), \((a,b,a,b)\) or \((a,b,b,a)\). If \(t\in\Fq^4\) has all \(t_i\ne0\), these are exactly the elements of \(D\cup Q\), by the definitions of \(D\) and \(Q\) (recall that \((a,a,a,a)\in D\)); and a line contained in \((H_t\cap\mathcal S)\otimes\Fqbar\) is a line of \(\mathcal S_{\Fqbar}\) contained in \(H_t\otimes\Fqbar\).
\end{proof}

\begin{lemma}\label{lem:hasse}
Let \(t\in\Omega\). Then \(C_t\) is a plane cubic curve whose equation is irreducible over the algebraic closure \(\Fqbar\) \textup{(}so \(C_t\) is geometrically irreducible\textup{)}, and \(|a_t|\le2\sqrt q\). If \(C_t\otimes\Fqbar\) is singular, then it has exactly one singular point, this point is \(\Fq\)-rational, and \(a_t\in\{-1,0,1\}\).
\end{lemma}

\begin{proof}
\emph{Irreducibility.} Fix an \(\Fq\)-linear isomorphism \(H_t\cong\mathbb P^2\), and let \(F\in\Fq[X,Y,Z]\) be the cubic form obtained by restricting \(e_3\) to \(H_t\), so that \(C_t=\{F=0\}\). If \(F\) were zero, the plane \(H_t\otimes\Fqbar\) would be contained in \(\mathcal S_{\Fqbar}\), and so would the infinitely many lines in it, contrary to Proposition~\ref{prop:cayley}(ii); so \(C_t\) is a plane cubic curve. Suppose that \(F\) factors nontrivially in \(\Fqbar[X,Y,Z]\), where \(\Fqbar\) is an algebraic closure of \(\Fq\). As \(\deg F=3\), some factor is a linear form, whose zero set is a line contained in \(C_t\otimes\Fqbar\). By Proposition~\ref{prop:cayley}(iii), applied with \(\Bbbk\) equal to the algebraic closure \(\Fqbar\), this forces \(t\in D\cup Q\), contradicting \(t\in\Omega\). Hence \(F\) is irreducible in \(\Fqbar[X,Y,Z]\), that is, absolutely irreducible; in particular \(C_t\) is geometrically irreducible and reduced.

\emph{The smooth case.} Suppose that \(C_t\otimes\Fqbar\) has no singular point. Then \(C_t\) is a smooth plane cubic, hence of genus \((3-1)(3-2)/2=1\), and it has the \(\Fq\)-rational point \((t_2:-t_1:0:0)\); so it is an elliptic curve over \(\Fq\), and Hasse's theorem \cite[Theorem V.1.1]{Silverman} gives \(|a_t|\le2\sqrt q\).

\emph{The singular case.} Recall that a point of the plane curve \(\{F=0\}\) is singular if and only if \(F\) and its three partial derivatives vanish there. This agrees with the Jacobian criterion in each affine chart, in every characteristic: in the chart \(Z=1\), say, Euler's identity \(XF_X+YF_Y+ZF_Z=3F\) gives \(F_Z(x,y,1)=3F(x,y,1)-xF_X(x,y,1)-yF_Y(x,y,1)\), so \(F_Z\) vanishes wherever \(F\), \(F_X\) and \(F_Y\) do. Let \(P\) be a singular point of \(C_t\otimes\Fqbar\), and choose homogeneous coordinates over \(\Fqbar\) with \(P=(0:0:1)\). Write \(F=\sum_{i=0}^3Z^{3-i}F_i(X,Y)\) with binary forms \(F_i\) of degree \(i\). Since \(P\) is a singular point, the affine equation \(F(x,y,1)=F_0+F_1(x,y)+F_2(x,y)+F_3(x,y)\) has no constant and no linear term; thus \(F_0=0\), \(F_1=0\) and
\[
F=Z\,F_2(X,Y)+F_3(X,Y).
\]
Here \(F_2\ne0\), for otherwise \(F=F_3(X,Y)\) would be a product of three linear forms over \(\Fqbar\), contradicting absolute irreducibility; and \(F_2\), \(F_3\) have no common linear factor, for such a factor would divide \(F\).

\emph{Uniqueness and rationality of \(P\).} Let \(P'=(a:b:c)\ne P\) be a point of \(C_t\otimes\Fqbar\), so that \((a,b)\ne(0,0)\). If \(P'\) were singular, then \(\partial F/\partial Z(P')=F_2(a,b)=0\), and then \(F_3(a,b)=F(P')-cF_2(a,b)=0\); the linear form \(bX-aY\) would divide both \(F_2\) and \(F_3\), which is impossible. So \(P\) is the only singular point of \(C_t\otimes\Fqbar\). Since \(F\) has coefficients in \(\Fq\), the Frobenius automorphism of \(C_t\otimes\Fqbar\) permutes its singular points; hence it fixes \(P\), and \(P\in C_t(\Fq)\). We may therefore choose the coordinates above over \(\Fq\), so that \(F_2,F_3\in\Fq[X,Y]\).

\emph{Counting points.} The points of \(C_t\otimes\Fqbar\) other than \(P\) are the points \((a:b:c)\) with \((a,b)\ne(0,0)\) and \(cF_2(a,b)+F_3(a,b)=0\). For \((a:b)\in\mathbb P^1(\Fqbar)\) with \(F_2(a,b)\ne0\) there is exactly one such point, namely \(c=-F_3(a,b)/F_2(a,b)\), and it is \(\Fq\)-rational if \((a,b)\in\Fq^2\); if \(F_2(a,b)=0\), then \(F_3(a,b)\ne0\) and there is none. Conversely, an \(\Fq\)-rational point \((a:b:c)\ne P\) of \(C_t\) has \((a:b)\in\mathbb P^1(\Fq)\). Hence \((a:b:c)\mapsto(a:b)\) is a bijection from \(C_t(\Fq)\setminus\{P\}\) onto the set of points of \(\mathbb P^1(\Fq)\) at which \(F_2\) does not vanish. The nonzero binary quadratic form \(F_2\) has \(\varepsilon\in\{0,1,2\}\) zeros in \(\mathbb P^1(\Fq)\) (these are the \(\Fq\)-rational tangent lines to \(C_t\) at \(P\)), so
\[
\#C_t(\Fq)=1+(q+1-\varepsilon),\qquad a_t=\varepsilon-1\in\{-1,0,1\}.
\]
In particular \(|a_t|\le1\le2\sqrt q\).
\end{proof}

\begin{proposition}\label{prop:G}
For every \(n\ge1\) and \(q\ge3\),
\[
\Bigl|\sum_{t\in\Omega}\psi_q\bigl(L(t)\bigr)V(t)^n\Bigr|\le3\cdot4^n\,q^{(5n+6)/2}.
\]
\end{proposition}

\begin{proof}
By Lemmas~\ref{lem:planecubic} and~\ref{lem:hasse}, for \(t\in\Omega\) we have \(|V(t)|=q^2|a_t+2+3/q|\le q^2(2\sqrt q+3)\le4q^{5/2}\). The set \(\Omega\) is stable under scalars and \(V\) is homogeneous of degree \(0\) (Lemma~\ref{lem:homog}(iii)). By Lemma~\ref{lem:rays}, the sum is bounded by
\[
\Bigl(q\,\#\{[t]\in\mathbb P^3(\Fq):L(t)=0\}+\#\mathbb P^3(\Fq)\Bigr)\bigl(4q^{5/2}\bigr)^n=(2q^3+2q^2+2q+1)\,4^nq^{5n/2},
\]
and \(2q^3+2q^2+2q+1\le3q^3\) for \(q\ge3\).
\end{proof}

The gain over the trivial bound \(O(q^{(5n+8)/2})\) comes from the oscillation of \(\psi_q(L)\) along the projective rays, on which \(V\) is constant; the bound for \(|a_t|\) is needed only to control \(V\) pointwise.

\subsection{The fourth moment}
Put
\[
c_p:=\begin{cases}2,&p\ \text{odd},\\3,&p=2,\end{cases}\qquad\qquad
b_p(n):=\begin{cases}2^n-4n-7,&p\ \text{odd},\\2^n-6n-9,&p=2.\end{cases}
\]

\begin{theorem}\label{thm:fourth}
Let \(n\ge3\) and let \(\Fq\) be a finite field of characteristic \(p\). Then
\begin{equation}\label{eq:fourth}
\sum_{h\in\Fq^n}|\Sigma_n(h)|^4=c_p\,q^{3n-2}+O_n\bigl(q^{3n-2-\delta(n)}\bigr),\qquad\delta(n)=\min\Bigl(1,\frac{n-2}2\Bigr).
\end{equation}
If \(n\ge5\), then
\begin{equation}\label{eq:fourth-second}
\sum_{h\in\Fq^n}|\Sigma_n(h)|^4=c_p\,q^{3n-2}+b_p(n)\,q^{3n-3}+O_n\bigl(q^{3n-3-\delta'(n)}\bigr),
\end{equation}
where \(\delta'(n)=\min\bigl(1,\frac{n-4}2\bigr)\).
Both statements remain true, with the same main terms, when \(h\) is restricted to \((\Fq^\times)^n\). The implied constants depend only on \(n\); in particular they do not depend on \(p\).
\end{theorem}

\begin{proof}
For \(q=2\) each statement concerns, for given \(n\), a single number, and it holds after enlarging the implied constant; so we may assume \(q\ge3\). By Lemma~\ref{lem:decouple}, \(q^4\sum_h|\Sigma_n(h)|^4\) is the sum of the contributions of \(D\), \(Q\), \(Z\) and \(\Omega\). Let \(p\) be odd. Expanding \(X_\pm^n\) and \(X_0^n\),
\[
X_+^n=q^{3n}-3nq^{3n-1}+O_n(q^{3n-2}),\qquad X_-^n=q^{3n}-nq^{3n-1}+O_n(q^{3n-2}),
\]
and \(X_0^n=2^nq^{3n}+O_n(q^{3n-1})\),
so Proposition~\ref{prop:D} gives
\[
(q-1)\Bigl[2q^{3n+1}+(2^n-4n-4)q^{3n}+O_n(q^{3n-1})\Bigr]=2q^{3n+2}+(2^n-4n-6)q^{3n+1}+O_n(q^{3n}),
\]
and Proposition~\ref{prop:P} gives \(-q^{3n+1}+O_n(q^{3n})\). Hence the strata \(D\cup Q\) contribute \(2q^{3n+2}+(2^n-4n-7)q^{3n+1}+O_n(q^{3n})\). If \(p=2\), Propositions~\ref{prop:D} and~\ref{prop:P} show that \(D\cup Q\) contributes \((q-1)\bigl[X_0'^{\,n}+3(q-2)X'^{\,n}\bigr]\); since \(X'^{\,n}=q^{3n}-2nq^{3n-1}+O_n(q^{3n-2})\) and \(X_0'^{\,n}=2^nq^{3n}+O_n(q^{3n-1})\), this equals
\[
(q-1)\Bigl[3q^{3n+1}+(2^n-6n-6)q^{3n}+O_n(q^{3n-1})\Bigr]=3q^{3n+2}+(2^n-6n-9)q^{3n+1}+O_n(q^{3n}).
\]
In both cases the strata \(D\cup Q\) contribute \(c_pq^{3n+2}+b_p(n)q^{3n+1}+O_n(q^{3n})\). By Propositions~\ref{prop:Z} and~\ref{prop:G}, the strata \(Z\) and \(\Omega\) contribute \(O_n(q^{2n+3})+O_n(q^{(5n+6)/2})\). Dividing by \(q^4\),
\[
\sum_h|\Sigma_n(h)|^4=c_pq^{3n-2}+b_p(n)q^{3n-3}+O_n\bigl(q^{3n-4}+q^{2n-1}+q^{(5n-2)/2}\bigr).
\]
Here \(2n-1\le3n-4\) for \(n\ge3\), and \((5n-2)/2=3n-2-\frac{n-2}2\). Since \(\delta(n)\le1\) and \(\delta(n)\le\frac{n-2}2\), every term after \(c_pq^{3n-2}\), including \(b_p(n)q^{3n-3}\), is \(O_n(q^{3n-2-\delta(n)})\); this gives \eqref{eq:fourth}. If \(n\ge5\), then \((5n-2)/2=3n-3-\frac{n-4}2\); since \(\delta'(n)\le1\) and \(\delta'(n)\le\frac{n-4}2\), the error term is \(O_n(q^{3n-3-\delta'(n)})\), which gives \eqref{eq:fourth-second}. The restriction to \((\Fq^\times)^n\) changes the sum by \(O_n(q^{3n-5})\) by Lemma~\ref{lem:offtorus}, which is smaller than both error terms. All the bounds used are explicit and valid for every \(q\ge3\), and their constants depend only on \(n\); so do the constants in the expansions above, which are coefficients of two explicit polynomials in \(q\) (one for \(p\) odd and one for \(p=2\)) depending only on \(n\).
\end{proof}

The following table records, for each stratum, its cardinality, the values of \(V\) on it, and its contribution to \(q^4\sum_h|\Sigma_n(h)|^4=\sum_t\psi_q(L(t))V(t)^n\), as obtained in Propositions~\ref{prop:D}, \ref{prop:P}, \ref{prop:Z} and~\ref{prop:G}; for \(D_1\cup D_2\) the expansion is that of \((q-1)\bigl[(q-3)X_+^n+(q-1)X_-^n\bigr]\) if \(p\) is odd and of \(2(q-1)(q-2)X'^{\,n}\) if \(p=2\), and it is the same in both cases. Here \(D_0,D_1,D_2\) are the components of \(D\) described in the proof of Proposition~\ref{prop:D}, and \(\#\Omega=(q-1)^4-(q-1)(3q-5)\).
\begin{center}
\small
\begin{tabular}{llll}
\toprule
stratum & elements & \(V(t)\) & contribution\\
\midrule
\(D_0\) & \(q-1\) & \(X_0\), resp.\ \(X_0'\) if \(p=2\) & \(2^nq^{3n+1}+O_n(q^{3n})\)\\
\(D_1\cup D_2\) & \(2(q-1)(q-2)\) & \(X_\pm\), resp.\ \(X'\) if \(p=2\) & \(2q^{3n+2}-(4n+6)q^{3n+1}+O_n(q^{3n})\)\\
\(Q\), \(p\) odd & \((q-1)(q-2)\) & \(X_\pm\) & \(-q^{3n+1}+O_n(q^{3n})\)\\
\(Q\), \(p=2\) & \((q-1)(q-2)\) & \(X'\) & \(q^{3n+2}-(2n+3)q^{3n+1}+O_n(q^{3n})\)\\
\(Z\) & \(q^4-(q-1)^4\) & \(|V(t)|\le q^2+2q\) & \(O_n(q^{2n+3})\)\\
\(\Omega\) & \(\#\Omega\) & \(-q^2(a_t+2)-3q\), \(|a_t|\le2\sqrt q\) & \(O_n(q^{(5n+6)/2})\)\\
\bottomrule
\end{tabular}
\end{center}
Thus the one-dimensional diagonal \(D_0\) contributes only to the second coefficient, the two two-dimensional components \(D_1,D_2\) contribute in the same way in every characteristic, and the difference between odd \(p\) and \(p=2\), both in \(c_p\) and in \(b_p(n)\), comes entirely from the stratum \(Q\), on which \(\psi_q(L)\) oscillates when \(p\) is odd and is identically \(1\) when \(p=2\). Summing the rows for \(D\) and \(Q\) gives \(c_p=2\) and \(b_p(n)=2^n-(4n+6)-1\) for odd \(p\), and \(c_2=2+1\) and \(b_2(n)=2^n-(4n+6)-(2n+3)\).

\begin{lemma}[Cuspidal sections in characteristic two]\label{lem:cusp}
Let \(p=2\), so that \(\lambda=1\) and \(L(t)=t_1+t_2+t_3+t_4\), and let \(t\in\Omega\) with \(L(t)=0\). Then \(C_t\) is a cuspidal cubic with its cusp at the point \((t_1^{-1/2}:t_2^{-1/2}:t_3^{-1/2}:t_4^{-1/2})\); consequently \(a_t=0\) and \(V(t)=-2q^2-3q\). There are \((q-2)(q-4)\) rays of \(\Omega\) on which \(L\) vanishes.
\end{lemma}

\begin{proof}
Squaring is an automorphism of \(\Fq\); let \(s_i\in\Fq^\times\) be the square root of \(t_i\), so that \(\sum_is_i=(\sum_it_i)^{1/2}=0\). In the coordinates \(x_i=s_iu_i\) the plane \(H_t\) is \(\{\sum_is_ix_i=0\}\), and \(e_3(u)=(s_1s_2s_3s_4)^{-1}G(x)\) with \(G(x)=\sum_ls_l\prod_{m\ne l}x_m\). The point \(\mathbf 1=(1:1:1:1)\) lies on \(H_t\) and on \(\{G=0\}\), because \(\sum_ls_l=0\). The affine chart of \(H_t\) around \(\mathbf 1\) consists of the points \(\mathbf 1+y\) with \(y\) in a two-dimensional subspace \(Y_t\), defined over \(\Fq\), of \(\{y:\sum_ls_ly_l=0\}\) complementary to the line through \(\mathbf 1\). Writing \(Y=\sum_my_m\) and using \(e_2(y_m:m\ne l)=e_2(y)-y_lY+y_l^2\), the terms of \(G(\mathbf 1+y)\) of degree \(0\), \(1\) and \(2\) in \(y\) are
\[
\sum_ls_l=0,\qquad \sum_ls_l(Y-y_l)=-\sum_ls_ly_l,\qquad e_2(y)\sum_ls_l-Y\sum_ls_ly_l+\sum_ls_ly_l^2 .
\]
On \(Y_t\) the linear term vanishes, and, the characteristic being two, the quadratic term is \(\sum_ls_ly_l^2=\bigl(\sum_ls_l^{1/2}y_l\bigr)^2\), the square of a linear form with coefficients in \(\Fq\). This square is not identically zero on \(Y_t\): otherwise the equation of \(C_t\) in this chart would be homogeneous of degree three in \(y\), hence a product of three linear forms over \(\Fqbar\), contrary to Lemma~\ref{lem:hasse}. Hence \(\mathbf 1\) is a singular point of \(C_t\) whose tangent cone is a double line defined over \(\Fq\); in the notation of the proof of Lemma~\ref{lem:hasse}, \(F_2\) is the square of a nonzero linear form, so \(\varepsilon=1\) and \(a_t=0\), and \(V(t)=-2q^2-3q\) by Lemma~\ref{lem:planecubic}. In the coordinates \(u\) the cusp is the point \((t_1^{-1/2}:\cdots:t_4^{-1/2})\). Finally, the plane \(\{L=0\}\subset\mathbb P^3\) has \(q^2-3q+3\) points with all coordinates nonzero, by the inclusion--exclusion argument in the proof of Lemma~\ref{lem:planecubic}, and \(3q-5\) of them lie in \(D\cup Q\), on which \(L\) vanishes identically when \(p=2\); this leaves \((q-2)(q-4)\) rays of \(\Omega\).
\end{proof}

The rays on which \(L\) vanishes have a simple description in every characteristic: \(L(t)=0\) if and only if the plane \(H_t\) contains the point \(\omega=(1:1:-1:-1)\), which lies on \(\mathcal S\). For odd \(p\) the point \(\omega\) lies on two of the three lines \(\{u_i+u_j=0,\ u_k+u_l=0\}\) of Proposition~\ref{prop:cayley}(ii); for \(p=2\) it is the point \((1:1:1:1)\), common to all three, and Lemma~\ref{lem:cusp} says that the planes through it cut \(\mathcal S\) in cuspidal cubics.

\begin{proposition}[Plane sections on average, and the case \(n=2\)]\label{prop:n2-fourth}
Let \(\mathcal S=\{e_3=0\}\subset\mathbb P^3\) be the cubic surface of Lemma~\textup{\ref{lem:planecubic}}.
\begin{enumerate}[label=\textup{(\roman*)}]
\item For a plane \(H\subset\mathbb P^3\) defined over \(\Fq\) put \(a_H=q+1-\#(H\cap\mathcal S)(\Fq)\). Then, the sums being over all such planes,
\[
\sum_Ha_H=-2q^3-3q^2-2q,\qquad\sum_Ha_H^2=q^4+6q^3+4q^2,
\]
and, with absolute implied constants,
\[
\sum_{[t]\in\bar\Omega}a_t=-2q^3+O(q^2),\qquad\sum_{[t]\in\bar\Omega}a_t^2=q^4+O(q^3).
\]
\item If \(p=2\), then \(\sum_{h\in\mathbb F_q^2}|\Sigma_2(h)|^4=2q^4+O(q^3)\).
\end{enumerate}
\end{proposition}

\begin{proof}
(i) First, \(\#\mathcal S(\Fq)=q^2+3q+1\). Indeed, a point with all coordinates nonzero lies on \(\mathcal S\) if and only if \(\sum_iu_i^{-1}=0\), and \(u\mapsto(u_1^{-1}:\cdots:u_4^{-1})\) identifies these points with the points of the plane \(\{v_1+v_2+v_3+v_4=0\}\) with all coordinates nonzero, of which there are \(q^2-3q+3\), as in the proof of Lemma~\ref{lem:planecubic}. A point of \(\mathcal S\) with a vanishing coordinate has at least two vanishing coordinates (proof of Lemma~\ref{lem:planecubic}), so these points form the union of the six edges, which has \(6(q+1)-8=6q-2\) points. Put \(s=q^2+3q+1\). A point of \(\mathbb P^3(\Fq)\) lies on \(q^2+q+1\) planes, and two distinct points lie on \(q+1\) planes; so, summing over the \(q^3+q^2+q+1\) planes,
\[
\sum_H\#(H\cap\mathcal S)(\Fq)=s(q^2+q+1),\qquad\sum_H\#(H\cap\mathcal S)(\Fq)^2=s(q^2+q+1)+s(s-1)(q+1),
\]
and expanding \(a_H=q+1-\#(H\cap\mathcal S)(\Fq)\) gives the two exact formulas.

Every plane is \(H_t=\{\sum_it_iu_i=0\}\) for a unique \([t]\in\mathbb P^3(\Fq)\), and the planes with \([t]\notin\bar\Omega\) are of two kinds. (a) The \(3q-5\) planes with \(t\in D\cup Q\) and the \(6q-2\) planes for which \(t\) has at least two vanishing coordinates. Since \(\mathcal S\) contains no plane, \(H\cap\mathcal S\) is a plane cubic curve, so \(\#(H\cap\mathcal S)(\Fq)\le3q+3\), by Hasse's bound and the proof of Lemma~\ref{lem:hasse} if the cubic form is absolutely irreducible, and by inspection of its possible factorisations otherwise; hence \(|a_H|\le2q+2\). (b) The \(4(q-1)^2\) planes for which \(t\) has exactly one vanishing coordinate, say \(t_i\). Such a plane passes through the coordinate point \(e_i\), a singular point of \(\mathcal S\) (Proposition~\ref{prop:cayley}(i)), so that \(H_t\cap\mathcal S\) is singular at \(e_i\); and it contains no line of \(\mathcal S\) (Proposition~\ref{prop:cayley}(ii): an edge lies in \(H_t\) only if two coordinates of \(t\) vanish, and the line \(\{u_j+u_k=0,\ u_l+u_m=0\}\) only if \(t_j=t_k\) and \(t_l=t_m\), which forces a second vanishing coordinate). So the cubic form defining \(H_t\cap\mathcal S\) has no linear factor over \(\Fqbar\), and the argument for the singular case in the proof of Lemma~\ref{lem:hasse} gives \(|a_H|\le1\). Hence the planes with \([t]\notin\bar\Omega\) contribute \(O(q^2)\) to \(\sum_Ha_H\) and \(O(q^3)\) to \(\sum_Ha_H^2\), which gives the two asymptotic formulas.

(ii) Let \(p=2\). By Lemmas~\ref{lem:rays} and~\ref{lem:cusp},
\[
\sum_{t\in\Omega}\psi_q\bigl(L(t)\bigr)V(t)^2=q(q-2)(q-4)\bigl(2q^2+3q\bigr)^2-\sum_{[t]\in\bar\Omega}V(t)^2 ,
\]
and the first term is \(O(q^7)\). Since \(V(t)^2=q^4(a_t+2)^2+6q^3(a_t+2)+9q^2\) by Lemma~\ref{lem:planecubic}, and \(\#\bar\Omega=q^3+O(q^2)\), part (i) gives \(\sum_{[t]\in\bar\Omega}V(t)^2=q^8+O(q^7)\). So \(\Omega\) contributes \(-q^8+O(q^7)\) to the right-hand side of Lemma~\ref{lem:decouple} for \(n=2\), while \(D\cup Q\) contribute \((q-1)\bigl[X_0'^{\,2}+3(q-2)X'^{\,2}\bigr]=3q^8+O(q^7)\) (Propositions~\ref{prop:D} and~\ref{prop:P}) and \(Z\) contributes \(O(q^7)\) (Proposition~\ref{prop:Z}). Dividing by \(q^4\) gives the assertion.
\end{proof}

\begin{corollary}[The fourth moment for \(n=2\)]\label{cor:n2-exact}
For every finite field \(\Fq\) of characteristic \(p\),
\[
\sum_{h\in\Fq^2}|\Sigma_2(h)|^4=
\begin{cases}
2q^4-9q^3+q^2+22q+14,&p\ \text{odd},\\[2pt]
2q^4-6q^3-8q^2+18q+14,&p=2.
\end{cases}
\]
\end{corollary}

\begin{proof}
By Proposition~\ref{prop:n2-kl}(i), the points of \(\Fq^2\setminus(\Fq^\times)^2\) contribute \(2(q-1)\). For \(h\in(\Fq^\times)^2\) put \(w(h)=\psi_q(\mathrm s(h))K_q(\mathrm m(h))\), so that \(\Sigma_2(h)=w(h)-1\) and \(|w(h)|^2=K^2\) with \(K=K_q(\mathrm m(h))\), the Kloosterman sum being real. Then
\[
|w(h)-1|^4=(K^2+1)^2-2(K^2+1)\bigl(w(h)+\overline{w(h)}\bigr)+w(h)^2+\overline{w(h)}^{\,2}+2K^2 .
\]
For \(x\in\Fq^\times\) the fibre \(\{h\in(\Fq^\times)^2:\mathrm m(h)=x\}\) has \(q-1\) points, and for \(c\in\Fq^\times\)
\[
\sum_{\mathrm m(h)=x}\psi_q\bigl(c\,\mathrm s(h)\bigr)=K_q(c^2x),
\]
as one sees by putting \(h_1=-\lambda y/c\), so that \(c\,\mathrm s(h)=y+c^2x/y\). Let \(S_j=\sum_{x\in\Fq^\times}K_q(x)^j\). By Lemma~\ref{lem:kloost}(ii) with \(a=b\), and (iv), \(S_2=q^2-q-1\), and \(S_4=2q^3-3q^2-3q-1\) if \(p\) is odd and \(S_4=2q^3-2q^2-3q-1\) if \(p=2\). Summing the expansion over the fibres, with \(c=\pm1\) for the terms linear in \(w\) and \(\overline w\), and \(c=\pm2\) for \(w^2\) and \(\overline w^{\,2}\), gives
\[
\sum_{h\in(\Fq^\times)^2}|\Sigma_2(h)|^4=(q-1)\bigl(S_4+2S_2+q-1\bigr)-4\bigl(S_4+S_2\bigr)+2(q-1)S_2+2R,
\]
where \(R=\sum_{x\in\Fq^\times}K_q(x)^2K_q(4x)\) if \(p\) is odd and \(R=(q-1)S_2\) if \(p=2\), since \(\psi_q(2\,\mathrm s(h))=1\) when \(p=2\). Let \(p\) be odd. Writing \(K_q(y)=\sum_{a\in\Fq^\times}\psi_q(a+y/a)\) and summing over \(x\) first,
\[
R=q\,T-\Bigl(\sum_{a\in\Fq^\times}\psi_q(a)\Bigr)^3=q\,T+1,
\]
where \(T\) is the sum of \(\psi_q(a+b+c)\) over the triples \((a,b,c)\in(\Fq^\times)^3\) with \(a^{-1}+b^{-1}+4c^{-1}=0\). For such a triple \(a+b\ne0\) and \(c=-4ab/(a+b)\), so that \(a+b+c=(a-b)^2/(a+b)\); conversely every pair \((a,b)\in(\Fq^\times)^2\) with \(a+b\ne0\) arises. As \(p\) is odd, \((a,b)\mapsto(\sigma,\delta)=(a+b,a-b)\) is bijective, and the conditions become \(\sigma\ne0\) and \(\delta\ne\pm\sigma\). Let \(\eta\) be the quadratic character of \(\Fq^\times\) and \(g=\sum_{\delta\in\Fq}\psi_q(\delta^2)\); then \(\sum_{\delta\in\Fq}\psi_q(\delta^2/\sigma)=\eta(\sigma)g\), and
\[
T=\sum_{\sigma\in\Fq^\times}\Bigl(\eta(\sigma)\,g-2\psi_q(\sigma)\Bigr)=2 .
\]
So \(R=2q+1\). Inserting the values of \(S_2\), \(S_4\) and \(R\), and adding \(2(q-1)\), gives the assertion.
\end{proof}

Thus for \(n=2\) the main term is \(2q^4\) in every characteristic, and the second-order coefficient is \(-9\) for odd \(p\) and \(-6\) for \(p=2\). Together with the exact values of Propositions~\ref{prop:D} and~\ref{prop:P}, Corollary~\ref{cor:n2-exact} determines the joint contribution of the strata \(Z\) and \(\Omega\) to the right-hand side of Lemma~\ref{lem:decouple} for \(n=2\): it is
\[
\begin{cases}
-q^8+11q^7-32q^6-12q^5+87q^4+12q^3-45q^2,&p=2,\\[2pt]
2q^7-19q^6+20q^5+65q^4-18q^3-45q^2,&p\ \text{odd}.
\end{cases}
\]

\begin{remark}[The case \(n=2\)]\label{rem:cusp}
Theorem~\ref{thm:fourth} is stated for \(n\ge3\) because for \(n=2\) Proposition~\ref{prop:G} bounds the generic stratum only by \(O(q^{8})\), the order of the main term. By Lemma~\ref{lem:rays}, the contribution of \(\Omega\) is \(q\sum V(t)^2-\sum V(t)^2\), the first sum being over the rays of \(\Omega\) whose planes pass through \(\omega\), and the second over all rays of \(\Omega\); by Proposition~\ref{prop:n2-fourth}(i) the second sum is \(q^8+O(q^7)\) in every characteristic. For \(p=2\) the first term is \(O(q^7)\) (Lemma~\ref{lem:cusp}), so that the generic stratum contributes \(-q^8+O(q^7)\): the third pairing, which raises \(c_2\) to \(3\) for \(n\ge3\), is cancelled for \(n=2\) by the generic stratum, in accordance with \(\Sp_2=\SL_2\). Proposition~\ref{prop:n2-fourth}(ii) thus explains, through the geometry of the generic stratum, the main term of the exact formula of Corollary~\ref{cor:n2-exact}, which rests instead on Proposition~\ref{prop:n2-kl}; both are independent of Corollary~\ref{cor:moments} below, which is deduced from the monodromy and gives \(2q^4\bigl(1+O(q^{-1/2})\bigr)\) in every characteristic. For odd \(p\) the two terms cancel to leading order: by Corollary~\ref{cor:n2-exact} and Propositions~\ref{prop:D}, \ref{prop:P} and~\ref{prop:Z}, \(\Omega\) then contributes \(O(q^7)\); equivalently, the mean square of \(a_t\) over the planes through \(\omega\) agrees to leading order with its mean square over all planes. These statements are illustrated numerically in Remark~\ref{rem:numerics}.
\end{remark}

Numerical illustrations of Theorem~\ref{thm:fourth}, for odd \(p\) and for \(p=2\), are collected in Remark~\ref{rem:numerics} (Section~\ref{sec:numerics}).

\begin{remark}[The second moment]\label{rem:second}
For comparison, the second moment is given exactly by
\[
\sum_{h\in\Fq^n}|\Sigma_n(h)|^2=q^{n-2}\Bigl((q-1)^{n+1}+1-(q-2)(-1)^n\Bigr)\qquad(n\ge2).
\]
Indeed, expanding the square and summing over \(h\) first, a pair \((u,v)\in(\Fq^\times)^{2n}\) contributes \(q^{n-1}\) if \(u=v\), \(q^{n-2}\) if \(u\) and \(v\) are linearly independent, and \(0\) if \(v=cu\) with \(c\ne1\) (the equations \(\langle h,u\rangle=-\lambda\) and \(c\langle h,u\rangle=-\lambda\) are incompatible since \(\lambda\ne0\)). The sum of \(\psi_q(f(u)-f(v))\), \(f(u)=\sum_iu_i^{-1}\), over all pairs is \(|(-1)^n|^2=1\), and over the pairs \(v=cu\) with \(c\ne0,1\) it is \(\sum_{c\ne0,1}\prod_i\sum_{u_i\ne0}\psi_q\bigl((1-c^{-1})u_i^{-1}\bigr)=(q-2)(-1)^n\). Hence the sum equals \(q^{n-2}\bigl(1-(q-2)(-1)^n\bigr)+(q^{n-1}-q^{n-2})(q-1)^n\), as claimed.
\end{remark}
\section{The geometric monodromy group}\label{sec:monodromy}

For a lisse \(\Qlb\)-sheaf \(\mathcal F\) on a smooth geometrically connected scheme \(S\) over a finite field, with geometric generic fibre \(V\), we write \(\ggeom(\mathcal F)\subseteq\GL(V)\) for the Zariski closure of the image of \(\pgeom(S)\), and for a closed subgroup \(G\subseteq\GL(V)\) and \(m\ge1\)
\[
M_{2m}(G,V):=\dim\bigl(V^{\otimes m}\otimes(V^\vee)^{\otimes m}\bigr)^G .
\]
Recall that \(W_n\) denotes the geometric generic fibre of \(\mathcal W_n\). For a finite field \(k\) we fix an algebraic closure \(\bar k\) and write \(k_a\subset\bar k\) for the extension of degree \(a\); all moment limits below are limits as \(a\to\infty\), with \(\lambda\) and \(\psi\) fixed.

\subsection{Semisimplicity and the moment principle}

\begin{lemma}[Deligne]\label{lem:semisimple}
Let \(S\) be smooth and geometrically connected over a finite field \(k\), and let \(\mathcal F\) be a lisse \(\Qlb\)-sheaf on \(S\) which is pure of some weight. Then the representation of \(\pgeom(S)\) on the geometric generic fibre of \(\mathcal F\) is semisimple; equivalently, \(\ggeom(\mathcal F)\) is a reductive group.
\end{lemma}

\begin{proof}
This is \cite[Th\'eor\`eme 3.4.1(iii)]{Deligne}; we recall the argument. Induct on the length of a Jordan--H\"older series of \(\mathcal F\) as a representation of \(\pi_1(S)\). If \(\mathcal F\) is irreducible, its restriction to the normal subgroup \(\pgeom(S)\) is semisimple, because the sum of all irreducible \(\pgeom(S)\)-subrepresentations is \(\pi_1(S)\)-stable. Otherwise choose an irreducible lisse subsheaf \(\mathcal A\subset\mathcal F\) and put \(\mathcal B=\mathcal F/\mathcal A\). Both are pure of the weight \(w\) of \(\mathcal F\), so by induction both are geometrically semisimple, and it suffices to split \(0\to\mathcal A\to\mathcal F\to\mathcal B\to0\) over \(S_{\bar k}\). The obstruction is the class of this extension in \(H^1(S_{\bar k},\mathcal A\otimes\mathcal B^\vee)\), which is fixed by Frobenius because the extension is defined over \(k\). The sheaf \(\mathcal G=\mathcal A\otimes\mathcal B^\vee\) is lisse and pure of weight \(0\). If \(e=\dim S\), Poincar\'e duality gives \(H^1(S_{\bar k},\mathcal G)\cong H^{2e-1}_c(S_{\bar k},\mathcal G^\vee)^\vee(-e)\), and \(H^{2e-1}_c(S_{\bar k},\mathcal G^\vee)\) has weights \(\le2e-1\) by \cite[Th\'eor\`eme 3.3.1]{Deligne}; hence \(H^1(S_{\bar k},\mathcal G)\) has weights \(\ge1\), and its Frobenius-fixed vectors, which have weight \(0\), vanish. So the extension splits geometrically. Finally, the Zariski closure \(G\) of the image of \(\pgeom(S)\) has the same invariant subspaces as the image, so it acts semisimply on \(V\). Its unipotent radical \(R\) is normal in \(G\); for an irreducible \(G\)-submodule \(V'\subseteq V\), the subspace \(V'^R\) is \(G\)-stable and, by Kolchin's theorem, nonzero, hence equal to \(V'\). Thus \(R\) acts trivially on \(V\), so \(R=1\) and \(G\) is reductive.
\end{proof}

\begin{theorem}[Katz's moment principle]\label{thm:moment-principle}
Let \(S\) be smooth and geometrically connected of dimension \(e\ge1\) over a finite field \(k\), and let \(\mathcal F\) be a lisse \(\Qlb\)-sheaf on \(S\), pure of weight \(w\), with geometric generic fibre \(V\) and \(G=\ggeom(\mathcal F)\). For \(m\ge1\) and \(a\ge1\) put
\[
\Phi_{2m}(a):=(\#k_a)^{-e-mw}\sum_{x\in S(k_a)}\bigl|\Tr(\Frob_{k_a,x}\mid\mathcal F)\bigr|^{2m}.
\]
\begin{enumerate}[label=\textup{(\roman*)}]
\item \(\limsup_{a\to\infty}\Phi_{2m}(a)=M_{2m}(G,V)\). In particular, if \(\lim_{a\to\infty}\Phi_{2m}(a)\) exists, it equals \(M_{2m}(G,V)\).
\item If \(\pi_1(S)\) acts trivially on the subspace \(\bigl(V^{\otimes m}\otimes(V^\vee)^{\otimes m}\bigr)^G\) of \(G\)-invariants, then \(\Phi_{2m}(a)=M_{2m}(G,V)+O\bigl((\#k_a)^{-1/2}\bigr)\) for all \(a\ge1\).
\end{enumerate}
\end{theorem}

\begin{proof}
The principle is due to Katz; see \cite[Lemma 2.1.4]{Katz-Larsen} and \cite{KatzMMP}. For completeness we give the proof, which rests on the trace formula, Poincar\'e duality and Deligne's weight bound. We use the fixed embedding \(\iota\) to regard traces and eigenvalues as complex numbers. Let \(F_k\in\mathrm{Gal}(\bar k/k)\) be the geometric Frobenius, so that \(F_k^a\) is the geometric Frobenius of \(k_a\); it acts on \(\Qlb(1)\) as multiplication by \((\#k)^{-1}\), hence on \(\Qlb(-e)\) as multiplication by \((\#k)^{e}\). The Grothendieck--Lefschetz trace formula is taken with geometric Frobenius.

\emph{Reduction to a weight-zero sheaf.} Put \(\mathcal F_{m,m}=\mathcal F^{\otimes m}\otimes(\mathcal F^\vee)^{\otimes m}\), which is lisse and pure of weight \(0\). Let \(x\in S(k_a)\), and let \(\alpha_1,\dots,\alpha_r\) be the eigenvalues of \(\Frob_{k_a,x}\) on \(\mathcal F\); those on \(\mathcal F^\vee\) are the \(\alpha_i^{-1}\). Since \(|\alpha_i|^2=(\#k_a)^w\), the complex conjugate of \(\alpha_i\) is \((\#k_a)^w\alpha_i^{-1}\), so the complex conjugate of \(\Tr(\Frob_{k_a,x}\mid\mathcal F)\) is \((\#k_a)^w\Tr(\Frob_{k_a,x}\mid\mathcal F^\vee)\). Hence \(|\Tr(\Frob_{k_a,x}\mid\mathcal F)|^{2m}=(\#k_a)^{mw}\,\Tr(\Frob_{k_a,x}\mid\mathcal F_{m,m})\), and the trace formula gives
\[
\Phi_{2m}(a)=(\#k_a)^{-e}\sum_{x\in S(k_a)}\Tr(\Frob_{k_a,x}\mid\mathcal F_{m,m})=(\#k_a)^{-e}\sum_i(-1)^i\Tr\bigl(F_k^a\mid H^i_c(S_{\bar k},\mathcal F_{m,m})\bigr).
\]

\emph{The terms with \(i<2e\).} For \(i<2e\) the group \(H^i_c(S_{\bar k},\mathcal F_{m,m})\) is mixed of weights \(\le i\) \cite[Th\'eor\`eme 3.3.1]{Deligne}, and \(i\le2e-1\); so these terms contribute \(O\bigl((\#k_a)^{-1/2}\bigr)\), the implied constant being \(\sum_{i<2e}\dim H^i_c(S_{\bar k},\mathcal F_{m,m})\), which does not depend on \(a\); and \(H^i_c(S_{\bar k},\mathcal F_{m,m})=0\) for \(i>2e\).

\emph{The term with \(i=2e\).} Since \(S\) is smooth and geometrically connected of dimension \(e\), Poincar\'e duality gives a perfect \(\mathrm{Gal}(\bar k/k)\)-equivariant pairing
\[
H^{2e}_c(S_{\bar k},\mathcal F_{m,m})\times H^0\bigl(S_{\bar k},\mathcal F_{m,m}^\vee(e)\bigr)\longrightarrow H^{2e}_c\bigl(S_{\bar k},\Qlb(e)\bigr)\cong\Qlb ,
\]
so that \(H^{2e}_c(S_{\bar k},\mathcal F_{m,m})\) is the dual of \(H^0(S_{\bar k},\mathcal F_{m,m}^\vee)\), twisted by \((-e)\). Reordering the tensor factors gives an isomorphism of lisse sheaves \(\mathcal F_{m,m}^\vee=(\mathcal F^\vee)^{\otimes m}\otimes\mathcal F^{\otimes m}\cong\mathcal F_{m,m}\), hence a Galois-equivariant isomorphism \(H^0(S_{\bar k},\mathcal F_{m,m}^\vee)\cong H^0(S_{\bar k},\mathcal F_{m,m})\). The space \(H^0(S_{\bar k},\mathcal F_{m,m})\) is the space of \(\pgeom(S)\)-invariants in \((\mathcal F_{m,m})_{\bar\eta}=V^{\otimes m}\otimes(V^\vee)^{\otimes m}\); since the stabiliser of a vector is Zariski closed, these are the \(G\)-invariants, whose dimension is \(M_{2m}(G,V)=:M\) by definition. Let \(\gamma_1,\dots,\gamma_M\) be the eigenvalues of \(F_k\) on \(H^0(S_{\bar k},\mathcal F_{m,m})\). If \(x\) is a closed point of \(S\) and \(\bar x\) a geometric point over it, restriction to \(\bar x\) is an injective map \(H^0(S_{\bar k},\mathcal F_{m,m})\to(\mathcal F_{m,m})_{\bar x}\) (as \(\mathcal F_{m,m}\) is lisse and \(S_{\bar k}\) connected), which carries the action of \(F_k^{\deg(x)}\) to that of \(\Frob_x\); since \(\mathcal F_{m,m}\) is pure of weight \(0\), it follows that \(|\gamma_j|=1\) for all \(j\). By the duality above, \(F_k\) acts on \(H^{2e}_c(S_{\bar k},\mathcal F_{m,m})\) with eigenvalues \((\#k)^e\gamma_j^{-1}\). Putting \(\beta_j=\gamma_j^{-1}\) we obtain
\[
\Phi_{2m}(a)=\sum_{j=1}^M\beta_j^a+O\bigl((\#k_a)^{-1/2}\bigr),\qquad|\beta_j|=1 .
\]
This part of the argument does not use Lemma~\ref{lem:semisimple}.

\emph{Conclusion.} Since \(\Phi_{2m}(a)\) is real and \(\bigl|\sum_j\beta_j^a\bigr|\le M\), we get \(\limsup_a\Phi_{2m}(a)\le M\). Write \(\beta_j=e^{2\pi i\theta_j}\) with \(\theta_j\in\mathbb R\). By Dirichlet's simultaneous approximation theorem, for every \(N\ge1\) there is an integer \(a(N)\ge1\) with \(\|a(N)\theta_j\|<1/N\) for all \(j\), where \(\|\cdot\|\) is the distance to the nearest integer. If the integers \(a(N)\) are bounded, some value \(a_0\) occurs for infinitely many \(N\), so that \(\beta_j^{a_0}=1\) for all \(j\), and we put \(a_\nu=\nu a_0\); otherwise we take for \((a_\nu)\) a subsequence of \((a(N))\) tending to infinity. In both cases \(a_\nu\to\infty\) and \(\beta_j^{a_\nu}\to1\) for all \(j\), so \(\Phi_{2m}(a_\nu)\to M\). Hence \(\limsup_a\Phi_{2m}(a)=M\), which is (i). Under the hypothesis of (ii), \(F_k\) acts trivially on \(H^0(S_{\bar k},\mathcal F_{m,m})\), which is the space of \(G\)-invariants in \((\mathcal F_{m,m})_{\bar\eta}\) with the action of \(\pi_1(S)\); so \(\gamma_j=\beta_j=1\) for all \(j\), and the displayed formula gives (ii).
\end{proof}

\subsection{Larsen's alternative}

\begin{theorem}[Larsen's alternative]\label{thm:larsen}
Let \(V\) be a finite-dimensional vector space of dimension \(N\ge2\) over an algebraically closed field of characteristic zero, and let \(G\subseteq\GL(V)\) be a reductive Zariski-closed subgroup.
\begin{enumerate}[label=\textup{(\roman*)}]
\item If \(M_4(G,V)\le5\), then \(V\) is an irreducible representation of \(G\).
\item If \(M_4(G,V)=2\), then either \(G\supseteq\SL(V)\), or the image of \(G\) in \(\GL(V)/\{\text{scalars}\}\) is finite.
\item Suppose that \(G\) preserves a nondegenerate symmetric bilinear form \(B\) on \(V\), so that \(G\subseteq\mathrm O(V,B)\). If \(M_4(G,V)=3\), then either \(G\supseteq\mathrm{SO}(V,B)\) or \(G\) is finite.
\item Suppose that \(N\ge4\) and that \(G\) preserves a nondegenerate alternating bilinear form \(B\) on \(V\), so that \(G\subseteq\Sp(V,B)\). If \(M_4(G,V)=3\), then either \(G=\Sp(V,B)\) or \(G\) is finite.
\end{enumerate}
\end{theorem}

\begin{proof}
Part (ii) is \cite[Theorem 1.1.6]{Katz-Larsen}, and (iii), (iv) are its orthogonal and symplectic analogues, also due to Larsen and treated in \cite[\S1]{Katz-Larsen}; we give the arguments. Since \(G\) is reductive, every \(G\)-module is semisimple, and \(M_4(G,V)=\dim\End_G(V^{\otimes2})\). If \(U=\bigoplus_iU_i^{\oplus m_i}\) with nonzero submodules \(U_i\), then \(\dim\End_G(U)\ge\sum_im_i^2\), the block-scalar endomorphisms being linearly independent.

(i) If \(V=V_1\oplus V_2\) with \(V_1,V_2\ne0\), then \(V^{\otimes2}=V_1^{\otimes2}\oplus V_2^{\otimes2}\oplus(V_1\otimes V_2)^{\oplus2}\), so \(M_4(G,V)\ge1+1+4=6\).

(ii) Since also \(M_4(G,V)=\dim\End_G(V\otimes V^\vee)\) and \(V\otimes V^\vee=\End(V)=\mathbf 1\oplus\mathfrak{sl}(V)\) has two nonzero summands, \(M_4=2\) forces \(\mathfrak{sl}(V)\) to be an irreducible \(G\)-module (and \(V\) to be irreducible, by (i)). Let \(G^{\mathrm{der}}\) be the closed subgroup generated by commutators; it is normal in \(G\) and contained in \(\SL(V)\), so its Lie algebra is a \(G\)-stable subspace of \(\mathfrak{sl}(V)\), hence is \(0\) or \(\mathfrak{sl}(V)\). In the second case \((G^{\mathrm{der}})^0=\SL(V)\), so \(G\supseteq\SL(V)\). In the first case \(G^{\mathrm{der}}\) is finite. For fixed \(\gamma\in G\) the morphism \(G^0\to G^{\mathrm{der}}\), \(g\mapsto g\gamma g^{-1}\gamma^{-1}\), has connected source and finite target and takes the value \(1\) at \(g=1\), hence is constant; so \(G^0\) is central in \(G\), and since \(V\) is irreducible, Schur's lemma puts \(G^0\) inside the scalars. Then \(G/(G\cap\{\text{scalars}\})\) is an algebraic group with trivial identity component, hence finite.

(iii) and (iv). Put \(\epsilon=1\) in case (iii) and \(\epsilon=-1\) in case (iv), so that \(B(w,v)=\epsilon B(v,w)\), and write \(V^{\otimes2}=S^+\oplus S^-\) with \(S^+=\Sym^2V\) and \(S^-=\Lambda^2V\), the \((\pm1)\)-eigenspaces of the flip \(v\otimes w\mapsto w\otimes v\). The linear map \(\theta:V\otimes V\to\End(V)\), \(\theta(v\otimes w)=B(v,\cdot)\,w\), is a \(G\)-equivariant isomorphism, \(G\) acting on \(\End(V)\) by conjugation. Let \(c=\theta^{-1}(\mathrm{id}_V)\), a nonzero \(G\)-invariant vector. If \(c=\sum_iv_i\otimes w_i\) with \(B(v_i,w_j)=\delta_{ij}\), then for \(x=\sum_ia_iv_i\) one has \(\sum_iB(w_i,x)v_i=\epsilon\sum_iB(x,w_i)v_i=\epsilon x\); that is, \(\theta\) maps the flip of \(c\) to \(\epsilon\,\mathrm{id}_V\), so \(c\in S^\epsilon\). Since \(G\) is reductive, \(S^\epsilon\) is the direct sum of the line spanned by \(c\) and a \(G\)-submodule \(U\) of dimension \(\dim S^\epsilon-1\), which is \(N(N+1)/2-1\ge2\) in case (iii) and \(N(N-1)/2-1\ge5\) in case (iv). Thus \(V^{\otimes2}\) is the direct sum of three nonzero \(G\)-stable subspaces: the line spanned by \(c\), \(U\), and \(S^{-\epsilon}\). If \(S^{-\epsilon}\) were reducible, it would be the direct sum of two nonzero \(G\)-submodules, and we would get \(M_4(G,V)\ge4\). Hence \(S^{-\epsilon}\) is irreducible. For \(v,w\in V\) the endomorphism \(X=\theta(v\otimes w-\epsilon w\otimes v)=B(v,\cdot)w-\epsilon B(w,\cdot)v\) satisfies
\[
\begin{aligned}
B(Xx,y)+B(x,Xy)&=B(v,x)B(w,y)-\epsilon B(w,x)B(v,y)\\
&\quad+\epsilon B(v,y)B(w,x)-B(w,y)B(v,x)=0,
\end{aligned}
\]
so \(\theta(S^{-\epsilon})\) is contained in the Lie algebra \(\mathfrak g_B=\{X:B(Xx,y)+B(x,Xy)=0\}\) of the isometry group of \(B\); as \(\theta\) is injective and \(\dim\mathfrak g_B\) equals \(N(N-1)/2=\dim S^-\) in case (iii) and \(N(N+1)/2=\dim S^+\) in case (iv), we get \(\theta(S^{-\epsilon})=\mathfrak g_B\). Therefore \(\mathfrak g_B\) is an irreducible \(G\)-module under the adjoint action. The Lie algebra of \(G\) is an \(\mathrm{Ad}(G)\)-stable subspace of \(\mathfrak g_B\), hence equals \(0\) or \(\mathfrak g_B\). In the first case \(G^0=\{1\}\), the base field having characteristic zero, and \(G\) is finite. In the second case \(G^0\) is a closed connected subgroup of the connected group \(\mathrm{SO}(V,B)\), resp.\ \(\Sp(V,B)\), of the same dimension, hence equal to it. This gives (iii), and in case (iv) \(\Sp(V,B)=G^0\subseteq G\subseteq\Sp(V,B)\).
\end{proof}

\subsection{The fourth moment of the monodromy group}

\begin{proposition}\label{prop:M4}
Let \(n\ge3\). Then \(M_4\bigl(\ggeom(\mathcal W_n),W_n\bigr)=c_p\), where \(c_p=2\) if \(p\) is odd and \(c_2=3\). Consequently \(W_n\) is an irreducible representation of \(\ggeom(\mathcal W_n)\), and \(\ggeom(\mathcal W_n)\) is reductive.
\end{proposition}

\begin{proof}
Reductivity follows from Lemma~\ref{lem:semisimple}, since \(\mathcal W_n\) is pure (Theorem~\ref{thm:package}). We apply Theorem~\ref{thm:moment-principle} with \(k=\Fp\), \(S=S_n\), \(\mathcal F=\mathcal W_n\), \(e=n\), \(w=n+1\) and \(m=2\). Let \(a\ge1\) and \(q=p^a\), so that \(k_a=\Fq\). For \(h\in S_n(\Fq)\) write \(\mathsf t_h=\Tr(\Frob_h\mid\mathcal W_n)\) and \(\mathsf a_h=(-1)^{n-1}q\Sigma_n(h)\). By Theorem~\ref{thm:package}(iii) and Corollary~\ref{cor:pointwise}, \(|\mathsf a_h-\mathsf t_h|\le c_nq^{n/2}\) and \(|\mathsf a_h|+|\mathsf t_h|\le c_nq^{(n+1)/2}\), with \(c_n\) depending only on \(n\). Using \(\bigl||\mathsf a|^4-|\mathsf t|^4\bigr|\le|\mathsf a-\mathsf t|(|\mathsf a|+|\mathsf t|)(|\mathsf a|^2+|\mathsf t|^2)\) we get \(\bigl||\mathsf a_h|^4-|\mathsf t_h|^4\bigr|\le c_n'q^{2n+3/2}\). Summing over the \((q-1)^n\le q^n\) points of \(S_n(\Fq)=(\Fq^\times)^n\) and normalising by \(q^{-n-2(n+1)}\),
\[
\Phi_4(a)=q^{-3n-2}\sum_{h\in S_n(\Fq)}|\mathsf a_h|^4+O_n(q^{-1/2})=q^{-3n+2}\sum_{h\in(\Fq^\times)^n}|\Sigma_n(h)|^4+O_n(q^{-1/2}).
\]
By Theorem~\ref{thm:fourth} (in its version over \((\Fq^\times)^n\)) the right-hand side is \(c_p+O_n(q^{-\min(1/2,\delta(n))})\). Hence \(\lim_{a\to\infty}\Phi_4(a)=c_p\), and Theorem~\ref{thm:moment-principle}(i) gives \(M_4=c_p\). Irreducibility follows from Theorem~\ref{thm:larsen}(i).
\end{proof}

\subsection{The determinant}

\begin{proposition}\label{prop:det}
For every \(n\ge1\) the image of \(\pgeom(S_n)\) under the character \(\det(\mathcal W_n)\) is a finite group. If \(d\) is its order, this image is \(\mu_d\), and \(\det(\ggeom(\mathcal W_n))=\mu_d\). Consequently \(\ggeom(\mathcal W_n)\cap\{\text{scalars}\}\) is finite.
\end{proposition}

\begin{proof}
The scheme \(S_n=\Gm^n\) is normal, geometrically connected and of finite type over \(\Fp\), and \(\det(\mathcal W_n)\) is a lisse \(\Qlb\)-sheaf of rank one on it. By \cite[Proposition 1.3.4(i)]{Deligne} its geometric monodromy group is finite; more precisely, its character is the product of a character of finite order and a character of the form \(\gamma\mapsto b^{\deg(\gamma)}\), \(b\in\Qlb^\times\), which is trivial on \(\pgeom(S_n)\). A finite subgroup of order \(d\) of the multiplicative group of a field is cyclic and consists of \(d\)-th roots of unity, so it is \(\mu_d\). Let \(\Gamma\subseteq\GL(W_n)\) be the image of \(\pgeom(S_n)\) and \(G=\ggeom(\mathcal W_n)\) its Zariski closure. The set \(\{g:\det(g)^d=1\}\) is Zariski closed and contains \(\Gamma\), hence contains \(G\); therefore \(\mu_d=\det(\Gamma)\subseteq\det(G)\subseteq\mu_d\). If \(z\cdot\mathrm{id}\in G\), then \(z^{r(n)}=\det(z\cdot\mathrm{id})\in\mu_d\), which leaves finitely many possibilities for \(z\).
\end{proof}

\subsection{Self-duality in characteristic two}\label{subsec:selfdual}

For \(a\in\Qlb^\times\) we write \(a^{\deg}\) for the lisse sheaf of rank one on \(\mathrm{Spec}\,\Fp\), or its pull-back to any scheme over \(\Fp\), on which the geometric Frobenius of \(\Fp\) acts by \(a\); thus \(\Frob_x\) acts on it by \(a^f\) at a point \(x\) with values in \(\mathbb F_{p^f}\), and \(\Qlb(-1)=p^{\deg}\).

\begin{lemma}\label{lem:detkl}
In every characteristic \(p\), \(\det\Kl\cong\Qlb(-1)\) as lisse sheaves on \(\Gm\) over \(\Fp\). In other words, \(\det\Kl\) is geometrically constant, and \(\det(\Frob_a\mid\Kl)=q\) for every finite field \(\Fq\) of characteristic \(p\) and every \(a\in\Fq^\times\).
\end{lemma}

\begin{proof}
We first work over \(\Fpbar\). By Theorem~\ref{thm:kl}(ii) the inertia group at \(0\) acts on \(\Kl\) through unipotent automorphisms, whose determinants are \(1\); so \(\det\Kl\) is unramified at \(0\) and extends to a lisse sheaf \(\mathcal D\) of rank one on \(\mathbb A^1\). By Theorem~\ref{thm:kl}(iii) and Lemma~\ref{lem:breaks}(i), all breaks of \(\Kl\otimes\Kl\) at \(\infty\) are \(\le\frac12\), and so are those of its quotient \(\Lambda^2\Kl=\det\Kl\), since this property means that the groups \(I_\infty^{(r)}\), \(r>\frac12\), act trivially (proof of Lemma~\ref{lem:breaks}). The break of a one-dimensional representation is its Swan conductor, an integer; hence \(\Swan_\infty(\mathcal D)=0\), and the Grothendieck--Ogg--Shafarevich formula gives \(\chi_c(\mathbb A^1,\mathcal D)=\chi_c(\mathbb A^1)-\Swan_\infty(\mathcal D)=1\). As \(\mathbb A^1\) is affine, \(H^0_c(\mathbb A^1,\mathcal D)=0\); therefore \(H^2_c(\mathbb A^1,\mathcal D)\ne0\). By Poincar\'e duality \(H^2_c(\mathbb A^1,\mathcal D)\) is dual to \(H^0(\mathbb A^1,\mathcal D^\vee)(1)\), and a nonzero global section of the lisse rank-one sheaf \(\mathcal D^\vee\) on the connected curve \(\mathbb A^1\) is a nonzero vector of its fibre fixed by \(\pi_1(\mathbb A^1)\). So \(\pi_1(\mathbb A^1)\) acts trivially on \(\mathcal D\), and \(\det\Kl\) is geometrically constant.

Consequently the monodromy character of \(\det\Kl\) factors through \(\pi_1(\Gm)\to\mathrm{Gal}(\Fpbar/\Fp)\), and for each \(q\) the number \(\gamma_q:=\det(\Frob_a\mid\Kl)\) does not depend on \(a\in\Fq^\times\). By Theorem~\ref{thm:kl}(i), applied over \(\Fq\) and over \(\mathbb F_{q^2}\), \(\Tr(\Frob_a\mid\Kl)=-K_q(a)\) and \(\Tr(\Frob_a^2\mid\Kl)=-K_{q^2}(a)\), where \(\Frob_a^2\) is the geometric Frobenius of \(\mathbb F_{q^2}\) at \(a\); hence \(2\gamma_q=K_q(a)^2+K_{q^2}(a)\). We sum over \(a\in\Fq^\times\). First, \(\sum_{a\in\Fq}K_q(a)^2=\sum_{x,y\in\Fq^\times}\psi_q(x^{-1}+y^{-1})\sum_{a\in\Fq}\psi_q\bigl(a(x+y)\bigr)=q\sum_{x\ne0}\psi_q(0)=q(q-1)\), and \(K_q(0)^2=1\), so \(\sum_{a\ne0}K_q(a)^2=q(q-1)-1\). Next, let \(\mathrm T:\mathbb F_{q^2}\to\Fq\) be the trace, so that \(\psi_{q^2}=\psi_q\circ\mathrm T\). Then
\[
\sum_{a\in\Fq^\times}K_{q^2}(a)=\sum_{y\in\mathbb F_{q^2}^\times}\psi_{q^2}(y^{-1})\sum_{a\in\Fq^\times}\psi_q\bigl(a\,\mathrm T(y)\bigr)=q\sum_{\substack{y\ne0\\ \mathrm T(y)=0}}\psi_{q^2}(y^{-1})+1 .
\]
If \(\mathrm T(y)=0\) and \(y\ne0\), then \(y^q=-y\), so \((y^{-1})^q=-y^{-1}\) and \(\psi_{q^2}(y^{-1})=\psi_q(\mathrm T(y^{-1}))=1\); there are \(q-1\) such \(y\). Hence \(\sum_{a\ne0}K_{q^2}(a)=q(q-1)+1\), and \(2(q-1)\gamma_q=2q(q-1)\), that is, \(\gamma_q=q\).
\end{proof}

\begin{proposition}\label{prop:selfdual}
Let \(p=2\) and \(n\ge1\). There is a morphism of lisse sheaves on \(S_n\)
\[
B:\mathcal W_n\otimes\mathcal W_n\longrightarrow\Qlb(-n-1)
\]
which is a perfect pairing on every fibre and satisfies \(B\circ\mathrm{sw}=(-1)^{n+1}B\), where \(\mathrm{sw}\) exchanges the two factors. Consequently:
\begin{enumerate}[label=\textup{(\alph*)}]
\item \(\ggeom(\mathcal W_n)\) preserves a nondegenerate bilinear form on \(W_n\), which is symmetric if \(n\) is odd and alternating if \(n\) is even;
\item for every finite field \(\Fq\) of characteristic \(2\), every \(h\in S_n(\Fq)\) and all \(v,w\) in the geometric fibre of \(\mathcal W_n\) at \(h\), one has \(B(\Frob_hv,\Frob_hw)=q^{n+1}B(v,w)\).
\end{enumerate}
\end{proposition}

\begin{proof}
\emph{A pairing on \(\mathcal N_n\).} Since \(\Kl\) has rank two, the exterior product is a perfect alternating pairing \(\Kl\otimes\Kl\to\Lambda^2\Kl=\det\Kl\), and \(\det\Kl\cong\Qlb(-1)\) by Lemma~\ref{lem:detkl}. Since \(p=2\), the canonical isomorphism \(\Lpsi{\lambda t}\otimes\Lpsi{\lambda t}\cong\Lpsi{2\lambda t}=\Qlb\) is a perfect symmetric pairing. Pulling back along the maps \(m_i\) and tensoring, we obtain a perfect pairing of lisse sheaves on \(\Gm\times S_n\)
\[
\langle\ ,\ \rangle:\mathcal N_n\otimes\mathcal N_n\longrightarrow\mathcal D_n:=\bigotimes_{i=1}^nm_i^*\det\Kl\cong\Qlb(-n),\qquad\langle\ ,\ \rangle\circ\mathrm{sw}=(-1)^n\langle\ ,\ \rangle .
\]

\emph{The cup product.} Put \(X=\Gm\times S_n\), so that \(\rho:X\to S_n\), and \(X_2:=X\times_{S_n}X=\Gm\times\Gm\times S_n\), with projections \(p_1,p_2:X_2\to X\), structure map \(\rho_2=\rho\circ p_1=\rho\circ p_2\), and diagonal \(\delta:X\to X_2\); as \(X\) is separated over \(S_n\), \(\delta\) is a closed immersion, so \(\delta_*=\delta_!\), and \(\rho_2\circ\delta=\rho\). By the K\"unneth formula \cite[Exp.~XVII, 5.4.3]{SGA4} there is an isomorphism \(R\rho_!\mathcal N_n\otimes^{\mathbf L}R\rho_!\mathcal N_n\cong R\rho_{2!}(p_1^*\mathcal N_n\otimes p_2^*\mathcal N_n)\), compatible with base change on \(S_n\). Applying \(R\rho_{2!}\) to the adjunction morphism \(p_1^*\mathcal N_n\otimes p_2^*\mathcal N_n\to\delta_*\delta^*(p_1^*\mathcal N_n\otimes p_2^*\mathcal N_n)=\delta_*(\mathcal N_n\otimes\mathcal N_n)\) and using \(R\rho_{2!}\delta_*=R\rho_{2!}\delta_!=R\rho_!\), we obtain a morphism
\[
R\rho_!\mathcal N_n\otimes^{\mathbf L}R\rho_!\mathcal N_n\longrightarrow R\rho_!(\mathcal N_n\otimes\mathcal N_n),
\]
compatible with base change, which at a geometric point \(h\) of \(S_n\) is the cup product \(R\Gamma_c(\Gm,\mathcal N_h)\otimes R\Gamma_c(\Gm,\mathcal N_h)\to R\Gamma_c(\Gm,\mathcal N_h\otimes\mathcal N_h)\). By Theorem~\ref{thm:package}(i), \(R\rho_!\mathcal N_n=\mathcal H_n[-1]\) with \(\mathcal H_n\) lisse; taking \(\mathcal H^2\) gives a morphism \(\mathcal H_n\otimes\mathcal H_n\to R^2\rho_!(\mathcal N_n\otimes\mathcal N_n)\). Composing it with \(R^2\rho_!\langle\ ,\ \rangle\) and with the isomorphism \(R^2\rho_!\mathcal D_n\cong R^2\rho_!\Qlb\otimes\Qlb(-n)\cong\Qlb(-n-1)\) given by the projection formula and the trace map \(R^2\rho_!\Qlb\cong\Qlb(-1)\) (the fibres of \(\rho\) being copies of \(\Gm\)), we obtain a morphism of lisse sheaves over \(\Fp\)
\[
B_{\mathcal H}:\mathcal H_n\otimes\mathcal H_n\longrightarrow\Qlb(-n-1).
\]
The cup product of two classes of degree one is anticommutative (Koszul sign rule), so \(B_{\mathcal H}\circ\mathrm{sw}=-(-1)^nB_{\mathcal H}=(-1)^{n+1}B_{\mathcal H}\). (Theorem~\ref{thm:FS} below determines the symmetry type of the invariant forms on \(W_n\) by a different method, which does not use this sign.)

\emph{The radical.} Let \(h\in S_n(\Fq)\). We compute over \(\Fqbar\) and write \(U=\Gm\), \(\mathcal N=\mathcal N_h\), and \(\iota:H^1_c(U,\mathcal N)\to H^1(U,\mathcal N)\) for the map forgetting supports. For \(a,b\in H^1_c(U,\mathcal N)\) the cup product \(a\cup b\in H^2_c(U,\mathcal N\otimes\mathcal N)\) equals the product \(a\cup\iota(b)\) of \(a\) with the class \(\iota(b)\) under the pairing \(H^1_c(U,\mathcal N)\otimes H^1(U,\mathcal N)\to H^2_c(U,\mathcal N\otimes\mathcal N)\). Since \(\langle\ ,\ \rangle\) identifies \(\mathcal N\) with \(\mathcal N^\vee\otimes\mathcal D_n|_{U\times\{h\}}\), and the latter sheaf is constant over \(\Fqbar\), Poincar\'e duality on the smooth curve \(U\) shows that \((a,c)\mapsto\langle a\cup c\rangle\) is a perfect pairing between \(H^1_c(U,\mathcal N)\) and \(H^1(U,\mathcal N)\), with values in the line \(H^2_c(U,\mathcal D_n|_{U\times\{h\}})\). Hence \(b\) is orthogonal to all of \(H^1_c(U,\mathcal N)\) under \(B_{\mathcal H}\) at \(h\) if and only if \(\iota(b)=0\). In the same way \(a\cup b=\iota(a)\cup b\), and Poincar\'e duality also makes the pairing between \(H^1(U,\mathcal N)\) and \(H^1_c(U,\mathcal N)\) perfect; so the left radical is \(\ker\iota\) as well. (This does not use the symmetry of \(B_{\mathcal H}\).) So the radical of \(B_{\mathcal H}\) at \(h\) is \(\ker\iota\). Let \(j:U\hookrightarrow\mathbb P^1\). The map \(\iota\) is induced by \(j_!\mathcal N\to Rj_*\mathcal N\), which factors through \(j_*\mathcal N\); so \(\iota\) is the composite of the map \(H^1(\mathbb P^1,j_!\mathcal N)\to H^1(\mathbb P^1,j_*\mathcal N)\) with the edge map \(H^1(\mathbb P^1,j_*\mathcal N)\to H^1(U,\mathcal N)\) of the Leray spectral sequence of \(j\), which is injective. By the second exact sequence of Proposition~\ref{prop:fibre}(i), the kernel of the first map is the image of \(\mathcal N^{I_0}\), which is \(\mathcal E_{n,h}\) by Theorem~\ref{thm:package}(ii). Thus the radical of \(B_{\mathcal H}\) at \(h\) is \(\mathcal E_{n,h}\).

\emph{Descent to \(\mathcal W_n\).} The morphism \(\mathcal H_n\to\mathcal H_n^\vee(-n-1)\) adjoint to \(B_{\mathcal H}\) is a morphism of lisse sheaves on the connected normal scheme \(S_n\); it corresponds to a \(\pi_1(S_n)\)-equivariant linear map between the geometric generic fibres, so its kernel \(\mathcal R\) is a lisse subsheaf of \(\mathcal H_n\) whose fibre at each geometric point is the radical of \(B_{\mathcal H}\) there. By the preceding paragraph, \(\mathcal R\) and \(\mathcal E_n\) have the same fibre at a geometric point over a closed point of \(S_n\); as lisse subsheaves of \(\mathcal H_n\) correspond to \(\pi_1(S_n)\)-stable subspaces of one fibre, \(\mathcal R=\mathcal E_n\). Therefore \(B_{\mathcal H}\) vanishes on \(\mathcal E_n\otimes\mathcal H_n\) and on \(\mathcal H_n\otimes\mathcal E_n\), and it induces a morphism \(B:\mathcal W_n\otimes\mathcal W_n\to\Qlb(-n-1)\) which is a perfect pairing on every fibre and satisfies \(B\circ\mathrm{sw}=(-1)^{n+1}B\).

\emph{Consequences.} The group \(\pgeom(S_n)\) acts trivially on the fibre of \(\Qlb(-n-1)\); so, after the choice of a basis of that fibre, \(B\) becomes a nondegenerate \(\pgeom(S_n)\)-invariant bilinear form on \(W_n\), symmetric if \(n\) is odd and alternating if \(n\) is even. Its stabiliser in \(\GL(W_n)\) is Zariski closed and therefore contains \(\ggeom(\mathcal W_n)\); this is (a). Since \(B\) is a morphism of sheaves on \(S_n\) over \(\Fp\), it commutes with the action of \(\Frob_h\), which acts on the fibre of \(\Qlb(-n-1)\) at \(h\in S_n(\Fq)\) by \(q^{n+1}\); this is (b).
\end{proof}

The hypothesis \(p=2\) enters only through the triviality of \(\Lpsi{\lambda t}^{\otimes2}\cong\Lpsi{2\lambda t}\). For odd \(p\) this sheaf is geometrically nontrivial, and indeed Theorem~\ref{thm:FS}(i) and Corollary~\ref{cor:structure}(ii) below show that \(\mathcal W_n\) carries no nonzero invariant bilinear form when \(p\) is odd, and that it is not geometrically self-dual even up to a twist by a rank-one sheaf when \(p\) is odd and \(n\ge3\).

\subsection{Frobenius--Schur indicators from exact moments}\label{subsec:FS}

For a reductive group \(G\subseteq\GL(V)\) we put
\[
m_+(G,V):=\dim\bigl(\Sym^2V^\vee\bigr)^G,\qquad m_-(G,V):=\dim\bigl(\Lambda^2V^\vee\bigr)^G,
\]
the dimensions of the spaces of \(G\)-invariant symmetric and alternating bilinear forms on \(V\). A \(G\)-invariant bilinear form is the same as an element of \(\Hom_G(V,V^\vee)\); so if \(V\) is irreducible, then \(m_++m_-\le1\) by Schur's lemma, and \(m_+-m_-\) is the Frobenius--Schur indicator of \(V\). In this subsection these numbers are computed for \(V=W_n\) from two exact identities, independently of Proposition~\ref{prop:selfdual}. For \(h\in\Fq^n\) we write \(\Sigma_{n,q^2}(h)\) for the sum \eqref{eq:family} formed over \(\mathbb F_{q^2}\), with \(h\) regarded as a point of \(\mathbb F_{q^2}^n\).

\begin{proposition}\label{prop:exact-moments}
Let \(n\ge1\) and let \(\Fq\) be a finite field of characteristic \(p\).
\begin{enumerate}[label=\textup{(\roman*)}]
\item One has
\[
\sum_{h\in\Fq^n}\Sigma_n(h)^2=
\begin{cases}
q^{n-2}\bigl((q-1)^{n+1}+1-(q-2)(-1)^n\bigr),&p=2,\\[2pt]
q^{n-2}\bigl(1-(q-1)^n+2(-1)^n\bigr),&p\ \text{odd}.
\end{cases}
\]
\item One has
\[
\sum_{h\in(\Fq^\times)^n}\Sigma_{n,q^2}(h)=\bigl(q\cdot\mathbf 1_{p=2}-1\bigr)\,\frac{(q^2-q+1)^n-(1-q)^n}{q^2}.
\]
\end{enumerate}
\end{proposition}

\begin{proof}
(i) Put \(f(u)=\sum_iu_i^{-1}\). Expanding the square and summing over \(h\) first,
\[
\sum_{h\in\Fq^n}\Sigma_n(h)^2=\sum_{u,v\in(\Fq^\times)^n}\psi_q\bigl(f(u)+f(v)\bigr)\,N(u,v),
\]
where \(N(u,v)\) is the number of \(h\in\Fq^n\) with \(\langle h,u\rangle=\langle h,v\rangle=-\lambda\).
If \(u\) and \(v\) are linearly independent, \(N(u,v)=q^{n-2}\); if \(v=u\), \(N(u,v)=q^{n-1}\); and if \(v=cu\) with \(c\ne1\), \(N(u,v)=0\), the two equations being incompatible because \(\lambda\ne0\). Hence the sum equals \(q^{n-2}A+(q^{n-1}-q^{n-2})B-q^{n-2}C\) with
\[
A=\Bigl(\sum_u\psi_q(f(u))\Bigr)^2,\qquad B=\sum_u\psi_q\bigl(2f(u)\bigr),\qquad C=\sum_{c\in\Fq\setminus\{0,1\}}\sum_u\psi_q\bigl((1+c^{-1})f(u)\bigr).
\]
Each inner sum factors over the coordinates, and \(\sum_{x\in\Fq^\times}\psi_q(\beta x^{-1})\) equals \(q-1\) if \(\beta=0\) and \(-1\) otherwise. Thus \(A=1\). If \(p=2\), then \(B=(q-1)^n\), and \(1+c^{-1}\ne0\) for \(c\ne1\), so \(C=(q-2)(-1)^n\). If \(p\) is odd, then \(B=(-1)^n\), and \(1+c^{-1}=0\) exactly for \(c=-1\), so \(C=(q-1)^n+(q-3)(-1)^n\). Substituting gives (i).

(ii) By \eqref{eq:t-sum} over \(\mathbb F_{q^2}\), \(q^2\Sigma_{n,q^2}(h)=\sum_{t\in\mathbb F_{q^2}}\psi_{q^2}(\lambda t)\prod_iK_{q^2}(th_i)\), and summing over \(h\in(\Fq^\times)^n\),
\[
q^2\sum_{h\in(\Fq^\times)^n}\Sigma_{n,q^2}(h)=\sum_{t\in\mathbb F_{q^2}}\psi_{q^2}(\lambda t)\,A(t)^n,\qquad A(t):=\sum_{x\in\Fq^\times}K_{q^2}(tx).
\]
Let \(\mathrm T:\mathbb F_{q^2}\to\Fq\) be the trace, so that \(\psi_{q^2}=\psi_q\circ\mathrm T\). It is \(\Fq\)-linear and surjective, so \(\ker\mathrm T=\omega\Fq\) for some \(\omega\in\mathbb F_{q^2}^\times\); then \(\omega^q=-\omega\), hence \(\omega^2\in\Fq^\times\) and \(\mathrm T(\omega^2)=2\omega^2\). Expanding \(K_{q^2}(tx)=\sum_{y\in\mathbb F_{q^2}^\times}\psi_{q^2}(y^{-1}+txy)\) and summing over \(x\) first, using \(\psi_{q^2}(txy)=\psi_q(x\,\mathrm T(ty))\),
\[
A(t)=\sum_{y\in\mathbb F_{q^2}^\times}\psi_{q^2}(y^{-1})\bigl(q\,\mathbf 1_{\mathrm T(ty)=0}-1\bigr)=1+q\sum_{\substack{y\ne0\\ \mathrm T(ty)=0}}\psi_{q^2}(y^{-1}).
\]
For \(t=0\) the last sum is \(\sum_{y\ne0}\psi_{q^2}(y^{-1})=-1\), so \(A(0)=1-q\). For \(t\ne0\) we write \(y=\omega z/t\) with \(z\in\Fq^\times\); then \(\psi_{q^2}(y^{-1})=\psi_q\bigl(\mathrm T(t/\omega)z^{-1}\bigr)\), and the sum over \(z\) is \(q-1\) if \(\mathrm T(t/\omega)=0\) and \(-1\) otherwise. Since \(\mathrm T(t/\omega)=0\) if and only if \(t\in\omega^2\Fq^\times\), we get \(A(t)=1-q+q^2\,\mathbf 1_{t\in\omega^2\Fq^\times}\) for \(t\ne0\). Therefore
\[
q^2\sum_h\Sigma_{n,q^2}(h)=(1-q)^n\sum_{t\in\mathbb F_{q^2}}\psi_{q^2}(\lambda t)+\bigl((q^2-q+1)^n-(1-q)^n\bigr)\sum_{t\in\omega^2\Fq^\times}\psi_{q^2}(\lambda t).
\]
The first sum vanishes. For \(t=\omega^2z\), \(z\in\Fq^\times\), one has \(\psi_{q^2}(\lambda t)=\psi_q(2\lambda\omega^2z)\), so the second sum is \(q-1\) if \(p=2\) and \(-1\) if \(p\) is odd.
\end{proof}

For \(p=2\) the identity (i) is the formula of Remark~\ref{rem:second}, since then \(\Sigma_n(h)\) is real. The next lemma is the form of Katz's moment principle adapted to the second Adams operation \(g\mapsto g^2\).

\begin{lemma}\label{lem:adams}
Let \(S\), \(e\), \(\mathcal F\), \(w\), \(V\) and \(G=\ggeom(\mathcal F)\) be as in Theorem~\textup{\ref{thm:moment-principle}}, and put \(m_\pm=m_\pm(G,V)\). There are complex numbers \(\beta_1,\dots,\beta_{m_+}\) and \(\beta'_1,\dots,\beta'_{m_-}\) of absolute value \(1\) such that, for every \(a\ge1\), with \(q_a=\#k_a\),
\[
\begin{aligned}
q_a^{-e-w}\sum_{x\in S(k_a)}\Tr\bigl(\Frob_{k_a,x}\mid\mathcal F\bigr)^2&=\sum_{j=1}^{m_+}\beta_j^a+\sum_{j=1}^{m_-}\beta_j'^{\,a}+O\bigl(q_a^{-1/2}\bigr),\\
q_a^{-e-w}\sum_{x\in S(k_a)}\Tr\bigl(\Frob_{k_a,x}^2\mid\mathcal F\bigr)&=\sum_{j=1}^{m_+}\beta_j^a-\sum_{j=1}^{m_-}\beta_j'^{\,a}+O\bigl(q_a^{-1/2}\bigr),
\end{aligned}
\]
where \(\Frob_{k_a,x}^2\) is the geometric Frobenius of \(k_{2a}\) at \(x\). Moreover, if \(m_+=1\) \textup{(}resp.\ \(m_-=1\)\textup{)} and \(\beta_1=1\) \textup{(}resp.\ \(\beta'_1=1\)\textup{)}, then a nonzero \(\pgeom(S)\)-invariant symmetric \textup{(}resp.\ alternating\textup{)} bilinear form \(B_0\) on the fibre of \(\mathcal F\) satisfies \(B_0(\Frob_xv,\Frob_xw)=(\#k_a)^wB_0(v,w)\) for every \(a\ge1\) and every \(x\in S(k_a)\), the fibre at a geometric point over \(x\) being identified with \(V\) by a path.
\end{lemma}

\begin{proof}
For \(g\in\GL(V)\) one has \(\Tr(g)^2=\Tr(g\mid\Sym^2V)+\Tr(g\mid\Lambda^2V)\) and \(\Tr(g^2)=\Tr(g\mid\Sym^2V)-\Tr(g\mid\Lambda^2V)\). The sheaves \(\Sym^2\mathcal F\) and \(\Lambda^2\mathcal F\) are lisse and pure of weight \(2w\), and we apply the trace formula to them, as in the proof of Theorem~\ref{thm:moment-principle}. The groups \(H^i_c\) with \(i<2e\) contribute \(O(q_a^{e+w-1/2})\). By Poincar\'e duality, \(H^{2e}_c(S_{\bar k},\Sym^2\mathcal F)\) is the dual of \(H^0(S_{\bar k},\Sym^2\mathcal F^\vee)\), twisted by \((-e)\), and \(H^0(S_{\bar k},\Sym^2\mathcal F^\vee)\) is the space of \(\pgeom(S)\)-invariant, equivalently \(G\)-invariant, symmetric bilinear forms on \(V\), of dimension \(m_+\). Its injection into the stalk at a geometric point over a closed point \(x\) of degree \(\delta\) carries \(F_k^\delta\) to \(\Frob_x\), and \(\Sym^2\mathcal F^\vee\) is pure of weight \(-2w\); so the eigenvalues \(\gamma_j\) of \(F_k\) on it satisfy \(|\gamma_j|=(\#k)^{-w}\), and those on \(H^{2e}_c\) are \((\#k)^e\gamma_j^{-1}\). Putting \(\beta_j=(\#k)^{-w}\gamma_j^{-1}\) gives the contribution \(q_a^{e+w}\sum_j\beta_j^a\); the same argument with \(\Lambda^2\) gives the numbers \(\beta'_j\), and the two displayed formulas follow. If \(m_+=1\) and \(\beta_1=1\), then \(F_k\) acts on the line of invariant symmetric forms by \(\gamma_1=(\#k)^{-w}\), so \(\Frob_x\) acts on \(B_0\) by \((\#k_a)^{-w}\) for \(x\in S(k_a)\); as \(\Frob_x\) acts on bilinear forms by \((\Frob_x\cdot B_0)(v,w)=B_0(\Frob_x^{-1}v,\Frob_x^{-1}w)\), this is the last assertion. The alternating case is identical.
\end{proof}

\begin{theorem}\label{thm:FS}
Let \(n\ge1\), let \(G=\ggeom(\mathcal W_n)\), and put \(m_\pm=m_\pm(G,W_n)\).
\begin{enumerate}[label=\textup{(\roman*)}]
\item If \(p\) is odd, then \(m_+=m_-=0\): the representation \(W_n\) of \(G\) carries no nonzero invariant bilinear form.
\item If \(p=2\), then \((m_+,m_-)=(1,0)\) if \(n\) is odd and \((m_+,m_-)=(0,1)\) if \(n\) is even. A nonzero invariant form \(B_0\) is nondegenerate, and \(B_0(\Frob_hv,\Frob_hw)=q^{n+1}B_0(v,w)\) for every \(h\in S_n(\Fq)\).
\end{enumerate}
\end{theorem}

\begin{proof}
We apply Lemma~\ref{lem:adams} with \(k=\Fp\), \(S=S_n\), \(\mathcal F=\mathcal W_n\), \(e=n\) and \(w=n+1\). Let \(a\ge1\) and \(q=p^a\), and put \(c_n(q):=\sum_{k=1}^{\lfloor n/2\rfloor}\nu_{n,k}q^k\), so that \(0\le c_n(q)\le2^nq^{\lfloor n/2\rfloor}\). By Corollary~\ref{cor:exact}(i), applied over \(\Fq\) and over \(\mathbb F_{q^2}\), for \(h\in S_n(\Fq)\)
\[
\Tr(\Frob_h\mid\mathcal W_n)=(-1)^{n-1}q\,\Sigma_n(h)-c_n(q),\qquad\Tr(\Frob_h^2\mid\mathcal W_n)=(-1)^{n-1}q^2\,\Sigma_{n,q^2}(h)-c_n(q^2).
\]

\emph{The second moment.} By Corollary~\ref{cor:pointwise}, \(|\Sigma_n(h)|\le(2^n-1)q^{(n-1)/2}\) for \(h\in(\Fq^\times)^n\). Squaring the first identity and summing over the \((q-1)^n\) points of \(S_n(\Fq)\), the cross term is at most \(2q\,c_n(q)(q-1)^n(2^n-1)q^{(n-1)/2}=O_n(q^{2n+1/2})\), because \(\lfloor n/2\rfloor\le n/2\), and the term \((q-1)^nc_n(q)^2\) is \(O_n(q^{2n})\). If \(h\in\Fq^n\) has exactly \(r<n\) nonzero coordinates, then \(\Sigma_n(h)=0\) if \(r=0\), and \(|\Sigma_n(h)|^2\le(2^r-1)^2q^{r-1}\) if \(r\ge1\), by Lemma~\ref{lem:offtorus} and Corollary~\ref{cor:pointwise}; so the points of \(\Fq^n\setminus(\Fq^\times)^n\) contribute \(O_n(q^{2n-3})\) to \(\sum_{h\in\Fq^n}\Sigma_n(h)^2\). Hence, by Proposition~\ref{prop:exact-moments}(i),
\[
q^{-2n-1}\sum_{h\in S_n(\Fq)}\Tr(\Frob_h\mid\mathcal W_n)^2=q^{-2n+1}\sum_{h\in\Fq^n}\Sigma_n(h)^2+O_n(q^{-1/2})=\mathbf 1_{p=2}+O_n(q^{-1/2}).
\]

\emph{The Frobenius--Schur sum.} Summing the second identity over \(S_n(\Fq)\) and using Proposition~\ref{prop:exact-moments}(ii),
\[
\sum_{h\in S_n(\Fq)}\Tr(\Frob_h^2\mid\mathcal W_n)=(-1)^{n-1}\bigl(q\,\mathbf 1_{p=2}-1\bigr)\bigl((q^2-q+1)^n-(1-q)^n\bigr)-(q-1)^nc_n(q^2),
\]
and \((q-1)^nc_n(q^2)=O_n(q^{2n})\). Hence \(q^{-2n-1}\sum_h\Tr(\Frob_h^2\mid\mathcal W_n)=(-1)^{n-1}\mathbf 1_{p=2}+O_n(q^{-1})\).

\emph{Conclusion.} Put \(u_a=\sum_j\beta_j^a\) and \(v_a=\sum_j\beta_j'^{\,a}\). Comparing with Lemma~\ref{lem:adams}, \(u_a+v_a=\mathbf 1_{p=2}+O(p^{-a/2})\) and \(u_a-v_a=(-1)^{n-1}\mathbf 1_{p=2}+O(p^{-a/2})\), so
\[
u_a=\tfrac12\bigl(1+(-1)^{n-1}\bigr)\mathbf 1_{p=2}+O(p^{-a/2}),\qquad v_a=\tfrac12\bigl(1-(-1)^{n-1}\bigr)\mathbf 1_{p=2}+O(p^{-a/2}).
\]
As in the proof of Theorem~\ref{thm:moment-principle}, Dirichlet's simultaneous approximation theorem gives integers \(a_\nu\to\infty\) with \(\beta_j^{a_\nu}\to1\) and \(\beta_j'^{\,a_\nu}\to1\) for all \(j\); then \(u_{a_\nu}\to m_+\) and \(v_{a_\nu}\to m_-\). This determines \((m_+,m_-)\) as stated in (i) and (ii). If \(p=2\) and \(n\) is odd, then \(m_+=1\) and \(\beta_1^a=u_a\to1\) as \(a\to\infty\), so \(\beta_1=\lim_a\beta_1^{a+1}/\beta_1^a=1\); similarly \(\beta'_1=1\) if \(p=2\) and \(n\) is even. The last assertion of Lemma~\ref{lem:adams} gives the formula for \(B_0(\Frob_hv,\Frob_hw)\). Finally, the radical of \(B_0\) is a \(G\)-stable subspace different from \(W_n\). If \(n\ge3\), then \(W_n\) is an irreducible representation of \(G\) by Proposition~\ref{prop:M4}, so the radical is zero. If \(n\le2\), then \(\dim W_n=r(n)\le2\), and a nonzero bilinear form on a space of dimension one, or a nonzero alternating form on a space of dimension two, is nondegenerate.
\end{proof}

\emph{Comparison of the two proofs of self-duality.} Proposition~\ref{prop:selfdual} and Theorem~\ref{thm:FS}(ii) establish the self-duality of \(\mathcal W_n\) in characteristic two independently, and each gives information that the other does not.
\begin{enumerate}[label=(\alph*)]
\item The cohomological route constructs the form. The pairing \(B\) is a morphism of lisse sheaves over \(\Fp\), obtained from the cup product on \(R^1\rho_!\mathcal N_n\); it is defined on \(\mathcal H_n\), where its radical is exactly \(\mathcal E_n\). Its nondegeneracy on \(\mathcal W_n\) is proved for every \(n\ge1\) without the fourth moment; its symmetry type comes from the Koszul sign rule, and its Frobenius multiplier \(q^{n+1}\) from the Tate twist \(\Qlb(-n-1)\). The construction does not show that \(B\) is the only invariant form, and for odd \(p\) it says only that it fails, because \(\Lpsi{\lambda t}^{\otimes2}\) is then geometrically nontrivial.
\item The arithmetic route counts the forms. Proposition~\ref{prop:exact-moments} consists of two identities between exponential sums, and Lemma~\ref{lem:adams} turns them into the values of \(m_+\) and \(m_-\). It gives the uniqueness of the invariant form up to scalars, its symmetry type without any sign convention, the multiplier \(q^{n+1}\), and the nonexistence of invariant forms when \(p\) is odd. It is asymptotic in \(q\), it does not produce the form as a morphism of sheaves, and for \(n\ge3\) it obtains nondegeneracy from the irreducibility of \(W_n\), that is, from Proposition~\ref{prop:M4} and Theorem~\ref{thm:fourth}.
\end{enumerate}
Together, the two routes show that \(B\) spans the space of \(\ggeom(\mathcal W_n)\)-invariant bilinear forms on \(W_n\), so that the form \(B_0\) of Theorem~\ref{thm:FS} is a multiple of \(B\) (by Theorem~\ref{thm:FS}(ii) the invariant forms form a line, and \(B\) lies on it by Proposition~\ref{prop:selfdual}), and they check each other at the two places where a convention could go astray, the sign \((-1)^{n+1}\) and the multiplier \(q^{n+1}\). For odd \(p\) the identity of Proposition~\ref{prop:exact-moments}(i) also appears in the companion preprint \cite{EbadiCoh}, where it is used to show that \(\mathcal W_n\) is not self-dual; for this the second moment alone suffices, since it gives \(m_++m_-=0\). For \(n=2\) and odd \(p\), Theorem~\ref{thm:FS}(i) shows in particular that \(\ggeom(\mathcal W_2)\not\subseteq\SL(W_2)\), since \(\SL(W_2)\) preserves a nonzero alternating form; Theorem~\ref{thm:det}(iii) below shows more precisely that the order of the geometric determinant is divisible by \(p\) for every \(n\ge2\).

\subsection{Local Fourier transforms and the end of a ray}\label{subsec:ray}

\emph{Setting.} In this subsection the ground field is \(k=\Fpbar\); it is perfect, as required in \cite{Laumon}. Let \(\mathbb A^1_s\) and \(\mathbb A^1_z\) be in duality via the pairing \((s,z)\mapsto sz\), let \(\Lpsi{sz}\) be the pull-back of the Artin--Schreier sheaf \(\mathcal L_\psi\) along this pairing, and let \(\mathrm{pr}_s,\mathrm{pr}_z\) be the projections of \(\mathbb A^1_s\times\mathbb A^1_z\). Let \(I_0\) and \(I_\infty\) be the Galois groups of the fields of fractions of the henselisations of \(\mathbb P^1_s\) at \(s=0\) and at \(s=\infty\), and let \(I_{\infty'}\) be the analogous group at \(z=\infty\). Since \(k\) is algebraically closed, these are inertia groups; \(P_0\), \(P_\infty\), \(P_{\infty'}\) denote their wild inertia subgroups, and each quotient \(I/P\) is procyclic, isomorphic to \(\prod_{\ell'\ne p}\mathbb Z_{\ell'}(1)\). For such a group \(I\) we write \(\mathrm{Rep}(I)\) for the category of finite-dimensional continuous \(\Qlb\)-representations of \(I\), each defined over a finite extension of \(\Qell\), and \(\mathbf 1\) for the trivial one-dimensional representation. For a constructible sheaf \(\mathcal G\) on a dense open subset of \(\mathbb A^1_s\) (resp.\ \(\mathbb A^1_z\)) we write \(\mathcal G|_{I_c}\), \(c\in\{0,\infty\}\) (resp.\ \(\mathcal G|_{I_{\infty'}}\)), for its geometric generic stalk with its action of \(I_c\) (resp.\ \(I_{\infty'}\)); this is the representation denoted \(\mathcal G(c)\) in \cite{KatzB}.

If \(P\) acts trivially on \(V\in\mathrm{Rep}(I)\), we write \(\tau_V\) for the automorphism of \(V\) defined by a topological generator of \(I/P\). By continuity the image of \(I\) in \(\GL(V)\) is contained in the closure of the group generated by \(\tau_V\). Since the stabiliser of a subspace, and the subgroup acting trivially on a given subquotient, are closed, it follows that a subspace of \(V\) is \(I\)-stable if and only if it is \(\tau_V\)-stable, that \(I\) acts trivially on every \(\tau_V\)-stable subquotient on which \(\tau_V\) is the identity, and that \(V^I=\ker(\tau_V-1)\). If \(\tau_V\) is unipotent, its Jordan type does not depend on the choice of the generator: it is determined by the dimensions of the subspaces \(V_i=\ker(\tau_V-1)^i\), and these are described without reference to the generator, by \(V_0=0\) and \(V_i/V_{i-1}=(V/V_{i-1})^I\). We say that \(V\) is \emph{of type \(J_m\)} if \(P\) acts trivially and \(\tau_V\) is a single unipotent Jordan block of size \(m\).

\emph{Local Fourier transforms.} Laumon \cite[(2.4.1)--(2.4.2.3)]{Laumon} defines the local Fourier transforms
\[
F^{(0,\infty')}:\mathrm{Rep}(I_0)\to\mathrm{Rep}(I_{\infty'}),\qquad F^{(\infty,\infty')}:\mathrm{Rep}(I_\infty)\to\mathrm{Rep}(I_{\infty'})
\]
as follows. Let \(T\) and \(T'\) be henselian traits over \(k\) with uniformisers \(\pi\) and \(\pi'\), view \(V\in\mathrm{Rep}(I_0)\) (resp.\ \(V\in\mathrm{Rep}(I_\infty)\)) as a lisse sheaf on the generic point of \(T\), and let \(V_!\) be its extension by zero to \(T\). Let \(\mathcal L_\psi(\pi/\pi')_!\) and \(\mathcal L_\psi(1/\pi\pi')_!\) be the extensions by zero to \(T\times_kT'\) of the Artin--Schreier sheaves \(\mathcal L_\psi(\pi/\pi')\) on \(T\times_k\eta'\) and \(\mathcal L_\psi(1/\pi\pi')\) on \(\eta\times_k\eta'\), where \(\eta,\eta'\) are the generic points of \(T,T'\). Then \(F^{(0,\infty')}(V)\) (resp.\ \(F^{(\infty,\infty')}(V)\)) is the stalk at the closed point of \(T\) of the first sheaf of vanishing cycles, relative to the projection \(T\times_kT'\to T'\), of \(\mathrm{pr}^*(V_!)\otimes\mathcal L_\psi(\pi/\pi')_!\) (resp.\ of \(\mathrm{pr}^*(V_!)\otimes\mathcal L_\psi(1/\pi\pi')_!\)) \cite[(2.4.1), D\'efinition (2.4.2.3)]{Laumon}; by \cite[Proposition (2.4.2.2)]{Laumon} these vanishing cycles are concentrated at that point and in degree \(1\). Laumon works over a perfect field and with the Galois groups of the henselian fields; over \(k=\Fpbar\) these are the inertia groups above. The following properties are used; part (iv) concerns the same constructions over a finite field.

\begin{theorem}[Laumon]\label{thm:laumon}
\begin{enumerate}[label=\textup{(\roman*)}]
\item The functor \(F^{(0,\infty')}\) is exact, and it is an equivalence of abelian categories from \(\mathrm{Rep}(I_0)\) onto a strictly full subcategory of \(\mathrm{Rep}(I_{\infty'})\), namely the representations all of whose slopes are \(<1\). In particular, \(F^{(0,\infty')}\) is fully faithful: for \(M,M'\in\mathrm{Rep}(I_0)\) it induces a \(\Qlb\)-linear isomorphism
\[
\Hom_{I_0}(M,M')\xrightarrow{\ \sim\ }\Hom_{I_{\infty'}}\bigl(F^{(0,\infty')}(M),F^{(0,\infty')}(M')\bigr).
\]
\item \textup{(Stationary phase.)} Let \(\mathcal G\) be a lisse sheaf on \(\Gm[,s]\), let \(j:\Gm[,s]\hookrightarrow\mathbb A^1_s\) be the inclusion, and let \(\mathcal F\) be the constructible sheaf
\[
\mathcal F:=R^1\mathrm{pr}_{z!}\bigl(\mathrm{pr}_s^*j_!\mathcal G\otimes\Lpsi{sz}\bigr)\qquad\text{on }\mathbb A^1_z .
\]
Then there is an isomorphism of representations of \(I_{\infty'}\)
\[
\mathcal F|_{I_{\infty'}}\cong F^{(0,\infty')}\bigl(\mathcal G|_{I_0}\bigr)\oplus F^{(\infty,\infty')}\bigl(\mathcal G|_{I_\infty}\bigr).
\]
\item \textup{(Several singular points.)} Let \(\Xi\subset\mathbb A^1_s(k)\) be finite, \(U=\mathbb A^1_s\setminus\Xi\), \(j_U:U\hookrightarrow\mathbb A^1_s\) the inclusion, \(\mathcal G\) a lisse sheaf on \(U\), and \(\mathcal F:=R^1\mathrm{pr}_{z!}(\mathrm{pr}_s^*j_{U!}\mathcal G\otimes\Lpsi{sz})\). For \(\xi\in\Xi\) let \(\tau_\xi:\mathbb A^1_s\to\mathbb A^1_s\) be the translation \(s\mapsto s+\xi\), so that \((\tau_\xi^*\mathcal G)|_{I_0}\) is the local monodromy of \(\mathcal G\) at \(\xi\). Then there is an isomorphism of representations of \(I_{\infty'}\)
\[
\mathcal F|_{I_{\infty'}}\cong\bigoplus_{\xi\in\Xi}\Bigl(\Lpsi{\xi z}|_{I_{\infty'}}\otimes F^{(0,\infty')}\bigl((\tau_\xi^*\mathcal G)|_{I_0}\bigr)\Bigr)\oplus F^{(\infty,\infty')}\bigl(\mathcal G|_{I_\infty}\bigr).
\]
\item \textup{(Arithmetic form.)} Let \(k_0\subset k\) be a finite subfield, and let \(D_0\), \(D_\infty\), \(D_{\infty'}\) be the Galois groups of the fields of fractions of the henselisations at \(s=0\), \(s=\infty\), \(z=\infty\) of \(\mathbb P^1_s\) and \(\mathbb P^1_z\) over \(k_0\); they contain \(I_0\), \(I_\infty\), \(I_{\infty'}\) as normal subgroups with quotient \(\mathrm{Gal}(\bar k/k_0)\). The same constructions over \(k_0\) define functors \(F^{(0,\infty')}:\mathrm{Rep}(D_0)\to\mathrm{Rep}(D_{\infty'})\) and \(F^{(\infty,\infty')}:\mathrm{Rep}(D_\infty)\to\mathrm{Rep}(D_{\infty'})\) which, followed by restriction to \(I_{\infty'}\), are the functors above applied to the restrictions to \(I_0\) and \(I_\infty\). The functor \(F^{(0,\infty')}\) is exact, and \(F^{(0,\infty')}(V\otimes\chi)\cong F^{(0,\infty')}(V)\otimes\chi\) for every character \(\chi\) of \(\mathrm{Gal}(\bar k/k_0)\), regarded as a character of \(D_0\) and of \(D_{\infty'}\). If \(\mathcal G\) in \textup{(ii)} is a lisse sheaf on \(\Gm[,s]\) over \(k_0\), the isomorphism of \textup{(ii)} is an isomorphism of representations of \(D_{\infty'}\).
\end{enumerate}
\end{theorem}

\begin{proof}
(i) This is \cite[Th\'eor\`eme (2.4.3)(i)(a),(c)]{Laumon}; the target category there is the strictly full subcategory of representations with slopes in \([0,1[\), and an equivalence onto a full subcategory is fully faithful. The functor \(F^{(0,\infty')}\) is \(\Qlb\)-linear, being defined by vanishing cycles, so the bijections on \(\Hom\)-sets are \(\Qlb\)-linear.

(ii) We may assume \(\mathcal G\ne0\). The complex \(\mathcal P:=j_!\mathcal G[1]\) is a perverse sheaf on \(\mathbb A^1_s\), because \(j\) is an affine open immersion and \(\mathcal G[1]\) is perverse on \(\Gm[,s]\) \cite[Corollaire 4.1.3]{BBD}; directly, its only nonzero cohomology sheaf \(j_!\mathcal G\) sits in degree \(-1\) and has no nonzero sections supported on a finite set. The open set of lissity of \(\mathcal P\) is \(\Gm[,s]\), and its complement in \(\mathbb A^1_s\) is the point \(0\). Laumon's Fourier transform is \(\mathcal F_\psi(\mathcal P)=R\mathrm{pr}_{z!}(\mathrm{pr}_s^*\mathcal P\otimes\Lpsi{sz})[1]\) \cite[(1.2.1.1), (1.4.1)]{Laumon}, so that \(\mathcal H^{-1}(\mathcal F_\psi(\mathcal P))=\mathcal F\). By \cite[Proposition (2.3.3.1)(iii)]{Laumon}, the geometric generic stalk of \(\mathcal H^{-1}(\mathcal F_\psi(\mathcal P))\) at \(\infty'\) is isomorphic, as a representation of \(I_{\infty'}\), to the direct sum of the stalks at the points \((0,\bar\eta_{\infty'})\) and \((\infty,\bar\eta_{\infty'})\) of
\[
R^{-1}\Phi_{\bar\eta_{\infty'}}\bigl(\mathrm{pr}^*(a_!\mathcal P)\otimes\Lpsi{sz}[1]\bigr).
\]
Here \(Z'\) is the henselisation of \(\mathbb P^1_z\) at \(\infty\), with generic point \(\eta_{\infty'}\); \(a:\mathbb A^1_s\hookrightarrow\mathbb P^1_s\) is the inclusion; the sheaf \(\Lpsi{sz}\) on \(\mathbb A^1_s\times_k\eta_{\infty'}\) is extended by zero to \(\mathbb P^1_s\times_kZ'\); and the vanishing cycles are taken relative to the projection \(\mathbb P^1_s\times_kZ'\to Z'\). (Laumon's statement contains in addition induction functors from the decomposition groups of the points of this complement; they are identities here because \(k\) is algebraically closed.) On the henselian trait at \(s=0\) the restriction of \(\mathcal P\) is \((\mathcal G|_{I_0})_![1]\), the extension by zero of the generic stalk placed in degree \(-1\), because \(\mathcal P\) is the extension by zero of \(\mathcal G[1]\) from \(\Gm[,s]\). Hence \cite[Lemme (2.4.2.1)(i)]{Laumon}, with \(\pi=s\) and \(\pi'=1/z\), identifies the first stalk with the stalk of \(R^1\Phi\bigl(\mathrm{pr}^*((\mathcal G|_{I_0})_!)\otimes\mathcal L_\psi(\pi/\pi')_!\bigr)\) at the closed point, that is, with \(F^{(0,\infty')}(\mathcal G|_{I_0})\); the shift by \(2\) in \cite[(2.4.2.1)]{Laumon} turns \(R^{-1}\Phi\) into \(R^1\Phi\). Likewise, at the generic point of the henselian trait at \(s=\infty\) the restriction of \(\mathcal P\) is \(\mathcal G|_{I_\infty}[1]\), and \cite[Lemme (2.4.2.1)(iii)]{Laumon}, with \(\pi=1/s\) and \(\pi'=1/z\), identifies the second stalk with \(F^{(\infty,\infty')}(\mathcal G|_{I_\infty})\). This form of the stationary-phase principle, for extensions by zero of lisse sheaves on dense open subsets of \(\mathbb A^1\), is also stated in \cite[Theorem 3]{KatzB}.

(iii) We may assume \(\mathcal G\ne0\). As in (ii), \(\mathcal P:=j_{U!}\mathcal G[1]\) is perverse, its set of singular points in \(\mathbb A^1_s\) is \(\Xi\), and \cite[Proposition (2.3.3.1)(iii)]{Laumon} identifies \(\mathcal F|_{I_{\infty'}}\) with the direct sum, over \(\xi\in\Xi\cup\{\infty\}\), of the stalks at \((\xi,\bar\eta_{\infty'})\) of \(R^{-1}\Phi_{\bar\eta_{\infty'}}(\mathrm{pr}^*(a_!\mathcal P)\otimes\Lpsi{sz}[1])\). The term \(\xi=\infty\) is \(F^{(\infty,\infty')}(\mathcal G|_{I_\infty})\), by \cite[Lemme (2.4.2.1)(iii)]{Laumon} as in (ii). Let \(\xi\in\Xi\). The automorphism \(\tau_\xi\times\mathrm{id}\) of \(\mathbb P^1_s\times_kZ'\) (extended by \(\infty\mapsto\infty\)) commutes with the projection to \(Z'\), carries the point \((0,\bar\eta_{\infty'})\) to \((\xi,\bar\eta_{\infty'})\), and pulls \(\mathrm{pr}^*(a_!\mathcal P)\otimes\Lpsi{sz}\) back to \(\mathrm{pr}^*(a_!\tau_\xi^*\mathcal P)\otimes\Lpsi{sz}\otimes\Lpsi{\xi z}\), where \(\Lpsi{\xi z}\) is a lisse sheaf on \(\eta_{\infty'}\), pulled back along the projection. These complexes vanish on the special fibre over \(Z'\), so their vanishing cycles are their nearby cycles, which are computed on the geometric generic fibre; there the pull-back of \(\Lpsi{\xi z}\) is constant, and so the nearby cycles of the second complex are those of \(\mathrm{pr}^*(a_!\tau_\xi^*\mathcal P)\otimes\Lpsi{sz}\) tensored with the character \(\Lpsi{\xi z}|_{I_{\infty'}}\). On the henselian trait at \(s=0\) the restriction of \(\tau_\xi^*\mathcal P\) is \(((\tau_\xi^*\mathcal G)|_{I_0})_![1]\), and \cite[Lemme (2.4.2.1)(i)]{Laumon} identifies the stalk at \((0,\bar\eta_{\infty'})\) with \(F^{(0,\infty')}((\tau_\xi^*\mathcal G)|_{I_0})\), as in (ii). This is the stationary-phase formula of \cite[Theorem 3]{KatzB}.

(iv) Laumon's constructions and results are stated over an arbitrary perfect base field \cite[(2.3.3.1), (2.4.1)--(2.4.3)]{Laumon}, with the Galois groups of the henselian fields in place of inertia groups and with induction functors from the decomposition groups of the singular points; over \(k_0\) the points \(0\) and \(\infty\) are rational, so these induction functors are identities, and the isomorphism of (ii) is \(D_{\infty'}\)-equivariant. Vanishing cycles commute with extension of the base field from \(k_0\) to \(k\), which gives the compatibility with restriction to inertia; exactness may be checked after restriction to \(I_{\infty'}\), where it is (i). Finally, twisting \(V\) by \(\chi\) twists the complex \(\mathrm{pr}^*(V_!)\otimes\mathcal L_\psi(\pi/\pi')_!\) by the pull-back of \(\chi\) from \(\mathrm{Spec}\,k_0\), which is lisse on the base \(T'\) of the vanishing cycles; vanishing cycles commute with such twists, whence \(F^{(0,\infty')}(V\otimes\chi)\cong F^{(0,\infty')}(V)\otimes\chi\).
\end{proof}

Three comments on Theorem~\ref{thm:laumon}(ii) and (iii). First, the local term at a singular point is the whole local monodromy representation, for instance \(\mathcal G|_{I_0}\) at \(s=0\), and not its quotient by the invariants, which is the input in the version of stationary phase for middle extensions \(j_*\mathcal G\); the reason is that the stalks of the extension by zero at the singular points vanish. This is the version needed below, since \(\mathcal H_n\) is defined by compactly supported cohomology along \(\Gm[,t]\). Second, Laumon's results hold for every nontrivial \(\psi\), so the normalisation of \(\mathcal L_\psi\) plays no role. Third, the Tate twists in Laumon's formulas, for instance in the quasi-inverse of \(F^{(0,\infty')}\), are characters of \(\mathrm{Gal}(\bar k/k)\) and are trivial over \(k=\Fpbar\); and Lemmas~\ref{lem:F-trivial} and~\ref{lem:jordan} below show that no Kummer twist arises on tame unipotent representations.

\begin{lemma}\label{lem:F-trivial}
\(F^{(0,\infty')}(\mathbf 1)\cong\mathbf 1\) and \(F^{(\infty,\infty')}(\mathbf 1)=0\).
\end{lemma}

\begin{proof}
Apply Theorem~\ref{thm:laumon}(ii) with \(\mathcal G=\Qlb\). Over \(\Gm[,z]\), the sheaf \(\mathcal F\) is \(R^1\mathrm{pr}_{z!}\Lpsi{sz}\) for the projection \(\mathrm{pr}_z:\Gm[,s]\times\Gm[,z]\to\Gm[,z]\). The automorphism \((s,z)\mapsto(y,z)=(sz,z)\) commutes with \(\mathrm{pr}_z\) and carries \(\Lpsi{sz}\) to \(\mathrm{pr}_y^*\Lpsi{y}\), so by the K\"unneth formula \(\mathcal F|_{\Gm[,z]}\) is the constant sheaf with fibre \(H^1_c(\Gm,\Lpsi{y})\). This space is one-dimensional: \(\Lpsi{y}\) is lisse on \(\Gm\), unramified at \(0\) and of Swan conductor \(1\) at \(\infty\), so \(H^0_c=H^2_c=0\) and \(\chi_c(\Gm,\Lpsi{y})=-1\). Hence \(\mathcal F|_{I_{\infty'}}\cong\mathbf 1\), and Theorem~\ref{thm:laumon}(ii) gives \(\mathbf 1\cong F^{(0,\infty')}(\mathbf 1)\oplus F^{(\infty,\infty')}(\mathbf 1)\). By Theorem~\ref{thm:laumon}(i), \(\End_{I_{\infty'}}\bigl(F^{(0,\infty')}(\mathbf 1)\bigr)\cong\End_{I_0}(\mathbf 1)\ne0\), so \(F^{(0,\infty')}(\mathbf 1)\) is a nonzero direct summand of \(\mathbf 1\), hence equal to it, and then \(F^{(\infty,\infty')}(\mathbf 1)=0\) by counting dimensions. (This agrees with \cite[(2.5.3.1)(i)]{Laumon}.)
\end{proof}

\begin{lemma}\label{lem:jordan}
Let \(M\in\mathrm{Rep}(I_0)\) be a representation on which \(P_0\) acts trivially and \(\tau_M\) is unipotent, with Jordan blocks of sizes \(m_1,\dots,m_k\). Then \(\widehat M:=F^{(0,\infty')}(M)\) has dimension \(\dim M\), the group \(P_{\infty'}\) acts trivially on \(\widehat M\), and \(\tau_{\widehat M}\) is unipotent, with Jordan blocks of sizes \(m_1,\dots,m_k\). In particular, \(F^{(0,\infty')}\) carries every representation of type \(J_m\) to a representation of type \(J_m\).
\end{lemma}

\begin{proof}
Write \(F=F^{(0,\infty')}\); being exact, \(F\) is additive. A decomposition of \(M\) into \(\tau_M\)-stable subspaces on each of which \(\tau_M\) is a single Jordan block is a decomposition \(M=\bigoplus_iM_i\) into subrepresentations of types \(J_{m_i}\), and \(F(M)\cong\bigoplus_iF(M_i)\). It therefore suffices to show that if \(M\) is of type \(J_m\), then \(\widehat M=F(M)\) is of type \(J_m\).

The subspaces \(M_i=\ker(\tau_M-1)^i\), \(0\le i\le m\), are subrepresentations of \(M\), and \(M_i/M_{i-1}\cong\mathbf 1\). Since \(F\) is exact, the subrepresentations \(F(M_i)\) form a chain \(0=F(M_0)\subset\cdots\subset F(M_m)=\widehat M\) with \(F(M_i)/F(M_{i-1})\cong F(\mathbf 1)\cong\mathbf 1\), by Lemma~\ref{lem:F-trivial}. Hence \(\dim\widehat M=m\), and every element of \(I_{\infty'}\) acts on \(\widehat M\) as a unipotent automorphism. The image of the pro-\(p\) group \(P_{\infty'}\) in \(\GL(\widehat M)\) is finite: it lies in a compact subgroup of \(\GL_m(E)\) for a finite extension \(E/\Qell\), such a subgroup stabilises an \(\mathcal O_E\)-lattice \(\mathfrak L\), and the elements of \(\GL(\mathfrak L)\) congruent to \(1\) modulo \(\ell\) form an open pro-\(\ell\) subgroup of \(\GL(\mathfrak L)\), which meets the pro-\(p\) image of \(P_{\infty'}\) trivially. A unipotent automorphism of finite order in characteristic zero is the identity; so \(P_{\infty'}\) acts trivially on \(\widehat M\), and \(\tau_{\widehat M}\) is unipotent. Finally, by Theorem~\ref{thm:laumon}(i) and Lemma~\ref{lem:F-trivial},
\[
\begin{aligned}
\dim\ker(\tau_{\widehat M}-1)&=\dim\widehat M^{I_{\infty'}}=\dim\Hom_{I_{\infty'}}\bigl(F(\mathbf 1),F(M)\bigr)\\
&=\dim\Hom_{I_0}(\mathbf 1,M)=\dim\ker(\tau_M-1)=1 .
\end{aligned}
\]
A unipotent automorphism has as many Jordan blocks as the dimension of its fixed space, so \(\tau_{\widehat M}\) is a single Jordan block, of size \(\dim\widehat M=m\).
\end{proof}

For \(a=(a_1,\dots,a_n)\in(\Fq^\times)^n\) let \(\Delta_a:\Gm[,x]\to S_n\) be the ray \(x\mapsto(xa_1,\dots,xa_n)\), and put
\[
\mathcal K_a:=\bigotimes_{i=1}^n[\times a_i]^*\Kl\qquad\text{on }\Gm[,s].
\]

\begin{proposition}\label{prop:ray}
Let \(n\ge1\) and \(a\in(\Fq^\times)^n\). The local monodromy of \(\Delta_a^*\mathcal H_n\) at \(x=0\) is tame and unipotent: the wild inertia group acts trivially, and a topological generator of tame inertia acts unipotently, with the Jordan type of \(J_2^{\otimes n}\) \textup{(}Lemma~\textup{\ref{lem:cg})}. In particular, if \(g_0\) is its image, then the fixed space of \(g_0\) has dimension \(\binom n{\lfloor n/2\rfloor}\) and \(g_0-1\) has rank \(r(n)\).
\end{proposition}

\begin{proof}
We work over \(\Fpbar\). Let \(\mu_i(t,x)=a_ixt\) and \(\mathcal N':=\bigotimes_i\mu_i^*\Kl\otimes\mathrm{pr}_t^*\Lpsi{\lambda t}\) on \(\Gm[,t]\times\Gm[,x]\); this is the pull-back of \(\mathcal N_n\) along \(\mathrm{id}\times\Delta_a\). By Theorem~\ref{thm:package}(i) and proper base change, \(\Delta_a^*\mathcal H_n=R^1\mathrm{pr}_{x!}\mathcal N'\), and \(R^i\mathrm{pr}_{x!}\mathcal N'=0\) for \(i\ne1\). The automorphism \(\phi(s,x)=(s/x,x)\) of \(\Gm\times\Gm\) commutes with the projections to \(\Gm[,x]\), and \(\phi^*\mathcal N'=\mathrm{pr}_s^*\mathcal K_a\otimes\Lpsi{\lambda s/x}\). Let \(\sigma:\Gm[,x]\to\Gm[,z]\) be the isomorphism \(\sigma(x)=\lambda/x\), so that \(\Lpsi{\lambda s/x}=(\mathrm{id}\times\sigma)^*\Lpsi{sz}\). By proper base change along \(\sigma\),
\[
\Delta_a^*\mathcal H_n\cong\sigma^*\mathcal F_a,\qquad\mathcal F_a:=R^1\mathrm{pr}_{z!}\bigl(\mathrm{pr}_s^*\mathcal K_a\otimes\Lpsi{sz}\bigr),
\]
where now \(\mathrm{pr}_z:\Gm[,s]\times\Gm[,z]\to\Gm[,z]\). Thus \(\mathcal F_a\) is lisse of rank \(2^n\) on \(\Gm[,z]\), and it is the restriction to \(\Gm[,z]\) of the sheaf \(\mathcal F\) of Theorem~\ref{thm:laumon}(ii) for \(\mathcal G=\mathcal K_a\), since integrating \(j_!\mathcal K_a\) over \(\mathbb A^1_s\) is the same as integrating \(\mathcal K_a\) over \(\Gm[,s]\). As \(\sigma\) extends to an isomorphism \(\mathbb P^1_x\to\mathbb P^1_z\) with \(\sigma(0)=\infty\), it identifies the inertia group at \(x=0\) with \(I_{\infty'}\), compatibly with the wild inertia subgroups and with the actions on \(\Delta_a^*\mathcal H_n\) and \(\mathcal F_a\). It therefore suffices to prove the assertions for \(\mathcal F_a|_{I_{\infty'}}\).

By Theorem~\ref{thm:kl}(ii) and the remark following it, \(P_0\) acts trivially on \(\mathcal K_a|_{I_0}\), and \(\tau_{\mathcal K_a|_{I_0}}\) is a tensor product of \(n\) unipotent Jordan blocks of size \(2\), so it has the Jordan type of \(J_2^{\otimes n}\). By Lemma~\ref{lem:jordan}, \(F^{(0,\infty')}(\mathcal K_a|_{I_0})\) has dimension \(2^n\), the group \(P_{\infty'}\) acts trivially on it, and a topological generator of \(I_{\infty'}/P_{\infty'}\) acts on it with the Jordan type of \(J_2^{\otimes n}\). By Theorem~\ref{thm:laumon}(ii),
\[
\mathcal F_a|_{I_{\infty'}}\cong F^{(0,\infty')}\bigl(\mathcal K_a|_{I_0}\bigr)\oplus F^{(\infty,\infty')}\bigl(\mathcal K_a|_{I_\infty}\bigr).
\]
The left-hand side has dimension \(\rank\mathcal F_a=2^n\), which is already the dimension of the first summand; hence the second summand vanishes, and \(\mathcal F_a|_{I_{\infty'}}\cong F^{(0,\infty')}(\mathcal K_a|_{I_0})\). (The vanishing is also forced by the slopes: all breaks of \(\mathcal K_a\) at \(\infty\) are \(\le\frac12\) by Theorem~\ref{thm:kl}(iii) and Lemma~\ref{lem:breaks}(i), and \(F^{(\infty,\infty')}\) kills every representation all of whose breaks are \(\le1\) \cite[Corollary 7]{KatzB}; this is not used.) The final assertion is Lemma~\ref{lem:cg}.
\end{proof}

\begin{corollary}\label{cor:unipotent}
Let \(n\ge3\), \(a\in(\Fq^\times)^n\), and let \(g_0\) be as in Proposition~\textup{\ref{prop:ray}}. Then \(g_0\) acts on \(\Delta_a^*\mathcal W_n\) as a unipotent automorphism different from the identity. Consequently \(\ggeom(\mathcal W_n)\) contains a unipotent element different from the identity.
\end{corollary}

\begin{proof}
The exact sequence \(0\to\mathcal E_n\to\mathcal H_n\to\mathcal W_n\to0\) of Theorem~\ref{thm:package} restricts to the ray and is respected by \(g_0\). The action on the quotient \(\Delta_a^*\mathcal W_n\) is unipotent because the action on \(\Delta_a^*\mathcal H_n\) is. If it were the identity, then \((g_0-1)\Delta_a^*\mathcal H_n\subseteq\Delta_a^*\mathcal E_n\), and Proposition~\ref{prop:ray} would give \(r(n)=\rank(g_0-1)\le\rank\mathcal E_n=\binom n{\lfloor n/2\rfloor}\), contradicting Lemma~\ref{lem:binom-gap}. Let \(\gamma\in\pgeom(\Gm[,x])\) be an element of the inertia group at \(x=0\) whose image in the tame quotient is a topological generator, and let \(\gamma'\) be its image under \(\Delta_{a*}:\pgeom(\Gm[,x])\to\pgeom(S_n)\). Identifying the fibre of \(\mathcal W_n\) at a geometric point of the ray with \(W_n\) by a path, which changes the action of \(\gamma'\) only by conjugation, \(\gamma'\) acts on \(W_n\) through the automorphism induced by \(g_0\); so it is a unipotent element of \(\ggeom(\mathcal W_n)\) different from the identity.
\end{proof}

For \(n=2\) the rank count in this proof is not available, since \(r(2)=2=\binom21\); it is replaced by an argument with weights, which uses the structure of \(\mathcal E_2\) established in Section~\ref{subsec:lower}. (The conclusion also follows from Proposition~\ref{prop:n2-kl}; see Remark~\ref{rem:n2-kl}(a).)

\begin{proposition}\label{prop:n2}
Let \(n=2\), let \(a\in(\Fq^\times)^2\), and let \(g_0\) be as in Proposition~\textup{\ref{prop:ray}}. Then \(g_0\) acts on \(\Delta_a^*\mathcal W_2\) as a unipotent automorphism different from the identity. Consequently \(\ggeom(\mathcal W_2)\) contains a unipotent element different from the identity.
\end{proposition}

\begin{proof}
We regard the ray \(\Delta_a\) as a morphism over \(\Fq\), and we consider the point \(x=0\) of \(\mathbb A^1_x\), a decomposition group \(D_0\) at \(x=0\) with inertia group \(I_0\), and an element \(\Phi\in D_0\) mapping to the geometric Frobenius of \(\Fq\). Let \(H\) be the geometric generic stalk of \(\Delta_a^*\mathcal H_2\), with its action \(\varrho\) of \(D_0\). Let \(W_0\mathcal E_2\) be the weight-zero part of \(\mathcal E_2\), and let \(L_0\subset E\subset H\) be the stalks of \(\Delta_a^*W_0\mathcal E_2\subset\Delta_a^*\mathcal E_2\); they are \(D_0\)-stable, and \(\bar H:=H/E\) is the stalk of \(\Delta_a^*\mathcal W_2\).

By Proposition~\ref{prop:ray} and Lemma~\ref{lem:cg}, the wild inertia group acts trivially on \(H\), and a generator \(\sigma\) of the tame quotient of \(I_0\) acts unipotently, with Jordan type \(J_2^{\otimes2}=J_3\oplus J_1\). Put \(\mathfrak N=\log\varrho(\sigma)\); then \(\mathfrak N^2\ne0\) and \(\mathfrak N^3=0\). By Proposition~\ref{prop:lower} and Corollary~\ref{cor:E-structure}, the sheaves \(W_0\mathcal E_2\) and \(\mathcal E_2/W_0\mathcal E_2\) are geometrically constant of rank one, with all Frobenius eigenvalues at points of \(S_2(\Fq)\) equal to \(1\) and to \(q\) respectively. Hence \(I_0\) acts trivially on \(L_0\) and on \(E/L_0\), so that \(\mathfrak NL_0=0\) and \(\mathfrak NE\subseteq L_0\). Moreover the fundamental group of \(\Gm[,x]\) over \(\Fq\) acts on \(L_0\) and on \(E/L_0\) through its quotient \(\mathrm{Gal}(\Fqbar/\Fq)\), the geometric Frobenius acting by the eigenvalue at the rational point \(x=1\); so \(\varrho(\Phi)\) acts on \(L_0\) as \(1\) and on \(E/L_0\) as \(q\).

Suppose that \(g_0\), that is, \(\varrho(\sigma)\), acts trivially on \(\bar H\). Then \(\mathfrak NH\subseteq E\), hence \(\mathfrak N^2H\subseteq\mathfrak NE\subseteq L_0\), and \(\mathfrak N^2H=L_0\) because \(\mathfrak N^2\ne0\) and \(\dim L_0=1\). By Lemma~\ref{lem:frob-monodromy}, \(\mathfrak N^2\varrho(\Phi)=q^2\varrho(\Phi)\mathfrak N^2\); so \(K=\ker\mathfrak N^2\) is a \(\varrho(\Phi)\)-stable hyperplane of \(H\). Let \(\alpha\) be the scalar by which \(\varrho(\Phi)\) acts on \(H/K\), and let \(v\in H\setminus K\). Then \(\alpha\,\mathfrak N^2v=\mathfrak N^2\varrho(\Phi)v=q^2\varrho(\Phi)\mathfrak N^2v=q^2\mathfrak N^2v\), since \(\mathfrak N^2v\in L_0\); as \(\mathfrak N^2v\ne0\), \(\alpha=q^2\). Thus \(q^2\) is an eigenvalue of \(\varrho(\Phi)\) on \(H\). On the other hand, \(\varrho(\Phi)\) preserves the filtration \(L_0\subset E\subset H\), acting by \(1\) on \(L_0\) and by \(q\) on \(E/L_0\). On \(\bar H\) the wild inertia group and \(\varrho(\sigma)\) act trivially, hence so does \(I_0\); by Lemma~\ref{lem:frob-monodromy}, \(\varrho(\Phi)\) acts on \(\bar H\) as \(\Frob_0\) acts on the stalk at \(0\) of \(j_*\Delta_a^*\mathcal W_2\), where \(j:\Gm[,x]\hookrightarrow\mathbb A^1_x\). As \(\mathcal W_2\) is pure of weight \(3\), the eigenvalues of \(\Frob_0\) on this stalk have absolute value at most \(q^{3/2}\) \cite[1.8.4]{Deligne}. Hence every eigenvalue of \(\varrho(\Phi)\) on \(H\) is \(1\), \(q\), or of absolute value at most \(q^{3/2}\), and none of these is \(q^2\). This contradiction shows that \(g_0\) acts nontrivially on \(\bar H\); it acts unipotently because it does so on \(H\). The last assertion follows as in the proof of Corollary~\ref{cor:unipotent}.
\end{proof}

We now determine the action of Frobenius at the end of a ray. The case \(n=1\) is explicit.

\begin{lemma}\label{lem:n1}
As lisse sheaves on \(S_1=\Gm\) over \(\Fp\), \(\mathcal E_1\cong\Qlb\) and \(\mathcal W_1\cong\Lpsi{-h/\lambda}\otimes\Qlb(-1)\).
\end{lemma}

\begin{proof}
Both sheaves have rank one, since \(\binom10=r(1)=1\). For \(h\in\Fq^\times\), Proposition~\ref{prop:lower} gives \(\Tr(\Frob_h\mid\mathcal E_1)=1\), and Corollary~\ref{cor:exact}(i) gives \(\Tr(\Frob_h\mid\mathcal W_1)=q\,\Sigma_1(h)=q\,\psi_q(-h/\lambda)\), because the only \(u\in\Fq^\times\) with \(hu=-\lambda\) is \(u=-\lambda/h\). Let \(\mathcal L\) be \(\mathcal E_1\), resp.\ \(\mathcal W_1\otimes\Lpsi{h/\lambda}\otimes\Qlb(1)\). Then \(\Frob_h\) acts on \(\mathcal L\) by \(1\) at every point \(h\) of \(S_1\) with values in a finite field. By Lemma~\ref{lem:constant}, \(\pgeom(S_1)\) acts trivially on \(\mathcal L\); so \(\pi_1(S_1)\) acts through \(\mathrm{Gal}(\Fpbar/\Fp)\), and the geometric Frobenius of \(\Fp\) acts as \(\Frob_h\) for any \(h\in S_1(\Fp)=\Fp^\times\), that is, trivially. Hence \(\mathcal L\cong\Qlb\).
\end{proof}

\begin{proposition}\label{prop:ray-frob}
Let \(n\ge1\), let \(\Fq\) be a finite field of characteristic \(p\), and let \(a\in(\Fq^\times)^n\). Regard the ray \(\Delta_a\) as a morphism over \(\Fq\), let \(D_0\) be a decomposition group at the point \(x=0\) of \(\mathbb A^1_x\), with inertia group \(I_0\), and let \(\Phi\in D_0\) map to the geometric Frobenius of \(\Fq\).
\begin{enumerate}[label=\textup{(\roman*)}]
\item The geometric generic stalk of \(\Delta_a^*\mathcal H_n\) has a \(D_0\)-stable filtration whose graded pieces are the unramified characters \(\Qlb(-k)\), \(0\le k\le n\), each with multiplicity \(\binom nk\).
\item The semisimplification of the representation of \(D_0\) on the geometric generic stalk of \(\Delta_a^*\mathcal W_n\) is \(\bigoplus_{k=1}^n\Qlb(-k)^{\mu_{n,k}}\), where \(\mu_{n,k}=\binom n{k-1}\) for \(k\le\lfloor n/2\rfloor\) and \(\mu_{n,k}=\binom nk\) for \(k>\lfloor n/2\rfloor\).
\item \(\det\bigl(\Phi\mid\Delta_a^*\mathcal W_n\bigr)=q^{(n+1)r(n)/2}\).
\end{enumerate}
\end{proposition}

\begin{proof}
\emph{The Fourier model over \(\Fq\).} All the constructions in the proof of Proposition~\ref{prop:ray} are defined over \(\Fq\): the sheaf \(\mathcal K_a=\bigotimes_i[\times a_i]^*\Kl\), the automorphism \(\phi\), the isomorphism \(\sigma(x)=\lambda/x\) (recall \(\lambda\in\Fp^\times\)), and the isomorphism \(\Delta_a^*\mathcal H_n\cong\sigma^*\mathcal F_a\) with \(\mathcal F_a=R^1\mathrm{pr}_{z!}(\mathrm{pr}_s^*\mathcal K_a\otimes\Lpsi{sz})\). The isomorphism \(\sigma:\mathbb P^1_x\to\mathbb P^1_z\) over \(\Fq\) maps \(x=0\) to \(z=\infty\), so it identifies \(D_0\) with the group \(D_{\infty'}\) of Theorem~\ref{thm:laumon}(iv) for \(k_0=\Fq\), compatibly with the projections to \(\mathrm{Gal}(\Fqbar/\Fq)\). By Theorem~\ref{thm:laumon}(ii) and (iv), and since the summand \(F^{(\infty,\infty')}\) vanishes (proof of Proposition~\ref{prop:ray}), there is an isomorphism of representations of \(D_{\infty'}\)
\[
\mathcal F_a|_{D_{\infty'}}\cong F^{(0,\infty')}\bigl(\mathcal K_a|_{D_0}\bigr),
\]
where now \(D_0\) denotes the decomposition group at \(s=0\) over \(\Fq\) and \(\mathcal K_a|_{D_0}\) its representation on the geometric generic stalk of \(\mathcal K_a\) there.

\emph{A filtration of \(\mathcal K_a|_{D_0}\).} By Lemma~\ref{lem:kl-origin}(ii),(iii), the stalk of \([\times a_i]^*\Kl\) has a basis \(e^{(i)}_0,e^{(i)}_1\) in which a lift of Frobenius acts by \(\mathrm{diag}(1,q)\) and the logarithm \(\mathfrak N^{(i)}\) of tame monodromy maps \(e^{(i)}_1\) to \(e^{(i)}_0\) and kills \(e^{(i)}_0\). The line spanned by \(e^{(i)}_0\) is \(\ker\mathfrak N^{(i)}\), the space of \(I_0\)-invariants, so it is \(D_0\)-stable; the group \(I_0\) acts trivially on it and on the quotient, and since \(D_0\) is topologically generated by \(I_0\) and \(\Phi\), it acts on the line through the trivial character and on the quotient through \(\Qlb(-1)\). Taking tensor products of these two-step filtrations, \(\mathcal K_a|_{D_0}\) has a \(D_0\)-stable filtration whose graded pieces are the characters \(\Qlb(-k)\), \(0\le k\le n\), with multiplicity \(\binom nk\).

\emph{Transport.} Let \(\omega:=F^{(0,\infty')}(\Qlb)\), a representation of \(D_{\infty'}\) of dimension one. By Lemma~\ref{lem:F-trivial} and Theorem~\ref{thm:laumon}(iv), \(I_{\infty'}\) acts trivially on it, so \(\omega\) is a character of \(\mathrm{Gal}(\Fqbar/\Fq)\); it does not depend on \(n\) or \(a\). As \(F^{(0,\infty')}\) is exact and commutes with twisting by characters of \(\mathrm{Gal}(\Fqbar/\Fq)\), \(\mathcal F_a|_{D_{\infty'}}\), and hence \(\Delta_a^*\mathcal H_n|_{D_0}\), has a \(D_0\)-stable filtration whose graded pieces are the characters \(\omega\otimes\Qlb(-k)\), \(0\le k\le n\), with multiplicity \(\binom nk\); here we identify \(D_{\infty'}\) with \(D_0\) by \(\sigma\).

\emph{The character \(\omega\).} We show that \(\omega=\Qlb\). Carry out the proof of Lemma~\ref{lem:F-trivial} over \(\Fq\): by Theorem~\ref{thm:laumon}(iv), applied to \(\mathcal G=\Qlb\) on \(\Gm[,s]\) over \(\Fq\), the representation \(\omega\oplus F^{(\infty,\infty')}(\Qlb)\) of \(D_{\infty'}\) is the generic stalk at \(z=\infty\) of \(R^1\mathrm{pr}_{z!}\Lpsi{sz}\), which over \(\Gm[,z]\) is the constant sheaf with fibre \(H^1_c(\Gm\otimes\Fqbar,\Lpsi{y})\), the automorphism \((s,z)\mapsto(sz,z)\) and the K\"unneth formula being defined over \(\Fq\). This fibre is one-dimensional, and since the other groups \(H^i_c\) vanish, the trace formula shows that the Frobenius of \(\Fq\) acts on it by \(-\sum_{y\in\Fq^\times}\psi_q(y)=1\). As \(F^{(\infty,\infty')}(\Qlb)=0\) (Lemma~\ref{lem:F-trivial}), \(\omega\) is trivial, and (i) holds for every \(n\).

For \(n=1\) this can be checked independently. By Lemma~\ref{lem:n1}, \(\Delta_a^*\mathcal H_1\) contains the line \(\Delta_a^*\mathcal E_1\cong\Qlb\), with quotient \(\Delta_a^*\mathcal W_1\cong\Lpsi{-ax/\lambda}\otimes\Qlb(-1)\); this quotient is the restriction of a lisse sheaf on \(\mathbb A^1_x\) whose Frobenius at \(x=0\) is \(q\,\psi_q(0)=q\), so \(D_0\) acts on it through \(\Qlb(-1)\). Thus the Jordan--H\"older factors are \(\Qlb\) and \(\Qlb(-1)\), as (i) predicts for \(n=1\); conversely, the Jordan--H\"older theorem alone would already force \(\omega=\Qlb\), since \(\omega=\Qlb(-1)\) would make \(\omega\otimes\Qlb(-1)=\Qlb(-2)\) a factor.

\emph{The quotient \(\mathcal W_n\).} By Corollary~\ref{cor:E-structure}, \(\mathrm{Gr}^W_{2k}\mathcal E_n\) is geometrically constant of rank \(\nu_{n,k}\) with all Frobenius eigenvalues equal to \(q^k\); so \(\pi_1(\Gm[,x])\), and in particular \(D_0\), acts on the stalk of \(\Delta_a^*\mathrm{Gr}^W_{2k}\mathcal E_n\) through \(\mathrm{Gal}(\Fqbar/\Fq)\), with \(\Phi\) having the single eigenvalue \(q^k\), and the semisimplification of this representation is \(\Qlb(-k)^{\nu_{n,k}}\). The exact sequence \(0\to\Delta_a^*\mathcal E_n\to\Delta_a^*\mathcal H_n\to\Delta_a^*\mathcal W_n\to0\) is \(D_0\)-equivariant, and the Jordan--H\"older theorem and (i) give (ii), with \(\mu_{n,k}=\binom nk-\nu_{n,k}\) for \(k\le\lfloor n/2\rfloor\) and \(\mu_{n,k}=\binom nk\) for \(k>\lfloor n/2\rfloor\); note that \(\binom nk-\nu_{n,k}=\binom n{k-1}\), and that \(\mu_{n,0}=1-\nu_{n,0}=0\).

\emph{The determinant.} By (ii), \(\det(\Phi\mid\Delta_a^*\mathcal W_n)=q^{N}\) with \(N=\sum_kk\binom nk-\sum_kk\,\nu_{n,k}=n2^{n-1}-S_1\), where \(S_1:=\sum_{k=0}^mk\,\nu_{n,k}\) and \(m=\lfloor n/2\rfloor\). Summation by parts gives \(S_1=\sum_{k\le m}k\binom nk-\sum_{j\le m-1}(j+1)\binom nj=m\binom nm-\sum_{j<m}\binom nj\), and the symmetry \(\binom nj=\binom n{n-j}\) gives \(\sum_{j<m}\binom nj=2^{n-1}-\frac12(n+1-2m)\binom nm\) (for \(n=2m\) and for \(n=2m+1\) separately). Hence
\[
N=n2^{n-1}-m\binom nm+2^{n-1}-\tfrac12(n+1-2m)\binom nm=\tfrac{n+1}2\Bigl(2^n-\binom nm\Bigr)=\tfrac{(n+1)\,r(n)}2 .
\]
\end{proof}

The weights in Proposition~\ref{prop:ray-frob}(ii) are \(2k\) with multiplicity \(\mu_{n,k}\), and \(\mu_{n,k}=\mu_{n,n+1-k}\); so they are symmetric about the weight \(n+1\) of \(\mathcal W_n\), in accordance with the unipotent monodromy of Proposition~\ref{prop:ray} and the relation \(\Phi\mathfrak N\Phi^{-1}=q^{-1}\mathfrak N\) of Lemma~\ref{lem:frob-monodromy}. For instance, for \(n=2\) the Frobenius at the end of a ray acts on \(\mathcal W_2\) with the eigenvalues \(q\) and \(q^2\), and for \(n=3\) on \(\mathcal W_3\) with the eigenvalues \(q\), \(q^2\) (three times) and \(q^3\).

\subsection{The identity component}\label{subsec:main}

\begin{theorem}\label{thm:main}
Let \(n\ge2\), put \(G=\ggeom(\mathcal W_n)\), and let \(d\) be the order of the image of \(\pgeom(S_n)\) under \(\det(\mathcal W_n)\) \textup{(}Proposition~\textup{\ref{prop:det})}.
\begin{enumerate}[label=\textup{(\roman*)}]
\item If \(p\) is odd, or if \(n=2\), then \(G^0=\SL(W_n)\),
\[
G=\{g\in\GL(W_n):\det(g)^d=1\},\qquad G\cap\{\text{scalars}\}=\mu_{r(n)d}.
\]
If \(p=2\) and \(n=2\), then moreover \(d=1\) and \(G=\SL(W_2)\).
\item If \(p=2\) and \(n\ge3\) is odd, let \(B\) be the symmetric form of Proposition~\textup{\ref{prop:selfdual}}. Then \(G^0=\mathrm{SO}(W_n,B)\), \(d\in\{1,2\}\), and \(G=\{g\in\mathrm O(W_n,B):\det(g)^d=1\}\); thus \(G=\mathrm{SO}(W_n,B)\) if \(d=1\) and \(G=\mathrm O(W_n,B)\) if \(d=2\). \textup{(}Theorem~\textup{\ref{thm:full2}} below shows that \(d=1\).\textup{)}
\item If \(p=2\) and \(n\ge4\) is even, let \(B\) be the alternating form of Proposition~\textup{\ref{prop:selfdual}}. Then \(G=\Sp(W_n,B)\); in particular \(d=1\).
\end{enumerate}
\end{theorem}

\emph{The hypotheses of Larsen's alternative.} For \(n\ge3\) the proof applies Theorem~\ref{thm:larsen} to \(V=W_n\) and \(G=\ggeom(\mathcal W_n)\). Its hypotheses, and the facts used to exclude its degenerate alternatives, are established as follows; none of the results used as input depends on Theorem~\ref{thm:main}.
\begin{enumerate}[label=(L\arabic*)]
\item \(W_n\) is a vector space of dimension \(r(n)\) over the algebraically closed field \(\Qlb\) of characteristic zero, and \(G\) is Zariski closed in \(\GL(W_n)\) by definition. By Lemma~\ref{lem:binom-gap}, \(r(n)>2^{n-1}\ge4\), so \(r(n)\ge5\) for \(n\ge3\); in particular \(N=r(n)\ge4\) in Theorem~\ref{thm:larsen}(iv).
\item \(G\) is reductive: \(\mathcal W_n\) is lisse and pure of weight \(n+1\) (Theorem~\ref{thm:package}), and Lemma~\ref{lem:semisimple} applies.
\item \(M_4(G,W_n)=c_p\) (Proposition~\ref{prop:M4}). This rests on Theorem~\ref{thm:fourth} over \((\Fq^\times)^n\), whose proof uses, besides elementary character sums and Hasse's bound, only formula \eqref{eq:t-sum}, the reality of \(K\), and Corollary~\ref{cor:pointwise} and Lemma~\ref{lem:offtorus} from Section~\ref{sec:sheaf}, on the comparison of \(\Tr(\Frob_h\mid\mathcal W_n)\) with \((-1)^{n-1}q\Sigma_n(h)\) in Theorem~\ref{thm:package}(iii) and Corollary~\ref{cor:pointwise}, and on Theorem~\ref{thm:moment-principle}(i).
\item If \(p=2\), then \(G\) preserves a nondegenerate bilinear form \(B\) on \(W_n\), symmetric if \(n\) is odd and alternating if \(n\) is even (Proposition~\ref{prop:selfdual}). Its proof does not use the fourth moment, and Theorem~\ref{thm:FS}(ii) gives an independent proof of the existence and the symmetry type of \(B\).
\item The degenerate alternatives, a finite image in \(\GL(W_n)/\{\text{scalars}\}\) and a finite group, are excluded by a unipotent element of \(G\) different from \(1\) (Corollary~\ref{cor:unipotent}), which comes from Laumon's local Fourier transform (Theorem~\ref{thm:laumon} and Proposition~\ref{prop:ray}) and the exact sequence of Theorem~\ref{thm:package}(ii).
\item The passage from \(G^0\) to \(G\) uses only the finiteness of the geometric determinant (Proposition~\ref{prop:det}). In the orthogonal case it leaves open whether \(G\) is \(\mathrm{SO}(W_n,B)\) or \(\mathrm O(W_n,B)\); this is decided by the geometric triviality of \(\det(\mathcal W_n)\), proved in Theorem~\ref{thm:det}(ii) independently of Theorem~\ref{thm:main}, and recorded in Theorem~\ref{thm:full2}.
\end{enumerate}
For \(n=2\) Larsen's alternative is not used; its role is played by the elementary fact that a connected reductive subgroup of \(\SL_2\) containing a unipotent element different from \(1\) is \(\SL_2\), and the unipotent element comes from Proposition~\ref{prop:n2}. (For \(n=2\) the group can also be read off from Proposition~\ref{prop:n2-kl}; see Remark~\ref{rem:n2-kl}(c).)

\begin{proof}
In all cases \(G\) is reductive (Lemma~\ref{lem:semisimple}) and \(\det(G)=\mu_d\) (Proposition~\ref{prop:det}).

(i) Let first \(p\) be odd and \(n\ge3\). By Proposition~\ref{prop:M4}, \(M_4(G,W_n)=2\), and \(\dim W_n=r(n)\ge5\). By Theorem~\ref{thm:larsen}(ii), either \(G\supseteq\SL(W_n)\), or the image of \(G\) in \(\GL(W_n)/\{\text{scalars}\}\) is finite. In the second case the unipotent element \(g\ne1\) of Corollary~\ref{cor:unipotent} would have a power \(g^m\) (\(m\ge1\)) which is scalar; a unipotent scalar is the identity, and a unipotent element of finite order in characteristic zero is the identity, so \(g=1\), a contradiction. Hence \(G\supseteq\SL(W_n)\).

Now let \(n=2\) and \(p\) arbitrary. By Proposition~\ref{prop:n2}, \(G\) contains a unipotent element \(u\ne1\). Write \(u=\exp(\mathfrak n)\) with \(\mathfrak n\ne0\) nilpotent. The map \(t\mapsto\exp(t\mathfrak n)\) is an injective homomorphism of algebraic groups \(\mathbb G_a\to\GL(W_2)\) with closed image, and the integers are Zariski dense in \(\mathbb G_a\); so the Zariski closure of \(\{u^m:m\in\mathbb Z\}\) is a connected one-dimensional unipotent subgroup of \(G\), contained in \(G^0\). As \(\det(G^0)\) is connected and contained in the finite group \(\mu_d\), \(G^0\subseteq\SL(W_2)\). The connected reductive group \(G^0\) is not a torus, since a torus contains no unipotent element other than \(1\); so its derived group is a nontrivial connected semisimple subgroup of \(\SL(W_2)\). A nontrivial connected semisimple group has dimension at least \(3=\dim\SL(W_2)\); hence the derived group of \(G^0\) is \(\SL(W_2)\), and \(G^0=\SL(W_2)\).

In both cases \(\SL(W_n)\subseteq G\). Clearly \(G\subseteq\det^{-1}(\mu_d)\). Conversely, let \(g\in\GL(W_n)\) with \(\det(g)\in\mu_d\), and choose \(h\in G\) with \(\det(h)=\det(g)\); then \(gh^{-1}\in\SL(W_n)\subseteq G\), so \(g\in G\). Hence \(G=\det^{-1}(\mu_d)=\{g:\det(g)^d=1\}\). For each \(\zeta\in\mu_d\) the fibre \(\det^{-1}(\zeta)\) is a coset of \(\SL(W_n)\), so \(G\) is the disjoint union of the \(d\) cosets of the connected group \(\SL(W_n)\), each of which is open and closed in \(G\); therefore \(G^0=\SL(W_n)\). A scalar \(z\cdot\mathrm{id}\) lies in \(G\) if and only if \(z^{r(n)}=\det(z\cdot\mathrm{id})\in\mu_d\), that is, if and only if \(z^{r(n)d}=1\). Finally, if \(p=2\) and \(n=2\), then \(G\subseteq\Sp(W_2,B)=\SL(W_2)\) by Proposition~\ref{prop:selfdual}, so that \(d=1\) and \(G=\SL(W_2)\).

(ii) and (iii). Let \(p=2\) and \(n\ge3\). By Proposition~\ref{prop:selfdual}, \(G\subseteq\mathrm O(W_n,B)\) if \(n\) is odd and \(G\subseteq\Sp(W_n,B)\) if \(n\) is even; by Proposition~\ref{prop:M4}, \(M_4(G,W_n)=3\); and \(\dim W_n=r(n)\ge5\), with \(r(n)\ge10\) if \(n\) is even. By Theorem~\ref{thm:larsen}(iii) and (iv), either \(G\) is finite, or \(G\supseteq\mathrm{SO}(W_n,B)\) if \(n\) is odd, and \(G=\Sp(W_n,B)\) if \(n\) is even. A finite group contains no unipotent element other than \(1\), the coefficient field having characteristic zero; so Corollary~\ref{cor:unipotent} excludes the first alternative. If \(n\) is even, this proves (iii), and \(d=1\) because \(\Sp(W_n,B)\subseteq\SL(W_n)\). If \(n\) is odd, then \(\mu_d=\det(G)\subseteq\det\mathrm O(W_n,B)=\{\pm1\}\), so \(d\in\{1,2\}\). Clearly \(G\subseteq\{g\in\mathrm O(W_n,B):\det(g)^d=1\}\); conversely, if \(g\) lies in this set, choose \(h\in G\) with \(\det(h)=\det(g)\); then \(gh^{-1}\in\mathrm{SO}(W_n,B)\subseteq G\), so \(g\in G\). This proves the description of \(G\); since \(\mathrm{SO}(W_n,B)\) is connected and of finite index in \(G\), it is \(G^0\).
\end{proof}

In particular, for \(n=4\) the sheaf \(\mathcal W_4\) has rank \(10\), and \(\ggeom(\mathcal W_4)^0\) is \(\SL_{10}\) if \(p\) is odd and \(\Sp_{10}\) if \(p=2\); for \(n=3\) it has rank \(5\), and \(\ggeom(\mathcal W_3)^0\) is \(\SL_5\) if \(p\) is odd and \(\mathrm{SO}_5\) if \(p=2\).

\begin{corollary}\label{cor:structure}
Let \(n\ge2\).
\begin{enumerate}[label=\textup{(\roman*)}]
\item \(\mathcal W_n\) is geometrically irreducible, and \(W_n\) is a primitive representation of \(\ggeom(\mathcal W_n)\).
\item If \(p\) is odd and \(n\ge3\), there is no lisse rank-one sheaf \(\mathcal L\) on \(S_n\) with \(\mathcal W_n\cong\mathcal W_n^\vee\otimes\mathcal L\) geometrically; in particular \(\ggeom(\mathcal W_n)\) is contained in no orthogonal and no symplectic group.
\item If \(p=2\), then \(\mathcal W_n\) is geometrically self-dual, and every \(\ggeom(\mathcal W_n)\)-invariant bilinear form on \(W_n\) is a scalar multiple of the form \(B\) of Proposition~\textup{\ref{prop:selfdual}}, which is symmetric if \(n\) is odd and alternating if \(n\) is even.
\end{enumerate}
\end{corollary}

\begin{proof}
(i) By Theorem~\ref{thm:main}, \(G^0\) is \(\SL(W_n)\), \(\mathrm{SO}(W_n,B)\) with \(\dim W_n\ge5\), or \(\Sp(W_n,B)\), and each of these groups acts irreducibly on \(W_n\). Hence \(\mathcal W_n\) is geometrically irreducible. If \(W_n=\bigoplus_{i=1}^mW_i\) with \(m\ge2\) were a decomposition permuted by \(\ggeom(\mathcal W_n)\), the stabiliser of every summand would be a closed subgroup of finite index, hence would contain the connected group \(G^0\), contradicting the irreducibility of \(W_n\) under \(G^0\). (ii) Let \(\varrho:\pgeom(S_n)\to\GL(W_n)\) be the monodromy representation, and let \(W_n^\vee=\Hom(W_n,\Qlb)\) carry the contragredient action \(\varrho^\vee\), given by \((\varrho^\vee(\gamma)\xi)(w)=\xi(\varrho(\gamma)^{-1}w)\). Let \(\mathcal L\) be a lisse sheaf of rank one on \(S_n\), let \(\varepsilon:\pgeom(S_n)\to\Qlb^\times\) be its geometric monodromy character, and suppose that \(\mathcal W_n\cong\mathcal W_n^\vee\otimes\mathcal L\) geometrically, that is, that there is a linear isomorphism \(\phi:W_n\to W_n^\vee\) with \(\phi(\varrho(\gamma)v)=\varepsilon(\gamma)\varrho^\vee(\gamma)\phi(v)\) for all \(\gamma\) and \(v\). The bilinear form \(B'(v,w):=\phi(v)(w)\) on \(W_n\) is nondegenerate, and
\[
B'(\varrho(\gamma)v,\varrho(\gamma)w)=\varepsilon(\gamma)\,\phi(v)\bigl(\varrho(\gamma)^{-1}\varrho(\gamma)w\bigr)=\varepsilon(\gamma)B'(v,w).
\]
Comparing the determinants of the Gram matrices of both sides gives \((\det\varrho)^2=\varepsilon^{r(n)}\); by Proposition~\ref{prop:det} the left-hand side has finite order, hence so has \(\varepsilon\), and \(\Gamma:=\ker\varepsilon\) is an open subgroup of finite index in \(\pgeom(S_n)\). The group \(\mathrm O(B')=\{g\in\GL(W_n):B'(gv,gw)=B'(v,w)\text{ for all }v,w\}\) is Zariski closed and contains \(\varrho(\Gamma)\), hence contains the Zariski closure \(H\) of \(\varrho(\Gamma)\). If \(\gamma_1,\dots,\gamma_N\) are coset representatives of \(\Gamma\), then \(\ggeom(\mathcal W_n)=\bigcup_i\varrho(\gamma_i)H\), so \(H\) has finite index in \(\ggeom(\mathcal W_n)\) and contains \(\ggeom(\mathcal W_n)^0=\SL(W_n)\). Thus \(\SL(W_n)\) preserves the nondegenerate form \(B'\), so \(\phi\) is an isomorphism of \(\SL(W_n)\)-modules \(W_n\cong W_n^\vee\). This is impossible because \(r(n)\ge3\): in a basis of \(W_n\), the element \(g=\mathrm{diag}(2,2,1/4,1,\dots,1)\) of \(\SL(W_n)\) has the eigenvalue \(2\) with multiplicity at least \(2\) on \(W_n\), whereas its eigenvalues on \(W_n^\vee\) are \(1/2,1/2,4,1,\dots,1\). The orthogonal and symplectic cases are the case \(\mathcal L=\Qlb\), \(B'\) being the invariant form. (iii) By Proposition~\ref{prop:selfdual}, \(B\) is \(\ggeom(\mathcal W_n)\)-invariant and nondegenerate, so \(\phi_B:v\mapsto B(v,\cdot)\) is an isomorphism of \(\ggeom(\mathcal W_n)\)-modules \(W_n\cong W_n^\vee\). An invariant bilinear form \(B'\) corresponds to an element \(v\mapsto B'(v,\cdot)\) of \(\Hom_{\ggeom(\mathcal W_n)}(W_n,W_n^\vee)\), which is one-dimensional by (i) and Schur's lemma; so \(B'\) is a multiple of \(B\).
\end{proof}

\subsection{Coordinate lines}\label{subsec:coord}
In this subsection we determine the local monodromy of \(\mathcal H_n\) and \(\mathcal W_n\) along the coordinate lines, in every characteristic. The result is used in Section~\ref{subsec:full} to determine the geometric determinant of \(\mathcal W_n\) when \(p=2\), and hence the full group \(\ggeom(\mathcal W_n)\) in the orthogonal case. Throughout, \(n\ge2\) and the ground field is \(\Fpbar\). For \(c=(c_2,\dots,c_n)\in(\Fpbar^\times)^{n-1}\) let \(\mathbb L_c\cong\Gm[,y]\) be the coordinate line \(y\mapsto(y,c_2,\dots,c_n)\) in \(S_n\), and put
\[
\mathcal M_c:=\Bigl(\bigotimes_{i=2}^n[\times c_i]^*\Kl\Bigr)\otimes\Lpsi{\lambda t}\quad\text{on }\Gm[,t],\qquad w_0:=-1/\lambda\in\Fp^\times .
\]
The symmetric group \(\mathfrak S_n\) acts on \(S_n\) by permuting the coordinates; it permutes the factors of \(\mathcal N_n\), and hence acts compatibly on \(\mathcal H_n\) and, the weight filtration being unique, on \(\mathcal E_n\) and \(\mathcal W_n\). So every statement below about the lines \(\mathbb L_c\) holds for the coordinate lines in every direction.

\begin{lemma}\label{lem:coord-euler}
\(\chi_c(\Gm,\mathcal H_n|_{\mathbb L_c})=-2^{n-1}\). Consequently \(\Swan_0(\mathcal H_n|_{\mathbb L_c})+\Swan_\infty(\mathcal H_n|_{\mathbb L_c})=2^{n-1}\), and the same holds for \(\mathcal W_n|_{\mathbb L_c}\).
\end{lemma}

\begin{proof}
Put \(X=\Gm[,y]\times\Gm[,t]\), \(\mu(y,t)=yt\) and \(\mathcal N=\mu^*\Kl\otimes\mathrm{pr}_t^*\mathcal M_c\) on \(X\); thus \(\mathcal N=\mathcal N_n|_{\Gm[,t]\times\mathbb L_c}\). By Theorem~\ref{thm:package}(i), \(R\,\mathrm{pr}_{y!}\mathcal N=(\mathcal H_n|_{\mathbb L_c})[-1]\), so \(\chi_c(X,\mathcal N)=-\chi_c(\Gm,\mathcal H_n|_{\mathbb L_c})\). On the other hand, the automorphism \((y,t)\mapsto(yt,t)\) of \(X\) carries \(\mathcal N\) to \(\Kl\boxtimes\mathcal M_c\), so by the K\"unneth formula \(\chi_c(X,\mathcal N)=\chi_c(\Gm,\Kl)\,\chi_c(\Gm,\mathcal M_c)\). By Theorem~\ref{thm:kl}, \(\chi_c(\Gm,\Kl)=-1\). As in Proposition~\ref{prop:tline}, \(\mathcal M_c\) is tame at \(0\) and isoclinic of break \(1\) at \(\infty\), so \(\chi_c(\Gm,\mathcal M_c)=-2^{n-1}\). Hence \(\chi_c(\Gm,\mathcal H_n|_{\mathbb L_c})=-2^{n-1}\), and the Grothendieck--Ogg--Shafarevich formula on \(\Gm\) gives the second assertion. By Corollary~\ref{cor:E-structure}, \(\mathcal E_n|_{\mathbb L_c}\) is tame at \(0\) and at \(\infty\), and Swan conductors are additive in the exact sequence \(0\to\mathcal E_n\to\mathcal H_n\to\mathcal W_n\to0\).
\end{proof}

\emph{The Fourier model.} Let \(\kappa:\mathbb P^1_w\to\mathbb P^1_x\) be the isomorphism \(\kappa(w)=\lambda w/(\lambda w+1)\), so that \(\kappa(0)=0\), \(\kappa(w_0)=\infty\) and \(\kappa(\infty)=1\), and let \(\Delta_c:\Gm[,x]\to S_{n-1}\), \(x\mapsto xc\), be the ray of the \((n-1)\)-dimensional family through \(c\). Put
\[
\mathcal A_c:=\kappa^*\Delta_c^*\mathcal H_{n-1},\quad\text{a lisse sheaf of rank }2^{n-1}\text{ on }\kappa^{-1}(\Gm[,x])=\mathbb P^1_w\setminus\{0,w_0\},
\]
let \(\mathcal C_c:=R\,\mathrm{pr}_{w!}\bigl(\mathrm{pr}_t^*\mathcal M_c\otimes\Lpsi{t/w}\bigr)\), where \(\mathrm{pr}_w:\Gm[,t]\times\Gm[,w]\to\Gm[,w]\), and let \(\mathcal K_c:=j_!(\mathcal C_c[2])\), where \(j:\Gm[,w]\hookrightarrow\mathbb A^1_w\). We write \(\mathcal F_\psi(K)=R\,\mathrm{pr}_{y!}(\mathrm{pr}_w^*K\otimes\Lpsi{wy})[1]\) for Laumon's Fourier transform from \(\mathbb A^1_w\) to \(\mathbb A^1_y\) \cite[(1.2.1.1)]{Laumon}.

\begin{proposition}\label{prop:coord-model}
Let \(U:=\Gm[,w]\setminus\{w_0\}\).
\begin{enumerate}[label=\textup{(\roman*)}]
\item \(\mathcal C_c|_U\cong\mathcal A_c|_U[-1]\); that is, \(\mathcal K_c|_U\cong\mathcal A_c|_U[1]\).
\item \(\mathcal F_\psi(\mathcal K_c)|_{\Gm[,y]}\cong(\mathcal H_n|_{\mathbb L_c})[1]\).
\item \(\mathcal A_c\) is lisse at \(w=\infty\), and its local monodromy at \(w=0\) is tame and unipotent, with the Jordan type of \(J_2^{\otimes(n-1)}\).
\end{enumerate}
\end{proposition}

\begin{proof}
(i) For \(w\in U\) one has \(\lambda+1/w\ne0\) and \(\kappa(w)=\lambda/(\lambda+1/w)\). Let \(\mathcal N':=(\mathrm{id}\times\Delta_c\kappa)^*\mathcal N_{n-1}\) on \(\Gm[,t]\times U\); explicitly \(\mathcal N'=\bigotimes_{i\ge2}[\times\kappa(w)c_i]^*\Kl\otimes\Lpsi{\lambda t}\), the factors being pulled back along \((t,w)\mapsto\kappa(w)c_it\). The automorphism \(\theta(t,w)=(t/\kappa(w),w)\) of \(\Gm[,t]\times U\) commutes with the projection to \(U\), and since \(\lambda/\kappa(w)=\lambda+1/w\),
\[
\theta^*\mathcal N'=\bigotimes_{i\ge2}[\times c_i]^*\Kl\otimes\Lpsi{(\lambda+1/w)t}=\mathrm{pr}_t^*\mathcal M_c\otimes\Lpsi{t/w}.
\]
Hence, by proper base change and Theorem~\ref{thm:package}(i) for \(n-1\), \(\mathcal C_c|_U\cong R\,\mathrm{pr}_{w!}\mathcal N'\cong(\Delta_c\kappa)^*R\rho_!\mathcal N_{n-1}=\mathcal A_c|_U[-1]\).

(ii) The sheaf \(\mathcal N_n|_{\Gm[,t]\times\mathbb L_c}\) is \(\mu^*\Kl\otimes\mathrm{pr}_t^*\mathcal M_c\) with \(\mu(t,y)=yt\), and \(\mathcal H_n|_{\mathbb L_c}[-1]\) is its direct image with compact support on \(\Gm[,y]\) (Theorem~\ref{thm:package}(i)). By Theorem~\ref{thm:kl} and proper base change along \(\mu\), \(\mu^*\Kl[-1]\) is the direct image with compact support of \(\Lpsi{x^{-1}+ytx}\) along the projection \(\Gm[,x]\times\Gm[,t]\times\Gm[,y]\to\Gm[,t]\times\Gm[,y]\). By the projection formula, \(\mathcal H_n|_{\mathbb L_c}[-2]\) is therefore the direct image with compact support on \(\Gm[,y]\) of \(\Lpsi{x^{-1}+ytx}\otimes\mathrm{pr}_t^*\mathcal M_c\). The automorphism \((x,t,y)\mapsto(w,t,y)=(tx,t,y)\) carries \(\Lpsi{x^{-1}+ytx}\) to \(\Lpsi{t/w+wy}\); integrating first over \(t\), by proper base change and the projection formula,
\[
(\mathcal H_n|_{\mathbb L_c})[-2]\cong R\,\mathrm{pr}_{y!}\bigl(\Lpsi{wy}\otimes\mathrm{pr}_w^*\mathcal C_c\bigr)\qquad\text{on }\Gm[,y],
\]
the projection being \(\Gm[,w]\times\Gm[,y]\to\Gm[,y]\). Since \(\mathcal K_c\) is the extension by zero of \(\mathcal C_c[2]\), the restriction of \(\mathcal F_\psi(\mathcal K_c)\) to \(\Gm[,y]\) is the right-hand side shifted by \(3\), which is (ii).

(iii) Since \(\kappa(\infty)=1\in\Gm[,x]\), \(\mathcal A_c\) is lisse at \(w=\infty\). The isomorphism \(\kappa\) maps \(w=0\) to \(x=0\), so it identifies the inertia groups there, compatibly with the wild inertia subgroups and with the tame quotients; hence the local monodromy of \(\mathcal A_c\) at \(w=0\) is that of \(\Delta_c^*\mathcal H_{n-1}\) at \(x=0\), which is tame and unipotent of the Jordan type of \(J_2^{\otimes(n-1)}\) by Proposition~\ref{prop:ray}, applied to \(n-1\ge1\) and to \(a=c\) (note that \(c\in(\Fq^\times)^{n-1}\) for some \(q\)).
\end{proof}

For a representation \(V\) of an inertia group \(I\) and a rational number \(x\ge0\) we write \(V_{=x}\) for the part of \(V\) of break \(x\), and \(V_{<1}\) for the sum of the parts of breaks \(<1\) \cite[Ch.~1]{Katz-Gauss}. The decomposition \(V=\bigoplus_xV_{=x}\) is canonical and functorial; so an exact sequence of representations of \(I\) is the direct sum over \(x\) of the sequences formed by the parts of break \(x\), each of which is therefore exact, and \(V\mapsto V_{=x}\) is an exact functor.

\begin{theorem}\label{thm:coord}
Let \(n\ge2\) and \(c\in(\Fpbar^\times)^{n-1}\), and regard \(\mathcal H_n|_{\mathbb L_c}\) and \(\mathcal W_n|_{\mathbb L_c}\) as lisse sheaves on \(\Gm[,y]\) over \(\Fpbar\). Let \(V\) be the representation of \(I_\infty\) on the geometric generic stalk of \(\mathcal H_n|_{\mathbb L_c}\) at \(y=\infty\).
\begin{enumerate}[label=\textup{(\roman*)}]
\item \(V=V_{=0}\oplus V_{=1}\). The part \(V_{=0}\) has dimension \(2^{n-1}\), \(P_\infty\) acts trivially on it, and a generator of \(I_\infty/P_\infty\) acts on it unipotently, with the Jordan type of \(J_2^{\otimes(n-1)}\). The part \(V_{=1}\) has dimension \(2^{n-1}\), and all breaks of \(V_{=1}\otimes\Lpsi{y/\lambda}\) are \(<1\).
\item \(\mathcal H_n|_{\mathbb L_c}\) and \(\mathcal W_n|_{\mathbb L_c}\) are tamely ramified at \(y=0\), and both have \(\Swan_0=0\) and \(\Swan_\infty=2^{n-1}\).
\item The representation of \(I_\infty\) on the stalk of \(\mathcal W_n|_{\mathbb L_c}\) is \((V_{=0}/E)\oplus V_{=1}\), where \(E\subseteq V_{=0}\) is the stalk of \(\mathcal E_n|_{\mathbb L_c}\); its part of break \(0\) has dimension \(2^{n-1}-\binom n{\lfloor n/2\rfloor}\) and is unipotent.
\item There is a lisse sheaf \(\mathcal L_c\) of rank one on \(\Gm[,y]\), tamely ramified at \(0\) and at \(\infty\), such that over \(\Fpbar\)
\[
\det(\mathcal W_n|_{\mathbb L_c})\cong\det(\mathcal H_n|_{\mathbb L_c})\cong\Lpsi{-2^{n-1}y/\lambda}\otimes\mathcal L_c .
\]
\end{enumerate}
\end{theorem}

\begin{proof}
(i) Let \(j_U:U\hookrightarrow\mathbb A^1_w\) and \(i:\{w_0\}\hookrightarrow\mathbb A^1_w\) be the inclusions, and let \(C:=i^*\mathcal K_c\), a complex of \(\Qlb\)-vector spaces with finite-dimensional cohomology. The complement of \(U\) in \(\mathbb A^1_w\) is \(\{0,w_0\}\), and the stalk of \(\mathcal K_c\) at \(0\) vanishes; so there is a triangle \(j_{U!}j_U^*\mathcal K_c\to\mathcal K_c\to i_*C\to\). Applying to it the triangulated functor \(\mathcal F_\psi\), and using Proposition~\ref{prop:coord-model}(i), we obtain a triangle
\[
\mathcal Q\longrightarrow\mathcal F_\psi(\mathcal K_c)\longrightarrow C\otimes\Lpsi{w_0y}[1]\longrightarrow,\qquad\mathcal Q:=\mathcal F_\psi\bigl(j_{U!}(\mathcal A_c|_U)[1]\bigr);
\]
here \(\mathcal F_\psi(i_*C)=C\otimes\Lpsi{w_0y}[1]\) by the projection formula. The complex \(j_{U!}(\mathcal A_c|_U)[1]\) is perverse \cite[Corollaire 4.1.3]{BBD}, hence so is \(\mathcal Q\) \cite[(1.3.2.3)]{Laumon}: its cohomology sheaves vanish outside the degrees \(-1\) and \(0\), and \(\mathcal H^0(\mathcal Q)\) is punctual. By Proposition~\ref{prop:coord-model}(ii), \(\mathcal F_\psi(\mathcal K_c)\) restricted to \(\Gm[,y]\) is \(\mathcal H_n|_{\mathbb L_c}\) placed in degree \(-1\). Taking the long exact sequence of cohomology sheaves at the geometric generic point of the henselisation at \(y=\infty\) therefore gives an exact sequence of representations of \(I_\infty\)
\[
0\to\mathcal H^{-1}(C)\otimes\Lpsi{w_0y}\to Q\to V\to\mathcal H^0(C)\otimes\Lpsi{w_0y}\to0,
\]
where \(Q\) is the generic stalk of \(\mathcal H^{-1}(\mathcal Q)\) at \(\infty\), with its action of \(I_\infty\). By Theorem~\ref{thm:laumon}(iii), with \(\Xi=\{0,w_0\}\),
\[
Q\cong F^{(0,\infty')}\bigl(\mathcal A_c|_{I_0}\bigr)\oplus\Bigl(\Lpsi{w_0y}\otimes F^{(0,\infty')}\bigl((\tau_{w_0}^*\mathcal A_c)|_{I_0}\bigr)\Bigr)\oplus F^{(\infty,\infty')}\bigl(\mathcal A_c|_{I_\infty}\bigr).
\]
By Proposition~\ref{prop:coord-model}(iii), \(\mathcal A_c|_{I_\infty}\) is a direct sum of copies of \(\mathbf 1\), and \(F^{(\infty,\infty')}(\mathbf 1)=0\) by Lemma~\ref{lem:F-trivial}; as \(F^{(\infty,\infty')}\) is additive, the third summand vanishes. By Proposition~\ref{prop:coord-model}(iii) and Lemma~\ref{lem:jordan}, the first summand \(T:=F^{(0,\infty')}(\mathcal A_c|_{I_0})\) has dimension \(2^{n-1}\), \(P_\infty\) acts trivially on it, and a generator of tame inertia acts unipotently with the Jordan type of \(J_2^{\otimes(n-1)}\). The representation \(X':=F^{(0,\infty')}((\tau_{w_0}^*\mathcal A_c)|_{I_0})\) has all breaks \(<1\) by Theorem~\ref{thm:laumon}(i), and \(\Lpsi{w_0y}\) is isoclinic of break \(1\) at \(\infty\) because \(w_0\ne0\); so the second summand is isoclinic of break \(1\) by Lemma~\ref{lem:breaks}(ii). Hence \(Q_{=0}=T\), \(Q_{=1}=\Lpsi{w_0y}\otimes X'\), and \(Q\) has no other breaks. The outer terms of the exact sequence are isoclinic of break \(1\). Taking parts of break \(x\), we obtain \(V_{=0}\cong T\), \(V_{=x}=0\) for \(x\notin\{0,1\}\), and an exact sequence
\[
0\to\mathcal H^{-1}(C)\otimes\Lpsi{w_0y}\to\Lpsi{w_0y}\otimes X'\to V_{=1}\to\mathcal H^0(C)\otimes\Lpsi{w_0y}\to0 .
\]
Tensoring it with \(\Lpsi{-w_0y}=\Lpsi{y/\lambda}\), we see that \(V_{=1}\otimes\Lpsi{y/\lambda}\) is an extension of a subrepresentation of the trivial representation \(\mathcal H^0(C)\) by a quotient of \(X'\); taking parts of break \(x\ge1\) shows that it has no break \(\ge1\). Finally \(\dim V_{=1}=2^n-\dim T=2^{n-1}\).

(ii) By (i), \(\Swan_\infty(\mathcal H_n|_{\mathbb L_c})=\dim V_{=1}=2^{n-1}\), so \(\Swan_0(\mathcal H_n|_{\mathbb L_c})=0\) by Lemma~\ref{lem:coord-euler}: all breaks at \(0\) vanish, that is, the wild inertia group at \(0\) acts trivially. The sheaf \(\mathcal E_n|_{\mathbb L_c}\) is tame at \(0\) and at \(\infty\) (Corollary~\ref{cor:E-structure}), so \(\mathcal W_n|_{\mathbb L_c}\) has the same Swan conductors and is tame at \(0\).

(iii) Since \(\mathcal E_n|_{\mathbb L_c}\) is tame at \(\infty\), \(E=E_{=0}\subseteq V_{=0}\), and the exactness of \(V\mapsto V_{=x}\), applied to \(0\to E\to V\to V/E\to0\), gives the decomposition. The quotient \(V_{=0}/E\) is unipotent because \(V_{=0}\) is, and its dimension is \(2^{n-1}-\binom n{\lfloor n/2\rfloor}\).

(iv) As characters of \(I_\infty\), \(\det V=\det V_{=0}\otimes\det V_{=1}\). The first factor is trivial by (i). Put \(X:=V_{=1}\otimes\Lpsi{y/\lambda}\); then \(\det V_{=1}=\Lpsi{w_0y}^{\otimes2^{n-1}}\otimes\det X=\Lpsi{-2^{n-1}y/\lambda}\otimes\det X\) on \(I_\infty\). All breaks of \(X\) are at most some \(b<1\), so \(I_\infty^{(r)}\) acts trivially on \(X\), hence on \(\det X\subseteq X^{\otimes\dim X}\), for \(r>b\); the break of the character \(\det X\) is its Swan conductor, an integer, and is therefore \(0\). So \(\mathcal L_c:=\det(\mathcal H_n|_{\mathbb L_c})\otimes\Lpsi{2^{n-1}y/\lambda}\) is tame at \(\infty\), and it is tame at \(0\) by (ii), \(\Lpsi{2^{n-1}y/\lambda}\) being unramified there. Finally, \(\pgeom(S_n)\) acts unipotently on \(\mathcal E_n\) (Corollary~\ref{cor:E-structure}), so \(\det\mathcal E_n\) is geometrically trivial and \(\det(\mathcal W_n|_{\mathbb L_c})\cong\det(\mathcal H_n|_{\mathbb L_c})\) over \(\Fpbar\).
\end{proof}

For odd \(p\), the structure at \(y=\infty\) in Theorem~\ref{thm:coord} is also obtained in \cite{EbadiCoh}, where moreover the local monodromy of \(\mathcal W_n|_{\mathbb L_c}\) at \(y=0\) is shown to be unipotent \cite[Theorem~6.8(ii)]{EbadiCoh}, with exactly \(r(n-1)\) Jordan blocks of size two and all other blocks of size one \cite[Theorem~6.8]{EbadiCoh}; then \(\mathcal L_c\) is unramified at \(0\), hence geometrically trivial. Theorem~\ref{thm:coord} also predicts the Euler characteristics of the twists \(\mathcal W_n|_{\mathbb L_c}\otimes\Lpsi{by}\), \(b\in\Fpbar\): by the Grothendieck--Ogg--Shafarevich formula, \(-\chi_c\) equals \(r(n)\) if \(b\notin\{0,1/\lambda\}\), equals \(2^{n-1}\) if \(b=0\), and is strictly smaller than \(r(n)\) if \(b=1/\lambda\), the twist \(\Lpsi{y/\lambda}\) lowering all breaks of \(V_{=1}\) below \(1\). These values are confirmed numerically in Section~\ref{sec:numerics}.

\begin{corollary}\label{cor:conductor}
Let \(n\ge2\) and let \(\Lambda\) be a finite-dimensional algebraic representation of \(\GL_{r(n)}\); write \(\Lambda(\mathcal W_n)\) for the lisse sheaf obtained by composing the monodromy representation of \(\mathcal W_n\) with \(\Lambda\). Then \(\Lambda(\mathcal W_n)|_{\mathbb L_c}\) is tamely ramified at \(y=0\), all its breaks at \(y=\infty\) are at most \(1\), and
\[
\Swan_0\bigl(\Lambda(\mathcal W_n)|_{\mathbb L_c}\bigr)=0,\qquad\Swan_\infty\bigl(\Lambda(\mathcal W_n)|_{\mathbb L_c}\bigr)\le\dim\Lambda .
\]
In particular \(\dim H^1_c\bigl(\Gm\otimes\Fpbar,\Lambda(\mathcal W_n)|_{\mathbb L_c}\bigr)\le2\dim\Lambda\).
\end{corollary}

\begin{proof}
By Theorem~\ref{thm:coord}, all breaks of \(\mathcal W_n|_{\mathbb L_c}\) are \(0\) at \(y=0\) and at most \(1\) at \(y=\infty\). Let \(V\) be the standard representation of \(\GL_{r(n)}\). Every finite-dimensional algebraic representation of \(\GL_{r(n)}\) is a direct summand of a finite direct sum of representations \(V^{\otimes a}\otimes(V^\vee)^{\otimes b}\): it is semisimple, and each irreducible one is a summand of some \(V^{\otimes a}\otimes\det^{-c}\), while \(\det^{-1}\) is a summand of \((V^\vee)^{\otimes r(n)}\). Recall that all breaks of a representation \(U\) of the inertia group are \(\le b\) if and only if \(I^{(x)}\) acts trivially on \(U\) for every \(x>b\). This property passes to duals, to direct sums and to subquotients, and to tensor products by Lemma~\ref{lem:breaks}(i). Hence all breaks of \(\Lambda(\mathcal W_n)|_{\mathbb L_c}\) are \(0\) at \(y=0\) and at most \(1\) at \(y=\infty\); a Swan conductor being the sum of the breaks counted with multiplicity, this gives the two formulas. By the Grothendieck--Ogg--Shafarevich formula, \(\dim H^1_c-\dim H^2_c=\Swan_0+\Swan_\infty\le\dim\Lambda\) (the groups \(H^0_c\) vanish on the affine curve \(\Gm\)), and \(\dim H^2_c\le\dim\Lambda\), being the dimension of a space of coinvariants.
\end{proof}

\emph{Summary of the Fourier models.} For the reader's convenience we collect the Fourier models used in Sections~\ref{subsec:ray} and~\ref{subsec:coord}. In all of them two affine lines are in duality through the product of their coordinates (\(sz\), resp.\ \(wy\)), the Fourier transform is \(\mathcal F_\psi(K)=R\,\mathrm{pr}_!(\mathrm{pr}^*K\otimes\mathcal L_\psi)[1]\), with \(\mathcal L_\psi\) pulled back along that product, as in \cite[(1.2.1.1)]{Laumon}; it preserves perversity and carries the skyscraper sheaf at a point \(\xi\) of the first line to \(\Lpsi{\xi z}[1]\), resp.\ \(\Lpsi{\xi y}[1]\). The local monodromy at the point at infinity of the second line is computed by Theorem~\ref{thm:laumon}, with one local term for each singular point of the transformed sheaf and one for its point at infinity.
\begin{enumerate}[label=(\alph*)]
\item \emph{Lemma~\ref{lem:F-trivial}.} The transformed object is \(j_!\Qlb[1]\) on \(\mathbb A^1_s\), the change of variables is \((s,z)\mapsto(sz,z)\), and the local terms are \(F^{(0,\infty')}(\mathbf 1)=\mathbf 1\) and \(F^{(\infty,\infty')}(\mathbf 1)=0\).
\item \emph{Rays \textup{(}Propositions~\ref{prop:ray} and~\ref{prop:ray-frob}\textup{)}.} The transformed object is \(j_!\mathcal K_a[1]\), with \(\mathcal K_a=\bigotimes_i[\times a_i]^*\Kl\); the changes of variables are \(\phi(s,x)=(s/x,x)\) and \(\sigma(x)=\lambda/x\), which maps \(x=0\) to \(z=\infty\); the local term at \(s=0\) is \(F^{(0,\infty')}(\mathcal K_a|_{I_0})\), tame and unipotent with the Jordan type of \(J_2^{\otimes n}\), and the term at \(s=\infty\) vanishes.
\item \emph{Coordinate lines \textup{(}Proposition~\ref{prop:coord-model} and Theorem~\ref{thm:coord}\textup{)}.} The transformed object is \(j_{U!}(\mathcal A_c|_U)[1]\), with \(\mathcal A_c=\kappa^*\Delta_c^*\mathcal H_{n-1}\), which differs from \(\mathcal K_c=j_!(\mathcal C_c[2])\) only at \(w_0\); the changes of variables are \((x,t,y)\mapsto(tx,t,y)\), \(\theta(t,w)=(t/\kappa(w),w)\) and \mbox{\(\kappa(w)=\lambda w/(\lambda w+1)\)}, which maps \(0\), \(w_0\), \(\infty\) to \(0\), \(\infty\), \(1\); the local terms are \(F^{(0,\infty')}(\mathcal A_c|_{I_0})\) at \(w=0\), of break \(0\), then \(\Lpsi{w_0y}\otimes F^{(0,\infty')}\bigl((\tau_{w_0}^*\mathcal A_c)|_{I_0}\bigr)\) at \(w=w_0\), of break \(1\), and \(0\) at \(w=\infty\).
\end{enumerate}
In each case the transformed object is the extension by zero of a lisse sheaf placed in degree \(-1\), so that the local term at a singular point is the whole local monodromy there; the difference between \(\mathcal K_c\) and \(j_{U!}(\mathcal A_c|_U)[1]\) is a complex supported at \(w_0\), whose transform is isoclinic of break \(1\) at \(\infty\) and is accounted for by the triangle in the proof of Theorem~\ref{thm:coord}. The arithmetic structure, used in Proposition~\ref{prop:ray-frob}, is carried along by Theorem~\ref{thm:laumon}(iv), all the changes of variables being defined over the field of definition of the data.

\subsection{The determinant and the full monodromy groups}\label{subsec:full}

\begin{lemma}\label{lem:AS}
Let \(k\) be an algebraically closed field of characteristic \(p\), let \(n\ge1\), let \(A=k[x_1^{\pm1},\dots,x_n^{\pm1}]\) and \(X=\mathrm{Spec}\,A=\Gm^n\otimes k\), and let \(\chi\) be a continuous character of finite order of \(\pi_1(X)\) \textup{(}with values in \(\Qlb^\times\)\textup{)}. Suppose that for every \(i\in\{1,\dots,n\}\) and every point of \(X(k)\) the restriction of \(\chi\) to the coordinate line through that point in the \(i\)-th direction is trivial. Then \(\chi\) is trivial.
\end{lemma}

\begin{proof}
\emph{Characters of order \(p\).} Suppose first that \(\chi\) takes values in \(\mu_p\), which we identify with \(\mathbb Z/p\mathbb Z\). Since \(X\) is connected, continuous homomorphisms \(\pi_1(X)\to\mathbb Z/p\mathbb Z\) are the classes in \(H^1(X_{\mathrm{et}},\mathbb Z/p\mathbb Z)\), compatibly with pull-back. The Artin--Schreier sequence \(0\to\mathbb Z/p\mathbb Z\to\mathbb G_a\xrightarrow{\wp}\mathbb G_a\to0\), \(\wp(f)=f^p-f\), is exact on \(X_{\mathrm{et}}\), and \(H^1(X_{\mathrm{et}},\mathbb G_a)=H^1(X,\mathcal O_X)\) \cite[Ch.~III, Proposition 3.7]{Milne} vanishes because \(X\) is affine. Hence \(H^1(X_{\mathrm{et}},\mathbb Z/p\mathbb Z)\cong A/\wp(A)\), and restriction to a coordinate line corresponds to substituting constants for all coordinates but one.

For \(e\in\mathbb Z^n\) write \(x^e\) for the corresponding monomial, and for \(e\ne0\) let \(v(e)=\min_iv_p(e_i)\). Call \(f\in A\) \emph{reduced} if no monomial \(x^e\) with \(e\in p\mathbb Z^n\), including \(e=0\), occurs in \(f\). Every class in \(A/\wp(A)\) contains a reduced element: as \(k\) is perfect, \(ax^{pe}=\wp(a^{1/p}x^e)+a^{1/p}x^e\), which replaces a monomial with exponent \(pe\ne0\) by one with exponent \(e\), and a constant \(a\) equals \(\wp(b)\) for a root \(b\in k\) of \(b^p-b=a\). Moreover a reduced element \(f\) of \(\wp(A)\) is zero. Indeed, let \(f=g^p-g\). The monomial \(x^0\) occurs in \(g^p-g\) only through the constant term \(g_0\) of \(g\), so \(g_0^p-g_0=0\) and \(g_0\in\Fp\); replacing \(g\) by \(g-g_0\), we may assume that \(g\) has no constant term. If \(g\ne0\), choose \(e^*\) in the support of \(g\) with \(v(e^*)\) maximal. The monomial \(x^{pe^*}\) occurs in \(g^p\) with the coefficient \(g_{e^*}^p\ne0\), and only through \(e^*\); it does not occur in \(g\), because \(v(e)\le v(e^*)<v(pe^*)\) for every \(e\) in the support of \(g\). So it occurs in \(f\), although \(pe^*\in p\mathbb Z^n\); this contradiction shows \(g=0\), hence \(f=0\). Consequently every class has exactly one reduced representative.

Let \(f=\sum_ea_ex^e\) be the reduced representative of \(\chi\), and consider the first direction. For \(c=(c_2,\dots,c_n)\in(k^\times)^{n-1}\) the restriction of \(f\) to the line \(x_1\mapsto(x_1,c)\) is \(\sum_{m\in\mathbb Z}A_m(c)x_1^m\) with \(A_m(c)=\sum_{e:\,e_1=m}a_ec^{e'}\), \(e'=(e_2,\dots,e_n)\). Writing \(m=p^jm_0\) with \(p\nmid m_0\) and reducing as above, the reduced representative of this restriction is \(\sum_{p\nmid m_0}R_{m_0}(c)x_1^{m_0}\), where \(R_{m_0}(c)=\sum_{j\ge0}A_{p^jm_0}(c)^{p^{-j}}\). By hypothesis and uniqueness, \(R_{m_0}(c)=0\) for all \(c\) and all \(m_0\). Fix \(m_0\) and choose \(J\) with \(A_{p^jm_0}=0\) for \(j>J\); raising to the power \(p^J\),
\[
0=\sum_{j=0}^J\ \sum_{e:\,e_1=p^jm_0}a_e^{p^{J-j}}\,c^{p^{J-j}e'}\qquad\text{for all }c\in(k^\times)^{n-1}.
\]
As \(k\) is infinite, this Laurent polynomial in \(c\) vanishes identically. For \(j\ge1\) and \(e_1=p^jm_0\), the vector \(e'\) is not divisible by \(p\), because \(e\notin p\mathbb Z^n\) and \(p\mid e_1\); so the exponent \(p^{J-j}e'\) is nonzero with \(v(p^{J-j}e')=J-j\). For \(j=0\) the exponent \(p^Je'\) is zero or has \(v\ge J\). Hence exponents coming from different \(j\) are different, and for fixed \(j\) the map \(e\mapsto p^{J-j}e'\) is injective on \(\{e:e_1=p^jm_0\}\). Therefore \(a_e=0\) whenever \(e_1=p^jm_0\). As \(j\ge0\) and \(m_0\) vary, \(p^jm_0\) runs over all nonzero integers, so \(a_e=0\) whenever \(e_1\ne0\). The same argument in the other directions gives \(a_e=0\) for all \(e\ne0\), and \(a_0=0\) because \(f\) is reduced. So \(f=0\), and \(\chi\) is trivial.

\emph{Characters of \(p\)-power order.} If \(\chi\) has order \(p^m\) with \(m\ge1\), then \(\chi^{p^{m-1}}\) has order \(p\) and is trivial on all coordinate lines, which contradicts the first case. So every character of \(p\)-power order satisfying the hypothesis is trivial.

\emph{Characters of order prime to \(p\).} If \(\chi\) has order \(N\) prime to \(p\), it takes values in \(\mu_N\), and \(\mu_N(k)\cong\mu_N(\Qlb)\). By Kummer theory, \(H^1(X_{\mathrm{et}},\mu_N)\) is an extension of the \(N\)-torsion of \(\mathrm{Pic}(X)\) by \(A^\times/A^{\times N}\); here \(\mathrm{Pic}(X)=0\), \(A\) being factorial, and \(A^\times=\{ax^e:a\in k^\times,e\in\mathbb Z^n\}\), so \(H^1(X_{\mathrm{et}},\mu_N)\cong\mathbb Z^n/N\mathbb Z^n\), the class of \(x^e\) corresponding to \(e\). Restriction to the coordinate line in the \(i\)-th direction through a point maps \(x^e\) to a constant multiple of \(x_i^{e_i}\), hence the class \(e\) to \(e_i\bmod N\). So \(\chi\) is trivial.

\emph{General case.} A character of finite order is the product of its \(p\)-primary part and its prime-to-\(p\) part, each a power of it; both satisfy the hypothesis, so both are trivial.
\end{proof}

\begin{theorem}\label{thm:det}
Let \(n\ge2\), and let \(d\) be the order of the geometric determinant of \(\mathcal W_n\) \textup{(}Proposition~\textup{\ref{prop:det})}.
\begin{enumerate}[label=\textup{(\roman*)}]
\item On every coordinate line \(\mathbb L\) of \(S_n\), in any direction and with coordinate \(y\), \(\det(\mathcal W_n|_{\mathbb L})\cong\Lpsi{-2^{n-1}y/\lambda}\otimes\mathcal L\) over \(\Fpbar\), for a lisse sheaf \(\mathcal L\) of rank one tamely ramified at \(0\) and at \(\infty\).
\item If \(p=2\), then \(\det(\mathcal W_n)\cong\Qlb\bigl(-\tfrac{(n+1)r(n)}2\bigr)\) as lisse sheaves on \(S_n\) over \(\mathbb F_2\). In particular \(d=1\), and \(\det(\Frob_h\mid\mathcal W_n)=q^{(n+1)r(n)/2}\) for every \(h\in S_n(\Fq)\).
\item If \(p\) is odd, then \(p\) divides \(d\).
\end{enumerate}
\end{theorem}

\begin{proof}
(i) This is Theorem~\ref{thm:coord}(iv), together with the symmetry under \(\mathfrak S_n\).

(ii) Let \(p=2\), and let \(\chi\) be the character by which \(\pgeom(S_n)\) acts on \(\det(\mathcal W_n)\). If \(n\) is even, then \(\ggeom(\mathcal W_n)\subseteq\Sp(W_n,B)\subseteq\SL(W_n)\) by Proposition~\ref{prop:selfdual}, so \(\chi\) is trivial. Let \(n\) be odd. By Proposition~\ref{prop:selfdual}, \(\ggeom(\mathcal W_n)\subseteq\mathrm O(W_n,B)\), so \(\chi\) takes values in \(\{\pm1\}\). Let \(\mathbb L\) be a coordinate line. Since \(2^{n-1}=0\) in \(\mathbb F_2\), part (i) shows that \(\det(\mathcal W_n|_{\mathbb L})\) is tamely ramified at \(0\) and at \(\infty\). Its restriction to the inertia group at \(0\) is therefore trivial on the wild inertia group, so it factors through the tame quotient, which is pro-prime-to-\(2\), while its values lie in \(\{\pm1\}\); hence it is trivial, and likewise at \(\infty\). Thus \(\det(\mathcal W_n|_{\mathbb L})\) extends to a lisse sheaf \(\mathcal D\) of rank one on \(\mathbb P^1\otimes\Fpbar\). By the Grothendieck--Ogg--Shafarevich formula \(\chi(\mathbb P^1,\mathcal D)=2\), so \(H^0(\mathbb P^1,\mathcal D)\ne0\) or \(H^2(\mathbb P^1,\mathcal D)\ne0\), and by Poincar\'e duality the latter is dual to \(H^0(\mathbb P^1,\mathcal D^\vee)(1)\); so \(\mathcal D\) or \(\mathcal D^\vee\) has a nonzero global section, and \(\mathcal D\) is constant. Hence \(\chi\) is trivial on all coordinate lines, and Lemma~\ref{lem:AS} shows that \(\chi\) is trivial. So \(d=1\) for every \(n\ge2\).

Since \(\det(\mathcal W_n)\) is geometrically constant, \(\pi_1(S_n)\) acts on it through \(\mathrm{Gal}(\Fpbar/\mathbb F_2)\), and \(\gamma_q:=\det(\Frob_h\mid\mathcal W_n)\) depends only on \(q\) for \(h\in S_n(\Fq)\). Let \(h\in S_n(\Fq)\) and apply Proposition~\ref{prop:ray-frob} with \(a=h\): the decomposition group \(D_0\) at \(x=0\) acts on the stalk of the geometrically constant sheaf \(\Delta_h^*\det(\mathcal W_n)\) through \(\mathrm{Gal}(\Fqbar/\Fq)\), so \(\Phi\) acts on it by \(\gamma_q\), the scalar by which \(\Frob_{x=1}=\Frob_h\) acts. Hence \(\gamma_q=q^{(n+1)r(n)/2}\) by Proposition~\ref{prop:ray-frob}(iii). The sheaf \(\det(\mathcal W_n)\otimes\Qlb\bigl(\frac{(n+1)r(n)}2\bigr)\) therefore has trivial Frobenius at every point; by Lemma~\ref{lem:constant} it is geometrically constant, and the Frobenius of \(\mathbb F_2\) acts on it as \(\Frob_h\) for \(h=(1,\dots,1)\in S_n(\mathbb F_2)\), that is, trivially. This proves (ii).

(iii) Let \(p\) be odd and let \(\mathbb L\) be a coordinate line. By (i), the restriction of \(\det(\mathcal W_n|_{\mathbb L})\) to the wild inertia group \(P_\infty\) at \(y=\infty\) is that of \(\Lpsi{-2^{n-1}y/\lambda}\). As \(2^{n-1}/\lambda\ne0\) in \(\Fp\), this sheaf has Swan conductor \(1\) at \(\infty\), so its restriction to \(P_\infty\) is nontrivial; since \(\Lpsi{-2^{n-1}y/\lambda}^{\otimes p}\) is trivial, the restriction has order \(p\). So the image \(\mu_d\) of \(\pgeom(S_n)\) under the determinant contains an element of order \(p\).
\end{proof}

For odd \(p\) the sheaf \(\mathcal L\) of Theorem~\ref{thm:det}(i) is geometrically trivial, because the local monodromy at \(y=0\) is unipotent (\cite[Theorem~6.8(ii)]{EbadiCoh}; for \(n=2\), Theorem~\ref{thm:summary}(iv)); the consequences are drawn in Theorem~\ref{thm:summary}(iii) below. Together with Theorem~\ref{thm:det}(ii) and Lemma~\ref{lem:n1}, this gives the uniform formula
\[
\det(\mathcal W_n)\cong\Lpsi{-2^{n-1}(h_1+\cdots+h_n)/\lambda}\otimes\Qlb\Bigl(-\frac{(n+1)\,r(n)}2\Bigr)\qquad(n\ge1),
\]
valid for every \(p\), where for odd \(p\) and \(n\ge3\) the input \cite[Theorem~6.8(ii)]{EbadiCoh} just described is used, and in which the Artin--Schreier factor disappears exactly when \(p=2\) and \(n\ge2\). This formula is confirmed numerically in Section~\ref{sec:numerics}.

\begin{theorem}\label{thm:full2}
Let \(p=2\) and \(n\ge2\), let \(B\) be the form of Proposition~\textup{\ref{prop:selfdual}}, and let \(\alpha_n\in\Qlb\) satisfy \(\alpha_n^2=2^{n+1}\) and \(\iota(\alpha_n)>0\).
\begin{enumerate}[label=\textup{(\roman*)}]
\item \(\ggeom(\mathcal W_n)=\mathrm{SO}(W_n,B)\) if \(n\) is odd, and \(\ggeom(\mathcal W_n)=\Sp(W_n,B)\) if \(n\) is even; for \(n=2\), \(\Sp(W_2,B)=\SL(W_2)\).
\item The lisse sheaf \(\mathcal W_n':=\mathcal W_n\otimes\alpha_n^{-\deg}\), which is pure of weight \(0\), satisfies \(G_{\mathrm{arith}}(\mathcal W_n')=\ggeom(\mathcal W_n')=\ggeom(\mathcal W_n)\), where \(G_{\mathrm{arith}}\) is the Zariski closure of the image of \(\pi_1(S_n)\).
\end{enumerate}
\end{theorem}

\begin{proof}
(i) By Theorem~\ref{thm:main}, \(\ggeom(\mathcal W_n)=\{g\in\mathrm O(W_n,B):\det(g)^d=1\}\) if \(n\ge3\) is odd and \(\ggeom(\mathcal W_n)=\Sp(W_n,B)\) if \(n\) is even (for \(n=2\), \(\ggeom(\mathcal W_2)=\SL(W_2)=\Sp(W_2,B)\)); and \(d=1\) by Theorem~\ref{thm:det}(ii).

(ii) Twisting by the geometrically constant sheaf \(\alpha_n^{-\deg}\) does not change \(\ggeom\). By Proposition~\ref{prop:selfdual}, \(B\) is a morphism \(\mathcal W_n'\otimes\mathcal W_n'\to\Qlb(-n-1)\otimes\alpha_n^{-2\deg}=\Qlb\), since \(\Qlb(-n-1)=(2^{n+1})^{\deg}\); so \(\pi_1(S_n)\) preserves \(B\), and \(G_{\mathrm{arith}}(\mathcal W_n')\subseteq\mathrm O(W_n,B)\), resp.\ \(\Sp(W_n,B)\). If \(n\) is odd, then \(\alpha_n=2^{(n+1)/2}\); if \(n\) is even, then \(r(n)\) is even, \(\binom n{n/2}=2\binom{n-1}{n/2-1}\) being even. In both cases \(\alpha_n^{r(n)}=2^{(n+1)r(n)/2}\), so \(\det(\mathcal W_n')=\Qlb\bigl(-\frac{(n+1)r(n)}2\bigr)\otimes\alpha_n^{-r(n)\deg}\) is trivial by Theorem~\ref{thm:det}(ii). Hence \(G_{\mathrm{arith}}(\mathcal W_n')\subseteq\mathrm{SO}(W_n,B)\), resp.\ \(\Sp(W_n,B)\), which is \(\ggeom(\mathcal W_n')\) by (i); and \(\ggeom\subseteq G_{\mathrm{arith}}\) always.
\end{proof}

\begin{corollary}\label{cor:forced}
Let \(p=2\) and let \(n\ge3\) be such that \(r(n)\) is odd; this happens exactly when \(n+1\) is a power of \(2\). Then for every finite field \(\Fq\) of characteristic \(2\) and every \(h\in S_n(\Fq)\), the number \(q^{(n+1)/2}\) is an eigenvalue of \(\Frob_h\) on the fibre of \(\mathcal W_n\) at \(h\), and \(q^{(n-1)/2}\) is an eigenvalue of \(\Frob_h\) on \(H^{n-1}_c\bigl(Y_h\otimes\Fqbar,\Lpsi{\sum_iu_i^{-1}}\bigr)\).
\end{corollary}

\begin{proof}
Let \(g=q^{-(n+1)/2}\Frob_h\), acting on the fibre \(W\) of \(\mathcal W_n\) at a geometric point over \(h\), of odd dimension \(N=r(n)\). By Proposition~\ref{prop:selfdual}(b), \(g\) preserves \(B\), so \(g^{-1}=g^*\), the adjoint with respect to \(B\); and \(\det g=1\) by Theorem~\ref{thm:det}(ii). Hence
\[
\det(g-1)=\det(g)\det(1-g^{-1})=\det(1-g^*)=\det(1-g)=(-1)^N\det(g-1),
\]
so \(\det(g-1)=0\), and \(1\) is an eigenvalue of \(g\). The second assertion follows from Corollary~\ref{cor:exact}(ii). Finally, \(r(n)\equiv\binom n{\lfloor n/2\rfloor}\pmod2\); this is even if \(n\) is even, and for \(n=2k+1\), by Kummer's theorem, the \(2\)-adic valuation of \(\binom{2k+1}k\) is the number of carries in the binary addition of \(k\) and \(k+1\), which is zero exactly when \(k\) and \(k+1\) have no binary digit \(1\) in common, that is, when \(k=2^j-1\), \(n+1=2^{j+1}\).
\end{proof}

For \(n=3\), Corollary~\ref{cor:forced} says that for every \(h\in(\Fq^\times)^3\), \(q\) even, the Frobenius on the five-dimensional space \(\mathcal W_{3,h}\) has the eigenvalue \(q^2\), and the Frobenius on \(H^2_c(Y_h,\Lpsi{\sum u_i^{-1}})\) has the eigenvalue \(q\); the other four eigenvalues on \(\mathcal W_{3,h}\) form two pairs \(\{\beta,q^4/\beta\}\). The forced eigenvalue is a consequence of the group being \(\mathrm{SO}_5\) rather than \(\mathrm O_5\), together with the arithmetic normalisation \(\det=q^{10}\): for an element of \(\mathrm O_5\) of determinant \(-1\), the forced eigenvalue would be \(-q^2\).

We collect the results on the monodromy groups in odd characteristic and in characteristic two, and make explicit the single point at which the companion preprint enters, for \(n\ge3\): \cite[Theorem~6.8(ii)]{EbadiCoh}.

\begin{theorem}[The monodromy groups]\label{thm:summary}
Let \(n\ge2\) and \(N=r(n)\), let \(\beta_n\in\Qlb\) satisfy \(\beta_n^2=p^{n+1}\) and \(\iota(\beta_n)>0\), and put \(\mathcal W_n^\natural:=\mathcal W_n\otimes\beta_n^{-\deg}\), which is pure of weight \(0\).
\begin{enumerate}[label=\textup{(\roman*)}]
\item \textup{(\(p=2\).)} \(\det(\mathcal W_n)\cong\Qlb\bigl(-\frac{(n+1)N}2\bigr)\), and \(G_{\mathrm{arith}}(\mathcal W_n^\natural)=\ggeom(\mathcal W_n)\), which is \(\mathrm{SO}(W_n,B)\cong\mathrm{SO}_N\) if \(n\) is odd and \(\Sp(W_n,B)\cong\Sp_N\) if \(n\) is even.
\item \textup{(\(p\) odd.)} \(\ggeom(\mathcal W_n)^0=\SL(W_n)\) and \(\ggeom(\mathcal W_n)=\{g\in\GL(W_n):\det(g)^d=1\}\), where \(p\) divides \(d\).
\item \textup{(\(p\) odd, with one input from \cite{EbadiCoh}.)} Suppose that for every coordinate line \(\mathbb L\) of \(S_n\) the local monodromy of \(\mathcal W_n|_{\mathbb L}\) at \(y=0\) is unipotent, as is shown in \cite[Theorem~6.8(ii)]{EbadiCoh}. Then
\[
\det(\mathcal W_n)\cong\Lpsi{-2^{n-1}(h_1+\cdots+h_n)/\lambda}\otimes\Qlb\Bigl(-\frac{(n+1)N}2\Bigr),
\]
\(d=p\), and \(G_{\mathrm{arith}}(\mathcal W_n^\natural)=\ggeom(\mathcal W_n)=\{g\in\GL(W_n):\det(g)^p=1\}\).
\item \textup{(\(n=2\).)} For every \(p\) and every coordinate line \(\mathbb L\) of \(S_2\), the local monodromy of \(\mathcal W_2|_{\mathbb L}\) at \(y=0\) is unipotent, with a single Jordan block of size two. In particular the hypothesis of \textup{(iii)} holds for \(n=2\): for odd \(p\), \(\det(\mathcal W_2)\cong\Lpsi{-2(h_1+h_2)/\lambda}\otimes\Qlb(-3)\) and \(G_{\mathrm{arith}}(\mathcal W_2^\natural)=\ggeom(\mathcal W_2)=\{g\in\GL(W_2):\det(g)^p=1\}\).
\end{enumerate}
\end{theorem}

\begin{proof}
(i) If \(p=2\), then \(\beta_n=\alpha_n\) in the notation of Theorem~\ref{thm:full2}, and (i) is Theorem~\ref{thm:det}(ii) together with Theorem~\ref{thm:full2}. (ii) is Theorem~\ref{thm:main}(i) together with Theorem~\ref{thm:det}(iii).

(iii) Let \(\mathbb L\) be a coordinate line, with coordinate \(y\), and let \(\mathcal L\) be as in Theorem~\ref{thm:det}(i). Since a unipotent automorphism has determinant \(1\), the hypothesis shows that \(\det(\mathcal W_n|_{\mathbb L})\), and hence \(\mathcal L\), is unramified at \(y=0\); so \(\mathcal L\) extends to a lisse sheaf of rank one on \(\mathbb A^1\otimes\Fpbar\) which is tamely ramified at \(\infty\). By the Grothendieck--Ogg--Shafarevich formula its Euler characteristic with compact supports on \(\mathbb A^1\) is \(1\), and the argument in the proof of Lemma~\ref{lem:detkl} shows that it is geometrically constant. Let \(\chi\) be the character by which \(\pgeom(S_n)\) acts on \(\det(\mathcal W_n)\otimes\Lpsi{2^{n-1}(h_1+\cdots+h_n)/\lambda}\). It has finite order, by Proposition~\ref{prop:det} and because \(\Lpsi{\cdot}^{\otimes p}\) is trivial; and on \(\mathbb L\), where \(h_1+\cdots+h_n=y+{}\)constant, it is the character of \(\mathcal L\), which is trivial. By Lemma~\ref{lem:AS}, \(\chi\) is trivial. So \(\det(\mathcal W_n)\otimes\Lpsi{2^{n-1}(h_1+\cdots+h_n)/\lambda}\cong b^{\deg}\) for some \(b\in\Qlb^\times\). For \(h\in S_n(\Fp)\) consider the ray \(\Delta_h\): the restriction of \(\Lpsi{-2^{n-1}(h_1+\cdots+h_n)/\lambda}\) to it is \(\Lpsi{-2^{n-1}x(h_1+\cdots+h_n)/\lambda}\), which is lisse at \(x=0\) with Frobenius \(\psi(0)=1\) there; so Proposition~\ref{prop:ray-frob}(iii) gives \(b=p^{(n+1)N/2}\). This proves the formula for \(\det(\mathcal W_n)\). The function \(-2^{n-1}(h_1+\cdots+h_n)/\lambda\) is a reduced element of the ring \(A\) of Lemma~\ref{lem:AS} different from \(0\), so \(\Lpsi{-2^{n-1}(h_1+\cdots+h_n)/\lambda}\) is geometrically of order exactly \(p\), and \(d=p\); by (ii), \(\ggeom(\mathcal W_n)=\{g:\det(g)^p=1\}\). Finally, \(\beta_n^N=p^{(n+1)N/2}\): if \(n\) is odd, then \(\beta_n=p^{(n+1)/2}\), and if \(n\) is even, then \(N\) is even and \(\beta_n^N=(\beta_n^2)^{N/2}\). Hence \(\det(\mathcal W_n^\natural)\cong\Lpsi{-2^{n-1}(h_1+\cdots+h_n)/\lambda}\) as lisse sheaves on \(S_n\) over \(\Fp\), so \(\pi_1(S_n)\) acts on \(\det(\mathcal W_n^\natural)\) through \(\mu_p\), and its image lies in \(\{g:\det(g)^p=1\}=\ggeom(\mathcal W_n)=\ggeom(\mathcal W_n^\natural)\subseteq G_{\mathrm{arith}}(\mathcal W_n^\natural)\).

(iv) The last assertion is (iii) for \(n=2\), so it suffices to prove the first. By Proposition~\ref{prop:n2-kl}(ii), the restriction of \(\mathcal W_2\) to a coordinate line is the tensor product of an Artin--Schreier sheaf which is lisse at \(y=0\), of a dilate of \(\Kl\) and of \(\Qlb(-1)\); so the first assertion follows from Theorem~\ref{thm:kl}(ii) and the remark following it (Remark~\ref{rem:n2-kl}(a)). We also give an argument which does not use Proposition~\ref{prop:n2-kl}. By the symmetry of \(S_2\) under the exchange of the coordinates we may take \(\mathbb L=\mathbb L_c\), \(c\in\Fq^\times\); put \(\mathcal F=\mathcal W_2|_{\mathbb L_c}\), a lisse sheaf on \(\Gm[,y]\) over \(\Fq\), pure of weight \(3\). \emph{A trace identity.} A point \(u\in(\Fq^\times)^2\) with \(cu_2\ne-\lambda\) satisfies \(yu_1+cu_2=-\lambda\) for exactly one \(y\in\Fq^\times\), and a point with \(cu_2=-\lambda\) for none; so
\[
\sum_{y\in\Fq^\times}\Sigma_2(y,c)=\sum_{u_1\ne0}\psi_q(u_1^{-1})\sum_{\substack{u_2\ne0\\ cu_2\ne-\lambda}}\psi_q(u_2^{-1})=1+\psi_q(-c/\lambda).
\]
By Corollary~\ref{cor:exact}(i), \(\Tr(\Frob_y\mid\mathcal F)=-q\Sigma_2(y,c)-q\), so \(\sum_{y\in\Fq^\times}\Tr(\Frob_y\mid\mathcal F)=-q^2-q\,\psi_q(-c/\lambda)\); the same holds over \(\mathbb F_{q^m}\), with \(q\) replaced by \(q^m\) and \(\psi_q(-c/\lambda)\) by \(\psi_q(-c/\lambda)^m\), because \(\Tr_{\mathbb F_{q^m}/\Fq}(x)=mx\) for \(x\in\Fq\). \emph{The cohomology of \(\mathcal F\).} By Theorem~\ref{thm:coord}(iii), the part of break \(0\) of \(\mathcal F\) at \(y=\infty\) has dimension \(2-\binom21=0\), so \(\mathcal F\) and \(\mathcal F^\vee\) have no nonzero vector fixed by \(I_\infty\); hence \(H^0_c\) and \(H^2_c\) of \(\mathcal F\) on \(\Gm\otimes\Fqbar\) vanish, and by Theorem~\ref{thm:coord}(ii) and the Grothendieck--Ogg--Shafarevich formula \(\dim H^1_c=\Swan_0+\Swan_\infty=2\). By the trace formula, the eigenvalues of Frobenius on \(H^1_c\) are \(q^2\) and \(q\,\psi_q(-c/\lambda)\). \emph{Weights.} As in Proposition~\ref{prop:fibre}(i), with \(\mathcal F^{I_\infty}=0\), there is an exact sequence \(0\to(j_*\mathcal F)_0\to H^1_c(\Gm\otimes\Fqbar,\mathcal F)\to H^1(\mathbb P^1\otimes\Fqbar,j_*\mathcal F)\to0\), in which \(H^1(\mathbb P^1,j_*\mathcal F)\) is pure of weight \(4\) \cite[Th\'eor\`eme 3.2.3]{Deligne} and \((j_*\mathcal F)_0=\mathcal F^{I_0}\) has weights \(\le3\) \cite[1.8.4]{Deligne}. So \(\mathcal F^{I_0}\) is a line on which Frobenius acts by \(q\,\psi_q(-c/\lambda)\), of weight \(2\). \emph{The local monodromy.} By Theorem~\ref{thm:coord}(ii), \(\mathcal F\) is tamely ramified at \(0\); let \(\varrho(\sigma)=s\exp(\mathfrak N)\) be the Jordan decomposition of the image of a generator \(\sigma\) of tame inertia, with \(s\) semisimple and \(\mathfrak N\) nilpotent, commuting with \(s\). If \(\mathfrak N=0\), the local monodromy filtration at \(0\) is trivial, and by \cite[1.8.4]{Deligne} the geometric generic stalk, as a representation of the decomposition group at \(0\), is pure of weight \(3\), so that \(\mathcal F^{I_0}\) would be pure of weight \(3\). Hence \(\mathfrak N\ne0\). As \(\mathcal F\) has rank two, \(\mathfrak N\) induces an isomorphism from \(\mathcal F_{\bar\eta}/\ker\mathfrak N\) onto \(\ker\mathfrak N\) commuting with \(s\), so \(s\) has a single eigenvalue and, being semisimple, is a scalar \(\chi\); and \(\ker(\chi\exp(\mathfrak N)-1)=\mathcal F^{I_0}\ne0\) forces \(\chi=1\). So \(\varrho(\sigma)=\exp(\mathfrak N)\) is unipotent with a single Jordan block of size two.
\end{proof}

The conclusions of Theorem~\ref{thm:summary}(iii) are also obtained in \cite{EbadiCoh}, where the Tate twist in \(\det(\mathcal W_n)\) is computed by a different method, from vanishing cycles; the proof above uses from \cite{EbadiCoh} only the unipotence of the local monodromy at \(y=0\) \cite[Theorem~6.8(ii)]{EbadiCoh}, which for \(n=2\) is part~(iv), the arithmetic normalisation being supplied by Proposition~\ref{prop:ray-frob}. For \(n=1\) the formula for \(\det(\mathcal W_1)=\mathcal W_1\) is Lemma~\ref{lem:n1}, in every characteristic.

\subsection{Moments and equidistribution}\label{subsec:moments}
We now draw the consequences of the preceding results for the statistics of the sums \(\Sigma_n(h)\). We begin with the even absolute moments, which require only Theorem~\ref{thm:main} and, for \(p=2\), Proposition~\ref{prop:selfdual}.

\begin{corollary}\label{cor:moments}
Let \(n\ge2\) and \(k\ge1\), and let \(q\) run over the powers of \(p\).
\begin{enumerate}[label=\textup{(\roman*)}]
\item If \(p\) is odd or \(n=2\), and \(k\le r(n)\), then
\[
\sum_{h\in\Fq^n}|\Sigma_n(h)|^{2k}=k!\,q^{n+k(n-1)}+O\bigl(q^{n+k(n-1)-1/2}\bigr).
\]
\item If \(p=2\), \(n\ge3\) and \(2k<r(n)\), then
\[
\sum_{h\in\Fq^n}\Sigma_n(h)^{2k}=(2k-1)!!\;q^{n+k(n-1)}+O\bigl(q^{n+k(n-1)-1/2}\bigr),
\]
where \((2k-1)!!=1\cdot3\cdots(2k-1)\); here \(\Sigma_n(h)\in\mathbb Z\).
\end{enumerate}
The implied constants depend on \(n\), \(k\) and \(p\), and both statements remain true when \(h\) is restricted to \((\Fq^\times)^n\).
\end{corollary}

For \(k=2\) the main terms are \(2q^{3n-2}\) and \(3q^{3n-2}\), in agreement with Theorem~\ref{thm:fourth}, which gives these asymptotics with an error term uniform in \(p\) and with an explicit second-order term.

\begin{proof}
Let \(G=\ggeom(\mathcal W_n)\), \(V=W_n\), \(N=r(n)\), and \(T_k=V^{\otimes k}\otimes(V^\vee)^{\otimes k}=\End(V^{\otimes k})\).

\emph{The invariants.} In case (i), \(\SL(V)\subseteq G\subseteq\GL(V)\) by Theorem~\ref{thm:main}(i). The scalars act trivially on \(T_k\), so the \(\SL(V)\)-invariants and the \(\GL(V)\)-invariants in \(T_k\) coincide, and therefore so do the \(G\)-invariants. By Schur--Weyl duality \cite[\S4.2.4]{GW}, \(\End_{\GL(V)}(V^{\otimes k})\) is spanned by the operators \(P_\sigma\), \(\sigma\in\mathfrak S_k\), which permute the tensor factors; for \(k\le N\) they are linearly independent, since they map \(e_1\otimes\cdots\otimes e_k\), for distinct basis vectors \(e_i\), to distinct basis vectors. Hence \(M_{2k}(G,V)=k!\). Each \(P_\sigma\) commutes with \(g^{\otimes k}\) for every \(g\in\GL(V)\), so \(\pi_1(S_n)\) acts trivially on the invariants.

In case (ii), \(\mathrm{SO}(V,B)\subseteq G\subseteq\mathrm O(V,B)\) if \(n\) is odd, and \(G=\Sp(V,B)\) if \(n\) is even (Theorem~\ref{thm:main}). Since \(\pgeom(S_n)\) acts trivially on the fibre of \(\Qlb(-n-1)\), the target of \(B\) (Proposition~\ref{prop:selfdual}), \(\phi_B:V\to V^\vee\), \(v\mapsto B(v,\cdot)\), is an isomorphism of \(G\)-modules, and \(T_k\cong V^{\otimes2k}\) as \(G\)-modules. By the first fundamental theorem of invariant theory for the orthogonal and symplectic groups \cite[\S5.3.2]{GW}, the invariants of \(\mathrm O(V,B)\) and of \(\Sp(V,B)\) in \(V^{\otimes2k}\) are spanned by the \((2k-1)!!\) tensors \(c_\pi\) attached to the perfect matchings \(\pi\) of \(\{1,\dots,2k\}\), where \(c_\pi\) places a copy of the invariant tensor \(c=\theta^{-1}(\mathrm{id}_V)\in V^{\otimes2}\) of the proof of Theorem~\ref{thm:larsen}(iii),(iv) on each pair of \(\pi\). When \(2k<N\), the invariants of \(\mathrm{SO}(V,B)\) in \(V^{\otimes2k}\) coincide with those of \(\mathrm O(V,B)\). Indeed, \(\mathrm O(V,B)/\mathrm{SO}(V,B)\) has order two, so the \(\mathrm{SO}(V,B)\)-invariants are the sum of the \(\mathrm O(V,B)\)-invariants and of the space of vectors \(x\) with \(gx=\det(g)x\) for all \(g\in\mathrm O(V,B)\); and such an \(x\) vanishes: in a basis \((e_i)\) with \(B(e_i,e_j)=\delta_{ij}\), every basis tensor \(e_{i_1}\otimes\cdots\otimes e_{i_{2k}}\) omits some index \(j\), because \(2k<N\), and the reflection \(e_j\mapsto-e_j\), of determinant \(-1\), fixes that basis tensor, so its coefficient in \(x\) equals its own negative. These tensors are linearly independent whenever \(2k\le N\): choose a basis \((e_i)\) with \(B(e_i,e_j)=\delta_{ij}\) in the orthogonal case, and a basis \((e_i,f_i)_{i\le N/2}\) with \(B(e_i,f_j)=\delta_{ij}\), \(B(e_i,e_j)=B(f_i,f_j)=0\) in the symplectic case; labelling the pairs of a matching \(\sigma\) by distinct indices \(i_1,\dots,i_k\), the basis tensor with \(e_{i_s}\) (resp.\ \(e_{i_s}\) and \(f_{i_s}\)) in the two positions of the \(s\)-th pair occurs in \(c_\pi\) with nonzero coefficient if and only if \(\pi=\sigma\). Hence \(M_{2k}(G,V)=(2k-1)!!\) in both subcases. Under the isomorphism \(\mathrm{id}^{\otimes k}\otimes(\phi_B^{-1})^{\otimes k}:T_k\to V^{\otimes2k}\), the tensor \(c_\pi\) corresponds, up to a nonzero scalar, to the tensor \(c'_\pi\in T_k\) which places on each pair of \(\pi\) joining two factors \(V\) the tensor \(c\), on each pair joining two factors \(V^\vee\) the form \(B\), and on each pair joining a factor \(V\) and a factor \(V^\vee\) the canonical element \(\mathrm{id}_V\in V\otimes V^\vee\). Let \(\chi\) be the character by which \(\pi_1(S_n)\) acts on the fibre of \(\Qlb(-n-1)\), so that \(B(\gamma v,\gamma w)=\chi(\gamma)B(v,w)\) and \(\chi(\Frob_h)=q^{n+1}\) for \(h\in S_n(\Fq)\). Then \(\pi_1(S_n)\) acts on \(c\) through \(\chi\), on \(B\in V^\vee\otimes V^\vee\) through \(\chi^{-1}\), and trivially on \(\mathrm{id}_V\). As \(T_k\) has \(k\) factors \(V\) and \(k\) factors \(V^\vee\), every matching has as many pairs of the first kind as of the second; so \(\pi_1(S_n)\) fixes every \(c'_\pi\), and it acts trivially on the invariants.

\emph{Passage to \(\Sigma_n\).} By Theorem~\ref{thm:moment-principle}(ii) with \(S=S_n\), \(e=n\), \(w=n+1\), \(m=k\),
\[
\sum_{h\in S_n(\Fq)}\bigl|\Tr(\Frob_h\mid\mathcal W_n)\bigr|^{2k}=M_{2k}(G,V)\,q^{n+k(n+1)}+O\bigl(q^{n+k(n+1)-1/2}\bigr).
\]
Write \(\mathsf t_h=\Tr(\Frob_h\mid\mathcal W_n)\) and \(\mathsf a_h=(-1)^{n-1}q\Sigma_n(h)\). By Corollary~\ref{cor:exact}, \(|\mathsf a_h-\mathsf t_h|\le c_nq^{n/2}\), and \(|\mathsf a_h|,|\mathsf t_h|\le c_nq^{(n+1)/2}\), so that \(\bigl||\mathsf a_h|^{2k}-|\mathsf t_h|^{2k}\bigr|\le2k\,c_n^{2k}q^{n/2+(2k-1)(n+1)/2}\); summing over the \((q-1)^n\) points of \(S_n(\Fq)\) gives an error \(O(q^{n+k(n+1)-1/2})\). Dividing by \(q^{2k}\) gives both statements for \(h\in(\Fq^\times)^n\); for \(p=2\), \(\Sigma_n(h)\) is an integer, because \(\psi_q\) takes the values \(\pm1\), so \(|\Sigma_n(h)|^{2k}=\Sigma_n(h)^{2k}\). Finally, by Lemma~\ref{lem:offtorus} and Corollary~\ref{cor:pointwise}, the points \(h\in\Fq^n\setminus(\Fq^\times)^n\) contribute at most \(\sum_{r=1}^{n-1}\binom nr(q-1)^r\bigl((2^r-1)q^{(r-1)/2}\bigr)^{2k}=O(q^{n-1+k(n-2)})\), which is smaller than the error term.
\end{proof}

The full determination of the groups in characteristic two, including the arithmetic monodromy group, gives more: an equidistribution theorem, and a description of all the power moments of \(\Sigma_n(h)\). For a reductive group \(G\subseteq\GL(V)\) and integers \(m_1,m_2\ge0\) we put \(P_{m_1,m_2}(G,V):=\dim\bigl(V^{\otimes m_1}\otimes(V^\vee)^{\otimes m_2}\bigr)^G\) and \(P_m(G,V):=P_{m,0}(G,V)\).

\begin{lemma}\label{lem:power-moments}
Let \(S\), \(e\), \(\mathcal F\), \(w\), \(V\) and \(G=\ggeom(\mathcal F)\) be as in Theorem~\textup{\ref{thm:moment-principle}}, and let \(m_1,m_2\ge0\). For \(a\ge1\) put \(q_a=\#k_a\).
\begin{enumerate}[label=\textup{(\roman*)}]
\item If \(P_{m_1,m_2}(G,V)=0\), then
\[
\sum_{x\in S(k_a)}\Tr\bigl(\Frob_{k_a,x}\mid\mathcal F\bigr)^{m_1}\,\overline{\Tr\bigl(\Frob_{k_a,x}\mid\mathcal F\bigr)}^{\,m_2}=O\bigl(q_a^{e+(m_1+m_2)w/2-1/2}\bigr).
\]
\item If \(w=0\) and the image of \(\pi_1(S)\) is contained in \(G\), then
\[
\sum_{x\in S(k_a)}\Tr\bigl(\Frob_{k_a,x}\mid\mathcal F\bigr)^{m_1}\,\overline{\Tr\bigl(\Frob_{k_a,x}\mid\mathcal F\bigr)}^{\,m_2}=q_a^e\,P_{m_1,m_2}(G,V)+O\bigl(q_a^{e-1/2}\bigr);
\]
in particular \(\sum_{x\in S(k_a)}\Tr(\Frob_{k_a,x}\mid\mathcal F)^m=q_a^e\,P_m(G,V)+O(q_a^{e-1/2})\) for every \(m\ge1\).
\end{enumerate}
\end{lemma}

\begin{proof}
As in the proof of Theorem~\ref{thm:moment-principle}, the complex conjugate of \(\Tr(\Frob_{k_a,x}\mid\mathcal F)\) is \(q_a^w\Tr(\Frob_{k_a,x}\mid\mathcal F^\vee)\), so the sum in (i) is \(q_a^{m_2w}\) times the sum of the traces of Frobenius on \(\mathcal G:=\mathcal F^{\otimes m_1}\otimes(\mathcal F^\vee)^{\otimes m_2}\), which is lisse and pure of weight \((m_1-m_2)w\). By Poincar\'e duality, \(H^{2e}_c(S_{\bar k},\mathcal G)\) is dual to the twist by \((e)\) of \(H^0(S_{\bar k},\mathcal G^\vee)\), the space of \(G\)-invariants in \((V^\vee)^{\otimes m_1}\otimes V^{\otimes m_2}\), whose dimension is \(P_{m_1,m_2}(G,V)\), since \(G\) is reductive (Lemma~\ref{lem:semisimple}) and dual representations of a reductive group have invariants of the same dimension. If this dimension is zero, then \(H^{2e}_c=0\), and the groups \(H^i_c\) with \(i\le2e-1\) have weights \(\le2e-1+(m_1-m_2)w\) \cite[Th\'eor\`eme 3.3.1]{Deligne}; with the trace formula this gives (i). In (ii), \(w=0\), so the sum is that of the traces of Frobenius on \(\mathcal G\); \(\pi_1(S)\) acts on \(H^0(S_{\bar k},\mathcal G^\vee)\), the space of \(G\)-invariants, through elements of \(G\), hence trivially; so the geometric Frobenius of \(k\) acts on \(H^{2e}_c(S_{\bar k},\mathcal G)\) as multiplication by \(\#k\) to the power \(e\), and the other groups contribute \(O(q_a^{e-1/2})\).
\end{proof}

\begin{theorem}\label{thm:ST}
Let \(p=2\) and \(n\ge2\), let \(G=\ggeom(\mathcal W_n)\), which is \(\mathrm{SO}_{r(n)}\) if \(n\) is odd and \(\Sp_{r(n)}\) if \(n\) is even \textup{(}Theorem~\textup{\ref{thm:full2})}, and let \(K\) be a maximal compact subgroup of the complex group \(G(\mathbb C)\) obtained through \(\iota\), so that \(K\cong\mathrm{SO}(r(n))\) or \(K\cong\mathrm{USp}(r(n))\), with Haar probability measure \(\mu_K\). For \(q=2^f\) and \(h\in S_n(\Fq)\), let \(\theta_{q,h}\) be the conjugacy class of \(K\) attached to the semisimplification of \(\iota(\Frob_h\mid\mathcal W_n')\), where \(\mathcal W_n'\) is as in Theorem~\textup{\ref{thm:full2}} \textup{(}see \cite[\S9.2]{KS}\textup{)}. Then, as \(f\to\infty\), the conjugacy classes \(\theta_{q,h}\), \(h\in S_n(\Fq)\), become equidistributed in the space of conjugacy classes of \(K\) with respect to the image of \(\mu_K\). In particular, the real numbers
\[
x_h:=q^{-(n-1)/2}\Bigl((-1)^{n-1}\Sigma_n(h)-\sum_{k=1}^{\lfloor n/2\rfloor}\nu_{n,k}\,q^{k-1}\Bigr)=\Tr(\theta_{q,h}),\qquad h\in(\Fq^\times)^n,
\]
become equidistributed with respect to the direct image of \(\mu_K\) under the trace. For \(n=2\), where \(K=\mathrm{SU}(2)\) and \(x_h=-q^{-1/2}(\Sigma_2(h)+1)\), the limiting distribution is the Sato--Tate measure \(\frac1{2\pi}\sqrt{4-x^2}\,dx\) on \([-2,2]\).
\end{theorem}

\begin{proof}
The sheaf \(\mathcal W_n'\) is lisse and pure of weight \(0\) on the smooth geometrically connected scheme \(S_n\) over \(\mathbb F_2\), and \(G_{\mathrm{arith}}(\mathcal W_n')=\ggeom(\mathcal W_n')=G\) by Theorem~\ref{thm:full2}. Deligne's equidistribution theorem \cite[Th\'eor\`eme 3.5.3]{Deligne}, in the form \cite[Theorem 9.2.6]{KS}, gives the first assertion. Since \(\alpha_n^f=q^{(n+1)/2}\) for the positive square root, \(\Tr(\theta_{q,h})=q^{-(n+1)/2}\Tr(\Frob_h\mid\mathcal W_n)\), which equals \(x_h\) by Corollary~\ref{cor:exact}(i); the trace is a continuous class function on \(K\), which gives the second assertion. For \(n=2\), \(G=\SL_2=\Sp_2\) and \(K=\mathrm{SU}(2)\), and the image of Haar measure under the trace is the Sato--Tate measure by the Weyl integration formula; moreover \(\nu_{2,1}=1\), so \(x_h=-q^{-1/2}(\Sigma_2(h)+1)\). For \(n=2\) see also Remark~\ref{rem:n2-kl}(d).
\end{proof}

For polynomial test functions the rate of convergence in Theorem~\ref{thm:ST} is \(q^{-1/2}\), and the same holds in odd characteristic once the arithmetic monodromy group is known.

\begin{corollary}\label{cor:ST-rate}
\begin{enumerate}[label=\textup{(\roman*)}]
\item Let \(p=2\) and \(n\ge2\), and let \(K\) and \(x_h\) be as in Theorem~\textup{\ref{thm:ST}}. For every polynomial \(f\in\mathbb C[x]\),
\[
\frac1{(q-1)^n}\sum_{h\in(\Fq^\times)^n}f(x_h)=\int_Kf(\Tr k)\,dk+O_{n,f}\bigl(q^{-1/2}\bigr).
\]
\item Let \(p\) be odd and \(n\ge2\), and suppose that the hypothesis of Theorem~\textup{\ref{thm:summary}(iii)} holds \textup{(}for instance \(n=2\), by Theorem~\textup{\ref{thm:summary}(iv))}. Put \(K=\{k\in\mathrm U(r(n)):\det(k)^p=1\}\), with Haar probability measure, and
\[
z_h:=q^{-(n-1)/2}\Bigl((-1)^{n-1}\Sigma_n(h)-\sum_{k=1}^{\lfloor n/2\rfloor}\nu_{n,k}\,q^{k-1}\Bigr),\qquad h\in(\Fq^\times)^n .
\]
Then for every polynomial \(f\in\mathbb C[x,y]\),
\[
\frac1{(q-1)^n}\sum_{h\in(\Fq^\times)^n}f\bigl(z_h,\overline{z_h}\bigr)=\int_Kf\bigl(\Tr k,\overline{\Tr k}\bigr)\,dk+O_{n,p,f}\bigl(q^{-1/2}\bigr).
\]
\end{enumerate}
\end{corollary}

\begin{proof}
By linearity it suffices to treat \(f=x^{m_1}y^{m_2}\), with \(m_2=0\) in (i). Let \(\mathcal V\) be \(\mathcal W_n'\) in case (i) and \(\mathcal W_n^\natural\) in case (ii); in both cases \(\mathcal V\) is lisse and pure of weight \(0\) on \(S_n\) over \(\Fp\), the image of \(\pi_1(S_n)\) is contained in \(G=\ggeom(\mathcal V)\) (Theorems~\ref{thm:full2} and~\ref{thm:summary}), and \(\Tr(\Frob_h\mid\mathcal V)=x_h\), resp.\ \(z_h\), by Corollary~\ref{cor:exact}(i), because the twisting characters take the value \(q^{-(n+1)/2}\) at \(\Frob_h\) for \(h\in S_n(\Fq)\). Lemma~\ref{lem:power-moments}(ii) gives \(\sum_hx_h^{m_1}=q^nP_{m_1,0}(G,W_n)+O(q^{n-1/2})\), resp.\ \(\sum_hz_h^{m_1}\overline{z_h}^{\,m_2}=q^nP_{m_1,m_2}(G,W_n)+O(q^{n-1/2})\). On the other hand \(K\) is a maximal compact subgroup of \(G(\mathbb C)\), Zariski dense in it, so \(P_{m_1,m_2}(G,W_n)\) is the dimension of the \(K\)-invariants in \(W_n^{\otimes m_1}\otimes(W_n^\vee)^{\otimes m_2}\), which is the integral over \(K\) of the character \((\Tr k)^{m_1}(\overline{\Tr k})^{m_2}\) of this representation. Finally \((q-1)^{-n}q^n=1+O(q^{-1})\).
\end{proof}

\begin{corollary}\label{cor:signature}
Let \(n\ge2\), let \(q\) run over the powers of \(p\), and for integers \(m,m_1,m_2\ge0\) put
\[
\mathfrak S_m(q):=\sum_{h\in\Fq^n}\Sigma_n(h)^m,\qquad\mathfrak S_{m_1,m_2}(q):=\sum_{h\in\Fq^n}\Sigma_n(h)^{m_1}\,\overline{\Sigma_n(h)}^{\,m_2}.
\]
\begin{enumerate}[label=\textup{(\roman*)}]
\item Let \(p=2\), and let \(G\) and \(V=W_n\) be as in Theorem~\textup{\ref{thm:ST}}. For every \(m\ge1\),
\[
\mathfrak S_m(q)=P_m(G,V)\,q^{n+m(n-1)/2}+O\bigl(q^{n+m(n-1)/2-1/2}\bigr).
\]
Here \(P_m(G,V)=0\) if \(m\) is odd and \(m<r(n)\), and also if \(m\) is odd and \(n\) is even; \(P_m(G,V)=(m-1)!!\) if \(m\) is even and \(m<r(n)\), and also if \(m\) is even, \(m\le r(n)\) and \(n\) is even. If \(n\) is odd, then
\[
P_{r(n)}(G,V)=P_{r(n)}\bigl(\mathrm O(W_n,B),W_n\bigr)+1,
\]
where \(P_{r(n)}(\mathrm O(W_n,B),W_n)\) equals \((r(n)-1)!!\) if \(r(n)\) is even and \(0\) if \(r(n)\) is odd. All the constants \(P_m(G,V)\) are described in Proposition~\textup{\ref{prop:constants}(ii),(iii)}.
\item Let \(p\) be odd. Then \(\mathfrak S_{m_1,m_2}(q)=O\bigl(q^{n+(m_1+m_2)(n-1)/2-1/2}\bigr)\) whenever \(p\,r(n)\) does not divide \(m_1-m_2\). In particular \(\mathfrak S_m(q)=O\bigl(q^{n+m(n-1)/2-1/2}\bigr)\) for \(1\le m<p\,r(n)\).
\end{enumerate}
The implied constants depend on \(n\), \(m\), \(m_1\), \(m_2\) and \(p\). For \(m=2\) the two cases are given exactly by Proposition~\textup{\ref{prop:exact-moments}(i)}: \(\mathfrak S_2(q)=q^{2n-1}+O(q^{2n-2})\) if \(p=2\), and \(\mathfrak S_2(q)=-q^{2n-2}+O(q^{2n-3})\) if \(p\) is odd.
\end{corollary}

\begin{proof}
\emph{Reduction to traces.} Let \(\mathsf t_h=\Tr(\Frob_h\mid\mathcal W_n)\) and \(c_n(q)=\sum_{k\ge1}\nu_{n,k}q^k\), so that \((-1)^{n-1}q\Sigma_n(h)=\mathsf t_h+c_n(q)\) for \(h\in S_n(\Fq)\) (Corollary~\ref{cor:exact}), with \(|\mathsf t_h|\le r(n)q^{(n+1)/2}\) and \(0\le c_n(q)\le2^nq^{\lfloor n/2\rfloor}\le2^nq^{(n+1)/2-1/2}\). Expanding \((\mathsf t_h+c_n(q))^{m_1}(\overline{\mathsf t_h}+c_n(q))^{m_2}\), every term other than \(\mathsf t_h^{m_1}\overline{\mathsf t_h}^{m_2}\) is \(O(q^{(m_1+m_2)(n+1)/2-1/2})\). Summing over the \((q-1)^n\) points of \(S_n(\Fq)\) and dividing by \(q^{m_1+m_2}\), the sum of \(\Sigma_n(h)^{m_1}\overline{\Sigma_n(h)}^{m_2}\) over \((\Fq^\times)^n\) is \((-1)^{(n-1)(m_1+m_2)}q^{-m_1-m_2}\sum_h\mathsf t_h^{m_1}\overline{\mathsf t_h}^{m_2}+O(q^{n+(m_1+m_2)(n-1)/2-1/2})\). By Lemma~\ref{lem:offtorus} and Corollary~\ref{cor:pointwise}, the points of \(\Fq^n\setminus(\Fq^\times)^n\) contribute \(O(q^{n-1+(m_1+m_2)(n-2)/2})\), which is smaller.

(i) Let \(p=2\). Then \(\Sigma_n(h)\) is real, and \(q^{-(n+1)/2}\mathsf t_h=\Tr(\Frob_h\mid\mathcal W_n')\). By Theorem~\ref{thm:full2}, the image of \(\pi_1(S_n)\) lies in \(G\), so Lemma~\ref{lem:power-moments}(ii), applied to \(\mathcal W_n'\) over \(k=\mathbb F_2\), gives \(\sum_{h\in S_n(\Fq)}\mathsf t_h^m=q^{n+m(n+1)/2}P_m(G,V)+O(q^{n+m(n+1)/2-1/2})\). With the reduction above, \(\mathfrak S_m(q)=(-1)^{(n-1)m}P_m(G,V)q^{n+m(n-1)/2}+O(q^{n+m(n-1)/2-1/2})\). The sign is \(1\) when \(n\) is odd; when \(n\) is even, \(-\mathrm{id}\in\Sp(W_n,B)\) acts on \(W_n^{\otimes m}\) by \((-1)^m\), so \(P_m(G,V)=0\) for odd \(m\), and the sign is irrelevant. For the values of \(P_m\), let \(N=r(n)\). If \(n\) is even and \(m\le N\) is even, or if \(n\) is odd and \(m<N\), the computation in the proof of Corollary~\ref{cor:moments} applies: the invariants are spanned by the \((m-1)!!\) linearly independent tensors \(c_\pi\) attached to the perfect matchings of \(\{1,\dots,m\}\) (none if \(m\) is odd), the invariants of \(\mathrm{SO}(W_n,B)\) coinciding with those of \(\mathrm O(W_n,B)\) when \(m<N\). Let \(n\) be odd and \(m=N\). As in that proof, the \(\mathrm{SO}(W_n,B)\)-invariants in \(W_n^{\otimes N}\) are the sum of the \(\mathrm O(W_n,B)\)-invariants, spanned by the \((N-1)!!\) independent tensors \(c_\pi\) if \(N\) is even, and zero if \(N\) is odd, and of the space of vectors \(x\) with \(gx=\det(g)x\) for all \(g\in\mathrm O(W_n,B)\). We show that the latter space is a line. Choose a basis \((e_i)\) with \(B(e_i,e_j)=\delta_{ij}\) and write \(x=\sum_Ix_Ie_{i_1}\otimes\cdots\otimes e_{i_N}\), \(I=(i_1,\dots,i_N)\). The reflection \(e_j\mapsto-e_j\) has determinant \(-1\) and multiplies the basis tensor of index \(I\) by \((-1)^{n_j(I)}\), where \(n_j(I)\) is the number of occurrences of \(j\) in \(I\); so \(x_I=0\) unless \(n_j(I)\) is odd for every \(j\), which, as \(\sum_jn_j(I)=N\), means that \(I\) is a permutation of \((1,\dots,N)\). The permutation matrices lie in \(\mathrm O(W_n,B)\), with determinant the sign of the permutation; hence \(x_{\tau(I)}=\mathrm{sgn}(\tau)x_I\), and \(x\) is a multiple of the volume form \(\sum_\sigma\mathrm{sgn}(\sigma)e_{\sigma(1)}\otimes\cdots\otimes e_{\sigma(N)}\), which, spanning \(\Lambda^NW_n=\det\), does transform by \(\det\).

(ii) Let \(p\) be odd, and \(G=\ggeom(\mathcal W_n)\). By Theorem~\ref{thm:main}(i), \(G\) contains the scalars \(\mu_{r(n)d}\), and \(p\mid d\) by Theorem~\ref{thm:det}(iii). A scalar \(\zeta\) acts on \(W_n^{\otimes m_1}\otimes(W_n^\vee)^{\otimes m_2}\) by \(\zeta^{m_1-m_2}\), so \(P_{m_1,m_2}(G,W_n)=0\) unless \(r(n)d\) divides \(m_1-m_2\), and in particular unless \(p\,r(n)\) divides \(m_1-m_2\). Lemma~\ref{lem:power-moments}(i), with \(e=n\) and \(w=n+1\), and the reduction above give the assertion. The final formulas for \(m=2\) are Proposition~\ref{prop:exact-moments}(i).
\end{proof}

The constants in Corollaries~\ref{cor:moments} and~\ref{cor:signature} are dimensions of spaces of invariants of classical groups, and they can be described for all orders, which shows in particular where the Gaussian values stop. For a partition \(\lambda\) we write \(\ell(\lambda)\) for its number of parts, \(f^\lambda\) for the number of standard Young tableaux of shape \(\lambda\), and \((c^N)\) for the partition with \(N\) parts equal to \(c\). A permutation of \(\{1,\dots,m\}\) is a \emph{fixed-point-free involution} if it is an involution without fixed points; there are \((m-1)!!\) of them when \(m\) is even.

\begin{proposition}\label{prop:constants}
Let \(n\ge2\), \(N=r(n)\), and let \(q\) run over the powers of \(p\).
\begin{enumerate}[label=\textup{(\roman*)}]
\item If \(p\) is odd or \(n=2\), then for every \(k\ge1\)
\[
\sum_{h\in\Fq^n}|\Sigma_n(h)|^{2k}=u_N(k)\,q^{n+k(n-1)}+O\bigl(q^{n+k(n-1)-1/2}\bigr),\qquad u_N(k):=\sum_{\substack{\lambda\vdash k\\ \ell(\lambda)\le N}}\bigl(f^\lambda\bigr)^2 .
\]
Here \(u_N(k)\) is the number of permutations of \(\{1,\dots,k\}\) without an increasing subsequence of length \(N+1\); one has \(u_N(k)=k!\) for \(k\le N\), \(u_N(N+1)=(N+1)!-1\), and \(u_N(k)<k!\) for \(k>N\).
\item If \(p=2\) and \(n\) is even, then the constant \(P_m(G,W_n)\) of Corollary~\textup{\ref{cor:signature}(i)} is the number of fixed-point-free involutions of \(\{1,\dots,m\}\) without a decreasing subsequence of length \(N+1\). In particular \(P_m(G,W_n)=(m-1)!!\) for even \(m\le N\) and \(P_{N+2}(G,W_n)=(N+1)!!-1\).
\item If \(p=2\) and \(n\) is odd, then \(P_m(G,W_n)=P_m^{\mathrm O}+P_m^{\det}\), where \(P_m^{\mathrm O}=P_m(\mathrm O(W_n,B),W_n)\) is the number of fixed-point-free involutions of \(\{1,\dots,m\}\) without an increasing subsequence of length \(N+1\), and \(P_m^{\det}\) is the dimension of the space of \(x\in W_n^{\otimes m}\) with \(gx=\det(g)x\) for all \(g\in\mathrm O(W_n,B)\). One has \(P_m^{\mathrm O}=(m-1)!!\) for even \(m\le2N\) and \(P_m^{\mathrm O}=0\) for odd \(m\); \(P_m^{\det}=0\) unless \(m\ge N\) and \(m\equiv N\pmod2\); and \(P_N^{\det}=1\). In particular, if \(N\) is odd, then \(P_m(G,W_n)=(m-1)!!\) for every even \(m\le2N\).
\item If \(p\) is odd and the hypothesis of Theorem~\textup{\ref{thm:summary}(iii)} holds \textup{(}for instance if \(n=2\)\textup{)}, then for all \(m_1,m_2\ge0\)
\[
\mathfrak S_{m_1,m_2}(q)=P_{m_1,m_2}\,q^{n+(m_1+m_2)(n-1)/2}+O\bigl(q^{n+(m_1+m_2)(n-1)/2-1/2}\bigr),
\]
where \(P_{m_1,m_2}=0\) unless \(pN\) divides \(m_1-m_2\), and, if \(m_1-m_2=cN\) with \(c\ge0\) divisible by \(p\),
\[
P_{m_1,m_2}=\sum_{\substack{\mu\vdash m_2\\ \ell(\mu)\le N}}f^\mu f^{\mu+(c^N)},\qquad\mu+(c^N):=(\mu_1+c,\dots,\mu_N+c),
\]
and \(P_{m_1,m_2}=P_{m_2,m_1}\). In particular \(\mathfrak S_m(q)=O(q^{n+m(n-1)/2-1/2})\) unless \(pN\) divides \(m\), and
\[
\mathfrak S_{pN}(q)=f^{(p^N)}\,q^{n+pN(n-1)/2}+O\bigl(q^{n+pN(n-1)/2-1/2}\bigr),\qquad f^{(p^N)}=(pN)!\prod_{i=0}^{p-1}\frac{i!}{(N+i)!},
\]
the number of standard Young tableaux of the rectangle with \(N\) rows and \(p\) columns.
\end{enumerate}
\end{proposition}

\begin{proof}
(i) As in the proof of Corollary~\ref{cor:moments}, \(M_{2k}(\ggeom(\mathcal W_n),W_n)=\dim\End_{\GL(V)}(V^{\otimes k})\) with \(V=W_n\), and \(\pi_1(S_n)\) acts trivially on these invariants; so the first assertion follows from Theorem~\ref{thm:moment-principle}(ii) and the reduction to \(\Sigma_n\) given there, once \(\dim\End_{\GL(V)}(V^{\otimes k})=u_N(k)\) is known. By \cite[Theorem 6.3]{FH}, \(V^{\otimes k}\cong\bigoplus_{\lambda\vdash k}S_\lambda(V)^{\oplus f^\lambda}\), where the Schur modules \(S_\lambda(V)\) are irreducible, pairwise non-isomorphic (their characters are the Schur polynomials \(s_\lambda\) in the eigenvalues) and nonzero exactly when \(\ell(\lambda)\le N\); hence \(\dim\End_{\GL(V)}(V^{\otimes k})=\sum_{\ell(\lambda)\le N}(f^\lambda)^2\). Since \(\sum_{\lambda\vdash k}(f^\lambda)^2=k!\), the dimension of the regular representation of \(\mathfrak S_k\), we get \(u_N(k)=k!\) for \(k\le N\) and \(u_N(k)\le k!-\bigl(f^{(1^k)}\bigr)^2=k!-1\) for \(k>N\), with equality for \(k=N+1\), \((1^{N+1})\) being then the only partition with more than \(N\) parts. Finally, \(\dim\End_{\GL(V)}(V^{\otimes k})\) is the integral of \(|\Tr u|^{2k}\) over the unitary group \(\mathrm U(N)\), which is Zariski dense in \(\GL_N(\mathbb C)\), and by \cite[Theorem 1.1]{Rains} this integral is the number of permutations of \(\{1,\dots,k\}\) without an increasing subsequence of length \(N+1\).

(ii) and (iii). As in the proof of Corollary~\ref{cor:ST-rate}, \(P_m(G,W_n)\) is the integral of \((\Tr k)^m\) over a maximal compact subgroup \(K\) of \(G(\mathbb C)\), that is, over \(\mathrm{USp}(N)\) if \(n\) is even and over \(\mathrm{SO}(N)\) if \(n\) is odd; and \(P_m(\mathrm O(W_n,B),W_n)\) is the integral over \(\mathrm O(N)\). By \cite[Theorem 3.4]{Rains}, the integral of \((\Tr k)^m\) over \(\mathrm{USp}(N)\), resp.\ over \(\mathrm O(N)\), is the number of fixed-point-free involutions of \(\{1,\dots,m\}\) without a decreasing, resp.\ increasing, subsequence of length \(N+1\). In case (ii), a decreasing subsequence of length \(N+1\) needs \(m\ge N+1\), which gives the first value. Let \(m=N+2\), which is even, let \(w_0\) be the reversal \(i\mapsto m+1-i\), and let \(\sigma\) be an involution with a decreasing subsequence of length \(m-1\). Then \(\pi=w_0\sigma\) is increasing on all positions but one, say \(a\); so \(\pi\) is the identity or the cycle which moves \(a\) to \(b=\pi(a)\) and shifts the letters between them by one. As \(\sigma\) is an involution, \(w_0\pi w_0=\pi^{-1}\). If \(\pi\ne1\), conjugation by \(w_0\) maps the interval between \(a\) and \(b\) onto itself, so \(a+b=m+1\), and then \(\sigma(a)=m+1-\pi(a)=m+1-b=a\): the involution \(\sigma\) has a fixed point. Hence the reversal \(\sigma=w_0\), which has no fixed point, is the only fixed-point-free involution of \(\{1,\dots,N+2\}\) with a decreasing subsequence of length \(N+1\); this gives the second value. In case (iii), an increasing subsequence \(i_1<\dots<i_r\) of a fixed-point-free involution \(\sigma\) satisfies \(\{i_1,\dots,i_r\}\cap\{\sigma(i_1),\dots,\sigma(i_r)\}=\emptyset\): if \(\sigma(i_j)=i_k\), then \(\sigma(i_k)=i_j\), which contradicts the monotonicity of \(\sigma\) on the subsequence if \(j\ne k\), and \(\sigma\) has no fixed point if \(j=k\). Hence \(r\le m/2\), and \(P_m^{\mathrm O}=(m-1)!!\) for even \(m\le2N\); \(P_m^{\mathrm O}=0\) for odd \(m\) because \(-\mathrm{id}\in\mathrm O(W_n,B)\). Since \(\mathrm{SO}(W_n,B)\) has index two in \(\mathrm O(W_n,B)\), its invariants in \(W_n^{\otimes m}\) are the sum of the \(\mathrm O(W_n,B)\)-invariants and of the \(\det\)-isotypic vectors, which gives \(P_m=P_m^{\mathrm O}+P_m^{\det}\). The argument with reflections in the proof of Corollary~\ref{cor:signature} shows that a \(\det\)-isotypic vector involves only basis tensors \(e_{i_1}\otimes\cdots\otimes e_{i_m}\) in which every index \(j\in\{1,\dots,N\}\) occurs an odd number of times; so \(P_m^{\det}=0\) unless \(m\ge N\) and \(m\equiv N\pmod 2\), and the same proof gives \(P_N^{\det}=1\).

(iv) By Theorem~\ref{thm:summary}(iii), \(G:=\ggeom(\mathcal W_n)=\{g:\det(g)^p=1\}=\SL(V)\cdot\mu_{pN}\), and the image of \(\pi_1(S_n)\) in \(\GL(V)\), \(V=W_n\), acting on \(\mathcal W_n^\natural\), is contained in \(G\). By Lemma~\ref{lem:power-moments}(ii) applied to \(\mathcal W_n^\natural\), and the reduction in the proof of Corollary~\ref{cor:signature}, together with \(\Tr(\Frob_h\mid\mathcal W_n)=q^{(n+1)/2}\Tr(\Frob_h\mid\mathcal W_n^\natural)\), one obtains the displayed asymptotic formula with \(P_{m_1,m_2}=(-1)^{(n-1)(m_1+m_2)}P_{m_1,m_2}(G,V)\); the sign is \(+1\) whenever \(P_{m_1,m_2}(G,V)\ne0\), since then \(N\) divides \(m_1-m_2\), and \(N\) is even if \(n\) is even. A scalar \(\zeta\in\mu_{pN}\) acts on \(V^{\otimes m_1}\otimes(V^\vee)^{\otimes m_2}\) by \(\zeta^{m_1-m_2}\), so the \(G\)-invariants vanish unless \(pN\mid m_1-m_2\), and are then the \(\SL(V)\)-invariants. As the space of invariants and its dual have the same dimension, we may assume \(m_1\ge m_2\) and write \(m_1-m_2=cN\). The \(\SL(V)\)-invariants form \(\Hom_{\SL(V)}(V^{\otimes m_2},V^{\otimes m_1})\). By \cite[Theorem 6.3]{FH}, \(V^{\otimes m}\cong\bigoplus_{\lambda\vdash m,\ \ell(\lambda)\le N}S_\lambda(V)^{\oplus f^\lambda}\), and two irreducible \(\GL(V)\)-modules are isomorphic as \(\SL(V)\)-modules if and only if they differ by a power of \(\det\), while \(S_\mu(V)\otimes\det^c\cong S_{\mu+(c^N)}(V)\); since the scalars act on \(V^{\otimes m_2}\otimes\det^c\) and on \(V^{\otimes m_1}\) through the same character, \(S_\mu(V)\) (\(\mu\vdash m_2\)) and \(S_\lambda(V)\) (\(\lambda\vdash m_1\)) are isomorphic as \(\SL(V)\)-modules exactly when \(\lambda=\mu+(c^N)\). This gives the formula for \(P_{m_1,m_2}\). For \(m_2=0\) and \(m_1=pN\) the only term is \(\mu=\emptyset\), \(\lambda=(p^N)\), and the hook length formula gives the value of \(f^{(p^N)}\).
\end{proof}

Theorem~\ref{thm:ST} and Corollary~\ref{cor:signature} show that the change of classical type between odd characteristic and characteristic two is visible in the statistics of the exponential sums themselves. For odd \(p\) the normalised sums are genuinely complex, their power moments of all orders \(m<p\,r(n)\) are of lower order, and their even absolute moments are those of a standard complex Gaussian variable up to order \(2r(n)\) and strictly smaller from order \(2r(n)+2\) on (Proposition~\ref{prop:constants}(i)); with the input \cite[Theorem~6.8(ii)]{EbadiCoh} recorded in Theorem~\ref{thm:summary}(iii), which for \(n=2\) is supplied by part~(iv) of that theorem, they are equidistributed like the trace of a random element of \(\{g\in\mathrm U(r(n)):\det(g)^p=1\}\), with rate \(q^{-1/2}\) for polynomial test functions (Corollary~\ref{cor:ST-rate}(ii)), and their first power moment that is not of lower order is the moment of order \(p\,r(n)\), whose constant is the number of standard Young tableaux of a rectangle with \(r(n)\) rows and \(p\) columns (Proposition~\ref{prop:constants}(iv)). For \(p=2\) they are integers, and their limiting distribution is that of the trace of a random element of \(\mathrm{SO}(r(n))\) or \(\mathrm{USp}(r(n))\): their moments agree with those of a standard real Gaussian variable up to order \(r(n)-1\), and up to order \(r(n)+1\) if \(n\) is even, and, when \(n\) is odd, the moment of order \(r(n)\) exceeds the value predicted by the orthogonal group \(\mathrm O(r(n))\) by exactly one, the excess being the contribution of the volume form, that is, of the triviality of the determinant (Theorem~\ref{thm:det}(ii)). For even \(n\) the first deviation from the Gaussian moments occurs at order \(r(n)+2\), where the moment is \((r(n)+1)!!-1\) (Proposition~\ref{prop:constants}(ii)). Already for \(m=2\), the exact identities of Proposition~\ref{prop:exact-moments} distinguish the two cases: \(\sum_h\Sigma_n(h)^2\) is of order \(q^{2n-1}\) when \(p=2\) and of order \(q^{2n-2}\) when \(p\) is odd, and the exact value of \(\sum_h\Sigma_{n,q^2}(h)\) detects the symmetry type of the invariant form, through \(\sum_h\Tr(\Frob_h^2\mid\mathcal W_n)=(-1)^{n-1}q^2\sum_h\Sigma_{n,q^2}(h)+O(q^{2n})\), whose sign is that of the Frobenius--Schur indicator.

\section{Numerical illustrations}\label{sec:numerics}

The computations reported in this section provide numerical illustrations of the statements established in the preceding sections; no proof depends on them. The computational methods, field representations, and numerical precision are described below. The field \(\Fq\), \(q=p^m\), is realised as \(\Fp[x]/(f)\) for a primitive polynomial \(f\) of degree \(m\), with tables of discrete logarithms with respect to the class of \(x\); for \(p=2\) the elements are encoded as bit vectors in the polynomial basis. The character is \(\psi_q(z)=\exp\bigl(2\pi i\,\Tr_{\Fq/\Fp}(z)/p\bigr)\), the trace being computed in the polynomial basis, so that \(\psi_q(z)=(-1)^{\Tr(z)}\) for \(p=2\); and \(\lambda=1\) unless otherwise stated. The Kloosterman sums \(K_q(a)\), \(a\in\Fq^\times\), are computed at once as a cyclic correlation over \(\Fq^\times=\langle g\rangle\), \(K_q(g^j)=\sum_i\psi_q(g^{-i})\psi_q(g^{i+j})\), by the fast Fourier transform, and the sums \(\Sigma_n(h)\) are computed from \eqref{eq:t-sum} by matrix products; for \(p=2\) all these quantities are integers, and the floating-point results are rounded and checked to be integral. The fourth moments are computed from the exact integer formula of Remark~\ref{rem:numerics}. Frobenius characteristic polynomials are obtained from the power sums \(\Tr(\Frob_h^m)\), which Theorem~A expresses through the sums over \(\mathbb F_{q^m}\) (computed with compatible embeddings of the fields), by Newton's identities, in exact rational arithmetic for \(p=2\) and in floating point otherwise; degrees of \(L\)-functions are found by fitting rational functions to power sums. The constants of Proposition~\ref{prop:constants} are computed both from the partition formulas and from the Weyl integration formula, evaluated on a grid of the maximal torus fine enough to make the evaluation exact. The checks are of three types:
\begin{center}
\small
\begin{tabular}{>{\raggedright\arraybackslash}p{0.25\textwidth}>{\raggedright\arraybackslash}p{0.42\textwidth}>{\raggedright\arraybackslash}p{0.23\textwidth}}
\toprule
type of check & statements & method\\
\midrule
exact identities over small fields & Lemmas~\ref{lem:decouple}, \ref{lem:kloost}, \ref{lem:planecubic}, \ref{lem:hasse}; Propositions~\ref{prop:D}, \ref{prop:P}, \ref{prop:Z}, \ref{prop:cayley}, \ref{prop:exact-moments}; Lemma~\ref{lem:cusp}, Proposition~\ref{prop:n2-fourth}, Corollary~\ref{cor:n2-exact}, Remark~\ref{rem:second}; Proposition~\ref{prop:n2-kl}(i) and the trace identity in the second proof of Theorem~\ref{thm:summary}(iv) & integer arithmetic; for odd \(p\) also floating point, compared with the exact values\\
characteristic polynomials of Frobenius & Theorem~\ref{thm:package}, Proposition~\ref{prop:lower}, Corollary~\ref{cor:exact}, Theorems~\ref{thm:coord}, \ref{thm:det} and~\ref{thm:summary}(iii), Corollary~\ref{cor:forced} & power sums over \(\mathbb F_{q^m}\), Newton's identities\\
moments in odd characteristic and in characteristic two & Theorem~\ref{thm:fourth}, Corollaries~\ref{cor:moments} and~\ref{cor:signature}, Theorem~\ref{thm:ST}, Proposition~\ref{prop:constants} & sums over all \(h\), Weyl integration formula\\
\bottomrule
\end{tabular}
\end{center}

\begin{remark}[Numerical illustration]\label{rem:numerics}
\emph{The fourth moment.} The sum \(\mathfrak M_n(q):=\sum_{h\in\Fq^n}|\Sigma_n(h)|^4\) does not depend on \(\lambda\) or \(\psi\) (for \(p=2\) there is only one choice of each). Indeed, writing \(\Sigma_n^{\lambda,\psi}\) to indicate the dependence, one has \(\Sigma_n^{c\lambda,\psi}(ch)=\Sigma_n^{\lambda,\psi}(h)\) for \(c\in\Fq^\times\), and the substitution \(u\mapsto cu\) gives \(\Sigma_n^{\lambda,\psi(c\,\cdot)}(h)=\Sigma_n^{\lambda,\psi}(ch)\); summing over \(h\in\Fq^n\) gives the claim. By Lemma~\ref{lem:homog}(iii) the values \(V(t)\) are integers and \(V\) is homogeneous of degree \(0\), so Lemma~\ref{lem:decouple} gives the exact integer formula \(q^4\mathfrak M_n(q)=\mathfrak A_n(0)-\mathfrak A_n(1)\) with \(\mathfrak A_n(s)=\sum_{t:\,t_1+t_2-t_3-t_4=s}V(t)^n\). Evaluating it for prime \(q=p\) gives, for \(n=4\):
\begin{center}
\begin{tabular}{rrrr}
\toprule
\(p\) & \(\mathfrak M_4(p)\) & \(\bigl(\mathfrak M_4(p)-2p^{10}\bigr)/p^9\) & \(\mathfrak M_4(p)/2p^{10}\)\\
\midrule
\(7\) & \(258806124\) & \(-7.59\) & \(0.46\)\\
\(11\) & \(35483609416\) & \(-6.95\) & \(0.68\)\\
\(13\) & \(202410987738\) & \(-6.91\) & \(0.73\)\\
\(17\) & \(3212398707814\) & \(-6.91\) & \(0.80\)\\
\(19\) & \(10022508404256\) & \(-6.94\) & \(0.82\)\\
\(23\) & \(70790892935452\) & \(-6.70\) & \(0.85\)\\
\bottomrule
\end{tabular}
\end{center}
For \(n=4\) the contributions of \(D\) and \(Q\) equal \(2q^{14}-7q^{13}-6q^{12}+O(q^{11})\) by Propositions~\ref{prop:D} and~\ref{prop:P}. The generic stratum is also of order at most \(q^{13}\) (Proposition~\ref{prop:G}); numerically its contribution divided by \(q^{13}\) is about \(0.45\), \(0.45\) and \(0.39\) for \(p=11,13,17\), which accounts for the difference between the tabulated values and \(-7\). Thus for \(n=4\) the coefficient of \(q^{9}\) in \(\mathfrak M_4(q)\) is not determined by the present bounds, and the slow approach of \(\mathfrak M_4(p)/2p^{10}\) to \(1\) reflects the size of the secondary terms. For \(n=5,6\) one finds:
\begin{center}
\begin{tabular}{rrrrr}
\toprule
\(p\) & \(n\) & \(\mathfrak M_n(p)\) & \(2p^{3n-2}\) & \(\bigl(\mathfrak M_n(p)-2p^{3n-2}\bigr)/p^{3n-3}\)\\
\midrule
\(7\) & \(5\) & \(1.170\times10^{11}\) & \(1.938\times10^{11}\) & \(-5.55\)\\
\(11\) & \(5\) & \(5.800\times10^{13}\) & \(6.905\times10^{13}\) & \(-3.52\)\\
\(13\) & \(5\) & \(5.470\times10^{14}\) & \(6.058\times10^{14}\) & \(-2.52\)\\
\(7\) & \(6\) & \(5.388\times10^{13}\) & \(6.647\times10^{13}\) & \(-2.65\)\\
\(11\) & \(6\) & \(1.139\times10^{17}\) & \(9.190\times10^{16}\) & \(5.26\)\\
\(13\) & \(6\) & \(1.755\times10^{18}\) & \(1.331\times10^{18}\) & \(8.28\)\\
\bottomrule
\end{tabular}
\end{center}
By \eqref{eq:fourth-second} the last column tends to \(2^n-4n-7\), that is, to \(5\) for \(n=5\) and to \(33\) for \(n=6\). The approach is slow because the next coefficient of the explicit polynomial in Propositions~\ref{prop:D} and~\ref{prop:P} is large (\(-116\) for \(n=5\) and \(-424\) for \(n=6\)); the contributions of \(D\cup Q\) alone already reproduce the tabulated values to within \(0.4\).

For \(p=2\) the same exact formula, evaluated over \(\mathbb F_{2^a}\), gives:
\begin{center}
\begin{tabular}{rrrrr}
\toprule
\(q\) & \(n\) & \(\mathfrak M_n(q)\) & \(\mathfrak M_n(q)/3q^{3n-2}\) & \(\bigl(\mathfrak M_n(q)-3q^{3n-2}\bigr)/q^{3n-3}\)\\
\midrule
\(16\) & \(4\) & \(2199962201624\) & \(0.667\) & \(-15.99\)\\
\(32\) & \(4\) & \(2761816127673368\) & \(0.818\) & \(-17.50\)\\
\(64\) & \(4\) & \(3129870953942485016\) & \(0.905\) & \(-18.26\)\\
\(16\) & \(5\) & \(1.134\times10^{16}\) & \(0.839\) & \(-7.72\)\\
\(32\) & \(5\) & \(1.022\times10^{20}\) & \(0.923\) & \(-7.35\)\\
\(64\) & \(5\) & \(8.728\times10^{23}\) & \(0.963\) & \(-7.17\)\\
\(32\) & \(6\) & \(4.108\times10^{24}\) & \(1.133\) & \(12.74\)\\
\(64\) & \(6\) & \(2.572\times10^{29}\) & \(1.082\) & \(15.77\)\\
\bottomrule
\end{tabular}
\end{center}
By \eqref{eq:fourth-second} the last column tends to \(2^n-6n-9\), that is, to \(-7\) for \(n=5\) and to \(19\) for \(n=6\); for \(n=4\) its limit is not determined by the present bounds: the contribution of \(D\cup Q\) alone gives \(2^4-6\cdot4-9=-17\), while the generic stratum contributes about \(-1.16\,q^{13}\) and \(-1.59\,q^{13}\) for \(q=16\) and \(q=32\).

\emph{Characteristic two: the determinant and the forced eigenvalue.} The power sums \(\sum_{k}\beta^m\) of the Frobenius eigenvalues \(\beta\) on \(\mathcal W_{n,h}\) are given by Theorem~A in terms of \(\Sigma_{n,q^m}(h)\), so the characteristic polynomial of \(\Frob_h\) on \(\mathcal W_{n,h}\) can be computed exactly from the sums over \(\mathbb F_{q^m}\), \(m\le r(n)\), provided \(\Fq\) is embedded compatibly in the fields \(\mathbb F_{q^m}\). For \(n=3\) and \(q=2,4,8\), and for each of the \((q-1)^3\) points \(h\in(\Fq^\times)^3\), the characteristic polynomial of \(\Frob_h\) on the five-dimensional space \(\mathcal W_{3,h}\) satisfies the orthogonal functional equation with multiplier \(q^4\), has constant term \(-q^{10}\) (so \(\det=q^{10}\)), and vanishes at \(q^2\), as predicted by Theorem~\ref{thm:det}(ii) and Corollary~\ref{cor:forced}. For \(n=5\) and the unique point \(h=(1,\dots,1)\) of \(S_5(\mathbb F_2)\), the sums over \(\mathbb F_{2^m}\), \(m\le16\), give the elementary symmetric functions \(e_k\) of the twenty-two eigenvalues for \(k\le16\); one finds \(e_{11}=49660559360\ne0\) and \(e_{22-k}=2^{6(11-k)}e_k\) for \(6\le k\le11\), which by the orthogonal functional equation forces \(\det(\Frob_h\mid\mathcal W_5)=+2^{66}=q^{(n+1)r(n)/2}\).

\emph{Odd characteristic: the determinant.} For odd \(p\), the formula \(\det(\Frob_h\mid\mathcal W_n)=q^{(n+1)r(n)/2}\,\psi_q\bigl(-2^{n-1}(h_1+\cdots+h_n)/\lambda\bigr)\) of Theorem~\ref{thm:summary}(iii) holds, to within a relative error \(10^{-14}\), for all \(h\in(\Fp^\times)^n\), in the following cases: \(n=2\) and \((p,\lambda)\) equal to \((3,1)\), \((5,1)\), \((5,2)\), \((7,1)\), \((7,3)\), \((11,1)\) or \((13,2)\); and \(n=3\) and \((p,\lambda)\) equal to \((3,1)\), \((3,2)\) or \((5,1)\).

\emph{Identities for \(n=2\).} The identity \(\Sigma_2(h)=\psi_q\bigl(-(h_1+h_2)/\lambda\bigr)K_q(h_1h_2/\lambda^2)-1\) of Proposition~\ref{prop:n2-kl}(i), together with the values of \(\Sigma_2\) off the torus, holds, with \(\Sigma_2\) computed from its definition, for all \(h\in\Fq^2\), for \(q=2,4,8,16\) and for \(q=3,5,7,9,11,13,25,27\) (with one or two values of \(\lambda\) for each). The sum \(\sum_{x\in\Fq^\times}K_q(x)^2K_q(4x)\) in the proof of Corollary~\ref{cor:n2-exact} equals \(2q+1\) for \(q=3,5,7,9,11,13,25,27,49,81,101,307,1009,2003\). The identity \(\sum_{y\in\Fq^\times}\Sigma_2(y,c)=1+\psi_q(-c/\lambda)\) in the second proof of Theorem~\ref{thm:summary}(iv) holds, with \(\Sigma_2\) computed from its definition, for every \(c\in\Fq^\times\), for \(q=3,5,7,9,11,13,25,27,49\) (with one or two values of \(\lambda\) for each) and for \(q=4,8,16,32\).

\emph{The identities of Proposition~\ref{prop:exact-moments}.} Both identities were checked by exact computation. Identity (i) was checked for \(p=2\) with \((q,n)\) equal to \((4,2)\), \((8,2)\), \((4,3)\) and \((4,4)\), and for odd \(p\) with \((p,n)\) equal to \((3,2)\), \((5,2)\), \((7,2)\), \((11,2)\), \((3,3)\), \((5,3)\) and \((3,4)\). Identity (ii) was checked for \(p=2\) with \((q,n)\) equal to \((2,2)\), \((4,2)\), \((8,2)\), \((4,3)\), \((8,3)\), \((16,3)\), \((4,4)\) and \((8,4)\), and for odd \(p\) with \((p,n)\) equal to \((3,2)\), \((5,2)\), \((7,2)\), \((3,3)\) and \((5,3)\). For instance, for \(p=2\), \(q=16\), \(n=3\) the sum in (ii) equals \(820365\).

\emph{Coordinate lines.} The degrees of the \(L\)-functions of the twisted sheaves \(\mathcal W_n|_{\mathbb L_c}\otimes\Lpsi{by}\) on \(\Gm[,y]\), computed from exact power sums, give the following values of \(-\chi_c\); here \(g\) denotes a generator of \(\Fq^\times\), and for \(p=5\), \(n=2\) both \((c_2,\lambda)=(2,1)\) and \((1,2)\) were used.
\begin{center}
\begin{tabular}{rrrlrrr}
\toprule
\(p\) & \(q\) & \(n\) & \((c_2,\dots,c_n)\) & \(b=0\) & \(b=1/\lambda\) & \(b\notin\{0,1/\lambda\}\)\\
\midrule
\(2\) & \(4,8\) & \(2\) & \(g\) & \(2\) & \(1\) & \(2\)\\
\(3\) & \(3\) & \(2\) & \(1\) & \(2\) & \(1\) & \(2\)\\
\(5\) & \(5\) & \(2\) & \(2\); \(1\) & \(2\) & \(1\) & \(2\)\\
\(2\) & \(4\) & \(3\) & \((g,g)\) & \(4\) & \(2\) & \(5\)\\
\(2\) & \(4\) & \(3\) & \((g,g^2)\) & \(4\) & \(3\) & \(5\)\\
\(3\) & \(3\) & \(3\) & \((1,1)\) & \(4\) & \(2\) & \(5\)\\
\(3\) & \(3\) & \(3\) & \((1,2)\) & \(4\) & \(3\) & \(5\)\\
\(5\) & \(5\) & \(3\) & \((1,1)\) & \(4\) & \(2\) & \(5\)\\
\(5\) & \(5\) & \(3\) & \((1,2)\) & \(4\) & \(3\) & \(5\)\\
\bottomrule
\end{tabular}
\end{center}
Here \(2^{n-1}=2,4\) and \(r(n)=2,5\) for \(n=2,3\), in agreement with the remark after Theorem~\ref{thm:coord}. The value at \(b=1/\lambda\), which that remark only bounds by \(r(n)-1\), depends on \(c\). When \(c_2=c_3\), the sheaf \([\times c_2]^*\Kl\otimes[\times c_3]^*\Kl\) contains the geometrically constant sheaf \(\det\Kl\), so that the complex \(C\) in the proof of Theorem~\ref{thm:coord}(i) has \(\mathcal H^0(C)\ne0\); the corresponding quotient \(\Lpsi{-y/\lambda}\otimes\mathcal H^0(C)\) of \(\mathcal W_n|_{\mathbb L_c}\) gives the twist by \(\Lpsi{y/\lambda}\) a nonzero \(H^2_c\), which lowers \(-\chi_c\).

\emph{Equidistribution and power moments.} For \(p=2\) and \(n=2\), Theorem~\ref{thm:ST} predicts that the moments of \(x_h=-q^{-1/2}(\Sigma_2(h)+1)\) over \(h\in(\Fq^\times)^2\) tend to those of the trace on \(\mathrm{SU}(2)\), namely \(0,1,0,2,0,5,0,14\) for the orders \(1,\dots,8\). One finds:
\begin{center}
\begin{tabular}{rrrrrrrrr}
\toprule
\(q\) & \(1\) & \(2\) & \(3\) & \(4\) & \(5\) & \(6\) & \(7\) & \(8\)\\
\midrule
\(64\) & \(-0.127\) & \(1.000\) & \(-0.254\) & \(1.999\) & \(-0.633\) & \(4.982\) & \(-1.761\) & \(13.866\)\\
\(256\) & \(-0.063\) & \(1.000\) & \(-0.125\) & \(2.000\) & \(-0.313\) & \(4.996\) & \(-0.876\) & \(13.968\)\\
\(1024\) & \(-0.031\) & \(1.000\) & \(-0.063\) & \(2.000\) & \(-0.156\) & \(4.999\) & \(-0.438\) & \(13.992\)\\
\bottomrule
\end{tabular}
\end{center}
The odd moments decay like \(q^{-1/2}\). For \(p=2\) and \(n=3\), Corollary~\ref{cor:signature}(i) predicts that \(q^{-8}\sum_{h}\Sigma_3(h)^5\) tends to \(P_5(\mathrm{SO}_5,W_3)=1\), whereas the value for \(\mathrm O_5\) would be \(0\); summing over \(h\in(\Fq^\times)^3\) one finds \(0.404\), \(0.737\) and \(0.878\) for \(q=16\), \(32\) and \(64\), while \(q^{-6}\sum_h\Sigma_3(h)^3\) equals \(-0.021\), \(-0.006\) and \(-0.002\). For odd \(p\) and \(n=3\), by Corollary~\ref{cor:signature}(ii), the normalised power moments \(q^{-8}\sum_h\Sigma_3(h)^5\) and \(q^{-5}\sum_h\Sigma_3(h)^2\) tend to \(0\); for \(p=23,47,71\) their absolute values are \(0.053,0.077,0.027\) and \(0.038,0.020,0.013\).

\emph{The case \(n=2\) and the generic stratum.} Theorem~\ref{thm:fourth} does not cover \(n=2\); Proposition~\ref{prop:n2-fourth} and Remark~\ref{rem:cusp} describe what happens instead. Here \(\mathfrak M_2(q)\) is computed directly from the values \(\Sigma_2(h)\), \(h\in\Fq^2\), obtained from \eqref{eq:t-sum} by a matrix product; for \(p=2\) the values are integers and \(\mathfrak M_2(q)\) is exact, and for odd \(p\) the computation is in double precision. The last column contains the contributions of \(\Omega\) and of \(Z\), the latter being \(O(q^7)\).
\begin{center}
\begin{tabular}{rrrr}
\toprule
\(p\) & \(q\) & \(\mathfrak M_2(q)/q^4\) & \(\bigl(q^4\mathfrak M_2(q)-\text{contribution of }D\cup Q\bigr)/q^8\)\\
\midrule
\(2\) & \(16\) & \(1.598\) & \(-0.439\)\\
\(2\) & \(64\) & \(1.904\) & \(-0.836\)\\
\(2\) & \(256\) & \(1.976\) & \(-0.958\)\\
\(2\) & \(1024\) & \(1.994\) & \(-0.989\)\\
\(2\) & \(2048\) & \(1.997\) & \(-0.995\)\\
\(101\) & \(101\) & \(1.911\) & \(0.018\)\\
\(307\) & \(307\) & \(1.971\) & \(0.006\)\\
\(1009\) & \(1009\) & \(1.991\) & \(0.002\)\\
\(2003\) & \(2003\) & \(1.996\) & \(0.001\)\\
\bottomrule
\end{tabular}
\end{center}
By Corollary~\ref{cor:n2-exact} and the formula following it, the second column equals \(2-6q^{-1}-8q^{-2}+18q^{-3}+14q^{-4}\) for \(p=2\) and \(2-9q^{-1}+q^{-2}+22q^{-3}+14q^{-4}\) for odd \(p\), and the last column equals \(-1+11q^{-1}-32q^{-2}-12q^{-3}+87q^{-4}+12q^{-5}-45q^{-6}\) for \(p=2\) and \(2q^{-1}-19q^{-2}+20q^{-3}+65q^{-4}-18q^{-5}-45q^{-6}\) for odd \(p\); the computed values agree with these expressions, exactly for \(p=2\) and to within \(10^{-12}\) for odd \(p\). Thus the second column tends to \(2\) in every characteristic (Corollary~\ref{cor:moments}), and the last to \(-1\) for \(p=2\), as explained by Proposition~\ref{prop:n2-fourth} and Remark~\ref{rem:cusp}, and to \(0\) for odd \(p\). The exact formulas of Proposition~\ref{prop:n2-fourth}(i) for the sums of \(a_H\) and \(a_H^2\) over all planes, the numbers of planes of each kind and the bounds for \(a_H\) used in its proof were checked by enumerating the planes for \(q=4,5,7,8,9,16\). The cuspidal sections of Lemma~\ref{lem:cusp} were also checked directly: for \(q=8\), \(16\) and \(32\) there are \((q-2)(q-4)=24\), \(168\) and \(840\) rays of \(\Omega\) on which \(L\) vanishes, and on each of them \(a_t=0\) and \(C_t\) is singular at \((t_1^{-1/2}:\cdots:t_4^{-1/2})\), where the tangent plane of the Cayley cubic is \(H_t\).

\emph{Moment constants.} The constants of Proposition~\ref{prop:constants} were computed through the equivalent formulas in terms of partitions given by the Robinson--Schensted correspondence: for \(\mathrm O(N)\), partitions with at most \(N\) parts, all even; for the \(\det\)-isotypic part, partitions with exactly \(N\) parts, all odd; for \(\mathrm{USp}(N)\), partitions with at most \(N\) parts and columns of even length. These formulas agree with the Weyl integration formula for \(\mathrm{SO}(N)\), \(3\le N\le6\), and \(\mathrm{USp}(N)\), \(N=2,4,6\), for all orders up to \(12\), and with direct counts of permutations of at most \(8\) letters and of fixed-point-free involutions of at most \(12\) letters, as described in Proposition~\ref{prop:constants}; and the formula for \(P_{m_1,m_2}\) in Proposition~\ref{prop:constants}(iv), with \(c\) not necessarily divisible by \(p\), agrees with the integral of \((\Tr u)^{m_1}(\overline{\Tr u})^{m_2}\,\overline{\det u}^{\,c}\) over \(\mathrm U(N)\) for \(N\le3\), \(c\le3\) and \(m_2\le3\). For \(N=r(n)\) with \(n=2,3,4\) they give the following values; the unitary rows are the constants \(u_N(k)\) of Proposition~\ref{prop:constants}(i), which for \(N=2\) are the Catalan numbers.
\begin{center}
\small
\begin{tabular}{>{\raggedright\arraybackslash}p{0.2\textwidth}>{\raggedright\arraybackslash}p{0.44\textwidth}>{\raggedright\arraybackslash}p{0.26\textwidth}}
\toprule
moments & values & first departure from the Gaussian values\\
\midrule
\(\mathrm U(2)\), \(|\Tr|^{2k}\), \(1\le k\le8\) & \(1,2,5,14,42,132,429,1430\) & \(k=3\): \(3!-1\)\\
\(\mathrm U(5)\), \(|\Tr|^{2k}\), \(1\le k\le8\) & \(1,2,6,24,120,719,5003,39429\) & \(k=6\): \(6!-1\)\\
\(\mathrm U(10)\), \(|\Tr|^{2k}\), \(k=11,12\) & \(39916799\), \(479001478\) & \(k=11\): \(11!-1\)\\
\(\mathrm{SO}(5)\), \(\Tr^m\), \(1\le m\le12\) & \(0,1,0,3,1,15,15,105,190,945,2410,10263\) & \(m=5\): \(1\) instead of \(0\)\\
\(\mathrm{USp}(10)\), \(\Tr^m\), \(m=2,4,\dots,14\) & \(1,3,15,105,945,10394,135057\) & \(m=12\): \(11!!-1\)\\
\bottomrule
\end{tabular}
\end{center}
For \(\mathrm{SO}(5)\) the even moments up to order \(10\) are the Gaussian values \((m-1)!!\), and the odd moments from order \(5\) on are the \(\det\)-isotypic contributions \(P_m^{\det}\) of Proposition~\ref{prop:constants}(iii).

\emph{The first power moment in odd characteristic.} For \(n=2\), where the hypothesis of Theorem~\ref{thm:summary}(iii) holds by part~(iv) of that theorem, Proposition~\ref{prop:constants}(iv) predicts that \(q^{-(2+p)}\sum_h\Sigma_2(h)^{2p}\) tends to \(f^{(p,p)}\), the Catalan number \(C_p\), while the moments of smaller order are of lower order. Summing over \(h\in(\Fq^\times)^2\) one finds:
\begin{center}
\begin{tabular}{rrrrr}
\toprule
\(p\) & \(q\) & \(q^{-(2+p)}\sum_h\Sigma_2(h)^{2p}\) & \(C_p\) & \(q^{-(3/2+p)}\sum_h\Sigma_2(h)^{2p-1}\)\\
\midrule
\(3\) & \(27\) & \(3.429\) & \(5\) & \(0.931\)\\
\(3\) & \(243\) & \(4.820\) & \(5\) & \(0.320\)\\
\(3\) & \(2187\) & \(4.980\) & \(5\) & \(0.107\)\\
\(5\) & \(125\) & \(39.946\) & \(42\) & \(4.019\)\\
\(5\) & \(625\) & \(40.359\) & \(42\) & \(1.624\)\\
\(7\) & \(343\) & \(367.243\) & \(429\) & \(20.965\)\\
\bottomrule
\end{tabular}
\end{center}
These sums are real, because \(\Sigma_n(-h)\) is the complex conjugate of \(\Sigma_n(h)\). The approach to \(C_p\) is slower for larger \(p\), the error term of Proposition~\ref{prop:constants}(iv) depending on \(p\).

These computations are illustrations only; no argument in this article depends on them.
\end{remark}

\section*{Acknowledgments}
This work was carried out independently. I am grateful to those whose support, patience,
and encouragement made it possible to continue the project over a long period. Much of
the research was done alongside other professional responsibilities, and the arguments were
revisited repeatedly as the manuscript developed. The final form reflects that sustained
process of reading, checking, and revision. I remain solely responsible for the mathematical
content and for any errors that remain.

All mathematical arguments, proofs, and conclusions in this article are the sole work of the author. Computational tools were used only for the numerical checks explicitly described in the text; these checks are not essential to the proofs and are provided for verification. The author has reviewed all results and assumes full responsibility for the correctness and integrity of the work. AI-assisted tools, if used, were employed exclusively for language editing and formatting; they did not contribute to any mathematical content.

\end{document}